\documentclass[11pt]{amsart}

\usepackage[a4paper,margin=1.05in]{geometry}
\usepackage{amsmath,amssymb,amsthm,mathtools}
\usepackage{mathrsfs}
\usepackage{enumitem}
\usepackage{hyperref}
\usepackage{microtype}
\usepackage{tikz}
\usetikzlibrary{positioning,arrows.meta,calc}
\usepackage{longtable,array}
\usepackage[table]{xcolor}

\hypersetup{
  colorlinks=true,
  linkcolor=blue,
  citecolor=blue,
  urlcolor=blue
}

\numberwithin{equation}{section}
\theoremstyle{plain}
\newtheorem{theorem}{Theorem}[section]
\newtheorem{proposition}[theorem]{Proposition}
\newtheorem{lemma}[theorem]{Lemma}
\newtheorem{corollary}[theorem]{Corollary}

\theoremstyle{definition}
\newtheorem{definition}[theorem]{Definition}
\newtheorem{remark}[theorem]{Remark}

\newcommand{\C}{\mathbb{C}}
\newcommand{\R}{\mathbb{R}}
\newcommand{\N}{\mathbb{N}}
\newcommand{\D}{\mathbb{D}}
\newcommand{\Z}{\mathbb{Z}}
\newcommand{\T}{\mathbb{T}}
\newcommand{\dd}{\,d}

\newcommand{\im}{\mathrm{Im}}
\newcommand{\re}{\mathrm{Re}}

\newcommand{\norm}[1]{\left\|#1\right\|}

\newcommand{\al}{\alpha}

\newcommand{\Mp}{M_p}

\newcommand{\Fpq}{\mathcal{F}^{p,q}_{\al}}

\newcommand{\FpqA}[3]{\mathcal{F}^{#1,#2}_{#3}}

\title[Sharp Gaussian mixed-norm theory]{Sharp Gaussian Mixed-Norm Theory for Fock Spaces}

\author{Xiang Fang}
\address{National Yang Ming Chiao Tung University, Taiwan}
\email{xfang@nycu.edu.tw}

\author{Pham Trong Tien}
\address{VNU University of Science, Vietnam National University, Hanoi}
\email{phamtien@vnu.edu.vn}

\date{\today}


\begin{document}

\begin{abstract}
We develop a structural theory of the Gaussian mixed-norm Fock spaces
\(\mathcal F_{\alpha}^{p,q}\), where \(\alpha>0\) and
\(0<p,q\leq\infty\), which separate angular \(L^p\)-means from radial
Gaussian \(L^q\)-summability. We first prove the sharp circle-mean estimate
\begin{equation*}
M_p(f,r)
\leq
C_{\alpha,q}(1+r)^{-\frac{1}{q}}
e^{\frac{\alpha}{2} r^2}
\lVert f\rVert_{\mathcal F_{\alpha}^{p,q}},
\qquad r\geq 0,
\end{equation*}
where \(M_p(f,r)\) denotes the angular \(L^p\)-mean of \(f\)
on the circle \(\lvert z\rvert=r\)
with the convention \(\frac{1}{\infty}=0\), and show that the exponent
\(-\frac{1}{q}\) is optimal. The proof is based on a local mass principle for
convex Laplace-type exponents, which also yields an annular
discretization of the mixed Gaussian norm.

Using this discretization, we characterize the continuous embeddings
between mixed-norm Fock spaces and obtain an atomic decomposition in
terms of normalized Fock kernels, with coefficients belonging to a
corresponding annular mixed sequence space. We further characterize
Carleson and vanishing Carleson measures and identify the continuous
duals in the Banach, quasi-Banach, and endpoint regimes. For \(p<1\),
the atomic decomposition relies on a quantitative re-centering theorem
that controls expansions of normalized Fock kernels under perturbations
of their centers. This theorem is of independent interest even for the
Hilbert Fock space \(\mathcal F_{\alpha}^{2}\).
\end{abstract}

\maketitle

\tableofcontents


\section{Introduction and main results}\label{sec:intro}

\subsection{Motivation}
\label{subsec:intro-background}

The growth of integral means on circles is one of the classical quantitative
invariants of analytic function theory.  If \(f\) is
analytic in the unit disk \(\mathbb D\), then for \(0<p<\infty\) one writes
\[
M_p(f,r)
 :=
 \left(\frac1{2\pi}\int_0^{2\pi}|f(re^{i\theta})|^p\,\dd\theta\right)^{\frac{1}{p}},
 \qquad 0<r<1,
\]
with the usual supremum modification when \(p=\infty\).  The asymptotic
behavior of \(M_p(f,r)\) as \(r\to1^{-}\) has been a central part of classical
function theory for more than a century.  For instance, in the Hardy space
\(H^p\) one has the sharp estimate
\[
M_p(f,r)\le \|f\|_{H^p},\qquad 0<r<1,
\]
whereas in the unweighted Bergman space \(A^p\), \(0<p<\infty\), the sharp
circle-mean estimate is
\[
M_p(f,r)\lesssim (1-r)^{-\frac{1}{p}}\|f\|_{A^p},
\qquad 0<r<1.
\]
The precise mean-growth rate is one of the standard entry points to the
structure theory of analytic function spaces on the disk, including Hardy
spaces \cite{DurenHp,KoosisHp}, Bergman spaces \cite{DS04, HKZ00}, and related
spaces.

\smallskip
\noindent
For entire functions, the corresponding optimal circle-mean problem has a
different geometry and scale.  
The exact polynomial correction does not appear to have been made explicit even
for the classical Fock space \(\mathcal F_\alpha^p\). In that setting the Gaussian
weight determines the leading exponential growth, and the estimate commonly
used in the literature gives
\[
M_p(f,r)\lesssim
e^{\frac{\alpha}{2}r^2}\|f\|_{\mathcal F_\alpha^p},
\qquad r\geq 0.
\]
The exponential factor \(e^{\frac{\alpha}{2} r^2}\) is optimal; see, for instance,
\cite[Corollary~2.8]{Zhu}.

\smallskip
\noindent
The starting point of the present paper is that the
optimal circle-mean growth contains a further polynomial correction:
\[
M_p(f,r)\lesssim
(1+r)^{-\frac{1}{p}}e^{\frac{\alpha}{2}r^2}
\|f\|_{\mathcal F_\alpha^p},
\qquad r\ge 0,
\]
and this polynomial factor is sharp.

\smallskip
\noindent
Such a polynomial correction becomes  more striking when viewed within the mixed-norm Fock scale
\(\mathcal F_\alpha^{p,q}\), where \(\alpha>0\) and \(0<p,q\le\infty\).
In this   scale the   estimate becomes
\[
M_p(f,r)\lesssim
(1+r)^{-\frac{1}{q}}e^{\frac{\alpha}{2}r^2}
\|f\|_{\mathcal F_\alpha^{p,q}},
\qquad r\ge 0,
\]
with the convention \(\frac{1}{\infty}=0\).  Thus the correction is independent of \(p\).

\smallskip
\noindent
From the perspective of disk theory, where sharp mean-growth estimates are
structural, determining this rate in Gaussian spaces of entire functions is a
natural problem.
This identifies the correct scaling geometry associated with mean growth and, together
with the methods developed below, leads to the structural results proved in this
paper for mixed-norm Fock spaces.

\subsection{Background}

The theory of Banach and quasi-Banach spaces of analytic functions on the unit
disc has developed over more than a century into a rich and highly structured
subject, exemplified by the classical theories of Hardy, Bergman, and Dirichlet
spaces and by their many weighted and mixed-norm extensions.  For entire
functions, Gaussian Fock spaces provide a natural whole-plane counterpart.
Although a substantial theory has been developed, particularly in the Hilbert
setting, its sharp, quasi-Banach, and endpoint structure is not comparably
complete; in particular, the sharp circle-mean growth problem treated in this
paper is not supplied by the existing theory even on the classical diagonal
scale.  General background on Fock spaces may be found in Zhu's
monograph~\cite{Zhu}.  Through the Bargmann--Fock model, these spaces also
arise naturally in phase-space harmonic analysis and mathematical physics; see
Folland~\cite{Folland89}.

\smallskip
\noindent
The mixed-norm scale is intrinsic to the radial--angular structure of an entire
function.  For each radius \(r\), the exponent \(p\) measures the angular
distribution of \(f\) along the circle \(\lvert z\rvert=r\), whereas the exponent
\(q\) measures how the resulting circle means accumulate as \(r\) varies.  These
are geometrically different operations.  The classical diagonal space
\(\mathcal F_\alpha^p=\mathcal F_\alpha^{p,p}\) ties them to a single exponent,
while \(\mathcal F_\alpha^{p,q}\) allows them to vary independently.  This
separation is not merely formal.  One of the principal results of the present
paper shows that the optimal polynomial correction to Gaussian circle-mean
growth is
\((1+r)^{-\frac{1}{q}}\): it is governed by the outer radial exponent \(q\), rather
than by the angular exponent \(p\).  The mixed scale therefore detects
information that is invisible when one restricts attention to the diagonal
family.

\smallskip
\noindent
The general mixed-norm point of view goes back to Benedek and
Panzone~\cite{BP61}.  In analytic function theory, radial--angular mixed norms
have appeared in the study of Hardy, Bergman, Lipschitz, and related spaces;
see, for example, Flett~\cite{Flett72}, Pavlovi\'c~\cite{Pav86,Pav87},
Blasco~\cite{B95}, and Buckley~\cite{Buckley00}.  At the level of radial
aggregation, the endpoint \(q=\infty\) in the present scale is analogous to the
Hardy norm, since it takes a supremum of angular \(L^p\)-means over the radial
parameter; see Duren~\cite{DurenHp} for the classical Hardy theory.  For
\(0<q<\infty\), mixed-norm Bergman spaces provide the closest disc analogue at
the level of norm structure.

\smallskip
\noindent
The classical Bergman theory is extensively developed in the monographs of
Hedenmalm, Korenblum and Zhu~\cite{HKZ00} and Duren and
Schuster~\cite{DS04}.  In the mixed-norm Bergman setting, embeddings,
projections, duality, atomic decompositions, Carleson measures, derivative
characterizations, and Littlewood--Paley type formulas have been studied in
considerable depth.  Relevant contributions include the work of
Jevti\'c~\cite{Jev87}, Gadbois~\cite{Gad88}, Gu~\cite{Gue92},
Pavlovi\'c and Pel\'aez~\cite{PP08}, Pel\'aez, R\"atty\"a and
Sierra~\cite{PRS2019}, and Moreno and Pel\'aez~\cite{MP26}.  This body of
work demonstrates that separating angular integrability from radial
summability leads to structural questions in which the two exponents interact
in an essential way.

\smallskip
\noindent
The Fock setting is not, however, a routine whole-plane transcription of the
Bergman setting.  Bergman-space localization is adapted to the boundary and to
Bergman or pseudohyperbolic geometry.  Fock-space localization takes place at
infinity and is governed instead by Euclidean geometry, Gaussian weights, and
Fock translations.  The spaces considered here combine this fixed-scale local
Euclidean--Gaussian geometry with global \(\ell^q\)-type accumulation across
annuli.  A successful structural theory must respect both mechanisms
simultaneously.  In particular, estimates that are local in the Euclidean
metric must be assembled without losing the outer radial summability, while
quasi-Banach exponents prevent a direct reliance on convexity or ordinary
duality.

\smallskip
\noindent
Several aspects of mixed-norm Fock spaces have already appeared in the
literature.  Random analytic functions in Fock-type mixed norms were studied
in~\cite{FT23}; growth and boundedness questions for Hausdorff-type operators
appear in~\cite{BG24}; Constantin and Pel\'aez~\cite{CP16} investigated
Littlewood--Paley type phenomena; and Liu~\cite{Liu24} developed a
projection-theoretic treatment of Fock projections on mixed-norm spaces.  These
works exhibit the relevance of the scale, but they do not provide a unified
structural theory throughout the full range
\(0<p,q\leq\infty\).

\smallskip
\noindent
The purpose of the present paper is to construct such a framework.  Its two
basic mechanisms are a local mass principle for convex Laplace-type exponents
and a quantitative re-centering principle for normalized Fock kernels.  The
first yields both the sharp circle-mean estimate and an annular discretization
of the mixed Gaussian norm.  The second supplies the stability needed to build
kernel expansions in the angular quasi-Banach range.  Together these mechanisms
lead to the embedding theorem, an atomic decomposition with an annular mixed
sequence model, Carleson and vanishing Carleson measure criteria, and duality,
including the quasi-Banach and endpoint regimes.  Thus the results form a
single structural theory rather than a collection of parallel extensions, and
they provide the basic function-space infrastructure needed for a systematic
operator theory on Gaussian mixed-norm Fock spaces.

\subsection{Mixed-norm Fock spaces and standing notation}\label{subsec:intro-def}

We fix the definitions used throughout.
Fix \(\alpha>0\) and \(0<p,q\le\infty\). For an entire function \(f\) and
\(r \geq 0\), let
\[
\Mp(f,r):=
\begin{cases}
\displaystyle
\left(\frac1{2\pi}\int_0^{2\pi}|f(re^{i\theta})|^p\,\dd\theta\right)^{\frac1p},
& 0<p<\infty,\\[0.8em]
\displaystyle
\sup_{\theta\in[0,2\pi]} |f(re^{i\theta})|,
& p=\infty.
\end{cases}
\]
Define the mixed quasi-norm
\[
\|f\|_{\Fpq}:=
\begin{cases}
\displaystyle
\left(\int_0^\infty \Mp^q(f,r)\,\dd\lambda_{\alpha q}(r)\right)^{\frac1q},
& 0<q<\infty,\\[1em]
\displaystyle
\sup_{r\geq 0}\Mp(f,r)e^{-\frac{\alpha}{2}r^2},
& q=\infty,
\end{cases}
\]
where
\[
\dd\lambda_{\alpha q}(r):=\alpha q e^{-\frac{\alpha q}{2}r^2}r\,\dd r .
\]

\smallskip
\noindent
The mixed-norm Fock space \(\mathcal F_\alpha^{p,q}\) consists of  entire
functions \(f\) with \(\|f\|_{\Fpq}<\infty\).  It is Banach if \(p,q\ge1\) and
otherwise quasi-Banach.  For \(p=q\) it is the classical Fock space
\(\mathcal F_\alpha^p\).

\smallskip
\noindent
\textbf{The little endpoint space.}
We define the \emph{little mixed-norm Fock space}
\(f_\alpha^{p,\infty}\) as the space of all entire functions \(f\) satisfying
\[
\Mp(f,r)=o\!\left(e^{\frac{\alpha}{2}r^2}\right),
\qquad r\to\infty .
\]
Then \(f_\alpha^{p,\infty}\) is a closed subspace of
\(\mathcal F_\alpha^{p,\infty}\). Moreover, \(f_\alpha^{p,\infty}\) coincides
with the closure of the polynomials in \(\mathcal F_\alpha^{p,\infty}\)
(see Corollary~\ref{cor:density-poly} below). In particular, for \(p=\infty\)
we recover the little Fock space \(f_\alpha^\infty\); see \cite[p.~39]{Zhu}.

\smallskip
\noindent
\textbf{Fock kernels.}
We use the standard Fock kernels
\[
K_{\alpha,w}(z):=e^{\alpha\overline w z},
\qquad
\kappa_{\alpha,w}(z):=
e^{\alpha\overline w z-\frac{\alpha}{2}|w|^2},
\qquad z,w\in\C.
\]
Thus \(\kappa_{\alpha,w}\) is the normalized Fock kernel in every classical
Fock space \(\mathcal F_\alpha^p\).  We write
\[
\widetilde\kappa_{\alpha,w}^{\,p,q}
:=
\frac{\kappa_{\alpha,w}}
{\|\kappa_{\alpha,w}\|_{\mathcal F_\alpha^{p,q}}}
\]
for the mixed-norm-normalized Fock kernel in
\(\mathcal F_\alpha^{p,q}\).  When \(0<p,q\le\infty\) are fixed, we abbreviate
this to \(\widetilde\kappa_{\alpha,w}\).  When \(\alpha,p,q\) are fixed, we
write
\[
K_w:=K_{\alpha,w},
\qquad
\kappa_w:=\kappa_{\alpha,w},
\qquad
\widetilde\kappa_w:=\widetilde\kappa_{\alpha,w}.
\]

\smallskip
\noindent
We use the following notation and conventions.

\smallskip
\noindent
\textbf{Basic notation.}
We write
\[
\N_0:=\{0,1,2,\ldots\}.
\]
The complex plane is denoted by \(\C\), and \(\dd A\) denotes Lebesgue area
measure on \(\C\).  We write
\[
\T:=\{e^{i\theta}:0\leq \theta<2\pi\}
=\{z\in\C:|z|=1\}
\]
and regard \(L^p(\T)\) as taken with respect to the normalized angular measure
\(\dd\theta/(2\pi)\).  Functions on \(\T\) are identified with
\(2\pi\)-periodic functions of \(\theta\); in particular, we write
\(a(\theta)\) for \(a(e^{i\theta})\) when no confusion is possible.  For
\(z\in\C\) and \(R>0\), \(B(z,R)\) denotes the open Euclidean ball with center
\(z\) and radius \(R\).  The space \(H(\C)\) denotes the space of entire
functions, endowed with the compact-open topology.

\smallskip
\noindent
\textbf{Parameters and conjugate exponents.}
We use \(\alpha,\beta>0\) for Gaussian parameters and
\(p,q\in(0,\infty]\) for exponents.  In \(\mathcal F_\alpha^{p,q}\), the
exponent \(p\) refers to angular circle means and \(q\) to radial summability.
The same exponent notation is used for the classical spaces
\(\mathcal F_\alpha^p\) and for the sequence spaces \(\ell^{p,q}\) and
\(\ell_w^{p,q}\).  Indexed exponents, such as \(p_1,p_2,q_1,q_2\), are used
when comparing different spaces.  Throughout,
\[
\frac{1}{\infty}:=0.
\]

\smallskip
\noindent
For \(0<p\leq\infty\), set
\[
p^*:=
\begin{cases}
\infty, & 0<p\leq 1,\\[0.3em]
\dfrac{p}{p-1}, & 1<p<\infty,\\[0.8em]
1, & p=\infty.
\end{cases}
\]
Thus \(p^*\) is the usual H\"older conjugate when \(1<p<\infty\), with the
above convention used in the quasi-Banach and endpoint cases.

\smallskip
\noindent
\textbf{Annular notation.}
For \(k\in\N_0\), we write
\[
A_k:=\{z\in\C:k\le |z|<k+1\}
\]
for the \(k\)-th annular layer; when needed, we use the convention
\(A_{-1}:=\emptyset\).
When the Gaussian parameter \(\alpha>0\) is fixed, we also write
\[
B_p(f;k):=
\sup_{r\in[k,k+1)}
M_p(f,r)e^{-\frac{\alpha}{2}r^2},
\qquad k\in\N_0.
\]

\smallskip
\noindent
\textbf{Comparison notation.}
Throughout the paper, \(C\) denotes a positive constant whose value may change
from line to line.  For nonnegative quantities \(X\) and \(Y\), we write
\(X\lesssim Y\), or equivalently \(Y\gtrsim X\), if \(X\leq CY\), where \(C\)
is independent of the varying functions, sequences, measures, indices, and
points under consideration.  The constant may depend on fixed structural
parameters; any required uniformity in such parameters will be stated
explicitly.  We write \(X\asymp Y\) if both \(X\lesssim Y\) and
\(Y\lesssim X\) hold.

\subsection{Main results and proof highlights}\label{sec:intro-main}

We state the main results and indicate their proofs.  The local mass principle
gives the circle-mean estimate and the annular discretization.  These lead to
the embedding theorem.  The atomic decomposition gives an annular coefficient
model, which is the main input for the Carleson measure theory.  The duality
theorem is proved by kernel testing, using the same growth and annular
estimates.

\subsubsection{Sharp Gaussian circle-mean growth}
The first result is the sharp circle-mean estimate.

\begin{theorem} 
\label{thm:estimate}
Let \(\alpha>0\) and \(0<p,q\le\infty\).  Then, for every
\(f\in\mathcal F_\alpha^{p,q}\),
\begin{equation}
\label{eq-newest}
M_p(f,r)
\lesssim
(1+r)^{-\frac1q}e^{\frac{\alpha}{2}r^2}
\|f\|_{\mathcal F_\alpha^{p,q}},
\qquad r\ge0,
\end{equation}
where the implicit constant depends only on \(\alpha\) and \(q\).
\end{theorem}

\smallskip
\noindent
\textbf{Proof highlight.}
The proof applies the
local mass principle in Lemma~\ref{lem:uniform-local-mass-convex} to convex
Laplace-type exponents
\[
F(x)=\Phi(x)-ae^{2x}+bx.
\]
With
\[
\Phi(x)=q\log M_p(f,e^x),\qquad
F(x)=q\log M_p(f,e^x)-\frac{\alpha q}{2}e^{2x}+2x,
\]
this converts the global radial norm into the pointwise bound with factor
\((1+r)^{-\frac{1}{q}}\).  The case \(q=\infty\) follows directly from the definition.

\smallskip
\noindent
The same local mass principle yields the annular discretization of the mixed
norm, the other basic tool used throughout the paper. Sharpness is
proved by monomial tests in Lemma~\ref{lem-sh}.

\subsubsection{Embeddings in the Gaussian mixed-norm scale}

The embedding theorem is as follows.

\begin{theorem} 
\label{thm:embeddings}
Let \(0<\alpha,\beta<\infty\) and
\(0<p_1,p_2,q_1,q_2\leq\infty\). Then the inclusion
\(
\mathcal F_{\alpha}^{p_1,q_1}\subset \mathcal F_{\beta}^{p_2,q_2}
\)
is continuous if and only if either
\begin{enumerate}
\item[\textup{(i)}] \(\alpha<\beta\), or
\item[\textup{(ii)}] \(\alpha=\beta\), \(q_1\leq q_2\), and
\[
\frac{1}{q_2}-\frac{1}{p_2}
\leq
\frac{1}{q_1}-\frac{1}{p_1}.
\]
\end{enumerate}
\end{theorem}

\smallskip
\noindent
For \(p_1=q_1\) and \(p_2=q_2\), this reduces to the usual nesting condition in
the classical Fock scale; see \cite[Theorem~2.10]{Zhu}.

\smallskip
\noindent
\textbf{Proof highlight.}
The proof separates the three obstructions: the Gaussian parameter, the outer
radial exponent, and the angular--radial balance.  Sufficiency follows from the sharp
circle-mean growth theorem.  For necessity, monomials detect the Gaussian
parameter and the outer radial exponent, while reproducing kernels detect the
remaining mixed-norm condition.   

\smallskip
\noindent
Section~\ref{sec:embeddings} proves the theorem and derives a Littlewood-type
consequence for random analytic functions in \(\mathcal F_{\alpha}^{p,q}\).

\subsubsection{Atomic decomposition}

We use a \(\delta\)-pocket tiling \(\mathcal P_\delta\), adapted to both the
Euclidean localization of Fock kernels and the radial shell structure of the
mixed norm.  The coefficient space is the weighted mixed sequence space
\(\ell^{p,q}_w\), built from inner \(\ell^p\)-blocks on the annular layers
\[
A_k=\{z\in\C:k\le |z|<k+1\}
\]
and an outer \(\ell^q\)-summation across shells.  The natural weight is
\[
w_k=(1+k)^{\frac1q-\frac1p}.
\]
This weight reflects the annular structure of the norm.  The geometric and
sequence-space setup is given in Section~\ref{sec:atomic}.

\begin{theorem} 
\label{thm:AD-main}
Let \(\alpha>0\) and \(0<p,q\le\infty\).  Then there exists
\(\delta_0=\delta_0(\alpha,p,q)\in(0,\frac12]\) such that, for every
\(\delta\in(0,\delta_0)\) and every \(\delta\)-pocket tiling
\(\mathcal P_\delta\) of \(\C\) with centers \(\{\zeta_Q\}_{Q\in\mathcal P_\delta}\),
the following assertions hold, with all implicit constants depending only on
\(\alpha,p,q\), and \(\delta\). 

\begin{enumerate}
\item[\textup{(i)}] \emph{Synthesis.} The synthesis operator defined by
\[
S^{\ker}_\delta c(z):=
\sum_{Q\in\mathcal P_\delta} c_Q\kappa_{\zeta_Q}(z),
\qquad
c=(c_Q)_{Q\in\mathcal P_\delta},
\]
is a bounded linear operator from \(\ell_w^{p,q}\) to
\(\mathcal F_\alpha^{p,q}\), and
\[
\|S^{\ker}_\delta c\|_{\mathcal F_\alpha^{p,q}}
\lesssim \|c\|_{\ell_w^{p,q}},
\qquad c\in\ell_w^{p,q}.
\]

\item[\textup{(ii)}] \emph{Coefficient recovery and optimality.} Conversely, for every
\(f\in\mathcal F_\alpha^{p,q}\) there exists
\(c=(c_Q)_{Q\in\mathcal P_\delta}\in\ell_w^{p,q}\) such that
\(f=S^{\ker}_\delta c\), and moreover
\[
\|f\|_{\mathcal F_\alpha^{p,q}}
\, \asymp \, 
\inf\left\{
\|c\|_{\ell_w^{p,q}}:
c\in\ell_w^{p,q},\ f=S^{\ker}_\delta c
\right\}.
\]
\end{enumerate}
\end{theorem}

\smallskip
\noindent
In the classical one-parameter Fock scale, atomic decompositions use normalized
reproducing kernels indexed by a uniform lattice, with coefficients in a single
\(\ell^p\)-space; see \cite[Theorem~2.34]{Zhu}.  Here the atoms are organized by
\(\mathcal P_\delta\).  The coefficient space has two levels: an inner
\(\ell^p\)-norm in each shell and an outer \(\ell^q\)-summation across shells.
The weight
\(
w_k=(1+k)^{\frac{1}{q}-\frac{1}{p}}
\)
is forced by the mixed-norm size of normalized Fock kernels.  The shellwise
organization is analogous to mixed-norm Bergman decompositions
\cite{PRS2019}, but the geometry is different: disk-boundary annuli are
replaced by Euclidean annuli and localization is Gaussian.

\smallskip
\noindent
\textbf{Proof highlight.}
The main issue is to obtain an exact expansion by normalized Fock kernels at
the fixed pocket centers.  The \(\delta\)-pocket tiling from
Subsection~\ref{subsec:atomic-pockettiling} provides the sampling geometry, and
the annular discretization from Section~\ref{sec:mixed-fock-tools} identifies
the coefficient space; the remaining task is stability.
For \(1\leq p\leq\infty\), stability follows from a kernel
analysis--synthesis scheme with a small remainder, inverted by a Neumann series.
For \(0<p<1\), we first use finite-order jet atoms, then eliminate the jets into
floating kernels.  The re-centering envelope theorem in
Section~\ref{sec:recentering-envelope} re-expands these floating kernels at the
fixed centers \(\zeta_Q\), with uniform envelope bounds on the coefficients.
An iteration of the remaining small defect then gives the exact decomposition
in Theorem~\ref{thm:AD-main}.

\smallskip
\noindent
Section~\ref{sec:atomic} contains the proof.  The equivalent formulation with
mixed-norm-normalized kernels \(\widetilde\kappa_{\zeta_Q}\) and the unweighted
coefficient space \(\ell^{p,q}\) is given in
Subsection~\ref{subsec:atomic-consequences}.

\subsubsection{Carleson measures}
\label{subsec:carmeas}

We characterize the positive Borel measures \(\mu\) on \(\C\) for which
the natural embedding
\[
i_\mu:\mathcal F_{\alpha}^{p,q}\longrightarrow L^s_\alpha(\mu)
\]
is bounded or compact.  The atomic decomposition from
Section~\ref{sec:atomic} reduces the problem to weighted sequence embeddings
on the annular coefficient space.  The resulting criteria are expressed in
terms of localized ball masses.
For a positive Borel measure \(\mu\) on \(\C\) and \(0<s<\infty\), set
\[
L^s_\alpha(\mu)
:=
\left\{
f\ \text{measurable on }\C:
\|f\|_{L^s_\alpha(\mu)}^s
:=
\int_{\C}|f(z)|^s e^{-\frac{\alpha s}{2}|z|^2}\,d\mu(z)
<\infty
\right\}.
\]
We denote by
\[
i_\mu:\mathcal F_\alpha^{p,q}\to L^s_\alpha(\mu),
\qquad
i_\mu f=f,
\]
the natural embedding map.  We say that \(\mu\) is an \textit{\(s\)-Carleson measure
for \(\mathcal F_\alpha^{p,q}\)} if \(i_\mu\) is bounded, and that \(\mu\) is a
\textit{vanishing \(s\)-Carleson measure for \(\mathcal F_\alpha^{p,q}\)} if \(i_\mu\)
is compact.

\smallskip
\noindent
To formulate the characterizations, fix \(R>0\) and a \(\delta\)-pocket tiling
\(\mathcal P_\delta\) of \(\C\) with distinguished centers \(\zeta_Q\).  Recall
that the \(\delta\)-pocket tiling satisfies the uniform ball control
\[
B(\zeta_Q,c_{\mathrm{in}}\delta)\subset Q
\subset B(\zeta_Q,C_{\mathrm{out}}\delta),
\qquad Q\in\mathcal P_\delta,
\]
with \(c_{\mathrm{in}}>0\) and \(C_{\mathrm{out}}>0\) independent of \(\delta\). Using the normalized Fock kernel \(\kappa_\zeta=\kappa_{\alpha,\zeta}\),
we define the continuous localized ball-mass function \(U_{\mu,R}\) and the
corresponding discrete ball-mass sequence \(u_{\mu,R}\) by
\[
U_{\mu,R}(\zeta):=\|\kappa_\zeta\|_{\mathcal F_\alpha^{p,q}}^{-s}\,\mu(B(\zeta,R)),
\quad \zeta\in\mathbb C,
\qquad \text{and} \qquad 
u_{\mu,R}(Q):=U_{\mu,R}(\zeta_Q),
\quad Q\in\mathcal P_\delta.
\]
The discrete and continuous characterizations involve the generalized conjugate
exponents corresponding to \(\frac{p}{s}\) and \(\frac{q}{s}\).  More precisely, we use the generalized conjugates
\[
p_s^*:=\left(\frac ps\right)^*,
\qquad
q_s^*:=\left(\frac qs\right)^*.
\]
Equivalently,
\[
p_s^*
=
\begin{cases}
\infty, & 0<p\le s,\\[0.3em]
\dfrac{p}{p-s}, & s<p<\infty,\\[0.8em]
1, & p=\infty,
\end{cases}
\qquad
q_s^*
=
\begin{cases}
\infty, & 0<q\le s,\\[0.3em]
\dfrac{q}{q-s}, & s<q<\infty,\\[0.8em]
1, & q=\infty.
\end{cases}
\]
We refer to \(p_s^*\) and \(q_s^*\) as the inner and outer exponents,
respectively.  For a sequence \(c=(c_Q)_{Q\in\mathcal P_\delta}\), and each
\(k\ge0\), we write
\[
c_k:=\bigl(c_Q\bigr)_{\zeta_Q\in A_k}, \qquad
A_k:=\{z\in\C:k\le |z|<k+1\},
\]
for its \(k\)-th shell block and define
\[
\|c\|_{\ell^{p_s^*,q_s^*}}
:=
\left\|
\bigl(\|c_k\|_{\ell^{p_s^*}}\bigr)_{k\ge0}
\right\|_{\ell^{q_s^*}},
\]
with the usual supremum interpretation when \(p_s^*=\infty\) or
\(q_s^*=\infty\).  Similarly, we write
\(L_{\mathrm{sh}}^{p_s^*,q_s^*}(\C)\) for the space of all measurable functions
\(F\) on \(\C\) such that
\[
\|F\|_{L_{\mathrm{sh}}^{p_s^*,q_s^*}}
:=
\left\|
\bigl(\|F\|_{L^{p_s^*}(A_k)}\bigr)_{k\ge0}
\right\|_{\ell^{q_s^*}}
<\infty,
\]
where
\[
\|F\|_{L^{p_s^*}(A_k)}
:=
\begin{cases}
\displaystyle
\left(\int_{A_k}|F(z)|^{p_s^*}\,dA(z)\right)^{\frac{1}{p_s^*}},
& p_s^*<\infty,\\[1em]
\displaystyle
\operatorname*{ess\,sup}_{z\in A_k}|F(z)|,
& p_s^*=\infty,
\end{cases}
\]
again with the usual supremum interpretation when \(q_s^*=\infty\).

\begin{theorem} 
\label{thm:carleson-intro}
Let \(\alpha>0\), \(0<p,q\le\infty\), and \(0<s<\infty\).  Let
\(\delta_0=\delta_0(\alpha,p,q)\) be as in
Theorem~\ref{thm:AD-main}, and fix \(0<\delta<\delta_0\).  Let
\(\mathcal P_\delta\) be a \(\delta\)-pocket tiling of \(\C\) with
distinguished centers \(\zeta_Q\).  Then, for every
\(R>C_{\mathrm{out}}\delta\), the following assertions are equivalent for a
positive Borel measure \(\mu\) on \(\C\):
\begin{enumerate}
\item[\textup{(i)}] \(\mu\) is an \(s\)-Carleson measure for
\(\mathcal F_\alpha^{p,q}\).
\item[\textup{(ii)}] The discrete ball-mass sequence \(u_{\mu,R}\) belongs to
\(\ell^{p_s^*,q_s^*}\).
\item[\textup{(iii)}] The continuous ball-mass function \(U_{\mu,R}\) belongs to
\(L_{\mathrm{sh}}^{p_s^*,q_s^*}(\C)\).
\end{enumerate}
Moreover,
\[
\|i_\mu\|^s
\asymp
\|u_{\mu,R}\|_{\ell^{p_s^*,q_s^*}}
\asymp
\|U_{\mu,R}\|_{L_{\mathrm{sh}}^{p_s^*,q_s^*}},
\]
where the implicit constants depend only on \(\alpha,p,q,s,\delta\), and \(R\).
\end{theorem}

\begin{theorem} 
\label{thm:vanishing-carleson-intro}
Under the assumptions and notation of Theorem~\ref{thm:carleson-intro}, the
following assertions are equivalent for a positive Borel measure \(\mu\) on
\(\C\):
\begin{enumerate}
\item[\textup{(i)}] \(\mu\) is a vanishing \(s\)-Carleson measure for
\(\mathcal F_\alpha^{p,q}\).
\item[\textup{(ii)}] The discrete ball-mass sequence \(u_{\mu,R}\) satisfies:
\begin{enumerate}
\item[\textup{(a)}] if \(0<q\le s\), then
\(
\|(u_{\mu,R})_k\|_{\ell^{p_s^*}}\to0
\)
as \(k\to\infty\);
\item[\textup{(b)}] if \(s<q\le\infty\), then
\(
u_{\mu,R}\in\ell^{p_s^*,q_s^*}\).
\end{enumerate}
\item[\textup{(iii)}] The continuous ball-mass function \(U_{\mu,R}\) satisfies:
\begin{enumerate}
\item[\textup{(a)}] if \(0<q\le s\), then
\(
\|U_{\mu,R}\|_{L^{p_s^*}(A_k)}\to0
\)
as \(k\to\infty\);
\item[\textup{(b)}] if \(s<q\le\infty\), then
\(
U_{\mu,R}\in L_{\mathrm{sh}}^{p_s^*,q_s^*}(\C)\).
\end{enumerate}
\end{enumerate}
\end{theorem}

\smallskip
\noindent
The choice of \(R\) is inessential: if one of the discrete or continuous
conditions in Theorems~\ref{thm:carleson-intro} and
\ref{thm:vanishing-carleson-intro} holds for some \(R>C_{\mathrm{out}}\delta\),
then it holds for every such \(R\), since assertion \({\rm(i)}\) is independent
of \(R\).

\smallskip
\noindent
When \(s<q\le\infty\), the bounded and vanishing conditions coincide and are
both characterized by \(u_{\mu,R}\in\ell^{p_s^*,q_s^*}\).  The additional
shellwise decay condition appears only for \(0<q\le s\).

\smallskip
\noindent
In the one-parameter Fock scale, localized Euclidean ball masses characterize
Carleson and vanishing Carleson measures: boundedness corresponds to uniform
mass bounds and compactness to vanishing at infinity; see
\cite[Theorems~2.3 and 2.4]{IZ10}.  For embeddings between different classical
Fock spaces, localized masses also satisfy an integrability condition in the
smaller-target-exponent regime, where boundedness and compactness coincide; see
\cite[Theorems~3.1 and 3.2]{HL11}.  Endpoint growth spaces were treated by
Mengestie~\cite[Proposition~2.6]{Meng13}, and related \(q\)-Carleson criteria
for growth Fock spaces by Abakumov and Doubtsov~\cite[Subsection~3.2]{AD16}.

\smallskip
\noindent
The mixed-norm result keeps the local Fock geometry but changes how localized
masses are measured.  Its shellwise structure is analogous to mixed-norm
Bergman Carleson theory, where radial levels and angular subdivisions interact
\cite[Theorem~3]{PRS2019}, although the underlying geometry is Gaussian rather
than Bergman.

\smallskip
\noindent
\textbf{Proof highlight.}
The main point is to convert local Fock testing into the mixed shell conditions
imposed by the norm.  Localized Euclidean ball masses suffice in the
one-parameter Fock scale; here they must be measured through the inner--outer
annular structure of the atomic coefficient model.
The proof proceeds by reduction to sequences.  Sharp kernel estimates identify
the relevant ball-mass weights, and the atomic decomposition from
Section~\ref{sec:atomic} reduces
\[
i_\mu:\Fpq\longrightarrow L^s_{\alpha}(\mu)
\]
to weighted embeddings
\[
i_\omega:\ell^{p,q}\longrightarrow \ell^s(\omega)
\]
on the annular coefficient space.  The discrete shell principles in
Subsection~\ref{subsec:carleson-discrete} then give the boundedness and
compactness criteria.  Necessity is read from \(u_{\mu,R}\); sufficiency passes
through an auxiliary Gaussian weight \(u^\sharp\), compares it with
\(u_{\mu,R}\) by off-diagonal decay and shell summation, and then transfers  to the continuous annular conditions for \(U_{\mu,R}\).

\smallskip
\noindent
The proofs of Theorems~\ref{thm:carleson-intro} and
\ref{thm:vanishing-carleson-intro} are given in
Section~\ref{sec:carleson}.

\subsubsection{Dual spaces}
\label{subsec:dual-spaces-intro}

In the Banach regime \(1\le p\le\infty\), \(1\le q<\infty\), the representing space is
governed by the usual conjugate exponents and no polynomial correction
is needed.  For \(0<p<1\) or \(0<q<1\), representing functions carry an
additional polynomial correction reflecting the imbalance between angular and
radial summability.  The endpoint \(q=\infty\) is treated through the little
mixed-norm Fock space.  We use the following weighted scale.

\smallskip
\noindent
Let \(\alpha>0\), \(\sigma\in\mathbb R\), and \(0<p,q\le\infty\).  We denote by
\(\mathcal F_{\alpha;\sigma}^{p,q}\) the space of all entire functions \(f\)
such that
\[
\|f\|_{\mathcal F_{\alpha;\sigma}^{p,q}}
:=
\begin{cases}
\displaystyle
\left(
\int_0^\infty
\bigl[M_p(f,r)(1+r)^\sigma\bigr]^q
\,d\lambda_{\alpha q}(r)
\right)^{\frac{1}{q}},
& 0<q<\infty,\\[1em]
\displaystyle
\sup_{r\ge0}
M_p(f,r)e^{-\frac{\alpha}{2}r^2}(1+r)^\sigma,
& q=\infty,
\end{cases}
\]
is finite, with the usual interpretation when \(p=\infty\).  For \(\sigma=0\)
this is \(\mathcal F_{\alpha}^{p,q}\).

\smallskip
\noindent
We use the generalized conjugate exponent from the notation section; thus
\(p^*\) and \(q^*\) denote the generalized conjugates of \(p\) and \(q\),
respectively.  We also set
\[
\sigma(p,q):=
\left(\frac1p-1\right)_+
-
\left(\frac1q-1\right)_+,
\qquad
(x)_+:=\max\{x,0\},
\]
with the convention \(\frac{1}{\infty}=0\).

\smallskip
\noindent
For each \(\alpha>0\), we use the sesquilinear Gaussian pairing
\[
\langle f,g\rangle_\alpha
:=
\frac{\alpha}{\pi}
\int_{\C} f(z)\overline{g(z)}e^{-\alpha|z|^2}\,dA(z),
\]
whenever the integral is absolutely convergent.  For \(\tau>0\), let
\[
(D_\tau g)(z):=g(\tau z),\qquad z\in\C.
\]
For \(\alpha,\gamma>0\), we define the transported \(\gamma\)-pairing with
base parameter \(\alpha\) by
\[
\langle f,g\rangle_{\alpha\to\gamma}^{\operatorname{tr}}
:=
\big\langle f,D_{\alpha/\gamma}g\big\rangle_\alpha .
\]
On polynomial pairs this transported pairing agrees with the ordinary Gaussian
pairing at parameter \(\gamma\).  Its well-definedness on the spaces below is
proved in Section~\ref{sec:dual-spaces}.

\begin{theorem} 
\label{thm:duality-intro}
Let \(\alpha,\beta>0\), \(0<p\le\infty\), \(0<q<\infty\), and
\(
\gamma=\sqrt{\alpha\beta}\).
Then the continuous dual of \(\mathcal F_\alpha^{p,q}\) can be identified with
\(
\mathcal F_{\beta;\sigma(p,q)}^{p^*,q^*}
\)
under the transported pairing
\(\langle\cdot,\cdot\rangle_{\alpha\to\gamma}^{\operatorname{tr}}\).
\end{theorem}

\smallskip
\noindent
We also record the corresponding duality theorem for the little endpoint space
\(f_\alpha^{p,\infty}\).

\begin{theorem} 
\label{thm:little-duality-intro}
Let \(\alpha,\beta>0\), \(0<p\le\infty\), and
\(
\gamma=\sqrt{\alpha\beta}\).
Then the continuous dual of \(f_\alpha^{p,\infty}\) can be identified with
\(
\mathcal F_{\beta;\sigma(p,\infty)}^{p^*,1}
\)
under the transported pairing
\(\langle\cdot,\cdot\rangle_{\alpha\to\gamma}^{\operatorname{tr}}\).
\end{theorem}

\smallskip
\noindent
Here and throughout the duality section, an identification \(X^*\simeq Y\)
under the pairing
\(\langle\cdot,\cdot\rangle_{\alpha\to\gamma}^{\operatorname{tr}}\) means that
each \(g\in Y\) defines a continuous linear functional \(F_g\) on \(X\) by
\[
F_g(f)
=
\langle f,g\rangle_{\alpha\to\gamma}^{\operatorname{tr}},
\qquad f\in X,
\]
and conversely every \(F\in X^*\) is represented uniquely in this way by some
\(g_F\in Y\), with
\[
\|F\|_{X^*}\asymp \|g_F\|_Y.
\]
The implicit constants in this equivalence depend only on the parameters
appearing in the corresponding duality theorem. 
Equivalently, the conjugate-linear map \(g\mapsto F_g\) identifies \(Y\) with
\(X^*\), with equivalent norms.

\smallskip
\noindent
These duality results extend two classical lines of theory.  First,
Theorem~\ref{thm:duality-intro} recovers the classical Fock-space duality with
the geometric relation
\(
\gamma=\sqrt{\alpha\beta}
\)
between the parameters of the primal space, the representing space, and the
pairing; see \cite[Theorems~2.23 and 2.24]{Zhu}.  In particular,
Theorem~\ref{thm:little-duality-intro} with \(p=\infty\) gives
\(
(f_\alpha^\infty)^*\simeq \mathcal F_\beta^1\),
which is the standard duality for the little Fock space; see
\cite[Theorem~2.26]{Zhu}.
Second, Theorem~\ref{thm:duality-intro} parallels mixed-norm Bergman duality,
studied for instance by Gadbois~\cite[Theorem~4.1]{Gad88} and by Moreno and
Pel\'aez~\cite[Theorem~1.1]{MP26}.  There the duality uses conjugate mixed
exponents and an adjusted radial weight; endpoint and small-exponent regimes
may also involve BMOA- or Bloch-type spaces.  In the present Fock setting the
correction is the polynomial factor
\(
(1+r)^{\sigma(p,q)}\),
while the Gaussian parameters retain the Fock-specific relation
\(\gamma=\sqrt{\alpha\beta}\).  

\smallskip
\noindent
\textbf{Proof highlight.}
The main difficulty is to identify the dual space uniformly in the Banach,
quasi-Banach, and endpoint regimes, while allowing different Gaussian
parameters in the primal space, the representing space, and the pairing.  In
the Banach mixed range \(1\leq p<\infty\) and \(1<q<\infty\), one may also use a
projection method: extend functionals to an ambient mixed-norm Lebesgue space,
apply the Benedek--Panzone duality theorem, and project back to the analytic
subspace; see \cite{BP61,Liu24}.  That approach is effective in this range, but
it relies on projection boundedness and Banach-space duality, and therefore
does not cover the quasi-Banach cases, the endpoint \(q=\infty\), or the full
weighted representing scale in Theorems~\ref{thm:duality-intro} and
\ref{thm:little-duality-intro}.

\smallskip
\noindent
Our proof is intrinsic to the analytic Fock spaces.  In the equal-parameter
case, every
\[
g\in \mathcal F_{\alpha;\sigma(p,q)}^{p^*,q^*}
\]
defines a bounded functional on \(\Fpq\) by
\[
F_g(f):=\langle f,g\rangle_\alpha .
\]
The estimate controls the Gaussian pairing on circles and then over radial
shells; the factor \((1+r)^{\sigma(p,q)}\) compensates for the quasi-Banach
losses and for the imbalance between \(p\) and \(q\).

\smallskip
\noindent
Conversely, for \(F\in(\Fpq)^*\), we define the canonical representative by
\[
g_F(w)=\overline{F(K_{\alpha,w})},\qquad w\in\mathbb C.
\]
Testing \(F\) on kernel families adapted to circles and shells, and using the
sharp kernel estimates together with the sharp circle-mean growth theorem from
Section~\ref{sec:mixed-fock-tools}, gives
\[
g_F\in \mathcal F_{\alpha;\sigma(p,q)}^{p^*,q^*}.
\]
Polynomial density and Gaussian orthogonality then identify \(F\) with pairing
against \(g_F\) and give uniqueness.  The general-parameter result follows from
the equal-parameter theorem by dilation transport, while the little endpoint
space \(f_\alpha^{p,\infty}\) is handled separately through finite shell testing,
which produces the outer \(\ell^1\)-type condition for the representative.

\smallskip
\noindent
Section~\ref{sec:dual-spaces} proves Theorems~\ref{thm:duality-intro} and
\ref{thm:little-duality-intro}.  Equal-parameter duality under
\(\langle\cdot,\cdot\rangle_\alpha\) is transported by dilation; the little
endpoint space is treated separately.

\subsection{Guide to the paper}
\label{subsec:guide-to-paper}
We close the introduction with the organization of the paper.

\smallskip
\noindent
\textbf{Main line of the paper.}
The paper passes from analytic estimates to discrete models, and then from
these models to structural consequences.  After the circle-mean estimate and
annular discretization are established, the later sections treat embeddings,
atomic decomposition, Carleson measure theory, and duality.

\smallskip
\noindent
\textbf{Proof architecture and structural mechanisms.}
Although the main theorems address different aspects of the mixed-norm theory,
their proofs are organized around three   mechanisms that connect the analytic
geometry of entire functions with the corresponding annular sequence model.

The first is a local mass principle for convex Laplace-type exponents. It
identifies the radial scale near which a weighted integral is concentrated and
converts that concentration into uniform two-sided estimates. In the present
setting, the same principle yields both the optimal polynomial correction in
the circle-mean estimate and the annular discretization of the mixed Gaussian
norm. Its formulation is not specific to Fock spaces and applies more generally
to Laplace-type integrals whose mass is concentrated near a moving maximizing
scale.

The second mechanism is the \(\delta\)-pocket discretization. It reconciles two
different geometries inherent in the problem: normalized Fock kernels are
localized at a fixed Euclidean scale, whereas the outer mixed norm records
global \(\ell^q\)-summability across radial annuli. The pocket model preserves
the local Euclidean--Gaussian geometry while retaining the correct annular
bookkeeping. It thereby provides the discrete framework for the embedding,
Carleson measure, and atomic decomposition theorems.

The third mechanism is a quantitative re-centering envelope for normalized Fock
kernels. It controls the cumulative effect of perturbing localized kernel
centers, both within individual pockets and across neighboring annular layers.
This stability becomes especially important when \(p<1\), where ordinary
convexity and Banach-space perturbation arguments are unavailable. The
re-centering envelope is the main new ingredient in the angular quasi-Banach
atomic decomposition and provides a robust method for transporting localized
kernel expansions without losing the mixed annular control.

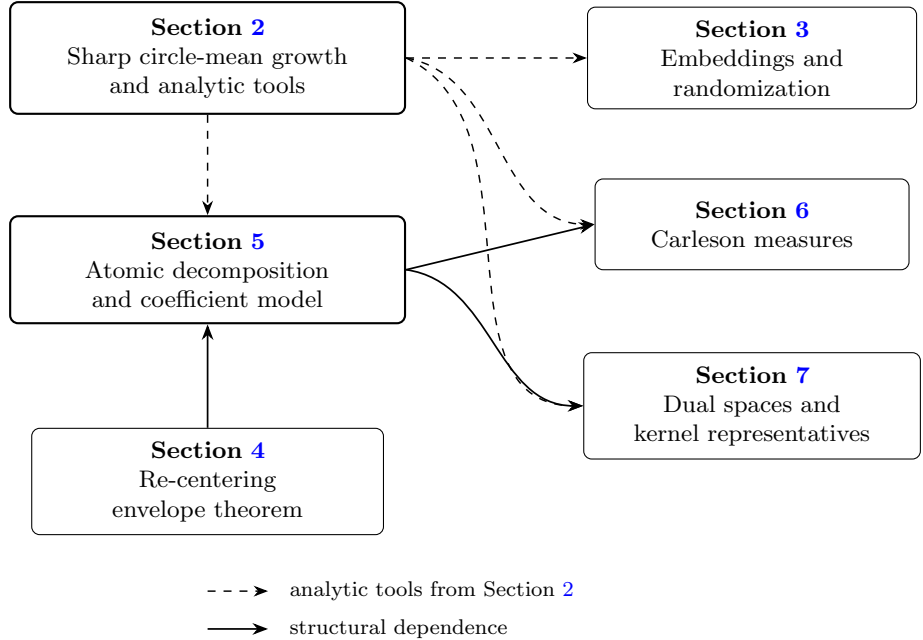
\begin{figure}[htbp]
\centering
\begin{tikzpicture}[
    >=Stealth,
    font=\footnotesize,
    secbox/.style={
        draw,
        rounded corners=3pt,
        align=center,
        inner sep=5pt,
        minimum height=1.2cm
    },
    mainbox/.style={
        draw,
        rounded corners=3pt,
        align=center,
        inner sep=6pt,
        minimum height=1.35cm,
        line width=0.8pt
    },
    dasheddep/.style={->, dashed, line width=0.55pt},
    soliddep/.style={->, line width=0.65pt}
]

\node[mainbox, text width=4.8cm] (sec2) at (0,2.8)
{\textbf{Section~\ref{sec:mixed-fock-tools}}\\
Sharp circle-mean growth\\
and analytic tools};

\node[mainbox, text width=4.8cm] (sec5) at (0,0)
{\textbf{Section~\ref{sec:atomic}}\\
Atomic decomposition\\
and coefficient model};

\node[secbox, text width=4.3cm] (sec4) at (0,-2.8)
{\textbf{Section~\ref{sec:recentering-envelope}}\\
Re-centering\\
envelope theorem};

\node[secbox, text width=4.0cm] (sec3) at (7.2,2.8)
{\textbf{Section~\ref{sec:embeddings}}\\
Embeddings and\\
randomization};

\node[secbox, text width=3.8cm] (sec6) at (7.2,0.6)
{\textbf{Section~\ref{sec:carleson}}\\
Carleson measures};

\node[secbox, text width=4.1cm] (sec7) at (7.2,-1.8)
{\textbf{Section~\ref{sec:dual-spaces}}\\
Dual spaces and\\
kernel representatives};

\draw[dasheddep] (sec2.east) -- (sec3.west);
\draw[dasheddep] (sec2.south) -- (sec5.north);
\draw[dasheddep] (sec2.east) to[out=-10,in=180] (sec6.west);
\draw[dasheddep] (sec2.east) to[out=-25,in=180] (sec7.west);

\draw[soliddep] (sec4.north) -- (sec5.south);
\draw[soliddep] (sec5.east) -- (sec6.west);
\draw[soliddep] (sec5.east) to[out=-8,in=180] (sec7.west);

\begin{scope}[shift={(0,-4.5)}, font=\scriptsize]
    \draw[dasheddep] (0,0.25) -- (0.8,0.25);
    \node[anchor=west] at (0.95,0.25)
    {analytic tools from Section~\ref{sec:mixed-fock-tools}};

    \draw[soliddep] (0,-0.25) -- (0.8,-0.25);
    \node[anchor=west] at (0.95,-0.25)
    {structural dependence};
\end{scope}

\end{tikzpicture}
\caption{Main dependencies among the sections.}
\label{fig:main-dependencies}
\end{figure}

\smallskip
\noindent
\textbf{Organization of the paper.}
Section~\ref{sec:intro} fixes notation and states the main results.
Section~\ref{sec:mixed-fock-tools} develops the circle-mean estimate, annular
discretization, reproducing formula, packet estimates, density results, and
auxiliary operator estimates.  Section~\ref{sec:embeddings} proves the
embedding theorem and the randomization application.
Section~\ref{sec:recentering-envelope} proves the re-centering envelope
theorem for normalized Fock kernels.  Section~\ref{sec:atomic} proves
the atomic decomposition, first in the angular Banach range and then in
the angular quasi-Banach range.  Section~\ref{sec:carleson}
characterizes Carleson and vanishing Carleson measures by mixed shell
conditions.  Section~\ref{sec:dual-spaces} proves the duality theorems,
including the transported Gaussian pairing and the little endpoint
space.

\section*{Index of notations and recurring objects}


\begingroup
\small
\renewcommand{\arraystretch}{1.12}
\setlength{\arrayrulewidth}{0.3pt}
\arrayrulecolor{black!28}

\begin{longtable}{
>{\raggedright\arraybackslash}p{0.29\textwidth}
!{\color{black!35}\vrule width 0.35pt}
>{\raggedright\arraybackslash}p{0.63\textwidth}
}
\rowcolor{black!5}
\textbf{Notation or object} & \textbf{Meaning and location} \\
\hline
\endfirsthead

\rowcolor{black!5}
\textbf{Notation or object} & \textbf{Meaning and location} \\
\hline
\endhead

\multicolumn{2}{l}{\textbf{A. Basic notation and mixed-norm Fock spaces}}\\
\hline

\(\mathbb N_0,\mathbb C,\mathbb T,dA\)
&
Basic sets and planar Lebesgue measure, Subsec.~\ref{subsec:intro-def}.
\\
\hline
\(B(z,R),H(\mathbb C)\)
&
Euclidean ball and space of entire functions, Subsec.~\ref{subsec:intro-def}.
\\
\hline
\(M_p(f,r)\)
&
Angular \(L^p\)-mean on \(|z|=r\), Subsec.~\ref{subsec:intro-def}.
\\
\hline

\(d\lambda_{\alpha q}(r)\)
&
Radial Gaussian measure, Subsec.~\ref{subsec:intro-def}.
\\
\hline

\(\mathcal F_\alpha^{p,q}\)
&
Mixed-norm Fock space, Subsec.~\ref{subsec:intro-def}.
\\
\hline

\(f_\alpha^{p,\infty}\)
&
Little endpoint mixed-norm Fock space, Subsec.~\ref{subsec:intro-def}.
\\
\hline

\(p^*,q^*\)
&
Generalized conjugate exponents, Subsec.~\ref{subsec:intro-def}.
\\
\hline

\(\mathcal F_{\alpha;\sigma}^{p,q}\)
&
Weighted mixed-norm Fock space, Subsec.~\ref{subsec:dual-spaces-intro}.
\\
\hline

\(K_{\alpha,w},\kappa_{\alpha,w}\)
&
Fock kernel and normalized Fock kernel, Subsec.~\ref{subsec:intro-def}.
\\
\hline

\(\widetilde\kappa_{\alpha,w}^{p,q}\)
&
Mixed-norm-normalized Fock kernel, Subsec.~\ref{subsec:intro-def}.
\\
\hline

\multicolumn{2}{l}{\textbf{B. Annular quantities and analytic tools}}\\
\hline

\(A_k\)
&
Unit annular layer, Subsec.~\ref{subsec:intro-def}.
\\
\hline

\(B_p(f;k)\)
&
Shell supremum of \(M_p(f,r)e^{-\alpha r^2/2}\), Subsec.~\ref{subsec:intro-def}.
\\
\hline

Local mass principle
&
Convexity principle for sharp radial control, Subsec.~\ref{subsec:local-mass-principle}.
\\
\hline

Sharp circle-mean growth
&
Optimal \(M_p\)-growth estimate, Theorem~\ref{thm:estimate}.
\\
\hline

Annular discretization
&
Equivalent shell norm for \(\mathcal F_\alpha^{p,q}\), Subsec.~\ref{subsec:annular-discretization}.
\\
\hline

\(P_\alpha,P_{\alpha,R}\)
&
Fock projection and truncation, Subsec.~\ref{subsec:kernel-projection}.
\\
\hline

\(G_r\)
&
Ring packet centered at radius \(r\), Subsec.~\ref{subsec:packets-tests}.
\\
\hline

\multicolumn{2}{l}{\textbf{C. Pocket tilings and coefficient spaces}}\\
\hline
\(\delta\)-pocket discretization
&
Euclidean-localized annular discretization of the plane, Sec.~\ref{sec:recentering-envelope} and \ref{sec:atomic}.
\\
\hline
\(\mathcal P_\delta\)
&
\(\delta\)-pocket decomposition/tiling, Sec.~\ref{sec:recentering-envelope} and Subsec.~\ref{subsec:atomic-pockettiling}.
\\
\hline

\(Q,\zeta_Q\)
&
Pocket and distinguished center, Sec.~\ref{sec:recentering-envelope} and Subsec.~\ref{subsec:atomic-pockettiling}.
\\
\hline

\(\mathcal P_{\delta,k}\)
&
Pockets in the annulus \(A_k\), Subsec.~\ref{subsec:atomic-pockettiling}.
\\
\hline

\(Q^*\)
&
Fattened pocket, Subsec.~\ref{subsec:atomic-pockettiling}.
\\
\hline

\(c_{\rm in},C_{\rm out}\)
&
Uniform pocket geometry constants, Sec.~\ref{sec:recentering-envelope} and Subsec.~\ref{subsec:atomic-pockettiling}.
\\
\hline

\(\Phi_\delta\)
&
Bi-Lipschitz reindexing by \(\mathbb Z^2\), Sec.~\ref{sec:recentering-envelope} and  Subsec.~\ref{subsec:atomic-pockettiling}.
\\
\hline

\(\ell^{p,q}\)
&
Unweighted pocket coefficient space, Subsec.~\ref{subsec:atomic-pockettiling}.
\\
\hline

\(\ell^{p,q}_w\)
&
Weighted pocket coefficient space, Subsec.~\ref{subsec:atomic-pockettiling}.
\\
\hline

\(w_k=(1+k)^{\frac{1}{q}-\frac{1}{p}}\)
&
Shell weight in \(\ell_w^{p,q}\), Subsec.~\ref{subsec:atomic-pockettiling}.
\\
\hline

\multicolumn{2}{l}{\textbf{D. Atomic decomposition and re-centering}}\\
\hline

\(C_\delta^{\ker}\)
&
Kernel analysis operator, Subsec.~\ref{subsec:atomic-banach}.
\\
\hline

\(S_\delta^{\ker}\)
&
Kernel synthesis operator, Subsec.~\ref{subsec:atomic-banach}.
\\
\hline

\(R_\delta^{\ker}\)
&
Kernel remainder operator, Subsec.~\ref{subsec:atomic-banach}.
\\
\hline

\(C_\delta^{\ker}(I-R_\delta^{\ker})^{-1}\)
&
Corrected Banach-range coefficient recovery, Subsec.~\ref{subsubsec:atomic-banach-exact}.
\\
\hline

Jet atoms
&
Finite-order local atoms for \(0<p<1\), Subsec.~\ref{subsubsec:atomic-jet-setup}.
\\
\hline

Floating kernels
&
Kernels with moving centers inside pockets, Subsec.~\ref{subsubsec:atomic-jet-elimination}.
\\
\hline

Re-centering envelope theorem
&
Re-expansion at fixed pocket centers, Sec.~\ref{sec:recentering-envelope}.
\\
\hline

\(b_Q(\eta)\)
&
Re-centering coefficients, Theorem~\ref{thm:envelope-recentering}.
\\
\hline

\(\mathcal A_{\gamma,\beta}^{\exp}(\mathcal P_\delta)\)
&
Subexponential matrix algebra, Subsec.~\ref{subsec:recenter-setup}.
\\
\hline

\(A_\delta^{\mathrm{ker}}\) &  Kernel coefficient-recovery operator in the angular quasi-Banach range, Subsec.~\ref{subsubsec:atomic-quasi-proof}.
\\
\hline

\(S_\delta,A_\delta\)
&
Normalized synthesis and coefficient recovery, Subsec.~\ref{subsec:atomic-consequences}.
\\
\hline


\multicolumn{2}{l}{\textbf{E. Carleson measure notation}}\\
\hline

\(L_\alpha^s(\mu)\)
&
Gaussian \(L^s\)-space with measure \(\mu\), Subsec.~\ref{subsec:carmeas}.
\\
\hline

\(i_\mu\)
&
Carleson embedding map, Subsec.~\ref{subsec:carmeas}.
\\
\hline

\(U_{\mu,R}\)
&
Continuous localized ball-mass function, Subsec.~\ref{subsec:carmeas}.
\\
\hline

\(u_{\mu,R}\)
&
Discrete localized ball-mass sequence, Subsec.~\ref{subsec:carmeas}.
\\
\hline

\(p_s^*,q_s^*\)
&
Inner and outer Carleson exponents, Subsec.~\ref{subsec:carmeas}.
\\
\hline
\(\ell^{p_s^*,q_s^*}\)
&
Discrete shell mixed-norm space for ball-mass sequences, Subsec.~\ref{subsec:carmeas}.
\\
\hline
\(L_{\rm sh}^{p_s^*,q_s^*}(\mathbb C)\)
&
Continuous shell mixed-norm space, Subsec.~\ref{subsec:carmeas}.
\\
\hline

\(\ell^s(\omega)\)
&
Weighted discrete target space, Subsec.~\ref{subsec:carleson-discrete}.
\\
\hline

\(i_\omega\)
&
Weighted identity map, Subsec.~~\ref{subsec:carleson-discrete}.
\\
\hline

\(u^\sharp\)
&
Auxiliary Gaussian weight in the Carleson sufficiency argument, Subsec.~\ref{subsubsec:carleson-atomic-reduction}.
\\
\hline

\multicolumn{2}{l}{\textbf{F. Duality notation and pairings}}\\
\hline

\(\langle f,g\rangle_\alpha\)
&
Gaussian pairing at parameter \(\alpha\), Subsec.~\ref{subsec:dual-spaces-intro}.
\\
\hline

\(D_\tau\)
&
Dilation operator, Subsec.~\ref{subsec:dual-spaces-intro}.
\\
\hline

\(\gamma=\sqrt{\alpha\beta}\)
&
Geometric pairing parameter, Theorem~\ref{thm:duality-intro}.
\\
\hline

\(\langle f,g\rangle_{\alpha\to\gamma}^{\rm tr}\)
&
Transported Gaussian pairing, Subsec.~\ref{subsec:dual-spaces-intro}.
\\
\hline

\(\sigma(p,q)\)
&
Polynomial correction in duality, Subsec.~\ref{subsec:dual-spaces-intro}.
\\
\hline

\(g_F\)
&
Canonical kernel representative, Subsec.~\ref{subsec:dual-prelim}.
\\
\hline

\(\mathcal F_{\beta;\sigma(p,q)}^{p^*,q^*}\)
&
Representing space in full-space duality, Theorem~\ref{thm:duality-intro}.
\\
\hline

\(\mathcal F_{\beta;\sigma(p,\infty)}^{p^*,1}\)
&
Representing space for the little endpoint, Theorem~\ref{thm:little-duality-intro}.
\\
\hline
\(T_r \) & Circle-test operator used in angular dual testing, Subsec.~\ref{subsec:dual-equal}.\\
\hline

\end{longtable}

\arrayrulecolor{black}
\endgroup

\section{Sharp circle-mean growth and analytic tools}
\label{sec:mixed-fock-tools}

This section develops the analytic tools used later for
\(\mathcal F_\alpha^{p,q}\).  We prove the circle-mean growth estimate from a
local mass principle for convex Laplace-type exponents, derive the annular
discretization of the mixed Gaussian norm, and record the reproducing formula,
packet and kernel estimates, density results, and auxiliary operator estimates.

\subsection{A local mass principle for convex Laplace exponents}
\label{subsec:local-mass-principle}

We use the following convexity principle for both the circle-mean estimate and
the annular discretization.

\begin{lemma}[Uniform local mass]
\label{lem:uniform-local-mass-convex}
Let \(\Phi:\mathbb R\to\mathbb R\) be convex, let \(a>0\), \(b\in\mathbb R\), and set
\[
F(x):=\Phi(x)-a e^{2x}+bx,\qquad x\in\mathbb R.
\]
Then for every \( c \in (0, \frac{1}{4}]\) and \(x \in \R\),
\begin{equation}
\label{eq:uniform-local-mass-convex}
\int_{x-h_x}^{x+h_x} e^{F(u)}\,du
\ge
2e^{-4ac^2} h_x e^{F(x)},
\qquad
h_x:=\min\left\{\frac14, c \, e^{-x}\right\}.
\end{equation}
\end{lemma}

\begin{proof}
For \(|t|\le \frac{1}{4}\), the convexity of $\Phi$ gives
\[
\Phi(x+t)+\Phi(x-t)\ge 2\Phi(x).
\]
Using this and the inequality $e^{2t}+e^{-2t}-2\le 8t^2$ for $|t| \leq \frac{1}{4}$ yields 
\[
F(x+t)+F(x-t)\ \ge\ 2F(x)- 8 a\,e^{2x}t^2,
\]
which by the arithmetic-geometric mean inequality implies that
\[
e^{F(x+t)}+e^{F(x-t)}
\ge 2\exp\!\Big(\frac{F(x+t)+F(x-t)}{2}\Big)
\ge 2e^{F(x)}e^{- 4 a e^{2x}t^2}.
\]
By the definition of \(h_x\), we have \(e^x h_x\le c\). Thus, 
\(
e^{-4ae^{2x}t^2}\ge e^{-4a c^2}\) for
\(0\le t\le h_x\).
Integrating in $t\in[0,h_x]$, we get
\[
\int_{x-h_x}^{x+h_x}e^{F(u)}\,du
=
\int_0^{h_x}\big(e^{F(x+t)}+e^{F(x-t)}\big)\,dt
\ge
2e^{-4a c^2}h_xe^{F(x)}.
\]
This is \eqref{eq:uniform-local-mass-convex}.
\end{proof}

\subsection{Sharp circle-mean growth and bounded point evaluations}
\label{subsec:sharp-growth}

We prove Theorem~\ref{thm:estimate} and record two consequences: sharpness
by monomial tests and the
point-evaluation estimate Lemma~\ref{lem:p-est}.

\begin{proof}[Proof of Theorem~\ref{thm:estimate}]
The case $q = \infty$ follows from the definition of
$\norm{\cdot}_{\FpqA{p}{\infty}{\alpha}}$.  Let $0 < q < \infty$ and
$0\ne f \in \Fpq$.
For $r >0$, write $x=\log r$ and set
\[
\Phi(x):=\log M_p\big(f,e^x\big),
\qquad
F(x):=q\,\Phi(x)-\frac{\alpha q}{2}e^{2x}+2x .
\]
Applying
Lemma~\ref{lem:uniform-local-mass-convex} to the convex function \(q\Phi\),
with $a = \frac{\alpha q}{2}, b = 2$, and $c = \frac{1}{4}$, we get
\[
\int_{x-h_x}^{x+h_x} e^{F(u)}\,du
\ \ge\ 2 e^{-\frac{\alpha q}{8}} \,h_x\,e^{F(x)},
\qquad
h_x:=\min\left\{\frac14, \frac{e^{-x}}{4} \right\}.
\]
For \(r\ge1\), \(x\ge0\), so \(h_x= \frac{e^{-x}}{4}=\frac{1}{4r}\).
Since $$e^{F(x)}= M_p^q(f,r) e^{-\frac{\alpha q }{2}r^2}\,r^2 $$ and $du= \frac{ds}{s}$ under $s=e^u$, we obtain
\[
M_p^q(f,r) e^{-\frac{\alpha q }{2}r^2}\,r \leq 2 e^{\frac{\alpha q}{8}} 
\int_{re^{-h_x}}^{re^{h_x}} M_p^q(f,s) e^{-\frac{\alpha q }{2}s^2}\,s\,ds \ \le \frac{ 2 e^{\frac{\alpha q}{8}}}{q \alpha}  \|f\|_{\Fpq}^q,
\]
which implies \eqref{eq-newest} for all $r \geq 1$. 

\smallskip
\noindent
For $0 \leq r \leq 1$, monotonicity of \(r\mapsto M_p(f,r)\) gives
\[
M_p^q(f, r) \leq M_p^q(f, 1) \leq  \left(q\alpha\int_{1}^{\infty} e^{-\frac{\alpha q}{2}t^2}tdt\right)^{-1} \left(q\alpha\int_{1}^{\infty} M^q_p(f,t) e^{-\frac{\alpha q}{2}t^2}tdt\right) \leq e^{\frac{\alpha q}{2}}   \|f\|^q_{\Fpq}.
\]
Since
\(
(1+r)^{-\frac{1}{q}}e^{\frac{\alpha}{2}r^2}\asymp 1\) for \(0\le r\le1\),
we obtain \eqref{eq-newest} also for \(0\le r\le1\).
\end{proof}

\medskip
\noindent
Sharpness follows from monomials.

\begin{lemma}\label{lem-sh}
Let $\alpha>0$ and $0<p,q\le\infty$. Then
\[
\sup_{n\ge 0}\frac{M_p(z^n, r)}{\|z^n\|_{\mathcal F_\alpha^{p,q}}}
\asymp (1+r)^{-\frac{1}{q}}\,e^{\frac{\alpha }{2}r^2},
\qquad r > 0,
\]
where the implicit constants depend only on $\alpha$ and $q$.
\end{lemma}

\begin{proof}
Since \(M_p(z^n,r)=r^n\) for every \(0<p\le\infty\), the monomial norm is
independent of \(p\).  For \(n\ge1\),
\[
\|z^n\|_{\mathcal F_\alpha^{p,q}}
=
\begin{cases}
\displaystyle
\left(\frac{2}{\alpha q}\right)^{\frac n2}
\Gamma\left(\frac{nq}{2}+1\right)^{\frac1q},
& 0<q<\infty,\\[1.2em]
\displaystyle
\left(\frac{n}{\alpha e}\right)^{\frac n2},
& q=\infty.
\end{cases}
\]
Moreover, \(\|1\|_{\mathcal F_\alpha^{p,q}}=1\). In particular, by Stirling's
formula and with the convention \(\frac{1}{\infty}=0\),
\[
\|z^n\|_{\Fpq}
\asymp
\left(\frac{n}{e\alpha}\right)^{\frac n2} n^{\frac{1}{2q}},
\qquad n\ge1,
\]
where the implicit constants depend only on \(\alpha\) and \(q\).
Using this and setting $\lambda:=\alpha r^{2}$, we get
\begin{equation*}
\frac{M_p(z^n, r)}{\|z^n\|_{\Fpq}}
\asymp
\Big(\frac{e\lambda}{n}\Big)^{\frac{n}{2}}\,n^{-\frac{1}{2q}}
=:b_n(\lambda),
\qquad n\ge 1.
\end{equation*}
Thus, up to constants depending only on $\alpha,q$,
$
\sup_{n\ge 0}\frac{M_p(z^n, r)}{\|z^n\|_{\Fpq}}
\asymp 1+\sup_{n\ge 1} b_n(\lambda).
$
It remains to show that
\begin{equation}\label{eq:reduction}
1 + \sup_{n\ge 1} b_n(\lambda) \asymp e^{\frac{\lambda}{2}}(1 + \lambda)^{-\frac{1}{2q}}, 
\qquad \lambda > 0.
\end{equation}
A direct computation gives
$$
\frac{b_{n+1}(\lambda)}{b_n(\lambda)} = \left( \frac{n}{n+1} \right)^{\frac{n}{2} + \frac{1}{2q}} \sqrt{\frac{e \lambda}{n+1}}\; \geq \; 1 
\quad \text{if and only if} \quad 
(n+1) \left(1 + \frac{1}{n}\right)^{n+\frac{1}{q}} \leq e \lambda.
$$
The quantity
\[
\tau_n:=e^{-1}(n+1)\left(1+\frac1n\right)^{n+\frac1q}
\]
satisfies
\[
0\le \tau_n-n\le C_q,\qquad n\ge1,
\]
for a constant \(C_q>0\) depending only on \(q\), since
\[
0\le
n\left[\left(1+\frac1n\right)^{\frac{1}{q}}-1\right]
\le \tau_n-n
\le
n\left[\left(1+\frac1n\right)^{1+\frac{1}{q}}-1\right]
\le C_q .
\]
Since
\[
\frac{b_{n+1}(\lambda)}{b_n(\lambda)}\ge1
\quad\Longleftrightarrow\quad
\tau_n\le \lambda,
\]
it follows that, for \(\lambda\ge C_q+1\), the sequence
\((b_n(\lambda))_{n\ge1}\) is increasing for \(n\le \lambda-C_q\) and
decreasing for \(n\ge \lambda+1\). Hence any maximizing index
\(n_0=n_0(\lambda)\) satisfies
\[
\lambda-C_q\le n_0\le \lambda+1.
\]
In particular, \(n_0=\lambda+O(1)\).
Hence
\[
\log b_{n_0}(\lambda)
=\frac{n_0}{2}\log\Big(\frac{e\lambda}{n_0}\Big)-\frac{1}{2q}\log n_0
 =  \frac{\lambda}{2}-\frac{1}{2q}\log\lambda+O(1).
\]
Thus \eqref{eq:reduction} holds for $\lambda \geq C_q + 1$. 

\smallskip
\noindent
If $0 < \lambda\le C_q + 1$, then
\[
b_n(\lambda)\le \left(\frac{e (C_q + 1)}{n}\right)^{\frac{n}{2}}\lesssim 1,
\ n\ge1, \quad \text{hence,} \quad 1+\sup_{n\ge1}b_n(\lambda)\asymp 1 \asymp e^{\frac{\lambda}{2}}(1+\lambda)^{-\frac{1}{2q}}.
\]
Thus \eqref{eq:reduction} holds for all \(\lambda>0\).
\end{proof}

We record the following estimate.

\begin{corollary}
For each $\alpha > 0$ and $0 < p, q \le \infty$,
\[
\sup_{\|f\|_{\Fpq} \leq 1} M_p(f, r) \; \asymp \; (1+r)^{-\frac{1}{q}}e^{\frac{\alpha }{2}r^2}, \quad r \geq 0,
\]
where the implicit constants depend only on $\alpha$ and $q$.
\end{corollary}
 
\begin{proof}
The upper bound is Theorem~\ref{thm:estimate}; the lower bound follows from
Lemma~\ref{lem-sh} by testing on monomials.
\end{proof}
 
\begin{lemma} \label{lem:p-est}
Let \(\alpha>0\) and \(0<p,q\leq\infty\).  Then, for every
\(f\in \mathcal F^{p,q}_\alpha\),
\[
|f(z)|
\lesssim
(1+|z|)^{\frac1p-\frac1q}
e^{\frac{\alpha}{2}|z|^2}
\|f\|_{\mathcal F^{p,q}_\alpha},
\qquad z\in\C,
\]
where the implicit constant depends only on \(\alpha,p\), and \(q\).
\end{lemma}

\begin{proof}
It suffices to control the circular supremum \(M_\infty(f,r)\).  We first
record the standard change of angular exponent.  If \(0<p_1\leq p_2\leq
\infty\), then
\begin{equation}\label{eq:angular-comparison}
M_{p_2}(f,r)
\lesssim
(1+r)^{\frac1{p_1}-\frac1{p_2}}M_{p_1}(f,r),
\qquad r\geq0.
\end{equation}
Indeed, for \(r>0\), set \(\rho=r+1\), \(t=\frac{r}{r+1}\), and
\(f_\rho(z)=f(\rho z)\) on \(\D\).  By \cite[Lemma~A]{A16},
\[
M_{p_2}(f_\rho,t)
\lesssim
(1-t)^{\frac1{p_2}-\frac1{p_1}}M_{p_1}(f_\rho,t),
\]
with a constant depending only on \(p_1,p_2\).  Since \(\rho t=r\) and
\(1-t=(1+r)^{-1}\), the displayed estimate follows; the case \(r=0\) follows
by continuity.

\smallskip
\noindent
Taking \(p_1=p\) and \(p_2=\infty\), we obtain
\[
M_\infty(f,r)\lesssim (1+r)^{\frac{1}{p}}M_p(f,r).
\]
The sharp circle-mean estimate in Theorem~\ref{thm:estimate} gives
\[
M_p(f,r)
\lesssim
(1+r)^{-\frac{1}{q}}e^{\frac{\alpha}{2}r^2}
\|f\|_{\mathcal F^{p,q}_\alpha}.
\]
Combining the last two estimates yields
\[
M_\infty(f,r)
\lesssim
(1+r)^{\frac1p-\frac1q}
e^{\frac{\alpha}{2}r^2}
\|f\|_{\mathcal F^{p,q}_\alpha}.
\]
Finally, with \(r=|z|\), we have \(|f(z)|\leq M_\infty(f,r)\), which proves
the claim.
\end{proof}

\subsection{Annular discretization of the mixed norm}
\label{subsec:annular-discretization}

The mixed norm has two levels: angular \(L^p\)-means on circles and radial
\(L^q\)-summability.  A second consequence of the local mass principle is that
the radial \(L^q\)-summability can be replaced, up to equivalence, by an outer
\(\ell^q\)-summation over the unit annuli
\[
A_k=\{z\in\C:k\le |z|<k+1\}.
\]
This annular model is used in the atomic decomposition, Carleson measure theory,
and duality.

\smallskip
\noindent
Recall that, for an entire function \(f\) and \(k\in\N_0\),
\[
B_p(f;k):=
\sup_{r\in[k,k+1)}
M_p(f,r)e^{-\frac{\alpha}{2}r^2}.
\]

\begin{lemma}[Local shell control]
\label{lem:window}
Let \(\alpha>0\), \(0<p\le\infty\), and \(0<q<\infty\). Then
\begin{equation}
\label{eq:window-B}
B_p^q(f;k)
\lesssim
\frac{1}{1+k}
\int_{[k-2,k+2]\cap[0,\infty)}
M_p^q(f,s)e^{-\frac{\alpha q}{2}s^2}s\,ds
\end{equation}
for every entire function \(f\) and every \(k\in\N_0\), where the  
constant depends   on \(\alpha\) and \(q\).
\end{lemma}

\begin{proof}
For \(k=0\), since \(r\mapsto M_p(f,r)\) is nondecreasing,
one has \(
B_p^q(f;0)
\le M_p^q(f,1).
\)
On the other hand,
\[
\int_1^2 M_p^q(f,s)e^{-\frac{\alpha q}{2}s^2}s\,ds
\ge
M_p^q(f,1)\int_1^2 e^{-\frac{\alpha q}{2}s^2}s\,ds .
\]
Since \([1,2]\subset [0,2]\), this gives \eqref{eq:window-B} for \(k=0\).

\smallskip
\noindent
It remains to consider \(k\ge1\); we may assume \(f\ne0\).  Fix
\(r\in[k,k+1)\) and put \(x=\log r\).  Applying
Lemma~\ref{lem:uniform-local-mass-convex} exactly as in the proof of
Theorem~\ref{thm:estimate}, with
\[
\Phi(x)=q\log M_p(f,e^x),\qquad
a=\frac{\alpha q}{2},\qquad b=2,\qquad c=\frac14,
\]
we obtain, since \(r\ge1\),
\[
M_p^q(f,r)e^{-\frac{\alpha q}{2}r^2}r
\lesssim
\int_{re^{-h_x}}^{re^{h_x}}
M_p^q(f,s)e^{-\frac{\alpha q}{2}s^2}s\,ds,
\qquad
h_x=\frac{e^{-x}}4=\frac1{4r}.
\]
It remains to locate the window.  Since \(h_x=\frac{1}{4r}\), the
inequalities \(e^{-t}\ge1-t\) and \(e^t\le1+t+t^2\) for
\(0<t\le \frac14\) give
\[
re^{-h_x}\ge r-\frac14 
\qquad \text{ and } \qquad
re^{h_x}\le r+\frac14+\frac1{16r}.
\]
Thus, for \(r\in[k,k+1)\),
\[
[re^{-h_x},re^{h_x}]
\subset [k-2,k+2]\cap[0,\infty).
\]
Since \(r\asymp 1+k\) on \([k,k+1)\), we conclude that
\[
M_p^q(f,r)e^{-\frac{\alpha q}{2}r^2}
\lesssim
\frac1{1+k}
\int_{[k-2,k+2]\cap[0,\infty)}
M_p^q(f,s)e^{-\frac{\alpha q}{2}s^2}s\,ds .
\]
Taking the supremum over \(r\in[k,k+1)\) gives \eqref{eq:window-B}.
\end{proof}

\begin{proposition}[Annular discretization]\label{prop:annular}
Let \(\alpha>0\) and $0 < p, q \le\infty$. Then 
\[
\|f\|_{\Fpq}^q
\;\asymp\;
\sum_{k=0}^\infty (1+k)\,B_p^q(f;k)
\qquad (0<q<\infty) \qquad \text{and} \qquad \|f\|_{\FpqA{p}{\infty}{\al}}=\sup_{k\ge 0} B_p(f;k)
\]
for every entire function $f$. In the case \(0<q<\infty\), the implicit
constants depend only on \(\alpha\) and \(q\).

\end{proposition}

\begin{proof}
The case $q = \infty$ follows from the definition of
$\norm{\cdot}_{\FpqA{p}{\infty}{\al}}$.  Let $0<q<\infty$.
Decomposing the radial integral into unit shells gives
\[
\|f\|_{\Fpq}^q
\asymp \sum_{k=0}^\infty\int_k^{k+1} M_p^q(f,r)\, e^{-\frac{\al q }{2}r^2}\,r\,dr
\le \sum_{k=0}^\infty B_p^q(f;k) \int_k^{k+1}r\,dr
\lesssim \sum_{k=0}^\infty (1+k)\,B_p^q(f;k).
\]
On the other hand, by Lemma~\ref{lem:window},
\[
(1+k)\,B_p(f;k)^q \lesssim  \int_{[k-2,k+2]\cap[0,\infty)} M_p^q(f,s)\, e^{-\frac{\al q }{2}s^2}\,s\,ds.
\]
Summing over $k\ge0$ and using the uniformly bounded overlap of the intervals
$[k-2,k+2]$, we obtain
\[
\sum_{k=0}^\infty (1+k)\,B_p(f;k)^q
\lesssim \int_0^\infty M_p^q(f,s)\, e^{-\frac{\al q }{2}s^2}\,s\,ds
= \frac{C_1}{q\alpha}\,\|f\|_{\Fpq}^q.
\]
\end{proof}

\subsection{Global reproducing formula on \texorpdfstring{\(\mathcal F_\alpha^{p,q}\)}{F alpha p,q}}
\label{subsec:kernel-projection}

Recall that
\[
K_{\alpha,w}(z)=e^{\alpha\overline w z},
\qquad
\kappa_{\alpha,w}(z)=
e^{\alpha\overline w z-\frac{\alpha}{2}|w|^2}.
\]
When the Gaussian parameter \(\alpha\) is fixed, we write
\(K_w=K_{\alpha,w}\) and \(\kappa_w=\kappa_{\alpha,w}\).

\begin{definition}[Fock projection]\label{def:Pal}
Define
\[
(P_{\alpha}f)(z):=\frac{\alpha}{\pi}\int_{\C} f(w)\,e^{\alpha z\overline w-\alpha|w|^2}\,dA(w),
\qquad z\in\C,
\]
whenever the integral is finite. For $R>0$, define the truncated operator
\[
(P_{\alpha,R}f)(z):=\frac{\alpha}{\pi}\int_{|w|\le R} f(w)\,e^{\alpha z\overline w-\alpha|w|^2}\,dA(w).
\]
\end{definition}

\begin{proposition}
[Reproducing formula]\label{prop:fock-reproducing}
Let $\alpha>0$ and $0<p,q\le \infty$.  For each
$f\in \mathcal F^{p,q}_\alpha$, the function $P_\alpha f$ is well defined and
entire, and $P_\alpha f=f$ on $\C$. Moreover,
\(
P_{\alpha,R}f \rightarrow f\) locally uniformly on \(\C\) as \(R\to\infty.
\)
\end{proposition}

\begin{proof}
Fix $f\in \mathcal F^{p,q}_\alpha$ and $K\Subset\C$, and let
$R_K:=\sup_{z\in K}|z|$.  By Lemma~\ref{lem:p-est},
\[
|f(w)|\lesssim \; \|f\|_{\mathcal F^{p,q}_\alpha}\, (1+|w|)^{\frac1p-\frac1q}\, e^{\frac{\alpha}{2}|w|^2} .
\]
Moreover, for $z\in K$ we have $\re(z\overline w)\le |z||w|\le R_K|w|
\le R_K^2+\frac{|w|^2}{4}$, and hence
\[
\left|e^{\alpha z\overline w-\alpha|w|^2}\right|
= e^{\alpha\Re(z\overline w)-\alpha|w|^2}
\le e^{\alpha R_K^2}e^{-\frac{3\alpha}{4}|w|^2}.
\]
Therefore,
\[
\left|f(w)e^{\alpha z\overline w-\alpha|w|^2}\right|
\lesssim \|f\|_{\mathcal F^{p,q}_\alpha} \,(1+|w|)^{\frac1p-\frac1q}e^{-\frac{\alpha}{4}|w|^2}\in L^1(\C),
\]
uniformly for $z\in K$.  Hence $P_\alpha f(z)$ is absolutely convergent for each
$z\in\C$, \(P_{\alpha,R}f\to P_\alpha f\) uniformly on $K$, and $P_\alpha f$ is
entire.

\smallskip
\noindent
Fix $0<\rho<1$ and set $f_\rho(z):=f(\rho z)$.  By Lemma~\ref{lem:p-est},
\[
|f_\rho(w)|^2e^{-\alpha|w|^2}
\lesssim \|f\|_{\mathcal F^{p,q}_\alpha}^2\,(1+\rho|w|)^{2\left(\frac1p-\frac1q\right)}
e^{-\alpha(1-\rho^2)|w|^2},
\]
and the right-hand side is integrable over $\C$. Thus $f_\rho\in\mathcal F^2_\alpha$.
The classical Hilbert--Fock reproducing formula gives \(P_\alpha f_\rho=f_\rho\).
Letting $\rho\uparrow1$ and using dominated convergence gives \(P_\alpha f=f\).
\end{proof}

\subsection{Ring packets and testing functions}
\label{subsec:packets-tests}

We introduce annulus-localized model functions used in sharpness and necessity
arguments.  For \(r\ge0\), the packet \(G_r\) is the normalized Fock kernel
centered on the positive real axis at distance \(r\).  Its mass is concentrated
near the circle of radius \(r\).

\smallskip
\noindent
The estimates below give the circle-mean profile of \(G_r\), its
mixed-norm size, and, as a consequence, the mixed-norm size of the reproducing
kernels \(K_w\).  They are used in the embedding theorem and in necessity
arguments.

\begin{definition}[Ring packets]\label{def:ring-packets}
For $r\ge 0$ define
\[
G_r(z):=e^{\alpha r z-\frac{\alpha}{2}r^2},\qquad z\in\C.
\]
\end{definition}

\begin{lemma} \label{lem:packet-profile}
Fix $\alpha>0$ and $0<p\le\infty$.
Then 
\begin{equation}\label{eq:packet-profile}
M_p(G_r,s)\,e^{-\frac{\alpha }{2}s^{2}} \asymp (1+\alpha r s)^{-\frac{1}{2p}}
e^{-\frac{\alpha }{2}(s-r)^2},
\qquad r,s\ge0.
\end{equation}
Moreover, if $r\ge1$, then
\begin{equation}\label{eq:packet-peak}
\sup_{|s-r|\le1}M_p(G_r,s)\,e^{-\frac{\alpha }{2}s^{2}} \asymp r^{-\frac{1}{p}}.
\end{equation}
In both estimates, the implicit constants depend only on \(\alpha\) and \(p\).
\end{lemma}

\begin{proof}
If $p=\infty$, then $M_\infty(G_r,s)=e^{-\frac{\alpha }{2}r^2}e^{\alpha r s}$, hence
$M_\infty(G_r,s)e^{-\frac{\alpha }{2}s^2} = \; e^{-\frac{\alpha }{2}(s-r)^2}$, which gives \eqref{eq:packet-profile}.
Assume $0<p<\infty$. A direct computation yields
\[
M_p^p(G_r,s)
= e^{-\frac{\alpha p}{2}r^2}\, I_0(p\alpha r s),
\]
where $I_{0}(t)=\frac{1}{2\pi}\int_{0}^{2\pi}e^{t\cos\theta}d\theta$ is the
modified Bessel function.  By the classical asymptotics for $I_0$ (see
\cite[\S 9.7]{AS}),
\[
I_0(t)\asymp \frac{e^t}{\sqrt{1+t}},\qquad t\ge0.
\]
Combining these identities gives \eqref{eq:packet-profile}.  The peak estimate
\eqref{eq:packet-peak} follows from
\eqref{eq:packet-profile} by noting that for $r\ge1$ and $|s-r|\le1$ one has
$s\asymp r$ and hence $(1+\alpha rs)^{-\frac{1}{2p}}\asymp r^{-\frac{1}{p}}$, while the Gaussian
factor stays between $e^{-\frac{\alpha}{2}}$ and $1$.
\end{proof}

\begin{proposition} \label{prop:T1}
Let  \(\alpha>0\) and $0<p,q\le\infty$. Then
\[
\|G_r\|_{\Fpq}\asymp (1+r)^{\frac1q-\frac1p},
\qquad r\ge0.
\]
In particular, 
\[
\|K_w\|_{\Fpq} \asymp (1+|w|)^{\frac1q-\frac1p} e^{\frac{\alpha}{2} |w|^2}, \ w \in \C.
\]
In both estimates, the implicit constants depend only on \(\alpha,p\), and \(q\).
\end{proposition}

\begin{proof}
By the definition of $\Fpq$,
\[
\|G_r\|_{\Fpq}^q=\int_0^\infty M_p^q(G_r,s)e^{-\frac{\alpha q }{2}s^2}\, s\,ds\quad (0<q<\infty),
\qquad
\|G_r\|_{\FpqA{p}{\infty}{\al}}=\sup_{s\ge0}M_p(G_r,s)e^{-\frac{\alpha }{2}s^2}.
\]
If $0\le r\le 2$, then \eqref{eq:packet-profile} implies $\|G_r\|_{\Fpq}\asymp 1 \asymp (1+r)^{\frac1q-\frac1p}$.
Assume $r\ge 2$. 
If $0<q<\infty$, then \eqref{eq:packet-profile} yields
\[
\|G_r\|_{\Fpq}^q \asymp \int_0^{\infty} (1+\alpha rs)^{-\frac{q}{2p}} e^{
-\frac{\alpha q }{2}(s-r)^2} s \, ds.
\]
On the unit ring $|s-r|\le1$ we have $s\asymp r$ and $(1+\alpha rs)^{-\frac{1}{2p}}\asymp r^{-\frac{1}{p}}$,
so
\[
\int_{|s-r|\le1} (1+\alpha rs)^{-\frac{q}{2p}} e^{
-\frac{\alpha q }{2}(s-r)^2} s \, ds \; \asymp \; r^{-\frac{q}{p}}\int_{|s-r|\le1}s\,ds \; \asymp \; r^{1-\frac{q}{p}}.
\]
For the complementary region $|s-r|>1$,
\begin{align*}
\int^{\infty}_{r+1}  (1+\alpha rs)^{-\frac{q}{2p}} e^{
-\frac{\alpha q }{2}(s-r)^2} s \, ds \;    \lesssim  \; r^{-\frac{q}{p}} \int_1^{\infty} (r + t) e^{-\frac{\alpha q }{2}t^2} \, dt \; \lesssim \; r^{1 - \frac{q}{p}}, 
\end{align*}
where we used $1+\alpha r(r+t)\gtrsim r^2$ for $t\ge1$ and the change of
variables $t=s-r$; moreover,
\begin{align*}
  \int_{0}^{r-1}  (1+\alpha rs)^{-\frac{q}{2p}} e^{
-\frac{\alpha q}{2}(s-r)^2} s \, ds \; & = \; \left( \int_{0}^{\frac{r}{2}} + \int^{r-1}_{\frac{r}{2}} \right) (1+\alpha rs)^{-\frac{q}{2p}} e^{
-\frac{\alpha q }{2}(s-r)^2} s \, ds  \\
& \lesssim \;  \; r^2 e^{-\frac{\al q }{8}r^2} + r^{1- \frac{q}{p}} \int^{r-1}_{\frac{r}{2}} e^{
-\frac{\alpha q }{2}(s-r)^2} \, ds \;
\lesssim \; r^{1-\frac{q}{p}},
\end{align*}
where $e^{
-\frac{\alpha q }{2}(s-r)^2} \le e^{
-\frac{\alpha q }{8}r^2}$ and $(1+\alpha rs)^{-\frac{q}{2p}}\le 1$ on $[0, \frac{r}{2}]$ and $s(1+\alpha rs)^{-\frac{q}{2p}}\lesssim r^{1-\frac{q}{p}}$ on $[\frac{r}{2},r-1]$. 
Thus, 
\[
\|G_r\|_{\Fpq}^q \; \asymp \; \int_0^{\infty} (1+\alpha rs)^{-\frac{q}{2p}} e^{
-\frac{\alpha q }{2}(s-r)^2} s \, ds \; \asymp \; (1 + r)^{1-\frac{q}{p}},
\]
which gives the desired estimate for $0 < q < \infty$.

\medskip
\noindent
For $q=\infty$, if $p = \infty$, then \eqref{eq:packet-profile} gives
$\|G_r\|_{\FpqA{\infty}{\infty}{\al}} = 1$.  Assume $0 < p < \infty$.  Then
\eqref{eq:packet-profile} yields
\[
\|G_r\|_{\FpqA{p}{\infty}{\al}} \asymp \sup_{s \geq 0} (1+\alpha rs)^{-\frac{1}{2p}} e^{
-\frac{\alpha }{2}(s-r)^2}.
\]
As above, on the unit ring $|s-r|\le1$,
\[
\sup_{|s-r|\le1} (1+\alpha rs)^{-\frac{1}{2p}} e^{
-\frac{\alpha }{2}(s-r)^2}  \; \asymp \; r^{-\frac{1}{p}},
\]
and for the complementary region $|s-r|>1$,
\[
\sup_{s \geq r + 1} (1+\alpha rs)^{-\frac{1}{2p}} e^{
-\frac{\alpha }{2}(s-r)^2}  \; \lesssim \; r^{-\frac{1}{p}},
\]
and
\[
\sup_{0 \leq s \leq \frac{r}{2}} (1+\alpha rs)^{-\frac{1}{2p}} e^{
-\frac{\alpha }{2}(s-r)^2}  \; \lesssim \; e^{
-\frac{\alpha }{8}r^2} \lesssim r^{-\frac{1}{p}} \quad \text{and} \quad \sup_{\frac{r}{2} \leq s \leq r - 1} (1+\alpha rs)^{-\frac{1}{2p}} e^{
-\frac{\alpha }{2}(s-r)^2}  \; \lesssim \; r^{-\frac{1}{p}}.
\]
These estimates prove the case $q = \infty$.
The kernel estimate follows from
\[
K_w(z)=e^{\frac{\alpha}{2}|w|^2}\,G_{|w|}\!\bigl(z\,e^{- i\,\arg w}\bigr), \qquad w \neq 0.
\]
\end{proof}

\begin{remark}
The packets \(G_r\) also show the order sharpness of
Theorem~\ref{thm:estimate}.  By Lemma~\ref{lem:packet-profile},
\[
M_p(G_r,r)e^{-\frac{\alpha}{2}r^2}\asymp (1+r)^{-\frac{1}{p}},
\]
while Proposition~\ref{prop:T1} gives
\[
\|G_r\|_{\Fpq}\asymp (1+r)^{\frac{1}{q}-\frac{1}{p}}.
\]
Hence
\[
\frac{M_p(G_r,r)}{\|G_r\|_{\Fpq}}
\asymp
(1+r)^{-\frac{1}{q}}e^{\frac{\alpha}{2}r^2},
\qquad r\ge0.
\]
The monomial sharpness in Lemma~\ref{lem-sh} will be useful later as a
separate test.
\end{remark}

\subsection{Structure and density of polynomials}
\label{subsec:structure-density}

We record the quasi-Banach structure of \(\Fpq\) and the polynomial
approximation properties used later.  The structure is standard; see also
\cite[Lemma~2.1]{FT23}.  The proof of Lemma~\ref{lem-norm} records the precise
\(s\)-triangle inequality in the present normalization.

\begin{lemma} \label{lem-norm}
Let $\alpha>0$ and $0<p,q\le\infty$. Set $s: = \min \{1, p, q\}$. Then, for all $f,g\in \mathcal F^{p,q}_\alpha$,
\begin{equation}\label{eq:s-triangle-compact}
\|f+g\|_{\Fpq}^{\,s}\le \|f\|_{\Fpq}^{\,s}+\|g\|_{\Fpq}^{\,s}.
\end{equation}
In particular, $\mathcal F^{p,q}_\alpha$ is Banach when $p,q\ge1$, and
otherwise complete quasi-Banach.
\end{lemma}

\begin{proof}
\noindent\emph{Quasi-triangle inequality.}
Set $s_1:=\min\{1,p\}$, so $s =\min\{s_1,q\}$.  Then
\begin{equation}\label{eq:inner-a-subadd}
\Mp^{s_1}(f+g,r)\le \Mp^{s_1}(f,r)+\Mp^{s_1}(g,r), \quad r \geq 0.
\end{equation}

\smallskip
\noindent If $q = \infty$, then $s= s_1$, and by \eqref{eq:inner-a-subadd},
\[
\|f+g\|_{\mathcal F^{p,\infty}_\alpha}^{s} 
\le \sup_{r \geq 0}
\Mp^{s}(f,r) e^{-\frac{\alpha s}{2} r^2} + \sup_{r \geq 0}
\Mp^{s}(g,r) e^{-\frac{\alpha s}{2} r^2} 
= \|f\|_{\mathcal F^{p,\infty}_\alpha}^{s}
+\|g\|_{\mathcal F^{p,\infty}_\alpha}^{s}.
\]

\smallskip
\noindent
Assume $0 < q < \infty$.
If $q\le s_1$, then, $s = q$ and by \eqref{eq:inner-a-subadd},
\begin{align*}
\|f+g\|_{\Fpq}^q
& \leq \int_0^\infty \big(\Mp^{s_1}(f,r)+\Mp^{s_1}(g,r)\big)^{\frac{q}{s_1}}\,d\lambda_{q\alpha}(r) \\
& \le \int_0^\infty \big(\Mp^{q}(f,r)+\Mp^{q}(g,r)\big)\,d\lambda_{q\alpha}(r) = \|f\|_{\Fpq}^q+\|g\|_{\Fpq}^q.
\end{align*}
If $q > s_1$, then $s = s_1$ and, by \eqref{eq:inner-a-subadd} and Minkowski inequality in $L^{\frac{q}{s_1}}(d\lambda_{q\alpha})$, we get 
\begin{align*}
  \|f+g\|_{\Fpq}^{s_1}& = \left\|\Mp^{s_1}(f + g,\cdot)\right\|_{L^{\frac{q}{s_1}}\left(\R^+, \, d\lambda_{q\alpha}\right)}  \le \left\|\Mp^{s_1}(f,\cdot) + \Mp^{s_1}(g,\cdot)\right\|_{L^{\frac{q}{s_1}}\left(\R^+, \, d\lambda_{q\alpha}\right)}\\
  & \leq \left\|\Mp^{s_1}(f,\cdot)\right\|_{L^{\frac{q}{s_1}}\left(\R^+, \, d\lambda_{q\alpha}\right)} + \left\| \Mp^{s_1}(g,\cdot)\right\|_{L^{\frac{q}{s_1}}\left(\R^+, \, d\lambda_{q\alpha}\right)} = \|f\|_{\Fpq}^{s_1}+\|g\|_{\Fpq}^{s_1}.  
\end{align*}

\smallskip
\noindent
\emph{Completeness.}
Let $(f_n)_{n \geq 1}$ be Cauchy in $\mathcal F^{p,q}_\alpha$. By \eqref{eq:s-triangle-compact},
$d(f,g):=\|f-g\|_{\Fpq}^{\,s}$ defines a translation-invariant metric on $\mathcal F^{p,q}_\alpha$,
and $(f_n)_{n \geq 1}$ is Cauchy with respect to $d$.
Fix $R>0$. By Lemma~\ref{lem:p-est}, for $|z|\le R$ and all $m,n$,
\[
|f_n(z)-f_m(z)|
\lesssim (1+R)^{\frac1p-\frac1q}e^{\frac{\alpha}{2}R^2}\,\|f_n-f_m\|_{\Fpq}.
\]
Hence $(f_n)_{n \geq 1}$ is Cauchy uniformly on $\overline{B(0,R)}$.  Since
$R>0$ is arbitrary, \(f_n\to f\) locally uniformly on \(\C\) for some entire
function \(f\).

\smallskip
\noindent
Fix $n\geq 1$ and set $g_{n,m}:=f_n-f_m$ and $g_n:=f_n-f$.
For each $r>0$, local uniform convergence implies $g_{n,m}\to g_n$ uniformly on $|z|=r$, and hence
$M_p(g_{n,m},r)\to M_p(g_n,r)$ as $m\to\infty$.
If $0<q<\infty$, Fatou's lemma yields
\[
\|g_n\|_{\Fpq}^q
\le \liminf_{m\to\infty}\|g_{n,m}\|_{\Fpq}^q
= \liminf_{m\to\infty}\|f_n-f_m\|_{\Fpq}^q.
\]
If $q=\infty$, taking the supremum over $r>0$ gives
\[
\|g_n\|_{\FpqA{p}{\infty}{\alpha}}\le \liminf_{m\to\infty}\|f_n-f_m\|_{\FpqA{p}{\infty}{\alpha}}.
\]
Since $(f_n)_n$ is Cauchy, the right-hand side tends to $0$ as $n\to\infty$, so
$\|f_n-f\|_{\Fpq}\to 0$.

\smallskip
\noindent
When $p,q\ge1$ we take $s=1$, so $\mathcal F^{p,q}_\alpha$ is Banach;
otherwise it is complete quasi-Banach.
\end{proof}

\medskip
\noindent
The following lemma was proved in \cite[Lemma~2.2]{FT23}.
\begin{lemma} \label{lem-fr}
Let $\alpha>0$ and $0<p,q\le\infty$, and $f\in\mathcal F^{p,q}_\alpha$. For $0< \tau <1$ set
$f_\tau(z)=f(\tau z)$.
\begin{itemize}
\item[\textup{(a)}] $\|f_\tau\|_{\Fpq}\to \|f\|_{\Fpq}$ as $\tau\to1^-$. Moreover, if $q<\infty$, then
$\|f_\tau -f\|_{\Fpq}\to0$ as $\tau\to1^-$.
\item[\textup{(b)}] For each fixed $0<\tau<1$, the Taylor polynomials of $f_\tau$ converge to $f_\tau$ in
$\mathcal F^{p,q}_\alpha$.
\end{itemize}
\end{lemma}

\medskip
\noindent
This gives the density statement.

\begin{corollary} \label{cor:density-poly}
Let $\alpha>0$, $0<p\le\infty$, and $0<q<\infty$. Then the set of complex polynomials is dense in
$\mathcal F^{p,q}_\alpha$ and $f^{p, \infty}_{\alpha}$.
\end{corollary}

\subsection{Auxiliary operator estimates}
\label{subsec:aux-operator-estimates}

We finish the section with two auxiliary operator estimates: a generalized
Schur test adapted to the outer \(\ell^q\)-structure of the annular coefficient
spaces, and a Neumann-series inversion lemma in complete quasi-normed spaces.

\subsubsection{Generalized Schur tests for outer \texorpdfstring{\(\ell^q\)}{ell q}}
\label{subsubsec:schur-tests}

We record the discrete estimates leading to a Schur-type bound for outer
\(\ell^q\)-norms.

\begin{lemma} \label{lem:peetre-weighted}
For all $\ell,k\in\mathbb N_0$ one has
\begin{equation}\label{eq:peetre}
\frac{1+\ell}{1+k}\le 1+|\ell-k|.
\end{equation}
Moreover, for $0<q<\infty$,
\begin{equation}\label{eq:peetre-q}
(1+k)^{-\frac{1}{q}}\le (1+\ell)^{-\frac{1}{q}}\,(1+|\ell-k|)^{\frac{1}{q}},
\qquad
(1+\ell)^{\frac{1}{q}}\le (1+k)^{\frac{1}{q}}\,(1+|\ell-k|)^{\frac{1}{q}}.
\end{equation}
\end{lemma}

\begin{proof}
Since $(1+\ell)\le (1+k)(1+|\ell-k|)$, \eqref{eq:peetre} follows by dividing by $(1+k)$.
The inequalities in \eqref{eq:peetre-q} follow by applying the monotonicity of
$t\mapsto t^{\pm \frac{1}{q}}$ on $(0,\infty)$ to \eqref{eq:peetre}.
\end{proof}

\begin{lemma} \label{lem:exp-rowsum}
Let $\beta>0$ and $a>0$. For $\gamma\ge0$ define
\[
S_\gamma(a):=\sum_{j\in\mathbb Z}(1+|j|)^\gamma e^{-a|j|^\beta}.
\]
Then $S_\gamma(a)<\infty$ for every $\gamma\ge0$. In particular, for every $k\in\mathbb N_0$,
\begin{equation}\label{eq:rowsum}
\sum_{\ell\ge0}(1+\ell)e^{-a|\ell-k|^\beta}\ \le\ C(a,\beta)\,(1+k),
\end{equation}
where one may take $C(a,\beta)=S_1(a)$.
\end{lemma}

\begin{proof}
The finiteness of $S_\gamma(a)$ follows from exponential decay.
For \eqref{eq:rowsum}, by \eqref{eq:peetre} we have $(1+\ell)\le (1+k)(1+|\ell-k|)$, hence
\[
\sum_{\ell\ge0}(1+\ell)e^{-a|\ell-k|^\beta}
\le (1+k)\sum_{\ell\ge0}(1+|\ell-k|)e^{-a|\ell-k|^\beta}.
\]
Changing variables $j=\ell-k$ and enlarging to $j\in\mathbb Z$ yields
\[
\sum_{\ell\ge0}(1+|\ell-k|)e^{-a|\ell-k|^\beta}
\le \sum_{j\in\mathbb Z}(1+|j|)e^{-a|j|^\beta}
=S_1(a),
\]
which proves \eqref{eq:rowsum}.
\end{proof}

\begin{lemma}[Outer Schur test]\label{lem:gen-schur-test}
Let $0<q\le \infty$, $0 < \beta < \infty$, and let $K:\mathbb N_0\times\mathbb N_0\to[0,\infty)$ satisfy, for some
constants $C_0,c>0$,
\[
K(\ell,k)\ \le\ C_0\,e^{-c|\ell-k|^{\beta}}\,(1+k)^{-\frac{1}{q}}, \qquad \ell,k\in\mathbb N_0.
\]
Define, for a nonnegative sequence $Y=\{Y_k\}_{k\ge0}$,
\[
(TY)_\ell:=\sum_{k\ge0}K(\ell,k)\,Y_k,\qquad \ell\in\mathbb N_0.
\]
Then there exists a constant $C=C(C_0,c,\beta, q)$ such that
\begin{equation}\label{eq:B2-qfinite}
\sum_{\ell\ge0}(1+\ell)\,(TY_\ell)^{q}\ \le\ C\,\sum_{k\ge0}Y_k^{\,q}\qquad(0<q<\infty),
\end{equation}
and at the endpoint
\begin{equation}\label{eq:B2-qinfty}
\sup_{\ell\ge0}\,TY_\ell\ \le\ C\,\sup_{k\ge0}\,Y_k\qquad(q=\infty).
\end{equation}
\end{lemma}

\begin{proof}
Throughout the proof, \(C\) denotes a positive constant which may change from
line to line and depends only on \(C_0,c,q\), and \(\beta\). Since the kernel
\(K\) is nonnegative, it is enough to consider nonnegative sequences \(Y\).

\medskip
\noindent\emph{Case 1: $0<q\le 1$.}
Since $t\mapsto t^q$ is subadditive on $[0,\infty)$,
\[
\bigl(TY_\ell\bigr)^q
=\Big(\sum_{k\ge0}K(\ell,k)Y_k\Big)^q
\le \sum_{k\ge0}K^q(\ell,k)\,Y_k^{\,q}.
\]
By the kernel hypothesis,
\[
K^q(\ell,k)\le C_0^q\,e^{-cq|\ell-k|^\beta}(1+k)^{-1}.
\]
Hence, by Lemma~\ref{lem:exp-rowsum} (with $a=cq$),
\[
\sum_{\ell\ge0}(1+\ell)\bigl(TY_\ell\bigr)^q
\le C_0^q\sum_{k\ge0}(1+k)^{-1}Y_k^{\,q}\sum_{\ell\ge0}(1+\ell)e^{-cq|\ell-k|^\beta}
\le C\,\sum_{k\ge0}Y_k^{\,q}.
\]

\medskip
\noindent\emph{Case 2:  $1<q<\infty$.}
Fix $\ell\in\mathbb N_0$.  By H\"older's inequality with exponents
$(q,\frac{q}{q-1})$,
\begin{equation}\label{eq:holder-appA}
\bigl(TY_\ell\bigr)^q
=\Big(\sum_{k\ge0}K(\ell,k)Y_k\Big)^q
\le \Big(\sum_{k\ge0}K(\ell,k)\Big)^{q-1}\sum_{k\ge0}K(\ell,k)\,Y_k^{\,q}.
\end{equation}
Let $R(\ell):=\sum_{k\ge0}K(\ell,k)$. Using Lemma~\ref{lem:peetre-weighted} and Lemma~\ref{lem:exp-rowsum},
\[
R(\ell)\le C_0(1+\ell)^{-\frac{1}{q}}\sum_{k\ge0}e^{-c|\ell-k|^\beta}(1+|\ell-k|)^{\frac{1}{q}}
\le C_0(1+\ell)^{-\frac{1}{q}}\,S_{\frac{1}{q}}(c).
\]
Thus $R^{q-1}(\ell)\le C\,(1+\ell)^{-\frac{q-1}{q}}$ with $C=C_0^{q-1}S^{q-1}_{\frac{1}{q}}(c)$.
Multiplying \eqref{eq:holder-appA} by $(1+\ell)$ and summing over $\ell$ gives
\[
\sum_{\ell\ge0}(1+\ell)\bigl(TY_\ell\bigr)^q
\le C\sum_{k\ge0}Y_k^{\,q}\sum_{\ell\ge0}(1+\ell)^{\frac{1}{q}}K(\ell,k).
\]
To estimate the remaining $\ell$-sum, use Lemma~\ref{lem:peetre-weighted} and \eqref{eq:peetre-q}:
\[
\sum_{\ell\ge0}(1+\ell)^{\frac{1}{q}}K(\ell,k)
\le C_0(1+k)^{-\frac{1}{q}}\sum_{\ell\ge0}e^{-c|\ell-k|^\beta}(1+\ell)^{\frac{1}{q}}
\le C_0\sum_{\ell\ge0}e^{-c|\ell-k|^\beta}(1+|\ell-k|)^{\frac{1}{q}}
\le C_0\,S_{\frac{1}{q}}(c).
\]
This proves \eqref{eq:B2-qfinite} for $1<q<\infty$.

\medskip
\noindent\emph{Case 3: $q=\infty$.}
Then
\[
\sup_{\ell \geq 0} \, TY_\ell
\ \le \ C_0 \, \sup_{k \geq 0} Y_k \sum_{k\ge0}e^{-c|\ell-k|^\beta} \ \leq \ C_0 S_0(c) \, \sup_{k \geq 0} Y_k, \ \text{where} \quad S_0(c): = \sum_{j \in \Z}e^{-c|j|^\beta} < \infty,
\]
which gives \eqref{eq:B2-qinfty}.
\end{proof}

\subsubsection{Neumann series in complete quasi-normed spaces}
\label{subsubsec:neumann-quasi}

\begin{lemma} \label{lem:neumann-quasi}
Let $(X,\|\cdot\|_{X})$ be a complete quasi-normed complex vector space. Assume that there exists
$s\in(0,1]$ such that
\[
\|x+y\|_{X}^{s}\le \|x\|_{X}^{s}+\|y\|_{X}^{s},
\qquad x,y\in X.
\]
Let $T$ be a bounded linear operator on $X$, and let
\[
\|T\|_{X\to X}:=\sup_{x\neq 0}\frac{\|Tx\|_{X}}{\|x\|_{X}},
\]
and assume that $\|T\|_{X\to X}<1$.
Then $I-T$ is invertible on $X$, and the Neumann series
\(\sum_{n=0}^{\infty}T^{n}\) converges in the induced operator quasi-norm to
\((I-T)^{-1}\). Moreover,
\[
\|(I-T)^{-1}\|_{X\to X}\le \left(1-\|T\|_{X\to X}^{s}\right)^{-1/s}.
\]
\end{lemma}

\begin{proof}
Put $\rho:=\|T\|_{X\to X}\in(0,1)$ and let $S_N:=\sum_{n=0}^{N}T^{n}$.
For $N>M\ge0$ and $x\in X$, the $s$-triangle inequality gives
\[
\|(S_N-S_M)x\|_X^{s}
=\Big\|\sum_{n=M+1}^{N}T^{n}x\Big\|_X^{s}
\le \sum_{n=M+1}^{N}\|T^{n}x\|_X^{s}
\le \|x\|_X^{s}\sum_{n=M+1}^{N}\rho^{ns}.
\]
Taking the supremum over $\|x\|_X=1$ yields
\[
\|S_N-S_M\|_{X\to X}^{s}\le \sum_{n=M+1}^{N}\rho^{ns}.
\]
Hence $(S_N)_{N \geq 0}$ is Cauchy in the operator quasi-norm.  For each fixed
$x$, completeness of $X$ gives \(Sx=\lim_N S_Nx\).
Letting $N\to\infty$ in the previous estimate gives
\[
\|S-S_M\|_{X\to X}^{s}\le \sum_{n=M+1}^{\infty}\rho^{ns}\xrightarrow[M\to\infty]{}0,
\]
so $S_N\to S$ in the operator quasi-norm.  In particular,
\[
\|S\|_{X\to X}^{s}\le \sum_{n=0}^{\infty}\|T^{n}\|_{X\to X}^{s}
\le \sum_{n=0}^{\infty}\rho^{ns}=\frac{1}{1-\rho^{s}}.
\]

\smallskip
\noindent
For each $N$ we have $(I-T)S_N=S_N(I-T)=I-T^{N+1}$.  Since
$\|T^{N+1}\|_{X\to X}\le \rho^{N+1}\to0$, passing to the limit yields
$(I-T)S=S(I-T)=I$. Thus $S=(I-T)^{-1}$.
\end{proof}

\section{Embeddings in the Gaussian mixed-norm scale}
\label{sec:embeddings}

We prove the embedding theorem for the Gaussian mixed-norm scale and
derive a Littlewood-type consequence for randomized entire functions.  The
proof uses the circle-mean growth estimate, the angular comparison
estimate, and the kernel-size estimates from
Section~\ref{sec:mixed-fock-tools}; these isolate the Gaussian parameter,
the outer radial exponent, and the angular--radial balance.

\begin{proof}[Proof of Theorem~\ref{thm:embeddings}]
We write
\(
\mathcal F_{\alpha}^{p_1,q_1}\hookrightarrow
\mathcal F_{\beta}^{p_2,q_2}
\)
for continuous inclusion.

\smallskip
\noindent
\emph{Sufficiency.}
Assume that one of \emph{(i)}--\emph{(ii)} holds. We show that
\[
\|f\|_{\mathcal F_{\beta}^{p_2,q_2}}
\lesssim
\|f\|_{\mathcal F_{\alpha}^{p_1,q_1}},
\qquad f\in \mathcal F_{\alpha}^{p_1,q_1}.
\]

\smallskip
\noindent\emph{(i)} $\alpha<\beta$.
For $f\in\mathcal F^{p_1,q_1}_\alpha$, Lemma~\ref{lem:p-est} gives
\[
M_\infty(f,r)\ \lesssim\ (1+r)^{\frac1{p_1}-\frac1{q_1}}e^{\frac{\alpha}{2}r^2}\,
\|f\|_{\FpqA{p_1}{q_1}{\alpha}}.
\]
Hence, since $M_{p_2}\le M_\infty$,
\[
M_{p_2}(f,r)e^{-\frac{\beta}{2}r^2}
\lesssim (1+r)^{\frac1{p_1}-\frac1{q_1}}e^{-\frac{\beta-\alpha}{2}r^2}\,
\|f\|_{\FpqA{p_1}{q_1}{\alpha}}.
\]
The last factor is integrable to any finite power and is bounded for
$q_2=\infty$, because $\beta-\alpha>0$. Thus
\[
\|f\|_{\FpqA{p_2}{q_2}{\beta}}
\lesssim \|f\|_{\FpqA{p_1}{q_1}{\alpha}}.
\]

\smallskip
\noindent\emph{(ii)} $\alpha=\beta$.
Assume $q_1\le q_2$ and $\frac1{q_2}-\frac1{p_2}\le \frac1{q_1}-\frac1{p_1}$.

\smallskip
\noindent\emph{Case 1: $p_2\le p_1$.}
Then $M_{p_2}\le M_{p_1}$. If $q_2=\infty$, Theorem~\ref{thm:estimate} gives
\[
\|f\|_{\FpqA{p_2}{\infty}{\alpha}}
\le \sup_{r\ge0}M_{p_1}(f,r)e^{-\frac{\alpha}{2}r^2}
\lesssim \|f\|_{\FpqA{p_1}{q_1}{\alpha}}.
\]
For $q_2<\infty$, use
\[
M_{p_1}(f,r)^{q_2}e^{-\frac{\alpha q_2}{2}r^2}
=\left(M_{p_1}(f,r)e^{-\frac{\alpha}{2}r^2}\right)^{q_2-q_1}\,
\left(M_{p_1}(f,r)^{q_1}e^{-\frac{\alpha q_1}{2}r^2}\right),
\]
and the corresponding supremum bound from Theorem~\ref{thm:estimate} to obtain
\[
\|f\|_{\FpqA{p_2}{q_2}{\alpha}}^{q_2}
\lesssim \|f\|_{\FpqA{p_1}{q_1}{\alpha}}^{q_2-q_1}\,
\|f\|_{\FpqA{p_1}{q_1}{\alpha}}^{q_1}
=\|f\|_{\FpqA{p_1}{q_1}{\alpha}}^{q_2}.
\]

\smallskip
\noindent\emph{Case 2: $p_1<p_2$.}
Recall that 
\[
M_{p_2}(f,r)\ \lesssim\ (1+r)^{\frac1{p_1}-\frac1{p_2}}\,M_{p_1}(f,r).
\]
If $q_2=\infty$, Theorem~\ref{thm:estimate} yields
\[
M_{p_2}(f,r)e^{-\frac{\alpha}{2}r^2}
\lesssim (1+r)^{\frac1{p_1}-\frac1{p_2}-\frac1{q_1}}\,\|f\|_{\FpqA{p_1}{q_1}{\alpha}},
\]
and the hypothesis is precisely $\frac1{p_1}-\frac1{p_2}-\frac1{q_1}\le0$. Hence
$\|f\|_{\FpqA{p_2}{\infty}{\alpha}}\lesssim \|f\|_{\FpqA{p_1}{q_1}{\alpha}}$.

\smallskip
\noindent
If $0<q_2<\infty$, then
\begin{align*}
\|f\|_{\FpqA{p_2}{q_2}{\alpha}}^{q_2}
&\lesssim \int_0^\infty M_{p_1}(f,r)^{q_2}(1+r)^{q_2(\frac1{p_1}-\frac1{p_2})}
e^{-\frac{\alpha q_2}{2}r^2}\,rdr\\
&=\int_0^\infty \left(M_{p_1}(f,r)^{q_1}e^{-\frac{\alpha q_1}{2}r^2}\right)
\left(M_{p_1}(f,r)e^{-\frac{\alpha}{2}r^2}\right)^{q_2-q_1}
(1+r)^{q_2\left(\frac1{p_1}-\frac1{p_2}\right)}\,rdr .
\end{align*}
Theorem~\ref{thm:estimate} gives
$M_{p_1}(f,r)e^{-\frac{\alpha}{2} r^2}\lesssim (1+r)^{-\frac{1}{q_1}}\|f\|_{\FpqA{p_1}{q_1}{\alpha}}$, hence
\[
\|f\|_{\FpqA{p_2}{q_2}{\alpha}}^{q_2}
\lesssim \|f\|_{\FpqA{p_1}{q_1}{\alpha}}^{q_2-q_1}
\int_0^\infty M_{p_1}(f,r)^{q_1}e^{-\frac{\alpha q_1}{2}r^2}(1+r)^{\sigma}\,rdr,
\]
where $\sigma:=q_2(\frac1{p_1}-\frac1{p_2})-\frac{q_2-q_1}{q_1}$.
The hypothesis is equivalent to $\sigma\le0$, so $(1+r)^\sigma\le1$ and
\[
\|f\|_{\FpqA{p_2}{q_2}{\alpha}}^{q_2}
\lesssim \|f\|_{\FpqA{p_1}{q_1}{\alpha}}^{q_2-q_1}\,
\|f\|_{\FpqA{p_1}{q_1}{\alpha}}^{q_1}
=\|f\|_{\FpqA{p_1}{q_1}{\alpha}}^{q_2}.
\]

\medskip
\noindent\emph{Necessity.}
Assume that the inclusion \(\mathcal F_{\alpha}^{p_1,q_1} \subset
\mathcal F_{\beta}^{p_2,q_2}\) is continuous, so
\begin{equation}\label{eq:emb-cont}
\|f\|_{\mathcal F_{\beta}^{p_2,q_2}}
\lesssim
\|f\|_{\mathcal F_{\alpha}^{p_1,q_1}},
\qquad f\in\mathcal F_{\alpha}^{p_1,q_1}.
\end{equation}

\smallskip
\noindent\emph{(a)} $\alpha\le\beta$.
The monomial norm estimate from the proof of Lemma~\ref{lem-sh} is
\[
\|z^n\|_{\FpqA{p}{q}{\alpha}}\asymp \Big(\frac{n}{e\alpha}\Big)^{\frac n2}n^{\frac1{2q}},
\qquad n\to\infty.
\] 
Testing \eqref{eq:emb-cont} on $z^n$ gives
\[
\Big(\frac{\alpha}{\beta}\Big)^{\frac n2}n^{\frac1{2q_2}-\frac1{2q_1}}\lesssim 1,
\]
which forces $\alpha\le\beta$.

\smallskip
\noindent\emph{(b)} If $\alpha=\beta$, then $q_1\le q_2$.
If $q_2<q_1$, the preceding test contradicts \eqref{eq:emb-cont}.

\smallskip
\noindent\emph{(c)} If $\alpha=\beta$ and $q_1\le q_2$, then
$\frac1{q_2}-\frac1{p_2}\le \frac1{q_1}-\frac1{p_1}$.
For $K_w(z):=e^{\alpha\overline w z}$, Proposition~\ref{prop:T1} gives
\[
\|K_w\|_{\FpqA{p}{q}{\alpha}}\asymp (1+|w|)^{\frac1q-\frac1p}e^{\frac{\alpha}{2}|w|^2}.
\]
Applying \eqref{eq:emb-cont} to $K_w$ and letting $|w|\to\infty$ yields
$\frac1{q_2}-\frac1{p_2}\le \frac1{q_1}-\frac1{p_1}$.

\smallskip
\noindent
This completes the proof.
\end{proof}

\begin{remark}
The necessity proof separates two mechanisms.  The monomial test detects
the Gaussian parameter and the outer radial exponent, while the reproducing
kernels detect the angular--radial balance:
\[
\|K_w\|_{\mathcal F_\alpha^{p,q}}
\asymp
(1+|w|)^{\frac1q-\frac1p}e^{\frac{\alpha}{2}|w|^2}
\]
identifies \(\frac1q-\frac1p\) as the relevant defect.
\end{remark}

\medskip
\noindent
We record the corresponding Littlewood-type statement.  The analogue
for mixed-norm spaces on the unit disk is due to~\cite{K22}.

\begin{definition}[Standard random sequences]\label{def:standard-random}
A random variable $X$ is called \emph{Bernoulli} if $\mathbb P(X=1)=\mathbb P(X=-1)=\tfrac12$, and
\emph{Steinhaus} if it is uniformly distributed on the unit circle.
We write $N(0,1)$ for the law of a (real) Gaussian random variable with mean $0$ and variance $1$.
A \emph{standard Bernoulli sequence} is a sequence $(\varepsilon_n)_{n\ge0}$ of i.i.d.\ Bernoulli variables.
A \emph{standard Steinhaus sequence} is a sequence $(e^{2\pi i \theta_n})_{n\ge0}$, where $(\theta_n)_{n\ge0}$
are i.i.d.\ and uniformly distributed on $[0,1)$.
A \emph{standard Gaussian sequence} is a sequence $(\xi_n)_{n\ge0}$ of i.i.d.\ $N(0,1)$ variables.
We call any of these three a \emph{standard random sequence}, denoted generically by $(X_n)_{n\ge0}$.
\end{definition}

\smallskip
\noindent
For an entire function $f$ with Taylor series $f(z)=\sum_{n=0}^{\infty}a_n z^n$, define its \emph{randomization}
\[
(\mathcal R f)(z):=\sum_{n=0}^{\infty} a_n X_n z^n .
\]

\smallskip
\noindent
Let $\mathcal X\subset H(\C)$ be a $p$-Banach space that contains all polynomials and has bounded point evaluations,
i.e.\ the functional $\delta_z(f)=f(z)$ is continuous on $\mathcal X$ for each $z\in\C$.
By the Hewitt--Savage zero--one law (see \cite[Theorem~2.5.4, p.~82]{D19}),
\[
\mathbb P(\mathcal R f\in\mathcal X)\in\{0,1\},\qquad f\in H(\C).
\]
This motivates the \emph{symbol space} of $\mathcal X$,
\[
\mathcal X_*:=\bigl\{f\in H(\C):\ \mathbb P(\mathcal R f\in\mathcal X)=1\bigr\}.
\]

\smallskip
\noindent
We characterize the parameters $p_1,p_2,q_1,q_2,\alpha,\beta\in(0,\infty)$ for which
\[
\mathcal R:\mathcal F^{p_1,q_1}_{\alpha}\hookrightarrow \mathcal F^{p_2,q_2}_{\beta},
\]
where $\mathcal R:E\hookrightarrow F$ means that $\mathcal R f\in F$ almost surely for every $f\in E$.

\smallskip
\noindent
By \cite[Theorem~2.6]{FT23}, $\left(\mathcal F^{p,q}_{\alpha}\right)_*=\mathcal F^{2,q}_{\alpha}$ for all
$p,q,\alpha\in(0,\infty)$. Theorem~\ref{thm:embeddings} therefore gives:

\begin{theorem}\label{thm:littlewood-fock}
Let $p_1,p_2,q_1,q_2,\alpha,\beta\in(0,\infty)$ and let $(X_n)_{n\ge0}$ be a standard random sequence.
Then $\mathcal R:\mathcal F^{p_1,q_1}_{\alpha}\hookrightarrow \mathcal F^{p_2,q_2}_{\beta}$ if and only if either
\begin{enumerate}
\item[\textup{(i)}] $\alpha<\beta$, or
\item[\textup{(ii)}] $\alpha=\beta$, $q_1\le q_2$, and $\frac1{p_1}+\frac1{q_2}-\frac1{q_1}\le \frac12$.
\end{enumerate}
\end{theorem}

\begin{proof}
The symbol-space identity gives
\[
\mathcal R:\mathcal F_{\alpha}^{p_1,q_1}
\hookrightarrow
\mathcal F_{\beta}^{p_2,q_2}
\quad\Longleftrightarrow\quad
\mathcal F_{\alpha}^{p_1,q_1}
\subset
\left(\mathcal F_{\beta}^{p_2,q_2}\right)_*
=
\mathcal F_{\beta}^{2,q_2}.
\]
The result follows from Theorem~\ref{thm:embeddings} with target angular
exponent \(2\): in the case \(\alpha=\beta\), its remaining condition is
\[
\frac1{p_1}+\frac1{q_2}-\frac1{q_1}\le \frac12.
\]
\end{proof}

\section{An envelope theorem for re-centering normalized Fock kernels}
\label{sec:recentering-envelope}

We prove an envelope theorem for re-centering normalized Fock kernels over
bi-Lipschitz \(\delta\)-pocket decompositions.  The theorem is used later for
the \(\delta\)-pocket tiling in the mixed-norm atomic decomposition, but it is
stated in a general geometric form that also covers the lattice model
\(\delta\mathbb Z^2\).

\smallskip
\noindent
The expansion part is related to the Hilbert-frame theory of Fock kernels and
coherent states; see \cite{GW92, Lyub92, Seip92, SW92, Zhu}.  The point needed here
is quantitative: the re-centering coefficients have a uniform subexponential
envelope, independent of the moving center and adapted to the pocket geometry.
This estimate is used in the angular quasi-Banach part of
Section~\ref{sec:atomic}, where floating kernels are re-expanded at fixed
pocket centers.

\smallskip
\noindent
The proof uses localized frames and inverse-closed algebras of
off-diagonally decaying matrices; see
\cite{Christensen, FG05, Gro04, GL06}.  The form required below is a
Fock-specific re-centering principle for bi-Lipschitz pocket decompositions.

\smallskip
\noindent
Subsection~\ref{subsec:recenter-setup} states the theorem and introduces the
subexponential matrix algebra.  Subsection~\ref{subsec:localized-gramian}
establishes the localized frame and Gramian estimates.
Subsection~\ref{subsec:inverse-closedness} obtains localized range inverses by
inverse-closedness.  Subsection~\ref{subsec:proof-envelope} applies these
estimates to moving centers.

\subsection{Setup, notation, and statement of the theorem}
\label{subsec:recenter-setup}

We fix the geometric framework, record the counting estimate, state the
envelope theorem, and introduce the subexponential matrix algebra.  In this
section \(\beta\in(0,1)\) and \(\gamma>0\) denote subexponential decay
parameters.

\begin{definition}[$\delta$-pocket decompositions]
\label{def:pocket-decomposition}
Let $\delta>0$ and let $0<c_{\mathrm{in}}<C_{\mathrm{out}}<\infty$.
A \emph{$(\delta,c_{\mathrm{in}},C_{\mathrm{out}})$-pocket decomposition}
of $\mathbb C$ is a countable family $\mathcal P_\delta$ of Borel sets
$Q\subset \mathbb C$, together with distinguished centers $\zeta_Q\in Q$
for $Q\in\mathcal P_\delta$, such that:

\begin{enumerate}
\item[\textup{(i)}] The pockets are pairwise disjoint up to null sets and cover
$\mathbb C$ up to a null set, that is,
\[
|Q\cap Q'|=0 \quad (Q\neq Q',\ Q, Q' \in\mathcal P_\delta),
\qquad
\left|\mathbb C\setminus \bigcup_{Q\in\mathcal P_\delta} Q\right|=0.
\]

\item[\textup{(ii)}] For every $Q\in\mathcal P_\delta$,
\[
B(\zeta_Q,c_{\mathrm{in}}\delta)\subset Q\subset
B(\zeta_Q,C_{\mathrm{out}}\delta).
\]
\end{enumerate}

\smallskip
\noindent
When the constants are understood, we simply call $\mathcal P_\delta$ a
\emph{$\delta$-pocket decomposition}.

\smallskip
\noindent
If, in addition, there exist a constant $C_L\ge 1$ and a bijection
\(
\Phi_\delta:\mathcal P_\delta\to \mathbb Z^2\)
such that
\begin{equation}
\label{eq:bilip-pocket}
C_L^{-1}\delta\,|\Phi_\delta(Q)-\Phi_\delta(Q')|
\le |\zeta_Q-\zeta_{Q'}|
\le C_L\delta\,|\Phi_\delta(Q)-\Phi_\delta(Q')|
\end{equation}
for all $Q, Q'\in\mathcal P_\delta$, then $\mathcal P_\delta$ is called a
\emph{bi-Lipschitz $\delta$-pocket decomposition}.

\smallskip
\noindent
The constants $c_{\mathrm{in}}$, $C_{\mathrm{out}}$, and, in the
bi-Lipschitz case, $C_L$, are called the \emph{geometric constants} of
$\mathcal P_\delta$.
\end{definition}

\smallskip
\noindent
Two examples of bi-Lipschitz $\delta$-pocket decompositions are the
standard lattice model $\delta\mathbb Z^2$, equipped with square cells
centered at lattice points, and the $\delta$-pocket tiling introduced in
Subsection~\ref{subsec:atomic-pockettiling}.

\smallskip
\noindent
The next lemma records the counting consequence of the pocket assumptions.

\begin{lemma} 
\label{lem:counting}
Let \(\mathcal P_\delta\) be a \(\delta\)-pocket decomposition of \(\mathbb C\) with
distinguished centers \(\zeta_Q\) and geometric constants \(c_{\mathrm{in}}\) and
\(C_{\mathrm{out}}\). Then, for every \(z\in\mathbb C\) and every \(R>C_{\mathrm{out}}\delta\),
\[
\left(\frac{R}{C_{\mathrm{out}}\delta}-1\right)^2
\le
\#\{Q\in\mathcal P_\delta:\zeta_Q\in B(z,R)\}
\le
\frac{C_{\mathrm{out}}^2}{c_{\mathrm{in}}^2}
\left(\frac{R}{C_{\mathrm{out}}\delta}+1\right)^2.
\]
Moreover, the upper bound remains valid for every \(R>0\).
\end{lemma}

\begin{proof}
This is the area comparison used later in Lemma~\ref{lem:geo-pockets}\textup{(d)}.
For every $z\in\mathbb C$ and every \(R>C_{\mathrm{out}}\delta\),
\[
B(z,R-C_{\mathrm{out}}\delta)
\, \subset
\bigcup_{\zeta_Q\in B(z,R)} Q
\, \subset
B(z,R+C_{\mathrm{out}}\delta).
\]
Hence
\[
\pi(R-C_{\mathrm{out}}\delta)^2
\le
\sum_{\zeta_Q\in B(z,R)} |Q|
\le
\pi(R+C_{\mathrm{out}}\delta)^2.
\]
Using
\(\pi c_{\mathrm{in}}^2\delta^2\le |Q|\le \pi C_{\mathrm{out}}^2\delta^2, Q\in\mathcal P_\delta\),
the result follows.
\end{proof}

\smallskip
\noindent
We state the re-centering theorem.

\begin{theorem}[Re-centering envelope theorem]
\label{thm:envelope-recentering}
Let $\alpha>0$ and let $\mathcal P_\delta$ be a bi-Lipschitz
$\delta$-pocket decomposition of $\mathbb C$ with geometric constants
$c_{\mathrm{in}}$, $C_{\mathrm{out}}$, and $C_L$.
Then there exists $\delta_0=\delta_0(\alpha,c_{\mathrm{in}},C_{\mathrm{out}},C_L)>0$
such that, for every $0<\delta<\delta_0$, the following statements hold.

\begin{enumerate}
\item[\textup{(i)}] For every $\eta\in\mathbb C$, there exists a complex sequence
$b(\eta)=\bigl(b_Q(\eta)\bigr)_{Q\in\mathcal P_\delta}$ such that
\[
\kappa_\eta
=
\sum_{Q\in\mathcal P_\delta} b_Q(\eta)\,\kappa_{\zeta_Q},
\]
with convergence in $\mathcal F_\alpha^2$.

\item[\textup{(ii)}] Let $0<\beta<1$ and $\gamma>0$.
Then the coefficients in \textup{(i)} can be chosen so that
\[
|b_Q(\eta)|
\le
C_{\mathrm{env}}\,e^{-\gamma |\zeta_Q-\eta|^\beta},
\qquad Q\in\mathcal P_\delta,\ \eta\in\mathbb C,
\]
where the constant $C_{\mathrm{env}}$ depends only on $\alpha,\beta,\gamma,\delta,c_{\mathrm{in}},C_{\mathrm{out}}$, and \(C_L\).
\end{enumerate}
\end{theorem}

\smallskip
\noindent
No uniqueness of the coefficient sequence is asserted.  The quantitative
content is the envelope estimate in \textup{(ii)}, which gives uniform decay
away from the moving center \(\eta\).

\smallskip
\noindent
The proof has four steps: construct a localized frame, place the Gramian in a
subexponential matrix algebra, obtain localized range inverses by
inverse-closedness, and apply the range inverse to the sampled kernel vector
associated with \(\eta\).

\smallskip
\noindent
We fix the Hilbert-space notation and the matrix algebra for subexponential
off-diagonal localization.
We write $\langle\cdot,\cdot\rangle_\alpha$ for the Hilbert space inner product
on $\mathcal F_\alpha^2$, namely
\[
\langle f,g\rangle_\alpha
:=
\frac{\alpha}{\pi}\int_{\mathbb C} f(z)\overline{g(z)}\,e^{-\alpha|z|^2}\,dA(z),
\qquad f,g\in\mathcal F_\alpha^2.
\]

\smallskip
\noindent
For a bi-Lipschitz $\delta$-pocket decomposition $\mathcal P_\delta$ of
$\mathbb C$ with distinguished centers $\zeta_Q$, \(Q\in\mathcal P_\delta\), we write
\[
\ell^2(\mathcal P_\delta)
:=
\left\{a=(a_Q)_{Q\in\mathcal P_\delta}:\ \sum_{Q\in\mathcal P_\delta}|a_Q|^2<\infty\right\}.
\]
Whenever the family
\[
e_Q:=\sqrt{|Q|}\,\kappa_{\zeta_Q},
\qquad Q\in\mathcal P_\delta,
\]
is a frame for $\mathcal F_\alpha^2$, we denote by
\[
\mathcal{C}:\mathcal F_\alpha^2\to \ell^2(\mathcal P_\delta),
\qquad
(\mathcal{C}f)_Q:=\langle f,e_Q\rangle_\alpha,
\]
the associated analysis operator, by
\[
\mathcal{C}^*:\ell^2(\mathcal P_\delta)\to \mathcal F_\alpha^2, \quad 
S:=\mathcal{C}^*\mathcal{C}, 
\quad\text{and}\quad G: = \mathcal{C} \mathcal{C}^*
\]
the synthesis operator, the frame operator,  and the Gramian, respectively.

\begin{definition}[Subexponential matrix algebra]
\label{def:subexp-matrix-algebra}
Let $\mathcal P_\delta$ be a $\delta$-pocket decomposition of $\mathbb C$ with distinguished centers $\zeta_Q$, \(Q\in\mathcal P_\delta\). Fix \(0<\beta<1\) and
\(\gamma>0\), and define
\[
\omega_{\gamma,\beta}(Q, Q')
:=
\exp\!\bigl(\gamma |\zeta_Q-\zeta_{Q'}|^\beta\bigr),
\qquad Q, Q'\in\mathcal P_\delta.
\]
For a matrix \(A=(A_{Q, Q'})_{Q, Q'\in\mathcal P_\delta}\), set
\[
\|A\|_{\mathcal A^{\exp}_{\gamma,\beta}(\mathcal P_\delta)}
:=
\max\left\{
\sup_{Q\in\mathcal P_\delta}\sum_{Q' \in\mathcal P_\delta}|A_{Q, Q'}|\,\omega_{\gamma,\beta}(Q, Q'),\
\sup_{Q' \in\mathcal P_\delta}\sum_{Q\in\mathcal P_\delta}|A_{Q,Q'}|\,\omega_{\gamma,\beta}(Q,Q')
\right\}.
\]
We denote by \(\mathcal A^{\exp}_{\gamma,\beta}(\mathcal P_\delta)\) the space
of all matrices \(A\) for which this quantity is finite.
\end{definition}

\begin{lemma} 
\label{lem:subexp-matrix-algebra}
Let \(\mathcal P_\delta\) be a \(\delta\)-pocket decomposition of \(\mathbb C\)
with distinguished centers \(\zeta_Q\), \(Q\in\mathcal P_\delta\). Fix
\(0<\beta<1\) and \(\gamma>0\). Then:

\begin{enumerate}
\item[\textup{(i)}] \(\mathcal A^{\exp}_{\gamma,\beta}(\mathcal P_\delta)\) is a
unital Banach algebra under matrix multiplication.

\item[\textup{(ii)}] Every
\(A\in \mathcal A^{\exp}_{\gamma,\beta}(\mathcal P_\delta)\) defines a bounded
operator on \(\ell^2(\mathcal P_\delta)\), and
\[
\|A\|_{\mathcal B(\ell^2(\mathcal P_\delta))}
\le
\|A\|_{\mathcal A^{\exp}_{\gamma,\beta}(\mathcal P_\delta)}.
\]
\end{enumerate}
\end{lemma}

\begin{proof}
Set \(\omega(Q,Q'):=\omega_{\gamma,\beta}(Q,Q')\).
For a matrix \(A=(A_{Q,Q'})\), write
\[
\|A\|_{\mathrm{row}}
:=
\sup_{Q\in\mathcal P_\delta}\sum_{Q'\in\mathcal P_\delta}|A_{Q,Q'}|\,\omega(Q,Q'),
\qquad
\|A\|_{\mathrm{col}}
:=
\sup_{Q'\in\mathcal P_\delta}\sum_{Q\in\mathcal P_\delta}|A_{Q,Q'}|\,\omega(Q,Q'),
\]
so that
\[
\|A\|_{\mathcal A^{\exp}_{\gamma,\beta}(\mathcal P_\delta)}
=
\max\{\|A\|_{\mathrm{row}},\|A\|_{\mathrm{col}}\}.
\]

\smallskip
\noindent
Since \(0<\beta<1\), the function \(t\mapsto t^\beta\) is concave on
\([0,\infty)\). Hence, for all \(Q, Q', Q''\in\mathcal P_\delta\),
\[
|\zeta_Q-\zeta_{Q'}|^\beta
\le
|\zeta_Q-\zeta_{Q''}|^\beta+|\zeta_{Q''}-\zeta_{Q'}|^\beta,
\]
and therefore
\begin{equation}
\label{eq:submult-weight}
\omega(Q,Q')\le \omega(Q,Q'')\,\omega(Q'',Q').
\end{equation}

\smallskip
\noindent
To prove \textup{(i)}, let \(A,B\in\mathcal A^{\exp}_{\gamma,\beta}(\mathcal P_\delta)\).
Then, using \eqref{eq:submult-weight},
\[
\sum_{Q'}|(AB)_{Q,Q'}|\,\omega(Q,Q')
\le
\sum_{Q''}|A_{Q,Q''}|\,\omega(Q,Q'')\sum_{Q'}|B_{Q'',Q'}|\,\omega(Q'',Q'),
\]
so
\[
\|AB\|_{\mathrm{row}}
\le
\|A\|_{\mathrm{row}}\|B\|_{\mathrm{row}}.
\]
Similarly,
\[
\|AB\|_{\mathrm{col}}
\le
\|A\|_{\mathrm{col}}\|B\|_{\mathrm{col}}.
\]
Hence
\[
\|AB\|_{\mathcal A^{\exp}_{\gamma,\beta}(\mathcal P_\delta)}
\le
\|A\|_{\mathcal A^{\exp}_{\gamma,\beta}(\mathcal P_\delta)}
\,
\|B\|_{\mathcal A^{\exp}_{\gamma,\beta}(\mathcal P_\delta)}.
\]
The identity matrix belongs to
\(\mathcal A^{\exp}_{\gamma,\beta}(\mathcal P_\delta)\) with norm \(1\), so the
algebra is unital.

\smallskip
\noindent
To prove completeness, let \((A^{(n)})_{n \geq 1}\) be Cauchy in
\(\mathcal A^{\exp}_{\gamma,\beta}(\mathcal P_\delta)\). Then, for each fixed
\(Q\), the weighted row vectors
\[
R_Q^{(n)}:=
\bigl(A^{(n)}_{Q,Q'}\omega(Q,Q')\bigr)_{Q'\in\mathcal P_\delta}
\]
form a Cauchy sequence in \(\ell^1(\mathcal P_\delta)\), and similarly, for
each fixed \(Q'\), the weighted column vectors
\[
C_{Q'}^{(n)}:=
\bigl(A^{(n)}_{Q,Q'}\omega(Q,Q')\bigr)_{Q\in\mathcal P_\delta}
\]
form a Cauchy sequence in \(\ell^1(\mathcal P_\delta)\). 
Let \(A=(A_{Q,Q'})_{Q,Q'\in\mathcal P_\delta}\)
be the entrywise limit. Passing to the limit in the weighted row and column
norms shows that \(A\in \mathcal A^{\exp}_{\gamma,\beta}(\mathcal P_\delta)\)
and \(A^{(n)}\to A\) in
\(\mathcal A^{\exp}_{\gamma,\beta}(\mathcal P_\delta)\). Thus
\(\mathcal A^{\exp}_{\gamma,\beta}(\mathcal P_\delta)\) is a Banach algebra.

\smallskip
\noindent
For \textup{(ii)}, since \(\omega(Q,Q')\ge 1\), we have
\[
\sup_Q\sum_{Q'} |A_{Q,Q'}|
\le
\|A\|_{\mathcal A^{\exp}_{\gamma,\beta}(\mathcal P_\delta)},
\qquad
\sup_{Q'}\sum_Q |A_{Q,Q'}|
\le
\|A\|_{\mathcal A^{\exp}_{\gamma,\beta}(\mathcal P_\delta)}.
\]
Therefore the standard Schur test yields
\[
\|A\|_{\mathcal B(\ell^2(\mathcal P_\delta))}
\le
\sqrt{
\left(\sup_Q\sum_{Q'} |A_{Q,Q'}|\right)
\left(\sup_{Q'}\sum_Q |A_{Q,Q'}|\right)
}
\le
\|A\|_{\mathcal A^{\exp}_{\gamma,\beta}(\mathcal P_\delta)}.
\]
This proves \textup{(ii)}.
\end{proof}

\subsection{Localized kernel frame and Gramian estimates}
\label{subsec:localized-gramian}

We establish the frame structure of
\(\{\sqrt{|Q|}\,\kappa_{\zeta_Q}\}_{Q\in\mathcal P_\delta}\), record the
reconstruction formula and the Gramian range projection, and derive
off-diagonal decay of the Gramian.

\begin{lemma}[Localized kernel frame]
\label{lem:kernel-frame}
Fix $\alpha>0$. Let $\mathcal P_\delta$ be a $\delta$-pocket decomposition of
$\mathbb C$ with distinguished centers $\zeta_Q$ and geometric constants
$c_{\mathrm{in}}$ and $C_{\mathrm{out}}$. Set
\(
\delta_0:=\frac{1}{2C_{\mathrm{out}}\sqrt{\alpha}}>0.
\)
Then, for every $0 < \delta < \delta_0$, the family
\[
\bigl\{e_Q=\sqrt{|Q|}\,\kappa_{\zeta_Q}:Q\in\mathcal P_\delta\bigr\}
\subset \mathcal F_\alpha^2
\]
is a frame for $\mathcal F_\alpha^2$ with frame bounds \(
0<a_\delta\le b_\delta<\infty \).

\smallskip
\noindent
Let
\(
\mathcal{C}:\mathcal F_\alpha^2\to \ell^2(\mathcal P_\delta),  (\mathcal{C}f)_Q: = \langle f, e_Q\rangle_{\alpha},\)
be the associated analysis operator, and let $S:=\mathcal{C}^*\mathcal{C}$ and $
G:=\mathcal{C} \mathcal{C}^*$ be the frame operator and the Gramian, respectively.
The following hold:
\begin{enumerate}
\item[\textup{(i)}] 
Define \(e_Q^\#:=S^{-1}e_Q\).
Then, for every \(f\in\mathcal F_\alpha^2\),
\[
f=\sum_{Q\in\mathcal P_\delta}
\bigl\langle f,\sqrt{|Q|}e_Q^\#\bigr\rangle_\alpha\,\kappa_{\zeta_Q},
\]
with unconditional convergence in \(\mathcal F_\alpha^2\).

\item[\textup{(ii)}] Set 
\(
\mathcal H:=\overline{\operatorname{Ran}(\mathcal{C})}\subset \ell^2(\mathcal P_\delta)\). Then $\mathcal H$ is $G$-invariant, the restriction
\(
G|_{\mathcal H}:\mathcal H\to\mathcal H\)
is invertible, and 
\[
\ker(G)=\mathcal H^\perp,
\qquad
\operatorname{Ran}(G)=\mathcal H,
\qquad
\operatorname{spec}(G)\subset \{0\}\cup [a_\delta,b_\delta].
\]

\item[\textup{(iii)}] If $\Gamma\subset\mathbb C$ is a positively oriented rectifiable Jordan curve such that
\(
[a_\delta,b_\delta]\subset \operatorname{int}(\Gamma)\) and \(
0\notin \overline{\operatorname{int}(\Gamma)}\),
then the orthogonal projection $P_{\mathcal H}$ onto $\mathcal H$ is given by
\[
P_{\mathcal H}
=
\frac{1}{2\pi i}\int_\Gamma (zI-G)^{-1}\,dz.
\]
\end{enumerate}
\end{lemma}

\begin{proof}
Set
\(
Z_\delta:=\{\zeta_Q:Q\in\mathcal P_\delta\}\).
Since
\(
B(\zeta_Q,c_{\mathrm{in}}\delta)\subset Q,
\ Q\in\mathcal P_\delta\),
and the pockets are pairwise disjoint up to null sets, the set $Z_\delta$ is
$c_{\mathrm{in}}\delta$-separated. Moreover, by Lemma~\ref{lem:counting},
\[
D^-(Z_\delta)
:=
\liminf_{R\to\infty}\inf_{z\in\mathbb C}
\frac{\#(Z_\delta\cap B(z,R))}{\pi R^2}
\ge
\liminf_{R\to\infty}
\frac{\left(\frac{R}{C_{\mathrm{out}}\delta}-1\right)^2}{\pi R^2}
=
\frac{1}{\pi C_{\mathrm{out}}^2\delta^2}.
\]
Hence, for every $0 < \delta < \delta_0$,
\[
D^-(Z_\delta)>\frac{\alpha}{\pi}.
\]
By the sampling theorem for Fock spaces \cite[Lemma~4.32]{Zhu}, $Z_\delta$ is
a sampling set for $\mathcal F_\alpha^2$. Therefore there exist constants
$a'_\delta,b'_\delta>0$ such that
\[
a'_\delta\|f\|_{\mathcal F_\alpha^2}^2
\le
\sum_{Q\in\mathcal P_\delta}
|f(\zeta_Q)|^2e^{-\alpha|\zeta_Q|^2}
\le
b'_\delta\|f\|_{\mathcal F_\alpha^2}^2,
\qquad f\in\mathcal F_\alpha^2.
\]
Using  
\(
\langle f,e_Q\rangle_\alpha
=
\sqrt{|Q|}\,f(\zeta_Q)e^{-\frac{\alpha}{2}|\zeta_Q|^2}
\), together with
\(
\pi c_{\mathrm{in}}^2\delta^2\le |Q|\le \pi C_{\mathrm{out}}^2\delta^2\),
we obtain
\[
a_\delta\|f\|_{\mathcal F_\alpha^2}^2
\le
\sum_{Q\in\mathcal P_\delta}|\langle f,e_Q\rangle_\alpha|^2
\le
b_\delta\|f\|_{\mathcal F_\alpha^2}^2,
\qquad f\in\mathcal F_\alpha^2,
\]
where
\(
a_\delta:=\pi c_{\mathrm{in}}^2\delta^2a'_\delta
\) and 
\(
b_\delta:=\pi C_{\mathrm{out}}^2\delta^2b'_\delta
\).
Thus $\{e_Q\}_{Q\in\mathcal P_\delta}$ is a frame for $\mathcal F_\alpha^2$.

\smallskip
\noindent
Part \textup{(i)} is the frame reconstruction formula for the canonical dual
\cite[Chapter~5]{Christensen}:
\[
f
=
\sum_{Q\in\mathcal P_\delta}
\bigl\langle f,e_Q^\# \bigr\rangle_\alpha\,e_Q,
\qquad f\in\mathcal F_\alpha^2,
\]
with unconditional convergence in \(\mathcal F_\alpha^2\).  The definitions of
\(e_Q\) and \(e_Q^\#\) give the stated formula.

\smallskip
\noindent
For part \textup{(ii)}, since \(\{e_Q\}_{Q\in\mathcal P_\delta}\) is a frame, the analysis operator
\(\mathcal{C}\) has closed range, so
\[
\mathcal H=\overline{\operatorname{Ran}(\mathcal{C})}=\operatorname{Ran}(\mathcal{C}).
\]
Moreover,
\(
\ker(G) = \ker(\mathcal{C}^*)=\mathcal H^\perp\),
and, because \(\operatorname{Ran}(\mathcal{C})\) is closed,
\[
\operatorname{Ran}(G)=\operatorname{Ran}(\mathcal{C} \mathcal{C}^*)=\operatorname{Ran}(\mathcal{C})=\mathcal H.
\]
In particular, \(\mathcal H\) is \(G\)-invariant.

\smallskip
\noindent
Now let \(c= \mathcal{C} f\in\mathcal H\). Then
\[
\|c\|_{\ell^2(\mathcal P_\delta)}^2
=
\|\mathcal{C} f\|_{\ell^2(\mathcal P_\delta)}^2
=
\langle Sf,f\rangle_\alpha,
\]
while
\[
\langle Gc,c\rangle_{\ell^2(\mathcal P_\delta)}
=
\langle \mathcal{C} \mathcal{C}^* \mathcal{C}f, \mathcal{C}f\rangle_{\ell^2(\mathcal P_\delta)}
=
\langle S^2f,f\rangle_\alpha.
\]
Since the frame bounds give
\(
a_\delta I\le S\le b_\delta I\),
it follows that
\(
a_\delta S\le S^2\le b_\delta S\),
and hence
\[
a_\delta\|c\|_{\ell^2(\mathcal P_\delta)}^2
\le
\langle Gc,c\rangle_{\ell^2(\mathcal P_\delta)}
\le
b_\delta\|c\|_{\ell^2(\mathcal P_\delta)}^2,
\qquad c\in\mathcal H.
\]
Therefore, \(G|_{\mathcal H}:\mathcal H\to\mathcal H\) is positive and bounded below by
\(a_\delta\).
Since $\operatorname{Ran}(G)=\mathcal H$, the restriction $G|_{\mathcal H}$ is surjective onto $\mathcal H$; together with the lower bound, this shows that \(G|_{\mathcal H}:\mathcal H\to\mathcal H\) is invertible, and
\[
\operatorname{spec}(G|_{\mathcal H})\subset [a_\delta,b_\delta].
\]
Since \(G=0\) on \(\mathcal H^\perp=\ker(G)\), we conclude that
\[
\operatorname{spec}(G)\subset \{0\}\cup[a_\delta,b_\delta].
\]

\smallskip
\noindent
For part \textup{(iii)}, the set \([a_\delta,b_\delta]\) is a clopen subset of
\(\operatorname{spec}(G)\). Hence, by the Riesz idempotent formula
(see \cite[(6.9), Chapter~VII, \S6]{ConwayFA}; see also
\cite[Proposition~4.11, Chapter~VII, \S4]{ConwayFA}),
\[
P:=
\frac{1}{2\pi i}\int_\Gamma (zI-G)^{-1}\,dz
\]
is the spectral projection of \(G\) associated with the spectral subset
\([a_\delta,b_\delta]\). Since \(G\) is self-adjoint and the corresponding spectral
subspace is precisely
\(\operatorname{Ran}(G)=\mathcal H\),
it follows that \(P\) is the orthogonal projection onto \(\mathcal H\), that is,
\(
P=P_{\mathcal H}\).
This proves the lemma.
\end{proof}

\medskip
\noindent
We identify the matrix of the Gramian with respect to the canonical basis of \(\ell^2(\mathcal P_\delta)\)
and prove its off-diagonal localization.
For each \(Q'\in\mathcal P_\delta\), let \(u^{(Q')}\in \ell^2(\mathcal P_\delta)\) denote the canonical basis vector, that is,
\[
u^{(Q')}_{Q}=
\begin{cases}
1,& Q=Q',\\
0,& Q\neq Q'.
\end{cases}
\]
Since
\(\mathcal{C}^*u^{(Q')}=e_{Q'}\),
it follows that, for \(Q, Q'\in\mathcal P_\delta\),
\[
(Gu^{(Q')})_Q
=
(\mathcal{C} \mathcal{C}^*u^{(Q')})_Q
=
(\mathcal{C}e_{Q'})_Q
=
\langle e_{Q'},e_Q\rangle_\alpha.
\]
Hence the Gramian \(G\) is represented by the matrix
\[
G=(G_{Q,Q'})_{Q,Q'\in\mathcal P_\delta},
\qquad
G_{Q,Q'}=\langle e_{Q'},e_Q\rangle_\alpha.
\]
We next show that this matrix belongs to the subexponential algebra.

\begin{lemma}[Localized Gramian]
\label{lem:gramian-localized}
In the setting of Lemma~\ref{lem:kernel-frame}, let 
$\mathcal P_\delta$ be a $\delta$-pocket decomposition of $\mathbb C$ with
distinguished centers $\zeta_Q$ and geometric constants
$c_{\mathrm{in}}$ and $C_{\mathrm{out}}$, where \(0 < \delta < \delta_0\). Let
\(G=(G_{Q, Q'})_{Q, Q'\in\mathcal P_\delta}\)
be the matrix of the Gramian $G=\mathcal{C} \mathcal{C}^*$ with respect to the canonical basis of
$\ell^2(\mathcal P_\delta)$.
Then, for every \(0<\beta<1\) and every \(\gamma>0\), \(G\in \mathcal A^{\exp}_{\gamma,\beta}(\mathcal P_\delta)\) 
and there is a constant $C = C(\alpha,\beta,\gamma,c_{\mathrm{in}},C_{\mathrm{out}})$ such that
\(
\|G\|_{\mathcal A^{\exp}_{\gamma,\beta}(\mathcal P_\delta)}
\le
C\).
\end{lemma}

\begin{proof}
For \(Q, Q' \in\mathcal P_\delta\), the reproducing-kernel identity gives
\[
|G_{Q,Q'}|
=
\sqrt{|Q||Q'|} \, \left|\langle \kappa_{\zeta_{Q'}},\kappa_{\zeta_Q}\rangle_\alpha \right| = \sqrt{|Q||Q'|}
\exp\!\left(-\frac{\alpha}{2}|\zeta_Q-\zeta_{Q'}|^2\right).
\]
Using \(|Q|,|Q'|\le \pi C_{\mathrm{out}}^2\delta^2\), we obtain
\begin{equation}
\label{eq:gramian-gaussian}
|G_{Q,Q'}|
\le
\pi C_{\mathrm{out}}^2 \delta^2
\exp\!\left(-\frac{\alpha}{2}|\zeta_Q-\zeta_{Q'}|^2\right).
\end{equation}

\smallskip
\noindent
Fix \(0<\beta<1\) and \(\gamma>0\). Then there exists a constant
\(C_1=C_1(\alpha,\beta,\gamma)\) such that
\[
\gamma t^\beta \le \frac{\alpha}{4}t^2 + C_1,
\qquad t\ge 0.
\]
Combining this with \eqref{eq:gramian-gaussian}, we get
\begin{equation}
\label{eq:gramian-subexp}
|G_{Q,Q'}|\,e^{\gamma |\zeta_Q-\zeta_{Q'}|^\beta}
\le
C_2\,\delta^2
\exp\!\left(-\frac{\alpha}{4}|\zeta_Q-\zeta_{Q'}|^2\right),
\qquad Q, Q'\in\mathcal P_\delta,
\end{equation}
where \(C_2=C_2(\alpha,\beta,\gamma,C_{\mathrm{out}})\).

\smallskip
\noindent
We estimate the row sums. Fix \(Q\in\mathcal P_\delta\), and for \(n\ge 0\) set
\[
\mathcal S_n(Q):=
\{Q'\in\mathcal P_\delta:\ n\le |\zeta_Q-\zeta_{Q'}|<n+1\}.
\]
Then, by Lemma \ref{lem:counting}, 
\[
\#\mathcal S_n(Q) \, \leq \, \# \left\{Q'\in\mathcal P_\delta:\ \zeta_{Q'} \in B(\zeta_Q, n+1)\right\} \, \leq \, \frac{C^2_{\mathrm{out}}}{c^2_{\mathrm{in}}} \left(\frac{n+1}{C_{\mathrm{out}}\delta}  + 1\right)^2.
\]
Using this and  \eqref{eq:gramian-subexp}, we get
\[
\sum_{Q'\in\mathcal P_\delta}
|G_{Q,Q'}|\,e^{\gamma |\zeta_Q-\zeta_{Q'}|^\beta}
\le
C_2\delta^2
\sum_{n=0}^\infty
e^{-\frac{\alpha}{4}n^2}\,\#\mathcal S_n(Q) \leq \frac{C_2}{c^2_{\mathrm{in}}} \sum_{n=0}^\infty \big(n+1 + C_{\mathrm{out} }\delta\big)^2 \,
e^{-\frac{\alpha}{4}n^2}.
\]
The series converges, and the bound is independent of \(Q\) and
\(\delta\in(0,\delta_0]\). Hence
\[
\sup_{Q\in\mathcal P_\delta}
\sum_{Q'\in\mathcal P_\delta}
|G_{Q,Q'}|\,e^{\gamma |\zeta_Q-\zeta_{Q'}|^\beta}
\le
C,
\]
where $C$ depends only on $\alpha,\beta,\gamma,c_{\mathrm{in}}$, and $C_{\mathrm{out}}$.
Since \(G_{Q,Q'}=\overline{G_{Q',Q}}\), the same estimate holds for the column
sums. Therefore
\(
\|G\|_{\mathcal A^{\exp}_{\gamma,\beta}(\mathcal P_\delta)}
\le
C\),
which proves the lemma.
\end{proof}

\subsection{Inverse-closedness and localized range inverses}
\label{subsec:inverse-closedness}

We transfer the localized Gramian estimate to the bi-Lipschitz index model,
invoke inverse-closedness of the matrix algebra, and deduce localization for
the range projection and the range inverse used in
Theorem~\ref{thm:envelope-recentering}.

\begin{lemma}[Bi-Lipschitz transfer]
\label{lem:app-bilip-transfer}
Let $\mathcal P_\delta$ be a bi-Lipschitz $\delta$-pocket decomposition of
$\mathbb C$ with bi-Lipschitz labeling
\(
\Phi_\delta:\mathcal P_\delta\to\mathbb Z^2
\)
and bi-Lipschitz constant $C_L$, that is,
\begin{equation}\label{eq:bilip-transfer}
C_L^{-1}\delta\,|\Phi_\delta(Q)-\Phi_\delta(Q')|
\le
|\zeta_Q-\zeta_{Q'}|
\le
C_L\delta\,|\Phi_\delta(Q)-\Phi_\delta(Q')|,
\qquad Q,Q'\in\mathcal P_\delta.
\end{equation}
Fix $0<\beta<1$ and $\gamma>0$. Set 
\[
\gamma_1:=\gamma\Bigl(\frac{\delta}{C_L}\Bigr)^\beta,
\qquad
\gamma_2: = \gamma (C_L\delta)^\beta.
\]
For a matrix $A=(A_{Q,Q'})_{Q,Q'\in\mathcal P_\delta}$, define the relabeled
matrix $\widetilde A=(\widetilde A_{\mathbf{m}, \mathbf{n}})_{\mathbf{m},\mathbf{n}\in\mathbb Z^2}$ by
\[
\widetilde A_{\mathbf{m},\mathbf{n}}
:=
A_{\Phi_\delta^{-1}(\mathbf{m}),\,\Phi_\delta^{-1}(\mathbf{n})},
\qquad \mathbf{m}, \mathbf{n}\in\mathbb Z^2.
\]
If $A\in \mathcal A^{\exp}_{\gamma,\beta}(\mathcal P_\delta)$, then $\widetilde A\in \mathcal A^{\exp}_{\gamma_1,\beta}(\mathbb Z^2)$. Conversely, if $\widetilde A\in \mathcal A^{\exp}_{\gamma_2,\beta}(\mathbb Z^2)$, then $A\in \mathcal A^{\exp}_{\gamma,\beta}(\mathcal P_\delta)$. Moreover,
\[
\|\widetilde A\|_{\mathcal A^{\exp}_{\gamma_1,\beta}(\mathbb Z^2)}
\le
\|A\|_{\mathcal A^{\exp}_{\gamma,\beta}(\mathcal P_\delta)}
\le
\|\widetilde A\|_{\mathcal A^{\exp}_{\gamma_2,\beta}(\mathbb Z^2)},
\]
where $\mathcal A^{\exp}_{\gamma,\beta}(\Z^2)$ is defined in the same
way as \(\mathcal A^{\exp}_{\gamma,\beta}(\mathcal P_\delta)\) in Definition \ref{def:subexp-matrix-algebra}, with
\(|\mathbf{m}-\mathbf{n}|\) in place of \(|\zeta_Q-\zeta_{Q'}|\).
\end{lemma}

\begin{proof}
For \(\mathbf{m}, \mathbf{n} \in\mathbb Z^2\), write
\(Q=\Phi_\delta^{-1}(\mathbf{m})\) and $Q'=\Phi_\delta^{-1}(\mathbf{n})$, respectively. 
Assume first \(A\in \mathcal A^{\exp}_{\gamma,\beta}(\mathcal P_\delta)\).
By \eqref{eq:bilip-transfer},
\[
|\zeta_Q-\zeta_{Q'}|
\ge
\frac{\delta}{C_L}\,|\Phi_\delta(Q)-\Phi_\delta(Q')|, \quad \text{ so, } \quad
e^{\gamma_1|\Phi_\delta(Q)-\Phi_\delta(Q')|^\beta}
\le
e^{\gamma|\zeta_Q-\zeta_{Q'}|^\beta},
\qquad Q, Q'\in\mathcal P_\delta.
\]
Fix \(\mathbf{m} \in\mathbb Z^2\). Then
\[
\sum_{\mathbf{n}\in\mathbb Z^2}
|\widetilde A_{\mathbf{m}, \mathbf{n}}|\,e^{\gamma_1|\mathbf{m} - \mathbf{n}|^\beta}
\le
\sum_{Q'\in\mathcal P_\delta}
|A_{Q,Q'}|\,e^{\gamma|\zeta_Q-\zeta_{Q'}|^\beta}.
\]
Taking the supremum over \(\mathbf{m}\) yields the row bound
\[
\sup_{\mathbf{m} \in\mathbb Z^2}
\sum_{\mathbf{n} \in\mathbb Z^2}
|\widetilde A_{\mathbf{m}, \mathbf{n}}|\,e^{\gamma_1|\mathbf{m} - \mathbf{n}|^\beta}
\le
\|A\|_{\mathcal A^{\exp}_{\gamma,\beta}(\mathcal P_\delta)}.
\]
The column estimate is identical.  Thus
\(\widetilde A\in \mathcal A^{\exp}_{\gamma_1,\beta}(\mathbb Z^2)\), and
\[
\|\widetilde A\|_{\mathcal A^{\exp}_{\gamma_1,\beta}(\mathbb Z^2)}
\le
\|A\|_{\mathcal A^{\exp}_{\gamma,\beta}(\mathcal P_\delta)}.
\]
Conversely, suppose \(\widetilde A\in \mathcal A^{\exp}_{\gamma_2,\beta}(\mathbb Z^2)\).
By \eqref{eq:bilip-transfer},
\[
|\zeta_Q-\zeta_{Q'}|
\le
C_L\delta\,|\Phi_\delta(Q)-\Phi_\delta(Q')|,
\quad \text{ so } \quad
e^{\gamma |\zeta_Q-\zeta_{Q'}|^\beta}
\le
e^{\gamma_2 |\Phi_\delta(Q)-\Phi_\delta(Q')|^\beta},
\qquad Q, Q'\in\mathcal P_\delta.
\]
Fix \(Q\in\mathcal P_\delta\). Then
\[
\sum_{Q'\in\mathcal P_\delta}
|A_{Q,Q'}|\,e^{\gamma|\zeta_Q-\zeta_{Q'}|^\beta}
\le
\sum_{\mathbf{n}\in\mathbb Z^2}
|\widetilde A_{\mathbf{m}, \mathbf{n}}|\,e^{\gamma_2 |\mathbf{m} - \mathbf{n}|^\beta}.
\]
Taking the supremum over \(Q\) gives the row bound
\[
\sup_{Q\in\mathcal P_\delta}
\sum_{Q'\in\mathcal P_\delta}
|A_{Q,Q'}|\,e^{\gamma|\zeta_Q-\zeta_{Q'}|^\beta}
\le
\|\widetilde A\|_{\mathcal A^{\exp}_{\gamma_2,\beta}(\mathbb Z^2)}.
\]
The column estimate is the same.  Hence
\(A\in \mathcal A^{\exp}_{\gamma,\beta}(\mathcal P_\delta)\), and
\[
\|A\|_{\mathcal A^{\exp}_{\gamma,\beta}(\mathcal P_\delta)}
\le
\|\widetilde A\|_{\mathcal A^{\exp}_{\gamma_2,\beta}(\mathbb Z^2)}.
\]
This proves the lemma.
\end{proof}

\begin{proposition}[Inverse-closedness of the subexponential algebra]
\label{prop:app-inverse-closedness}
Let $\mathcal P_\delta$ be a bi-Lipschitz $\delta$-pocket decomposition of
$\mathbb C$ with bi-Lipschitz labeling
\(
\Phi_\delta:\mathcal P_\delta\to\mathbb Z^2
\)
and bi-Lipschitz constant $C_L$. Fix $0<\beta<1$ and $\gamma>0$, and set
\(
\gamma':=\frac{\gamma}{C_L^{2\beta}}\).
If
\(
A\in \mathcal A^{\exp}_{\gamma,\beta}(\mathcal P_\delta)
\)
is invertible as a bounded operator on $\ell^2(\mathcal P_\delta)$, then
\(
A^{-1}\in \mathcal A^{\exp}_{\gamma',\beta}(\mathcal P_\delta)\).
\end{proposition}

\begin{proof}
Let $\widetilde A$ be the relabeled matrix on $\mathbb Z^2$ associated with $A$,
as in Lemma~\ref{lem:app-bilip-transfer}. Then
\[
\widetilde A\in \mathcal A^{\exp}_{\gamma_1,\beta}(\mathbb Z^2)
\quad \text{ with } \quad
\gamma_1:=\gamma\Bigl(\frac{\delta}{C_L}\Bigr)^\beta.
\]
Let \(U:\ell^2(\mathcal P_\delta)\to \ell^2(\mathbb Z^2)\) be the map
\[
(Ux)(\mathbf{m}):=x_{\Phi_\delta^{-1}(\mathbf{m})},
\qquad \mathbf{m} \in\mathbb Z^2.
\]
The map \(U\) is unitary, being a permutation of the coordinates:
\[
\|Ux\|_{\ell^2(\mathbb Z^2)}^2
=
\sum_{\mathbf{m}\in\mathbb Z^2}|x_{\Phi_\delta^{-1}(\mathbf{m})}|^2
=
\sum_{Q\in\mathcal P_\delta}|x_Q|^2
=
\|x\|_{\ell^2(\mathcal P_\delta)}^2.
\]
For \(x\in \ell^2(\mathcal P_\delta)\) and \(\mathbf{m}\in\mathbb Z^2\), writing
\(Q=\Phi_\delta^{-1}(\mathbf{m})\), we compute
\[
(UAx)(\mathbf{m})
=
(Ax)_Q
=
\sum_{Q'\in\mathcal P_\delta} A_{Q,Q'}x_{Q'}
=
\sum_{\mathbf{n} \in\mathbb Z^2}
A_{\Phi_\delta^{-1}(\mathbf{m}),\,\Phi_\delta^{-1}(\mathbf{n})}(Ux)(\mathbf{n})
=
(\widetilde A Ux)(\mathbf{m}).
\]
Hence \(\widetilde A=UAU^{-1}\).  Thus \(\widetilde A\) is invertible on
\(\ell^2(\mathbb Z^2)\), and
\(
\widetilde A^{-1}=UA^{-1}U^{-1}\).

\smallskip
\noindent
Now consider the weight
\[
v_{\gamma_1}(\mathbf{m}):=e^{\gamma_1 |\mathbf{m}|^\beta},
\qquad \mathbf{m}\in\mathbb Z^2.
\]
Since $0<\beta<1$, the weight $v_{\gamma_1}$ is admissible and satisfies the weak growth
assumption in \cite[Corollary~7]{GL06}. Therefore Corollary~7 of \cite{GL06}
applies and yields
\(
\widetilde A^{-1}\in \mathcal A^{\exp}_{\gamma_1,\beta}(\mathbb Z^2)\).
\smallskip
\noindent
Applying the converse direction of Lemma~\ref{lem:app-bilip-transfer} to
$\widetilde A^{-1}$ with the parameter $\gamma_1$ gives 
\(
A^{-1}\in \mathcal A^{\exp}_{\gamma',\beta}(\mathcal P_\delta)\) 
with
\(
\gamma' = \frac{\gamma_1}{(C_L\delta)^\beta}
=
\frac{\gamma}{C_L^{2\beta}}
\).
\end{proof}

\begin{proposition}[Localized range projection and inverse]
\label{prop:app-localized-range-projection}
In the setting of Lemma~\ref{lem:kernel-frame}, assume in addition that
$\mathcal P_\delta$ is a bi-Lipschitz $\delta$-pocket decomposition with
bi-Lipschitz constant $C_L$.
Fix $0<\beta<1$ and $\gamma>0$.
Let $\Gamma\subset\mathbb C$ be a positively oriented rectifiable Jordan curve such that
\[
[a_\delta,b_\delta]\subset \operatorname{int}(\Gamma), \qquad  
0\notin \overline{\operatorname{int}(\Gamma)}, \qquad \text{ 
and } \qquad \operatorname{dist}\bigl(z,\{0\}\cup[a_\delta,b_\delta]\bigr)\ge \frac{a_\delta}{4}, \quad
 z\in\Gamma.
\]
Then the following hold:
\begin{enumerate}
\item[\textup{(i)}] The orthogonal projection $P_{\mathcal H}$ onto
\(\mathcal H=\overline{\operatorname{Ran}(\mathcal{C})}\)
belongs to
\(\mathcal A^{\exp}_{\gamma,\beta}(\mathcal P_\delta)\).
\item[\textup{(ii)}] Defining
\[
N_{\mathcal H}:=
\frac{1}{2\pi i}\int_\Gamma z^{-1}(zI-G)^{-1}\,dz,
\]
one has
\[
N_{\mathcal H}\in \mathcal A^{\exp}_{\gamma,\beta}(\mathcal P_\delta),
\qquad
N_{\mathcal H}c=(G|_{\mathcal H})^{-1}c,\quad c\in\mathcal H, \quad \text{ and } \quad 
N_{\mathcal H} = N_{\mathcal H} P_{\mathcal H}=P_{\mathcal H}N_{\mathcal H}.
\]
\end{enumerate}
Moreover, for any fixed admissible choice of the contour  \(\Gamma\), the
\(\mathcal A^{\exp}_{\gamma,\beta}(\mathcal P_\delta)\)-norm of \(N_{\mathcal H}\) is bounded
by a constant depending only on
\(\alpha,\beta,\gamma,\delta,c_{\rm in},C_{\rm out}, C_L\), and on that fixed choice of \(\Gamma\).
\end{proposition}

\begin{proof}
For each \(z\in\Gamma\), Lemmas~\ref{lem:subexp-matrix-algebra} and
\ref{lem:gramian-localized} give
\(z I-G\in\mathcal{A}^{\exp}_{\gamma_1,\beta}(\mathcal{P}_{\delta})\), with
\(\gamma_1=C_L^{2\beta}\gamma\).  Since
\(\Gamma\subset\mathbb C\setminus\operatorname{spec}(G)\), \(zI-G\) is
invertible on \(\ell^2(\mathcal P_\delta)\).  Proposition
\ref{prop:app-inverse-closedness} gives
\((zI-G)^{-1}\in\mathcal A^{\exp}_{\gamma,\beta}(\mathcal P_\delta)\).

\smallskip
\noindent
We claim that
\(
\Gamma\ni z\longmapsto (zI-G)^{-1}
\in \mathcal A^{\exp}_{\gamma,\beta}(\mathcal P_\delta)\)
is continuous. 
Fix \(z_0\in\Gamma\), and set
\[
M_0:=(z_0I-G)^{-1}\in \mathcal A^{\exp}_{\gamma,\beta}(\mathcal P_\delta).
\]
For \(z\in\Gamma\), the resolvent identity gives
\[
(zI-G)^{-1}-(z_0I-G)^{-1}
=
(z_0-z)(zI-G)^{-1}(z_0I-G)^{-1}.
\]
Hence
\[
\left\|(zI-G)^{-1}-(z_0I-G)^{-1}\right\|_{\mathcal A^{\exp}_{\gamma,\beta}}
\le
|z-z_0|\,
\left\|(zI-G)^{-1}\right\|_{\mathcal A^{\exp}_{\gamma,\beta}}\,
\|M_0\|_{\mathcal A^{\exp}_{\gamma,\beta}}.
\]
It suffices to show local boundedness in
\(\mathcal A^{\exp}_{\gamma,\beta}(\mathcal P_\delta)\), which follows from the
Neumann-series representation:
for \(|z-z_0|<\|M_0\|_{\mathcal A^{\exp}_{\gamma,\beta}}^{-1}\), we have
\[
zI-G
=
(z_0I-G)\bigl(I+(z-z_0)M_0\bigr),
\]
and hence
\[
(zI-G)^{-1}
=
\bigl(I+(z-z_0)M_0\bigr)^{-1}M_0
=
\sum_{n=0}^\infty (-1)^n (z-z_0)^n M_0^{\,n+1}
\]
in \(\mathcal A^{\exp}_{\gamma,\beta}(\mathcal P_\delta)\). In particular,
\(z\mapsto (zI-G)^{-1}\) is locally bounded near \(z_0\), hence continuous at
\(z_0\).
Since \(z_0\in\Gamma\) was arbitrary, the claim follows.

\smallskip
\noindent
For part \textup{(i)}, Lemma~\ref{lem:kernel-frame}\textup{(iii)} gives
\[
P_{\mathcal H}
=
\frac{1}{2\pi i}\int_\Gamma (zI-G)^{-1}\,dz
\]
as an operator on \(\ell^2(\mathcal P_\delta)\). It follows from the claim that the contour integral
\[
P:=
\frac{1}{2\pi i}\int_\Gamma (zI-G)^{-1}\,dz
\]
is well defined as a Bochner integral in
\(\mathcal A^{\exp}_{\gamma,\beta}(\mathcal P_\delta)\). Let
\[
\iota:\mathcal A^{\exp}_{\gamma,\beta}(\mathcal P_\delta)
\hookrightarrow
\mathcal B\bigl(\ell^2(\mathcal P_\delta)\bigr)
\]
denote the canonical bounded inclusion. Since \(\iota\) is bounded and linear, it commutes with
Bochner integration, and therefore
\[
\iota(P)
=
\frac{1}{2\pi i}\int_\Gamma \iota\bigl((zI-G)^{-1}\bigr)\,dz
=
\frac{1}{2\pi i}\int_\Gamma (zI-G)^{-1}\,dz
=
P_{\mathcal H}.
\]
Hence the matrix of \(P_{\mathcal H}\) belongs to
\(
\mathcal A^{\exp}_{\gamma,\beta}(\mathcal P_\delta)\)
as claimed.

\smallskip
\noindent
For part \textup{(ii)}, since \(0\notin \overline{\operatorname{int}(\Gamma)}\), the scalar function
\(z\mapsto z^{-1}\) is continuous on \(\Gamma\). Hence
\[
\Gamma\ni z\longmapsto z^{-1}(zI-G)^{-1}
\in \mathcal A^{\exp}_{\gamma,\beta}(\mathcal P_\delta)
\]
is continuous, and therefore
\[
N_{\mathcal H}=
\frac{1}{2\pi i}\int_\Gamma z^{-1}(zI-G)^{-1}\,dz
\]
is well defined as a Bochner integral in
\(\mathcal A^{\exp}_{\gamma,\beta}(\mathcal P_\delta)\). In particular,
\(
N_{\mathcal H}\in \mathcal A^{\exp}_{\gamma,\beta}(\mathcal P_\delta)\).

\smallskip
\noindent
If \(c\in \mathcal H^\perp\), then \(c\in \ker(G)\) by
Lemma~\ref{lem:kernel-frame}\textup{(ii)}. Thus, for each $z\in\Gamma$,
\(
(zI-G)c=zc\),
and hence
\((zI-G)^{-1}c=z^{-1}c\).
Therefore
\[
N_{\mathcal H}c
=
\frac{1}{2\pi i}\int_\Gamma z^{-2}\,dz\,c
=
0,
\]
since \(z\mapsto z^{-2}\) is holomorphic on a neighborhood of
\(\Gamma\cup\operatorname{int}(\Gamma)\).

\smallskip
\noindent
Now let \(c\in\mathcal H\), and set
\[
T:=G|_{\mathcal H}:\mathcal H\to\mathcal H.
\]
By Lemma~\ref{lem:kernel-frame}\textup{(ii)}, the subspace \(\mathcal H\) is \(G\)-invariant,
the operator \(T\) is invertible, and
\[
\operatorname{spec}(T)\subset [a_\delta,b_\delta]\subset \operatorname{int}(\Gamma).
\]
Hence, for \(z\in\Gamma\),
\[
(zI-G)^{-1}c=(zI_{\mathcal H}-T)^{-1}c.
\]

\smallskip
\noindent
By the spectral theorem for bounded self-adjoint operators, see
\cite[Chapter~IX, \S2, Theorem~2.2]{ConwayFA}, there exists a spectral measure
\(E\) on \(\operatorname{spec}(T)\) such that
\[
T=\int_{\operatorname{spec}(T)} \lambda\,dE(\lambda).
\]
Hence, by the bounded Borel functional calculus
(see also \cite[Chapter~IX, \S1, Proposition~1.10 and (1.11)]{ConwayFA}),
\[
(zI_{\mathcal H}-T)^{-1}
=
\int_{\operatorname{spec}(T)} \frac{1}{z-\lambda}\,dE(\lambda),
\qquad z\in\mathbb C\setminus\operatorname{spec}(T).
\]
For \(c,d\in\mathcal H\), define the complex measure
\[
\mu_{c,d}(B):=\langle E(B)c,d\rangle_{\ell^2(\mathcal P_\delta)}.
\]
Then
\[
\langle (zI_{\mathcal H}-T)^{-1}c,d\rangle
=
\int_{\operatorname{spec}(T)} \frac{1}{z-\lambda}\,d\mu_{c,d}(\lambda),
\qquad z\in\Gamma.
\]
Since \(\Gamma\) is compact and bounded away from \(\operatorname{spec}(T)\), Fubini's theorem yields
\[
\langle N_{\mathcal H}c,d\rangle
=
\int_{\operatorname{spec}(T)}
\left(
\frac{1}{2\pi i}\int_\Gamma \frac{1}{z(z-\lambda)}\,dz
\right)
d\mu_{c,d}(\lambda).
\]
For \(\lambda\in \operatorname{spec}(T)\subset [a_\delta,b_\delta]\), one has  \(\lambda\in \operatorname{int}(\Gamma)\) and \(0\notin \operatorname{int}(\Gamma)\), and Cauchy's integral formula gives
\[
\frac{1}{2\pi i}\int_\Gamma \frac{1}{z(z-\lambda)}\,dz
=
\frac{1}{\lambda}.
\]
Consequently,
\[
\langle N_{\mathcal H} c,d\rangle
=
\int_{\operatorname{spec}(T)} \lambda^{-1}\,d\mu_{c,d}(\lambda)
=
\langle T^{-1}c,d\rangle.
\]
Since \(d\in\mathcal H\) is arbitrary, it follows that
\[
N_{\mathcal H}c=T^{-1}c=(G|_{\mathcal H})^{-1}c,
\qquad c\in\mathcal H.
\]

\smallskip
\noindent
Every \(c\in \ell^2(\mathcal P_\delta)\) decomposes as
\(
c=P_{\mathcal H}c+(I-P_{\mathcal H})c
\)
with \((I-P_{\mathcal H})c\in\mathcal H^\perp\). Since \(N_{\mathcal H}\) vanishes on \(\mathcal H^\perp\),
\(
N_{\mathcal H}c=N_{\mathcal H}P_{\mathcal H}c\),
so \(N_{\mathcal H}=N_{\mathcal H}P_{\mathcal H}\). On the other hand, if \(c\in\mathcal H\), then
\(N_{\mathcal H}c\in\mathcal H\), while if \(c\in\mathcal H^\perp\), then \(N_{\mathcal H}c=0\in\mathcal H\).
Hence \(\operatorname{Ran}(N_{\mathcal H})\subset \mathcal H\), and therefore
\(
P_{\mathcal H}N_{\mathcal H}=N_{\mathcal H}\).
This proves the proposition.
\end{proof}

\subsection{Proof of the envelope theorem}
\label{subsec:proof-envelope}

We prove Theorem~\ref{thm:envelope-recentering} by applying the localized
range inverse to the sampled kernel vector and deriving the required
subexponential envelope for the resulting re-centering coefficients.

\begin{proof}[Proof of Theorem~\ref{thm:envelope-recentering}]
Let \(\delta_0\) be as in Lemma~\ref{lem:kernel-frame}, and fix \(0<\delta<\delta_0\).
We use the notation introduced in Lemma~\ref{lem:kernel-frame} and Proposition~\ref{prop:app-localized-range-projection}.

\smallskip
\noindent
(i)
For \(\eta\in\mathbb C\), define
\(
b_Q(\eta):=\bigl\langle \kappa_\eta,\sqrt{|Q|}e_Q^\#\bigr\rangle_\alpha\), \(
Q\in\mathcal P_\delta\).
Then Lemma~\ref{lem:kernel-frame}\textup{(i)} gives
\[
\kappa_\eta
=
\sum_{Q\in\mathcal P_\delta}
b_Q(\eta)\,\kappa_{\zeta_Q}
\]
with convergence in \(\mathcal F_\alpha^2\), proving \textup{(i)}.

\smallskip
\noindent
(ii)
Fix \(0<\beta<1\) and \(\gamma>0\).  For the coefficients chosen in
\textup{(i)}, set
\[
g(\eta):= \mathcal{C}\kappa_\eta\in\mathcal H \quad \text{ and } \quad 
a(\eta):=\bigl(\bigl\langle \kappa_\eta,e_Q^\# \bigr\rangle_\alpha\bigr)_{Q\in\mathcal P_\delta}.
\]
Since \(S^{-1}\) is self-adjoint, we have
\(
a(\eta)= \mathcal{C} (S^{-1}\kappa_\eta)\in \operatorname{Ran}(\mathcal{C})=\mathcal H\).
Moreover,
\[
Ga(\eta)
=
\mathcal{C} \mathcal{C}^*\mathcal{C}(S^{-1}\kappa_\eta)
=
\mathcal{C}(\mathcal{C}^*\mathcal{C})S^{-1}\kappa_\eta
=
\mathcal{C}SS^{-1}\kappa_\eta
=
\mathcal{C}\kappa_\eta
=
g(\eta).
\]
Hence,
\(
a(\eta)=(G|_{\mathcal H})^{-1}g(\eta)\).
By Proposition~\ref{prop:app-localized-range-projection}\textup{(ii)},
\[
N_{\mathcal H}\in \mathcal A^{\exp}_{\gamma,\beta}(\mathcal P_\delta)
\qquad\text{and}\qquad
N_{\mathcal H}c=(G|_{\mathcal H})^{-1}c,\quad c\in\mathcal H.
\]
Since \(g(\eta)\in\mathcal H\), \(a(\eta)=N_{\mathcal H}g(\eta)\).  Therefore
\[
b_Q(\eta)
=
\sqrt{|Q|}\,a_Q(\eta)
=
\sqrt{|Q|}\,(N_{\mathcal H}g(\eta))_Q,
\qquad Q\in\mathcal P_\delta.
\]

\smallskip
\noindent
For \(Q'\in\mathcal P_\delta\),
\(
g_{Q'}(\eta)=\langle \kappa_\eta,e_{Q'}\rangle_\alpha\), so
\(|g_{Q'}(\eta)|
=
\sqrt{|Q'|}\,e^{-\frac{\alpha}{2}|\zeta_{Q'}-\eta|^2}\).
Since \(N_{\mathcal H}\in \mathcal A^{\exp}_{\gamma,\beta}(\mathcal P_\delta)\), we have
\[
\sum_{Q'\in\mathcal P_\delta}
|(N_{\mathcal H})_{Q,Q'}|\,e^{\gamma|\zeta_Q-\zeta_{Q'}|^\beta}
\le
\|N_{\mathcal H}\|_{\mathcal A^{\exp}_{\gamma,\beta}(\mathcal P_\delta)},
\qquad Q\in\mathcal P_\delta.
\]
Moreover, since \(0<\beta<1\), there exists \(C_0=C_0(\alpha,\beta,\gamma)>0\) such that
\(
e^{-\frac{\alpha}{2}t^2}\le C_0 e^{-\gamma t^\beta}\) for all \(t\ge0\).
Using also
\[
|\zeta_Q-\eta|^\beta
\le
|\zeta_Q-\zeta_{Q'}|^\beta+|\zeta_{Q'}-\eta|^\beta \quad \text{and} \quad |Q|,\ |Q'|\le \pi C_{\mathrm{out}}^2\delta^2,
\qquad Q,Q'\in\mathcal P_\delta,
\]
we obtain, for every \(Q\in\mathcal P_\delta\),
\begin{align*}
|b_Q(\eta)|
& \le
\sqrt{|Q|}\sum_{Q' \in\mathcal P_\delta}|(N_{\mathcal H})_{Q,Q'}|\,|g_{Q'}(\eta)| \le
\sum_{Q'\in\mathcal P_\delta}
\sqrt{|Q|\,|Q'|}\,|(N_{\mathcal H})_{Q,Q'}|\,
e^{-\frac{\alpha}{2}|\zeta_{Q'}-\eta|^2} \\
&\le
C_0\pi C_{\mathrm{out}}^2\delta^2
\sum_{Q'\in\mathcal P_\delta}
|(N_{\mathcal H})_{Q,Q'}|\,e^{\gamma|\zeta_Q-\zeta_{Q'}|^\beta}
e^{-\gamma|\zeta_Q-\zeta_{Q'}|^\beta-\gamma|\zeta_{Q'}-\eta|^\beta} \\
&\le
C_0\pi C_{\mathrm{out}}^2\delta^2
e^{-\gamma|\zeta_Q-\eta|^\beta}
\sum_{Q'\in\mathcal P_\delta}
|(N_{\mathcal H})_{Q,Q'}|\,e^{\gamma|\zeta_Q-\zeta_{Q'}|^\beta} \\
&\le
C_0\pi C_{\mathrm{out}}^2\delta^2
\|N_{\mathcal H}\|_{\mathcal A^{\exp}_{\gamma,\beta}(\mathcal P_\delta)}
e^{-\gamma|\zeta_Q-\eta|^\beta}.
\end{align*}
Thus, for every $\eta \in\mathbb C$ and $Q\in\mathcal P_\delta$,
\[
|b_Q(\eta)|\le C_{\mathrm{env}}\,e^{-\gamma|\zeta_Q-\eta|^\beta}
\qquad \text{with} \quad 
C_{\mathrm{env}}=
C_0\pi C_{\mathrm{out}}^2\delta^2
\|N_{\mathcal H}\|_{\mathcal A^{\exp}_{\gamma,\beta}(\mathcal P_\delta)}.
\]
With a fixed admissible contour \(\Gamma\) chosen in terms of \([a_\delta,b_\delta]\), Proposition~\ref{prop:app-localized-range-projection}\textup{(ii)}
shows that \(C_{\rm env}\) depends only on
\(\alpha,\beta,\gamma,\delta,c_{\rm in},C_{\rm out}\), and \(C_L\).
This proves (ii).
\end{proof}

\section{Atomic decomposition}
\label{sec:atomic}

We prove Theorem~\ref{thm:AD-main}.  Throughout this section
\(\alpha>0\) and \(0<p,q\le\infty\) are fixed.  We use a
\(\delta\)-pocket tiling \(\mathcal P_\delta\), with \(\delta>0\) sufficiently
small, and the associated coefficient spaces from
Subsection~\ref{subsec:atomic-pockettiling}.  The auxiliary Schur and Neumann
estimates are in Subsection~\ref{subsec:aux-operator-estimates}; the
re-centering step for \(0<p<1\) uses Section~\ref{sec:recentering-envelope}.

\smallskip
\noindent
The proof splits into the ranges \(1\le p\le\infty\) and \(0<p<1\).  In the
first range, the kernel analysis--synthesis scheme is stable after a small
remainder estimate and Neumann inversion.  In the second range, we pass through
finite-order jet atoms, eliminate the jets locally into floating kernels,
re-center the floating kernels by Section~\ref{sec:recentering-envelope}, and
iterate the remaining small defect.

\smallskip
\noindent
Subsection~\ref{subsec:atomic-pockettiling} sets up the geometry and
coefficient spaces.  Subsections~\ref{subsec:atomic-banach} and
\ref{subsec:atomic-quasi} prove Theorem~\ref{thm:AD-main} in the two angular
ranges.  Subsection~\ref{subsec:atomic-consequences} gives the equivalent
mixed-norm-normalized formulation.


\subsection{Geometric discretization and coefficient sequence spaces}\label{subsec:atomic-pockettiling}

This subsection sets up the discrete framework.  We introduce a
\(\delta\)-polar Whitney tiling \(\mathcal P_\delta\), define the coefficient
spaces \(\ell^{p,q}\) and \(\ell^{p,q}_w\), and use the annular layers
\[
A_k=\{z\in\C:k\le |z|<k+1\}.
\]
We then record ball control, bounded overlap, counting estimates, and a
bi-Lipschitz reindexing by \(\mathbb Z^2\), used later in the re-centering
step.

\medskip
\noindent
For the construction we use indices \(k,\nu,\mu\) for the radial level, radial
subdivision, and angular subdivision.  Afterward we return to the intrinsic
notation \(Q\in\mathcal P_\delta\) with center \(\zeta_Q\).

\begin{definition}[$\delta$-polar Whitney tiling]\label{def:pockets}
Fix $\delta\in(0, \frac{1}{2}]$ and let $N_\delta:=\lceil \delta^{-1}\rceil$.
We construct a countable family \(\mathcal P_\delta\) of half-open Borel sets
\(Q\), called pockets, which tiles \(\C\) up to null sets, together with
distinguished centers \(\zeta_Q\in Q\).

\smallskip
\noindent\textbf{Angular counts and angles.}
For $k\ge0$ and $\nu=0,1,\dots,N_\delta-1$, define
\[
N_{\delta,k,\nu}:=
\begin{cases}
2\nu+1, & k=0,\\
(2k+1)N_\delta, & k\ge1,
\end{cases}
\qquad \text{and} \qquad
\theta_{k,\nu,\mu}:=\frac{2\pi\mu}{N_{\delta,k,\nu}},
\quad \mu=0,1,\dots,N_{\delta,k,\nu},
\]
with indices modulo $N_{\delta,k,\nu}$ on the unit circle.

\smallskip
\noindent\textbf{Radial levels.}
For $k\ge1$ define radial levels on annuli $A_k: = \{z: k \leq |z| < k+1\}$ by equal increments in $r^2$:
\[
r_{k,\nu}^2:=k^2+\frac{\nu}{N_\delta}\bigl((k+1)^2-k^2\bigr)
= k^2+\frac{\nu}{N_\delta}(2k+1),
\qquad \nu=0,1,\dots,N_\delta,
\]
and for $k=0$ on the unit disc $A_0: = \{z:  |z| < 1\}$ set
\[
r_{0,\nu}:=\frac{\nu}{N_\delta},\qquad \nu=0,1,\dots,N_\delta.
\]

\smallskip
\noindent\textbf{Pockets and centers.}
For each $k\ge0$, $0\le\nu\le N_\delta-1$, and $0\le\mu\le N_{\delta,k,\nu}-1$, define the pocket
\[
Q_{k,\nu,\mu}:=\Bigl\{re^{i\theta}:\ r\in[r_{k,\nu},r_{k,\nu+1}),\ 
\theta\in[\theta_{k,\nu,\mu},\theta_{k,\nu,\mu+1})\Bigr\}.
\]
Define the radial midpoint by
\[
\rho_{k,\nu}:=\frac{r_{k,\nu}+r_{k,\nu+1}}{2}\qquad (k\ge0),
\]
and the angular midpoint by
\[
\vartheta_{k,\nu,\mu}:=\frac{\theta_{k,\nu,\mu}+\theta_{k,\nu,\mu+1}}{2}.
\]
Then set the center
\[
\zeta_{k,\nu,\mu}:=\rho_{k,\nu}e^{i\vartheta_{k,\nu,\mu}}\in Q_{k,\nu,\mu}.
\]

\smallskip
\noindent
We call $\mathcal{P}_{\delta}$ \emph{the $\delta$-polar Whitney tiling of $\C$};
its elements are called \emph{$\delta$-pockets}.  When no confusion arises, we
say \emph{$\delta$-pocket tiling}.  We write \(Q\in\mathcal P_\delta\) with
center \(\zeta_Q\), and denote
\[
\mathcal P_{\delta,k}:=\{Q \in  \mathcal{P}_{\delta}:  Q \subset A_k\},
\qquad
Z(\mathcal P_\delta):=\{\zeta_Q:\ Q\in\mathcal P_\delta\}.
\]
\smallskip
\noindent\textbf{Simplification for $k\ge1$.}
Since $N_{\delta,k,\nu}$ is independent of $\nu$ for $k\ge1$, we may write
\[
N_{\delta,k}:=N_{\delta,k,\nu}=(2k+1)N_\delta,
\qquad
\theta_{k,\mu}:=\theta_{k,\nu,\mu},
\qquad
\vartheta_{k,\mu}:=\vartheta_{k,\nu,\mu} \quad \text{whenever} \ k \geq 1.
\]
\end{definition}

\begin{remark}
We distinguish $k=0$ from $k\ge1$ because near the origin the angular
resolution must depend on the ring index $\nu$; the choice $N_{\delta, 0,\nu}=2\nu+1$
prevents degeneration of pocket shapes and preserves the equal-area property.

\smallskip
\noindent
Lemma~\ref{lem:geo-pockets} below shows that \(\mathcal P_\delta\) is a
\(\delta\)-pocket decomposition in the sense of
Definition~\ref{def:pocket-decomposition}, with absolute geometric constants
\(c_{\rm in}\) and \(C_{\rm out}\).
\end{remark}

\begin{definition}[Fattened pockets]\label{def:fattened-pockets}
Let $\delta\in(0, \frac{1}{2}]$ and let $\mathcal P_\delta$ be the $\delta$-pocket tiling of $\C$.
For each pocket $Q\in\mathcal P_\delta$ we define a \emph{fattened pocket} $Q^*\supset Q$
by enlarging $Q$ by one neighbor in the radial direction and one neighbor in the angular direction.

\smallskip
\noindent
For $Q=Q_{k,\nu,\mu}\subset A_k$ ($k\ge0$), set
\[
Q^*_{k,\nu,\mu}
:=
\Bigl\{re^{i\theta}:\ r\in[r_{k,\nu-1},\,r_{k,\nu+2}),\ 
\theta\in[\theta_{k,\nu,\mu-1},\,\theta_{k,\nu,\mu+2})\Bigr\},
\]
where $\mu$ is read modulo $N_{\delta,k,\nu}$.
The radial indices are interpreted using the conventions
\[
r_{0,-1}:=0,\qquad
r_{k,-1}:=r_{k-1,N_\delta-1}\ (k\ge1),\qquad
r_{k,N_\delta+1}:=r_{k+1,1}\ (k\ge0).
\]
We call $\mathcal P_\delta^*:=\{Q^*: Q\in\mathcal P_\delta\}$ the \emph{fattened $\delta$-pocket family}
(of the $\delta$-pocket tiling), and refer to its elements as \emph{fattened $\delta$-pockets}.
\end{definition}

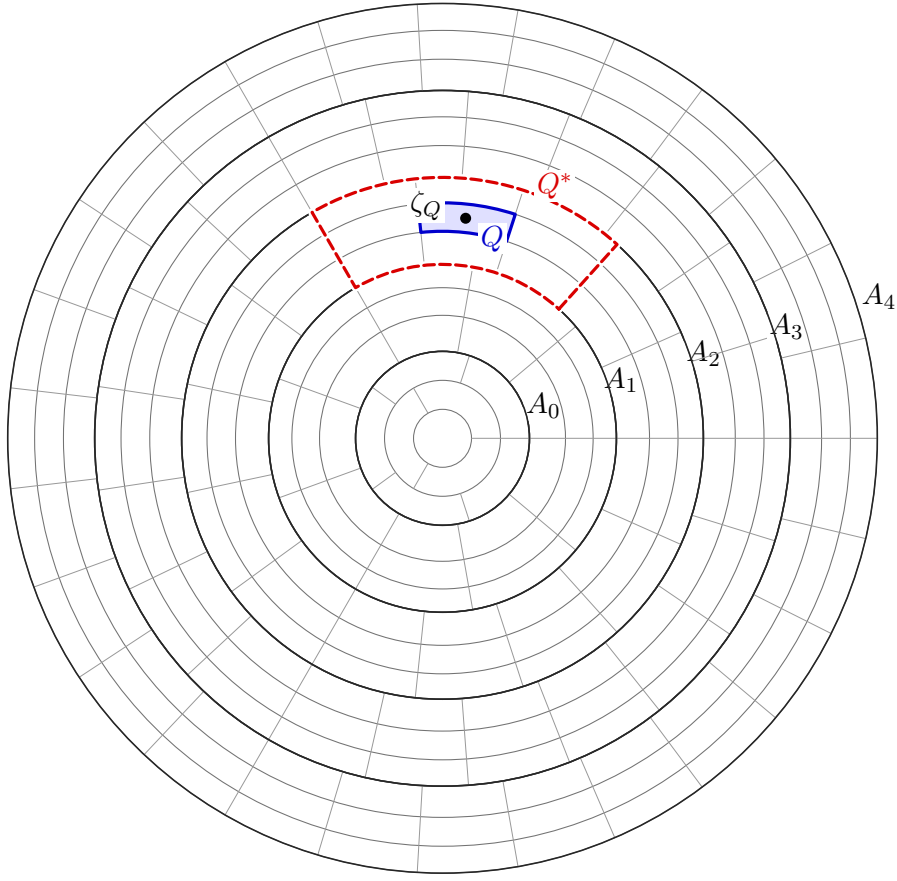
\begin{figure}[ht]
\centering

\begin{tikzpicture}[scale=1.15, line cap=round, line join=round]

\def\Ndelta{3}   
\def\kmax{4}     

\def\kQ{2}
\def\nuQ{1}
\def\muQ{3}

\tikzset{
  pocketcircle/.style={black!55, thin},
  pocketray/.style={black!40, very thin},
  outline/.style={black!85, semithick},
  qfill/.style={fill=blue!12, draw=blue!80!black, very thick},
  qstarline/.style={draw=red!85!black, dash pattern=on 4pt off 2pt, line width=1.2pt}
}

\newcommand{\rkNu}[2]{%
  \pgfmathsetmacro{\tempR}{sqrt((#1)^2 + (#2)*(2*(#1)+1)/(\Ndelta))}%
}

\pgfmathtruncatemacro{\Ndmone}{\Ndelta-1}

\foreach \nu in {0,...,\Ndelta}{
  \pgfmathsetmacro{\rnu}{\nu/(\Ndelta)}
  \draw[pocketcircle] (0,0) circle (\rnu);
}
\draw[outline] (0,0) circle (1);
\node[fill=white, inner sep=1pt] at (18:1.22) {$A_0$};

\foreach \nu in {1,...,\Ndmone}{
  \pgfmathtruncatemacro{\Mnu}{2*\nu+1}
  \pgfmathtruncatemacro{\Mnumone}{\Mnu-1}
  \pgfmathsetmacro{\dth}{360/(\Mnu)}
  \pgfmathsetmacro{\rlo}{\nu/(\Ndelta)}
  \pgfmathsetmacro{\rhi}{(\nu+1)/(\Ndelta)}
  \foreach \mu in {0,...,\Mnumone}{
    \pgfmathsetmacro{\ang}{\mu*\dth}
    \draw[pocketray] (\ang:\rlo) -- (\ang:\rhi);
  }
}

\foreach \k in {1,...,\kmax}{

  \rkNu{\k}{0}\pgfmathsetmacro{\rinner}{\tempR}
  \rkNu{\k}{\Ndelta}\pgfmathsetmacro{\router}{\tempR}

  \pgfmathtruncatemacro{\Nk}{(2*\k+1)*\Ndelta}
  \pgfmathtruncatemacro{\Nkmone}{\Nk-1}
  \pgfmathsetmacro{\dtheta}{360/(\Nk)}

  \foreach \nu in {0,...,\Ndelta}{
    \rkNu{\k}{\nu}
    \draw[pocketcircle] (0,0) circle (\tempR);
  }

  \foreach \mu in {0,...,\Nkmone}{
    \pgfmathsetmacro{\ang}{\mu*\dtheta}
    \draw[pocketray] (\ang:\rinner) -- (\ang:\router);
  }

  \draw[outline] (0,0) circle (\rinner);
  \draw[outline] (0,0) circle (\router);

  \ifnum\k=4
  \pgfmathsetmacro{\labR}{\router+0.28}
\else
  \pgfmathsetmacro{\labR}{\router+0.16}
  \fi
\node[fill=white, inner sep=1pt] at (18:\labR) {$A_{\k}$};
}


\pgfmathtruncatemacro{\NkQ}{(2*\kQ+1)*\Ndelta}
\pgfmathsetmacro{\dthetaQ}{360/(\NkQ)}

\rkNu{\kQ}{\nuQ-1}\pgfmathsetmacro{\rQstarin}{\tempR}
\rkNu{\kQ}{\nuQ}\pgfmathsetmacro{\rQin}{\tempR}
\rkNu{\kQ}{\nuQ+1}\pgfmathsetmacro{\rQout}{\tempR}
\rkNu{\kQ}{\nuQ+2}\pgfmathsetmacro{\rQstarout}{\tempR}

\pgfmathsetmacro{\thQm}{(\muQ-1)*\dthetaQ}
\pgfmathsetmacro{\thQa}{\muQ*\dthetaQ}
\pgfmathsetmacro{\thQb}{(\muQ+1)*\dthetaQ}
\pgfmathsetmacro{\thQp}{(\muQ+2)*\dthetaQ}

\pgfmathsetmacro{\rhoQ}{0.5*(\rQin+\rQout)}
\pgfmathsetmacro{\varthetaQ}{0.5*(\thQa+\thQb)}

\pgfmathsetmacro{\QlabelAngle}{\varthetaQ-8}
\pgfmathsetmacro{\QlabelRadius}{\rhoQ-0.18}

\pgfmathsetmacro{\zlabelAngle}{\varthetaQ+10}
\pgfmathsetmacro{\zlabelRadius}{\rhoQ+0.14}

\pgfmathsetmacro{\QstarLabelAngle}{0.5*(\thQm+\thQp)-18}
\pgfmathsetmacro{\QstarLabelRadius}{\rQstarout+0.18}

\draw[white, line width=3pt]
  (\thQm:\rQstarin) -- (\thQm:\rQstarout)
  arc[start angle=\thQm, end angle=\thQp, radius=\rQstarout]
  -- (\thQp:\rQstarin)
  arc[start angle=\thQp, end angle=\thQm, radius=\rQstarin]
  -- cycle;

\draw[qstarline]
  (\thQm:\rQstarin) -- (\thQm:\rQstarout)
  arc[start angle=\thQm, end angle=\thQp, radius=\rQstarout]
  -- (\thQp:\rQstarin)
  arc[start angle=\thQp, end angle=\thQm, radius=\rQstarin]
  -- cycle;

\filldraw[qfill]
  (\thQa:\rQin) -- (\thQa:\rQout)
  arc[start angle=\thQa, end angle=\thQb, radius=\rQout]
  -- (\thQb:\rQin)
  arc[start angle=\thQb, end angle=\thQa, radius=\rQin]
  -- cycle;

\fill (\varthetaQ:\rhoQ) circle (1.8pt);

\node[text=blue!80!black, fill=white, inner sep=1pt]
  at (\QlabelAngle:\QlabelRadius) {$Q$};

\node[text=red!85!black, fill=white, inner sep=1pt]
  at (\QstarLabelAngle:\QstarLabelRadius) {$Q^*$};

\node[fill=white, inner sep=1pt]
  at (\zlabelAngle:\zlabelRadius) {$\zeta_Q$};

\end{tikzpicture}

\caption{A pocket \(Q\in\mathcal P_\delta\), its center \(\zeta_Q\), and the associated fattened pocket \(Q^*\).}

\end{figure}

\begin{definition} \label{def:coeff-spaces}
Let $0<p,q\le\infty$, and let $\mathcal P_{\delta}$ be a $\delta$-pocket tiling of $\C$. For a sequence $c=(c_Q)_{Q\in\mathcal P_\delta}$ and  each $k \geq 0$, we denote by
\(
c_k:=\bigl(c_Q\bigr)_{Q\in\mathcal P_{\delta,k}}
\)
its $k$-th block and define
\[
\|c_k\|_{\ell^p}:=
\begin{cases}
\Big(\displaystyle\sum_{Q\in\mathcal P_{\delta,k}} |c_Q|^p\Big)^{\frac{1}{p}}, & 0<p<\infty,\\[2mm]
\displaystyle\sup_{Q\in\mathcal P_{\delta,k}} |c_Q|, & p=\infty.
\end{cases}
\]
The mixed sequence space $\ell^{p,q}$ consists of all $c$ such that
\[
\|c\|_{\ell^{p,q}}:=
\begin{cases}
\Big(\displaystyle\sum_{k\ge0}\|c_k\|_{\ell^p}^{\,q}\Big)^{\frac{1}{q}}, & 0<q<\infty,\\[2mm]
\displaystyle\sup_{k\ge0}\|c_k\|_{\ell^p}, & q=\infty,
\end{cases}
\]
is finite. The weighted space $\ell^{p,q}_w$ is defined by the quasi-norm
\[
\|c\|_{\ell^{p,q}_w}:=
\begin{cases}
\Big(\displaystyle\sum_{k\ge0}\bigl[w_k\,\|c_k\|_{\ell^p}\bigr]^{q}\Big)^{\frac{1}{q}}, & 0<q<\infty,\\[2mm]
\displaystyle\sup_{k\ge0} w_k\,\|c_k\|_{\ell^p}, & q=\infty,
\end{cases}
\qquad\text{where}\qquad
w_k:=(1+k)^{\frac1q-\frac1p}.
\]
\end{definition}
\begin{remark}
Throughout the paper, \(\ell^{p,q}\) and \(\ell_w^{p,q}\) refer to the
coefficient sequence spaces associated with the fixed tiling
\(\mathcal P_\delta\), equivalently with the block decomposition
\(
\mathcal P_\delta=\bigsqcup_{k\ge0}\mathcal P_{\delta,k}
\).
We suppress the dependence on \(\mathcal P_\delta\) in the notation.
\end{remark}

\subsubsection{Geometric properties of the $\delta$-pocket tiling}

We record the geometric properties of the tiling: exact area, uniform ball
control, counting estimates, and bounded overlap of fattened pockets.

\begin{lemma} \label{lem:geo-pockets}
Let $\delta\in(0, \frac{1}{2}]$ and let $\mathcal P_\delta$ be the $\delta$-pocket tiling of $\C$ in
Definition~\ref{def:pockets}. Then:
\begin{enumerate}
\item[\textup{(a)}]\emph{(Partition)} $\displaystyle \C=\bigsqcup_{Q\in\mathcal P_\delta}Q$.

\item[\textup{(b)}] \emph{(Exact area)} $\displaystyle |Q|= \frac{\pi}{N_\delta^2}$ for all $Q\in\mathcal P_\delta$.

\item[\textup{(c)}] \emph{(Uniform ball control)} There exist absolute constants $0<c_{\rm in}<C_{\rm out}<\infty$ (for instance,
$c_{\rm in}= \frac{1}{32}$ and $C_{\rm out}=4$) such that
\[
B(\zeta_Q,c_{\rm in}\delta)\subset Q\subset B(\zeta_Q,C_{\rm out}\delta),
\qquad Q\in\mathcal P_\delta .
\]

\item[\textup{(d)}] \emph{(Counting)} For every $z\in\C$ and every $R>C_{\rm out}\delta$,
\[
\left(\frac{R}{\delta}-C_{\rm out}\right)^2
\le
\#\{Q\in\mathcal P_\delta:\ \zeta_Q\in B(z,R)\}
\le 4\left(\frac{R}{\delta}+C_{\rm out}\right)^2 .
\]
Moreover, the upper bound remains valid for every \(R>0\).
\end{enumerate}
\end{lemma}

\begin{proof}
\textup{(a)}
This is immediate from the half--open polar construction of $Q_{k,\nu,\mu}$.

\smallskip
\noindent\textup{(b)}
By the definition of $Q = Q_{k, \nu, \mu}$,
\[
|Q|=\int_{\theta_{k, \nu, \mu}}^{\theta_{k, \nu, \mu + 1}}\int_{r_{k,\nu}}^{r_{k,\nu + 1}} r\,dr\,d\theta
=\frac12 \left(r_{k,\nu + 1}^2-r_{k,\nu}^2\right)\left(\theta_{k, \nu, \mu + 1}-\theta_{k, \nu, \mu}\right).
\]
For $k\ge1$ we have $r_{k,\nu + 1}^2-r_{k,\nu}^2=\frac{2k+1}{N_\delta}$ and $\theta_{k, \nu, \mu + 1}-\theta_{k, \nu, \mu}=\frac{2\pi}{(2k+1)N_\delta}$, hence
$|Q|=\frac{\pi}{N_\delta^2}$. For $k=0$ one has $r_{k,\nu + 1}^2-r_{k,\nu}^2=\frac{2\nu+1}{N_\delta^2}$ and
$\theta_{k, \nu, \mu + 1}-\theta_{k, \nu, \mu}=\frac{2\pi}{2\nu+1}$, giving the same value.

\smallskip
\noindent\textup{(c)}
Write $\zeta_Q=\rho e^{i\vartheta}$ (with $\rho$ and $\vartheta$ the midpoints of the radial and angular
intervals of $Q$) and let $z=re^{i\theta}\in Q$. Then
\begin{equation}\label{eq:polar-tri-short}
|z-\zeta_Q|\le |r-\rho|+\rho|\theta-\vartheta|.
\end{equation}
Let $\Delta r:=r_{k,\nu + 1}-r_{k,\nu}$ and $\Delta\theta:=\theta_{k, \nu, \mu + 1}-\theta_{k, \nu, \mu}$.
By construction, $|r-\rho|\le \frac{\Delta r}{2}$ and $|\theta-\vartheta|\le \frac{\Delta\theta}{2}$. Moreover:
\begin{itemize}
\item if $k\ge1$, then $\displaystyle \frac{1}{2 N_{\delta}} \leq \Delta r \leq \frac{3}{2 N_{\delta}}$ and
$\displaystyle \frac{2 \pi}{3N_\delta} \leq \rho\Delta\theta \leq  \frac{2 \pi}{N_\delta}$ (uniformly in $\nu,\mu$);
\item if $k=0$, then $\Delta r=\frac{1}{N_\delta}$ and $\rho\Delta\theta 
=\frac{\pi}{N_\delta}$.
\end{itemize}
Inserting these bounds into \eqref{eq:polar-tri-short} yields $|z-\zeta_Q|\le 4\delta$,
hence $Q\subset B(\zeta_Q,C_{\rm out}\delta)$ for some absolute $C_{\rm out} > 0$ (for instance, $C_{\rm out}: = 4$).

\smallskip
\noindent
For the inner inclusion, the distance from $\zeta_Q$ to each radial side equals $\frac{\Delta r}{2} \geq \frac{1}{4 N_\delta}$.
The distance to each angular side is $\rho\sin(\frac{\Delta\theta}{2})\geq \frac{1}{\pi}  \rho\Delta\theta\geq \frac{2}{3N_\delta}$
(for $\nu=0$ there is no angular restriction).
Thus $\textup{dist}(\zeta_Q,\partial Q)\geq \frac{1}{ 8N_\delta}\geq \frac{ \delta}{16}$, proving
$B(\zeta_Q,c_{\rm in}\delta)\subset Q$ for some absolute $c_{\rm in} > 0$ (for instance, $c_{\rm in} : = \frac{1}{32}$).

\smallskip
\noindent\textup{(d)}
Let $R>C_{\rm out}\delta$. If $w\in B(z,R-C_{\rm out}\delta)$ and $Q_w$ is the (a.e. unique) pocket
containing $w$, then \textup{(c)} gives $|\zeta_{Q_w}-w|\le C_{\rm out}\delta$, hence $\zeta_{Q_w}\in B(z,R)$.
Therefore
\[
B(z,R-C_{\rm out}\delta)\subset \bigcup_{\zeta_Q\in B(z,R)}Q,
\]
and disjointness in (a) and $|Q|=\frac{\pi}{N_\delta^2}$ in (b) yield the lower bound. The upper bound follows similarly from
$Q\subset B(\zeta_Q,C_{\rm out}\delta)$.
\end{proof}

\begin{lemma} \label{lem:geo-fattened}
Let $\delta\in(0, \frac{1}{2}]$, let $\mathcal P_\delta$ be the $\delta$-pocket tiling of $\C$, and let
$\mathcal P_\delta^*=\{Q^*:Q\in\mathcal P_\delta\}$ be the fattened family defined in Definition~\ref{def:fattened-pockets}. Then there exist absolute constants
$C_*,N_*\in(1,\infty)$, independent of $\delta$, such that:

\begin{enumerate}
\item[\textup{(a)}] \emph{(Containment and size)} For every $Q\in\mathcal P_\delta$,
\[
Q\subset Q^*\subset B(\zeta_Q,C_*\delta),
\qquad
|Q^*|\asymp \delta^2,
\qquad
\rm{diam}(Q^*)\asymp \delta,
\]
where the implicit constants are absolute.

\item[\textup{(b)}] \emph{(Bounded overlap)} For every $z\in\C$,
\[
\sum_{Q\in\mathcal P_\delta}\mathbf 1_{Q^*}(z)\ \le\ N_*.
\]
\end{enumerate}
\end{lemma}

\begin{proof}
We use Lemma~\ref{lem:geo-pockets}\textup{(c)}: there exist absolute constants
$c_{\rm in},C_{\rm out}>0$ such that
\begin{equation}\label{eq:basic-ball-fattened}
B(\zeta_Q,c_{\rm in}\delta)\subset Q\subset B(\zeta_Q,C_{\rm out}\delta),
\qquad Q\in\mathcal P_\delta.
\end{equation}

\smallskip
\noindent\textup{(a)}
The inclusion $Q\subset Q^*$ is immediate from Definition~\ref{def:fattened-pockets}.
Fix $Q=Q_{k,\nu,\mu}$ and write $\zeta_Q=\rho e^{i\vartheta}$. For $z=re^{i\theta}\in Q^*$ we have
\begin{equation}\label{eq:polar-tri-fattened-final}
|z-\zeta_Q|\le |r-\rho|+\rho|\theta-\vartheta|.
\end{equation}
By the construction of $Q^*$,
\[
|r-\rho|\le \tfrac12\bigl(r_{k,\nu+2}-r_{k,\nu-1}\bigr)
\le \frac{9}{4N_\delta},
\qquad
\rho|\theta-\vartheta|
\le \tfrac{\rho}{2}\bigl(\theta_{k,\nu,\mu+2}-\theta_{k,\nu,\mu-1}\bigr)
\le \frac{3\pi}{N_\delta},
\]
where we used the bounds for $\Delta r$ and $\rho\Delta\theta$ established in the proof of
Lemma~\ref{lem:geo-pockets}. Inserting these estimates into \eqref{eq:polar-tri-fattened-final} gives
\[
|z-\zeta_Q|\le \Bigl(3\pi+\frac94\Bigr)\frac{1}{N_\delta}
\le \Bigl(3\pi+\frac94\Bigr)\delta,
\]
and therefore $Q^*\subset B(\zeta_Q,C_*\delta)$ with $C_*:=3\pi+\frac94$.
Moreover, the bounds on $|Q^*|$ and $\rm{diam}(Q^*)$ now follow from the inclusions
$B(\zeta_Q,c_{\rm in}\delta)\subset Q\subset Q^*\subset B(\zeta_Q,C_*\delta)$.

\medskip
\noindent\textup{(b)}
Fix $z\in\C$. If $z\in Q^*$, then by \textup{(a)} we have $\zeta_Q\in B(z,C_*\delta)$. Hence
\[
\#\{Q\in\mathcal P_\delta:\ z\in Q^*\}
\le
\#\{Q\in\mathcal P_\delta:\ \zeta_Q\in B(z,C_*\delta)\}.
\]
Applying Lemma~\ref{lem:geo-pockets}\textup{(d)} with $R=C_*\delta$ (and noting that $C_*>C_{\rm out}$)
yields the uniform bound
\[
\#\{Q\in\mathcal P_\delta:\ \zeta_Q\in B(z,C_*\delta)\}\ \le\ 4(C_*+C_{\rm out})^2=:N_*,
\]
which is independent of $z$ and $\delta$. This gives the asserted overlap bound.
\end{proof}

\subsubsection{Bi-Lipschitz reindexing of pockets}
We reindex the pocket centers by \(\mathbb Z^2\) in a bi-Lipschitz way,
allowing localized estimates on pockets to be transferred to a standard
lattice setting.  This reindexing is used with the envelope theorem in
Section~\ref{sec:recentering-envelope} during the angular quasi-Banach
re-centering step.

\begin{lemma} \label{lem:biLip}
Let $\delta\in(0, \frac{1}{2}]$ and let $\mathcal P_\delta$ be the $\delta$-pocket tiling of $\C$ with centers  
$\zeta_Q$. Then there exist an absolute constant
$C_L\ge1$ and a bijection $\Phi_\delta:\mathcal P_\delta\to\Z^2$ such that
\begin{equation}\label{eq:bilip-main}
C_L^{-1}\,\delta\,|\Phi_\delta(Q)-\Phi_\delta(R)|
\le |\zeta_Q-\zeta_R|
\le C_L\,\delta\,|\Phi_\delta(Q)-\Phi_\delta(R)|,
\qquad Q,R\in\mathcal P_\delta.
\end{equation}
Moreover, $\Phi_\delta$ can be chosen so that
\begin{equation}\label{eq:bilip-displacement-natural}
\sup_{Q\in\mathcal P_\delta}\left|\zeta_Q-\frac{\sqrt{\pi}}{N_\delta}\Phi_\delta(Q)\right|\le C_L \delta.
\end{equation}
\end{lemma}

\begin{proof}
After rescaling the tiling to unit-area tiles, the scaled centers form a
Delone set satisfying a uniform discrepancy bound for polyominoes.
Laczkovich's theorem then gives a bounded-displacement bijection with
\(\Z^2\), and scaling back gives the required bi-Lipschitz estimates.

\smallskip
\noindent\textbf{Scaling.}
By Lemma~\ref{lem:geo-pockets}\textup{(b)}, each pocket has area $|Q|= \frac{\pi}{N_\delta^2}$.
Set
\[
s_\delta:=\frac{N_\delta}{\sqrt{\pi}},
\qquad
\widetilde Q:=s_\delta Q,
\qquad
x_Q:=s_\delta\zeta_Q,
\qquad
\mathcal E_\delta:=\{x_Q:\ Q\in\mathcal P_\delta\}.
\]
Then $|\widetilde Q|=1$ for all $Q$ and $\{\widetilde Q\}$ tiles $\R^2$ up to null sets.

\smallskip
\noindent\textbf{Delone property.}
By Lemma~\ref{lem:geo-pockets}\textup{(c)}, there exist $c_{\rm in},C_{\rm out}>0$ such that
$B(\zeta_Q,c_{\rm in}\delta)\subset Q\subset B(\zeta_Q,C_{\rm out}\delta)$.
Using $\delta\asymp N_\delta^{-1}$ and multiplying by $s_\delta$ yields absolute
$c_0,C_0>0$ such that
\begin{equation}\label{eq:scaled-balls-polished}
B(x_Q,c_0)\subset \widetilde Q\subset B(x_Q,C_0),\qquad Q\in\mathcal P_\delta.
\end{equation}
In particular, $\mathcal E_\delta$ is uniformly discrete (distinct points are $c_0$-separated) and
relatively dense (every $y\in\R^2$ lies within distance $C_0$ of some $x_Q$). Hence $\mathcal E_\delta$
is a Delone set with parameters independent of $\delta$.

\smallskip
\noindent\textbf{A discrepancy bound for polyominoes.}
Let $H\subset\R^2$ be a finite union of half-open unit squares and let $p(H)$ denote its perimeter.
Define
\[
\mathcal I(H):=\{Q \in \mathcal{P}_{\delta}:\ \widetilde Q\subset H\},\qquad
\mathcal B(H):=\{Q \in \mathcal{P}_{\delta} :\ \widetilde Q\cap H\neq\emptyset,\ \widetilde Q\not\subset H\}.
\]
Since $|\widetilde Q|=1$ and $\{\widetilde Q\}$ tiles $\R^2$, one has
\[
|H|-\#\mathcal I(H)=\sum_{Q\in\mathcal B(H)}|H\cap \widetilde Q|, \quad \textup{hence, } \ 0\le |H|-\#\mathcal I(H)\le \#\mathcal B(H).
\]
Also, $
\#(\mathcal E_\delta\cap H)
=
\#\mathcal I(H)+\#\{Q\in\mathcal B(H):\ x_Q\in H\}$, so
$|\#(\mathcal E_\delta\cap H)-\#\mathcal I(H)|\le \#\mathcal B(H)$.
Therefore
\begin{equation}\label{eq:disc-boundary}
\bigl|\#(\mathcal E_\delta\cap H)-|H|\bigr|\ \le\ 2\,\#\mathcal B(H).
\end{equation}
If $Q\in\mathcal B(H)$ then $\widetilde Q$ meets $\partial H$, and \eqref{eq:scaled-balls-polished}
implies $\textup{dist}(x_Q,\partial H)\le C_0$. By separation, the balls $B(x_Q, \frac{c_0}{2})$ are disjoint, and all lie
in the $C_1$-neighborhood of $\partial H$ with $C_1:=C_0+\frac{c_0}{2}$. Hence
\[
\#\mathcal B(H)\cdot \pi\left(\frac{c_0}{2}\right)^2 \ \le\ |\{y:\textup{dist}(y,\partial H)\le C_1\}|.
\]
Since $\partial H$ consists of $p(H)$ unit axis-parallel segments, the $C_1$-neighborhood of $\partial H$
has area $\lesssim p(H)$ (with an absolute constant depending only on $C_1$). Consequently,
\begin{equation}\label{eq:lacz-disc-polished}
\bigl|\#(\mathcal E_\delta\cap H)-|H|\bigr|\ \le\ C_{\rm disc}\,p(H),
\end{equation}
uniformly in $\delta$.

\smallskip
\noindent\textbf{Bounded displacement to $\Z^2$.}
By the Delone property and \eqref{eq:lacz-disc-polished}, we may apply
\cite[Theorem~1.1]{L92} (with density $1$) to obtain a bijection $\Psi_\delta:\mathcal E_\delta\to\Z^2$
and an absolute constant $M>0$ such that
\begin{equation}\label{eq:bd-polished}
|\Psi_\delta(x)-x|\le M,\qquad x\in\mathcal E_\delta.
\end{equation}
Define $\Phi_\delta(Q):=\Psi_\delta(x_Q)\in\Z^2$. Then $\Phi_\delta$ is a bijection.

\smallskip
\noindent\textbf{Bi-Lipschitz bounds.}
Let $m=\Phi_\delta(Q)$ and $n=\Phi_\delta(R)$. From \eqref{eq:bd-polished},
\[
\bigl||x_Q-x_R|-|m-n|\bigr|\le |x_Q-m|+|x_R-n|\le 2M.
\]
Using the separation $|x_Q-x_R|\ge 2c_0$ for $Q\neq R$, this additive error is absorbed into a
multiplicative constant: there exists an absolute $C_{\rm bi}\ge1$ such that
\begin{equation}\label{eq:bilip-scaled-polished}
C_{\rm bi}^{-1}|m-n|\le |x_Q-x_R|\le C_{\rm bi}|m-n|,
\qquad Q,R\in\mathcal P_\delta.
\end{equation}
Since $x_Q=s_\delta\zeta_Q$ and $s_\delta=\frac{N_\delta}{\sqrt{\pi}}$, multiplying by $\frac{\sqrt{\pi}}{N_\delta}$
and using $ N_\delta^{-1}\asymp\delta$ yields \eqref{eq:bilip-main} with an absolute constant $C_L$.

\smallskip
\noindent\textbf{Bounded displacement in the natural mesh.}
From \eqref{eq:bd-polished}, $|x_Q-\Phi_\delta(Q)|\le M$, hence
\[
\Bigl|\zeta_Q-\frac{\sqrt{\pi}}{N_\delta}\Phi_\delta(Q)\Bigr|
=\frac{\sqrt{\pi}}{N_\delta}|x_Q-\Phi_\delta(Q)|
\le \frac{M\sqrt{\pi}}{N_\delta}
\le 2M\sqrt{\pi}\,\delta,
\]
which implies \eqref{eq:bilip-displacement-natural} after enlarging $C_L$.
\end{proof}


\subsection{Angular Banach range $1 \leq p \leq \infty$}
\label{subsec:atomic-banach}

We prove Theorem~\ref{thm:AD-main} for \(1\le p\le\infty\) using the kernel
analysis and synthesis operators
\[
C^{\ker}_\delta:\mathcal F_\alpha^{p,q}\to \ell^{p,q}_w,
\qquad
S^{\ker}_\delta:\ell^{p,q}_w\to\mathcal F_\alpha^{p,q}
\]
and show that the remainder operator
\[
R^{\ker}_\delta:=I-S^{\ker}_\delta C^{\ker}_\delta
\]
is small for \(\delta>0\) sufficiently small.  Neumann inversion of
\(I-R^{\ker}_\delta=S^{\ker}_\delta C^{\ker}_\delta\) gives the exact
decomposition and the two-sided coefficient estimate.

\subsubsection{Kernel operators and stability}
\label{subsubsec:atomic-operators-banach}

\medskip
\noindent

\begin{definition} \label{def:kernel-analysis-synthesis}
Fix $\delta\in(0, \frac{1}{2}]$ and let $\mathcal P_\delta$ be the $\delta$-pocket tiling of $\C$ with centers
$\{\zeta_Q\}_{Q\in\mathcal P_\delta}$. We define the \emph{kernel analysis} operator
$C_{\delta}^{\rm ker}$ by
\begin{equation}\label{eq:Cdelta-ker}
(C^{\mathrm ker}_{\delta}f)_{Q}
:=\frac{\alpha}{\pi}\int_{Q} f(z)\,e^{-\frac{\alpha}{2}|z|^2}\,
e^{\,i\alpha\mathrm{Im}(\zeta_Q\overline{z})}\,dA(z),
\qquad Q\in\mathcal P_\delta,
\end{equation}
whenever the right-hand side is well defined.  Define the \emph{kernel synthesis} operator \(S_{\delta}^{\rm ker}\)
by
\begin{equation}\label{eq:Sdelta-ker}
(S^{\mathrm ker}_{\delta}c)(z)
:= \sum_{Q\in\mathcal P_\delta} c_Q\,\kappa_{\zeta_Q}(z),
\qquad z\in\C,
\end{equation}
for coefficient sequences \(c=(c_Q)_{Q\in\mathcal P_\delta}\) for which the
series converges, and define the \emph{kernel remainder} operator by
\[
R_\delta^{\ker}:=I-S_\delta^{\ker}C_\delta^{\ker}.
\]
\end{definition}

\bigskip
\noindent
We start with boundedness of the kernel analysis operator.

\begin{lemma} \label{lem:analysis-bdd}
Let $\al > 0$, $1\le p\le\infty$, and
$0<q\le\infty$. Let $\mathcal P_\delta$ be a $\delta$-pocket tiling of $\C$ with
$\delta\in(0,\tfrac12]$. Then 
\[
\|C_\delta^{\rm ker}f\|_{\ell^{p,q}_w}\ \lesssim \,  \delta^{1- \frac{1}{p}} \, \|f\|_{\Fpq},
\qquad f\in\mathcal F^{p,q}_\alpha,
\]
where
the implicit constant depends only on \(\alpha, p, q\),  and is
independent of \(\delta\).
In particular, $C_{\delta}^{\rm{ker}}: \Fpq \to \ell^{p,q}_{w}$ is a bounded linear operator.
\end{lemma}

\begin{proof}
Since $\bigl|e^{\,i\alpha\, \im(\zeta_Q\overline z)}\bigr|=1$, Definition~\ref{def:kernel-analysis-synthesis}
gives the pointwise bound
\begin{equation}\label{eq:Cd-ptwise}
|(C^{\mathrm ker}_\delta f)_Q|
\le \frac{\alpha}{\pi}\int_Q |f(z)|\,e^{-\frac{\alpha}{2}|z|^2}\,dA(z),
\qquad Q\in\mathcal P_\delta.
\end{equation}
For $k\ge0$, recall  $\mathcal P_{\delta,k} =\{Q\in\mathcal P_\delta:\,Q\subset A_k\}$ and denote by
$(C^{\mathrm ker}_\delta f)_k$ the block of coefficients indexed by $Q \in \mathcal P_{\delta,k}$. We claim
that
\begin{equation}\label{eq:block-estimate}
\|(C^{\mathrm ker}_\delta f)_k\|_{\ell^p}\ \lesssim\, \delta^{1 - \frac{1}{p}}\,(1+k)^{\frac{1}{p}}\,B_p(f;k),
\qquad k\ge0,
\end{equation}
where 
\[
B_p(f;k):=\sup_{r\in[k,k+1)} M_p(f,r)\,e^{-\frac{\alpha}{2}r^2}
\]
is the annular block quantity used in the annular discretization principle in Proposition~\ref{prop:annular}.

\medskip
\noindent\emph{Step 1: Proof of \eqref{eq:block-estimate} for $1\le p<\infty$ and $k\ge1$.}
For $k\ge1$ the pockets in $A_k$ are rectangles in polar coordinates of the form
$[r_{k,\nu},r_{k,\nu+1}) \times [\theta_{k,\mu},\theta_{k,\mu+1})$, with $N_\delta$ radial slices and
$N_{\delta,k}$ angular slices (notation as in the construction of $\mathcal P_\delta$). Using
\eqref{eq:Cd-ptwise}, inserting the indicator
$\mathbf 1_{[r_{k,\nu}, r_{k,\nu + 1})}(r)$, and applying Minkowski's integral inequality in the Banach space
$\ell^p(\{(\nu,\mu)\})$ with $p \geq 1$, we obtain
\begin{align*}
\|(C^{\mathrm ker}_\delta f)_k\|_{\ell^p}
& \le \frac{\alpha}{\pi}
\left(\sum_{\nu=0}^{N_\delta-1}\sum_{\mu=0}^{N_{\delta,k}-1}\left[\int_{r_{k,\nu}}^{r_{k,\nu + 1}}\int_{\theta_{k,\mu}}^{\theta_{k,\mu + 1}}
|f(re^{i\theta})|\,e^{-\frac{\alpha}{2}r^2}\,r\,d\theta\,dr\right]^p\right)^{\frac1p} \\
& \le \frac{\alpha}{\pi}
\left(\sum_{\nu=0}^{N_\delta-1}\sum_{\mu=0}^{N_{\delta,k}-1}\left[\int_{k}^{k+1} \mathbf 1_{[r_{k,\nu}, r_{k,\nu + 1})}(r) \int_{\theta_{k,\mu}}^{\theta_{k,\mu + 1}}
|f(re^{i\theta})|\,e^{-\frac{\alpha}{2}r^2}\,r\,d\theta\,dr\right]^p\right)^{\frac1p} \\
& \leq \frac{\alpha}{\pi}\int_k^{k+1}
\left(\sum_{\nu=0}^{N_\delta-1}\sum_{\mu=0}^{N_{\delta,k}-1}\left[\mathbf 1_{[r_{k,\nu}, r_{k,\nu + 1})}(r) \int_{\theta_{k,\mu}}^{\theta_{k,\mu + 1}}|f(re^{i\theta})|\,d\theta\right]^p
\right)^{\frac1p}
e^{-\frac{\alpha}{2}r^2}\,r\,dr.
\end{align*}
For each fixed $r\in[k,k+1)$, exactly one $\nu$ satisfies $\mathbf 1_{[r_{k,\nu}, r_{k,\nu + 1})}(r)=1$. Therefore the $\nu$-sum collapses and we obtain
\[
\|(C^{\mathrm ker}_\delta f)_k\|_{\ell^p}
\lesssim \int_k^{k+1}
\left(\sum_{\mu=0}^{N_{\delta,k}-1}\left[\int_{\theta_{k,\mu}}^{\theta_{k,\mu+1}}
|f(re^{i\theta})|\,d\theta\right]^p\right)^{\frac{1}{p}}
e^{-\frac{\alpha}{2}r^2}\,r\,dr.
\]
For each fixed $r\in[k,k+1)$, H\"older's inequality on each arc
$I_{k,\mu}: = [\theta_{k, \mu}, \theta_{k, \mu + 1})$ gives
\[
\int_{I_{k,\mu}} |f(re^{i\theta})|\,d\theta
\le |I_{k,\mu}|^{1-\frac1p}\left(\int_{I_{k,\mu}}|f(re^{i\theta})|^p\,d\theta\right)^{\frac{1}{p}}.
\]
Raising to the $p$-th power and summing over $\mu$ yields
\[
\left(\sum_{\mu=0}^{N_{\delta,k}-1}\left[\int_{I_{k,\mu}}|f(re^{i\theta})|\,d\theta\right]^p\right)^{\frac{1}{p}}
\le |I_{k,\mu}|^{1-\frac1p}\left(\int_0^{2\pi}|f(re^{i\theta})|^p\,d\theta\right)^{\frac{1}{p}} \asymp  |I_{k,\mu}|^{1-\frac1p} \; M_p(f,r).
\]
Since $|I_{k,\mu}|=\Delta\theta_k= \frac{2\pi}{(2k+1)N_{\delta}}$ for
$k\ge1$, we have
\[
|I_{k,\mu}|^{1-\frac1p}\asymp N_\delta^{\frac1p-1}(1+k)^{\frac1p-1}\asymp \delta^{\,1-\frac1p}(1+k)^{\frac1p-1}.
\]
Hence
\[
\|(C^{\mathrm ker}_\delta f)_k\|_{\ell^p}
\lesssim \delta^{1 - \frac1p}(1+k)^{\frac1p-1}\int_k^{k+1} M_p(f,r)e^{-\frac{\alpha}{2}r^2}\,r\,dr
\le \delta^{1 - \frac1p}(1+k)^{\frac1p}\,B_p(f;k),
\]
which is \eqref{eq:block-estimate} for $k\ge1$ and $1\le p<\infty$.

\medskip
\noindent\emph{Step 2: The block $k=0$ and $1\le p<\infty$.}
On $A_0$ we use a crude estimate based on equal area of pockets. Since $|Q|\asymp \delta^{2}$ and
\eqref{eq:Cd-ptwise} holds, H\"older's inequality in area yields
\[
|(C^{\mathrm ker}_\delta f)_Q|
\lesssim |Q|^{1-\frac1p}\Bigl(\int_Q |f(z)|^p e^{-\frac{\alpha p}{2}|z|^2}\,dA(z)\Bigr)^{\frac1p}.
\]
Since the pockets partition $A_0$, taking the $\ell^p$-norm over $Q\in\mathcal P_{\delta,0}$
gives
\[
\|(C^{\mathrm ker}_\delta f)_0\|_{\ell^p}
\lesssim |Q|^{1-\frac1p}\Bigl(\int_{A_0} |f(z)|^p e^{-\frac{\alpha p}{2}|z|^2}\,dA(z)\Bigr)^{\frac1p}
\lesssim \delta^{ 1 - \frac1p}\,B_p(f;0),
\]
which is again \eqref{eq:block-estimate}, where we used $\delta^{2\left( 1 - \frac1p\right)} \leq \delta^{1 - \frac1p} $.

\medskip
\noindent\emph{Step 3: The case $p=\infty$.}
If $Q\subset A_k$, then $|f(z)|e^{-\frac{\alpha}{2}|z|^2}\le B_\infty(f;k)$ for $z\in Q$, hence
\[
|(C^{\mathrm ker}_\delta f)_Q|
\lesssim \int_Q B_\infty(f;k)\,dA
=|Q|\,B_\infty(f;k)\asymp \delta^{2}B_\infty(f;k).
\]
Therefore,
\[
\|(C^{\mathrm ker}_\delta f)_k\|_{\ell^\infty}\lesssim \delta^{2}B_\infty(f;k),
\]
which is \eqref{eq:block-estimate}, since $ \delta^2 \leq \delta$.

\medskip
\noindent\emph{Step 4: Passage to $\ell^{p,q}_w$.}
Multiply \eqref{eq:block-estimate} by the weight $w_k=(1+k)^{\frac1q-\frac1p}$ to get
\[
w_k\,\|(C^{\mathrm ker}_\delta f)_k\|_{\ell^p}\lesssim \, \delta^{1 - \frac1p} \,(1+k)^{\frac1q}\,B_p(f;k).
\]
If $0<q<\infty$, raise to the $q$th power and sum over $k$:
\[
\|C^{\mathrm ker}_\delta f\|_{\ell^{p,q}_w}^q
=\sum_{k\ge0}\bigl[w_k\|(C^{\mathrm ker}_\delta f)_k\|_{\ell^p}\bigr]^q
\lesssim \delta^{q\big(1 - \frac1p\big)}\sum_{k\ge0}(1+k)\,B_p^q(f;k).
\]
By the annular discretization principle in Proposition~\ref{prop:annular},
$\sum_{k\ge0}(1+k)\,[B_p(f;k)]^q\asymp \|f\|_{\Fpq}^q$, which yields the desired bound.
If $q=\infty$, take the supremum in $k$ instead and use the corresponding endpoint form of
Proposition~\ref{prop:annular}.
\end{proof}

\bigskip
\noindent
For synthesis, Lemma~\ref{lem:kernel-series-entire} first shows that the
series defines an entire function.  The ring-sum estimate in
Lemma~\ref{lem:kernel-ring-sum} then yields the single-shell and shellwise
bounds in Lemmas~\ref{lem:single-shell-synthesis} and
\ref{lem:per_shell_synthesis}; Lemma~\ref{lem:gen-schur-test} and
Proposition~\ref{prop:annular} give the boundedness of \(S^{\ker}_\delta\).

\begin{lemma} 
\label{lem:kernel-series-entire}
Let $1\le p\le\infty$ and $0<q\le\infty$. Let $\mathcal P_\delta$ be a $\delta$-pocket tiling of $\C$ with $\delta\in(0,\frac12]$. Then, for every coefficient sequence
\(
c=(c_Q)_{Q\in\mathcal P_\delta}\in \ell_w^{p,q}\),
the series
\[
\sum_{Q\in\mathcal P_\delta} c_Q\,\kappa_{\zeta_Q}(z)
\]
converges absolutely and locally uniformly on\/ $\C$. In particular, it defines an entire function, and hence the kernel synthesis operator
\[
S_\delta^{\ker}:\ell_w^{p,q}\to H(\C),
\qquad
(S_\delta^{\ker}c)(z):=\sum_{Q\in\mathcal P_\delta} c_Q\,\kappa_{\zeta_Q}(z),
\]
is well-defined.
\end{lemma}

\begin{proof}
Fix $R>0$ and an integer $N \geq R$. For each  $Q\in\mathcal P_{\delta,k}$ with $k \geq N$ and  $|z|\le R$, we have
\[
|\kappa_{\zeta_Q}(z)| = e^{\frac{\alpha}{2}|z|^2}e^{-\frac{\alpha}{2}|z-\zeta_Q|^2}
\le
e^{\frac{\alpha}{2}R^2}e^{-\frac{\alpha}{2}(|\zeta_Q|-R)^2}
\lesssim
e^{-\frac{\alpha}{4}(k-R)^2}
\]
with the implicit constant depending only on $\alpha$ and $R$.
Therefore
\[
\sum_{Q\in\mathcal P_{\delta,k}} |c_Q|\,|\kappa_{\zeta_Q}(z)|
\lesssim
e^{-\frac{\alpha}{4}(k-R)^2}\sum_{Q\in\mathcal P_{\delta,k}} |c_Q|.
\]
Since $p\ge1$, by the definition of $\mathcal{P}_{\delta}$ (Definition \ref{def:pockets}),
\[
\sum_{Q\in\mathcal P_{\delta,k}} |c_Q|
\le
\bigl(\#\mathcal P_{\delta,k}\bigr)^{1-\frac1p}\|c_k\|_{\ell^p} \lesssim \delta^{-2\big(1-\frac1p\big)}(1+k)^{1-\frac1p}\|c_k\|_{\ell^p}.
\]
Hence, 
\[
\sum_{Q\in\mathcal P_{\delta,k}} |c_Q|\,|\kappa_{\zeta_Q}(z)|
\lesssim \delta^{-2\big(1-\frac1p\big)}
e^{-\frac{\alpha}{4}(k-R)^2}
(1+k)^{1-\frac1q} w_k\|c_k\|_{\ell^p}.
\]

\smallskip
\noindent
Summing over $k \geq N$:
If $0<q\le1$, then
\[
\sum_{k\ge N} w_k\|c_k\|_{\ell^p}
\le
\Biggl(\sum_{k\ge N}\bigl[w_k\|c_k\|_{\ell^p}\bigr]^q\Biggr)^{\frac1q} \to 0, \quad \text{as} \ N \to \infty,
\]
and hence
\[
\sum_{k\ge N} \sum_{Q\in\mathcal P_{\delta,k}} |c_Q|\,|\kappa_{\zeta_Q}(z)|
\lesssim 
\, \delta^{-2\big(1-\frac1p\big)} \,
\left(\sup_{k\ge0}e^{-\frac{\alpha}{4}(k-R)^2}(1+k)^{1-\frac1q}\right)
\Biggl(\sum_{k\ge N}\bigl[w_k\|c_k\|_{\ell^p}\bigr]^q\Biggr)^{\frac1q} \to 0,
\]
as $N \to \infty$.
If $1<q<\infty$, then H\"older's inequality yields
\begin{align*}
\sum_{k\ge N} \sum_{Q\in\mathcal P_{\delta,k}} |c_Q|\,|\kappa_{\zeta_Q}(z)| 
&\lesssim \, \delta^{-2\big(1-\frac1p\big)}
\sum_{k\ge N}
e^{-\frac{\alpha}{4}(k-R)^2}(1+k)^{1-\frac1q}\,
w_k\|c_k\|_{\ell^p} \\
&\le \, \delta^{-2\big(1-\frac1p\big)}
\Biggl(
\sum_{k\ge N}
e^{-\frac{\alpha q}{4(q-1)}(k-R)^2}(1+k)
\Biggr)^{\frac{q-1}{q}}
\|c\|_{\ell_w^{p,q}} \to 0,
\end{align*}
as $N \to \infty$. 
If $q=\infty$, then
\[
\sum_{k\ge N} \sum_{Q\in\mathcal P_{\delta,k}} |c_Q|\,|\kappa_{\zeta_Q}(z)| 
\lesssim \, \delta^{-2\big(1-\frac1p\big)}
\|c\|_{\ell_w^{p,\infty}}
\sum_{k\ge N}
e^{-\frac{\alpha}{4}(k-R)^2}(1+k) \to 0,
\]
as $N \to \infty$.
Thus,  the tail
\[
\sum_{k\ge N}\sum_{Q\in\mathcal P_{\delta,k}} |c_Q|\,|\kappa_{\zeta_Q}(z)|
\]
tends to zero uniformly for $|z|\le R$. Since the finitely many shells $k<N$ contribute only a finite sum, it follows that the series
\[
\sum_{k\ge 0} \sum_{Q\in\mathcal P_{\delta,k}} |c_Q|\,|\kappa_{\zeta_Q}(z)| 
\]
converges uniformly for $|z|\le R$. Since $R>0$ was arbitrary, the series
\[
\sum_{Q\in\mathcal P_\delta} c_Q\,\kappa_{\zeta_Q}(z)
\]
converges absolutely and locally uniformly on\/ $\C$. By Weierstrass' theorem, its sum is entire.
\end{proof}

\begin{lemma} \label{lem:kernel-ring-sum}
Let $\al > 0$, $\delta\in(0,\tfrac12]$ and let $\mathcal P_\delta$ be the $\delta$-pocket tiling of
$\C$.  Then, for all $k\ge0$ and all $0 \leq \nu \leq N_{\delta} - 1$,
\begin{equation}\label{eq:ring-sum}
\sup_{\theta\in[0,2\pi]}\ \sum_{\mu=0}^{N_{\delta,k,\nu}-1}
\left|\kappa_{\zeta_{k,\nu,\mu}}(re^{i\theta})\right|\,e^{-\frac{\alpha}{2}r^2}
\ \lesssim \ \delta^{-1}\, e^{-\frac{\alpha}{8}(r-k)^2},
\qquad r\ge0,
\end{equation}
where the implicit constant depends only on \(\alpha\), and as in Definition \ref{def:pockets}, $\zeta_{k, \nu, \mu} = \rho_{k, \nu}e^{i\,\vartheta_{k, \nu, \mu}}$ is the center of the pocket $Q_{k, \nu, \mu}$, and $N_{\delta,k,\nu}=2\nu+1$ for $k=0$, while $N_{\delta,k,\nu}=N_{\delta,k}=(2k+1)N_\delta$ for $k\ge1$.
\end{lemma}

\begin{proof}
Fix $k\ge 0 $ and $0 \leq \nu \leq N_{\delta} - 1$. Writing $z=re^{i\theta}$,
 a standard computation with the kernel $\kappa_{\zeta_{k, \nu, \mu}}$ gives \[
\left|\kappa_{\zeta_{k, \nu, \mu}}(re^{i\theta})\right|\,e^{-\frac{\alpha}{2}r^2}
=e^{-\frac{\alpha}{2}(r-\rho_{k,\nu})^2}\,
g_t(\theta-\vartheta_{k,\nu,\mu}),
\qquad
g_t(\phi):=\exp\!\big(-t(1-\cos\phi)\big),
\]
where $t:=\alpha r\rho_{k,\nu}\ge0$. Hence
\begin{equation}\label{eq:reduce-angular-sum}
\sum_{\mu=0}^{N_{\delta,k, \nu}-1}\left|\kappa_{\zeta_{k, \nu, \mu}}(re^{i\theta})\right|\,e^{-\frac{\alpha}{2}r^2}
= e^{-\frac{\alpha}{2}(r-\rho_{k,\nu})^2}\,
\sum_{\mu=0}^{N_{\delta,k, \nu}-1} g_t(\theta-\vartheta_{k,\nu,\mu}).
\end{equation}

\smallskip
\noindent
\smallskip
\noindent\emph{Case $k=0$.} Here $N_{\delta,0,\nu}=2\nu+1\lesssim \delta^{-1}$ and $0\le\rho_{0,\nu}<1$. Using
$g_t\le1$ in \eqref{eq:reduce-angular-sum} gives
\[
\sup_{\theta \in [0, 2\pi]}\sum_{\mu=0}^{N_{\delta,0,\nu}-1}\left|\kappa_{\zeta_{0, \nu, \mu}}(re^{i\theta})\right|\,e^{-\frac{\alpha}{2}r^2}
\lesssim \delta^{-1} \, e^{-\frac{\alpha}{2}(r-\rho_{0,\nu})^2}
\lesssim \delta^{-1}\, e^{-\frac{\alpha}{8}r^2},
\]
which is compatible with \eqref{eq:ring-sum} when $k=0$.

\smallskip
\noindent
\smallskip
\noindent\emph{Case $k \geq 1$.} In this case $N_{\delta, k, \nu} = N_{\delta, k}$ and $\vartheta_{k, \nu, \mu} = \vartheta_{k, \mu}$, and the angles $\{\vartheta_{k, \mu}\}_{\mu = 0}^{N_{\delta, k}-1}$ form a
uniform grid with mesh $\Delta\theta_k:=\frac{2\pi}{N_{\delta,k}}$.

\smallskip
\noindent\emph{Step 1: a uniform sampling bound for periodic bounded variation functions.}
Let $g$ be $2\pi$-periodic of bounded variation, and let $\{\vartheta_{k, \mu}\}_{\mu = 0}^{N_{\delta, k}-1}$ be a uniform grid with mesh
$\Delta \theta_k:= \frac{2\pi}{N_{\delta,k}}$. Then
\begin{equation}\label{eq:BV-sampling}
\sup_{\theta\in[0,2\pi]}\sum_{\mu=0}^{N_{\delta,k}-1} g(\theta-\vartheta_{k, \mu})
\ \le\ \frac{1}{\Delta \theta_k}\int_0^{2\pi} g(\phi)\,d\phi + V(g),
\end{equation}
where $V(g)=\int_0^{2\pi}|g'(\phi)|\,d\phi$ is the total variation. This follows by comparing each
sample to the average on its cell and summing oscillations; $$\sum_{\mu = 
0}^{N_{\delta,k} - 1} \textup{osc}_{I_{k,\mu}}(g)\le V(g) \quad \text{with} \quad I_{k, \mu} = [\theta_{k, \mu}, \theta_{k, \mu+1}).
$$
We apply \eqref{eq:BV-sampling} to $g=g_t$. Since $g_t$ is $C^1$ and $2\pi$-periodic,
\[
\int_0^{2\pi} g_t(\phi)\,d\phi = 2\pi e^{-t}I_0(t),
\qquad
g_t'(\phi)=-t\sin\phi\,g_t(\phi),
\]
and a direct computation gives
\[
V(g_t)=\int_0^{2\pi}|g_t'(\phi)|\,d\phi
=t\int_0^{2\pi}|\sin\phi|\,e^{-t(1-\cos\phi)}\,d\phi
=2(1-e^{-2t})\le 2.
\]
Moreover, the standard bound for the modified Bessel function $I_0$ yields
\begin{equation}\label{eq:bessel-bound}
e^{-t}I_0(t)\ \asymp \ (1+t)^{-\frac12},
\qquad t\ge0.
\end{equation}
Inserting these into \eqref{eq:BV-sampling} gives
\begin{equation}\label{eq:angular-sum}
\sup_{\theta \in [0, 2\pi]}\sum_{\mu=0}^{N_{\delta,k}-1} g_t(\theta-\vartheta_{k,\nu,\mu})
\ \lesssim\ (\Delta \theta_k)^{-1}(1+t)^{-\frac12}+1.
\end{equation}

\smallskip
\noindent\emph{Step 2: exploit the tiling geometry.}
For $k\ge1$, $(\Delta \theta_k)^{-1}= \frac{N_{\delta,k}}{2\pi} \asymp  (1+k)\delta^{-1}$, and
$\rho_{k,\nu}\asymp  1+k$. Thus \eqref{eq:angular-sum} implies
\begin{equation}\label{eq:angular-sum2}
\sup_{\theta \in [0, 2\pi]}\sum_{\mu=0}^{N_{\delta,k}-1} g_t(\theta-\vartheta_{k,\nu,\mu})
\ \lesssim\ \delta^{-1}\Bigl(1+\frac{\rho_{k,\nu}}{\sqrt{1+\alpha r\rho_{k,\nu}}}\Bigr).
\end{equation}

\smallskip
\noindent\emph{Step 3: absorb the slowly varying factor.}
Set $\rho:=\rho_{k,\nu} $. We claim that
\begin{equation}\label{eq:slow-absorb}
\sup_{r\ge0}\Bigl(1+\frac{\rho}{\sqrt{1+\alpha r\rho}}\Bigr) e^{-\frac{\alpha}{4}(r-\rho)^2}
\ \lesssim\ 1,
\qquad \rho\ge1.
\end{equation}
If $r\ge  \frac{\rho}{2}$, then $1+\alpha r\rho\ge \frac{\alpha}{2}\rho^2$, hence
 the left-hand side is $\lesssim 1$.
If $0\le r< \frac{\rho}{2}$, then 
\[
\Bigl(1+\frac{\rho}{\sqrt{1+\alpha r\rho}}\Bigr) e^{-\frac{\alpha}{4}(r-\rho)^2}
\le (1+\rho)\,e^{-\frac{\alpha}{16}\rho^2}\ \lesssim\ 1.
\]
Combining \eqref{eq:reduce-angular-sum}, \eqref{eq:angular-sum2}, and \eqref{eq:slow-absorb} yields
\[
\sup_{\theta \in [0, 2\pi]}\sum_{\mu=0}^{N_{\delta,k}-1}\left|\kappa_{\zeta_{k,\nu,\mu}}(re^{i\theta})\right|\,e^{-\frac{\alpha}{2}r^2}
\ \lesssim\ \delta^{-1}\ e^{-\frac{\alpha}{4}(r-\rho_{k,\nu})^2}.
\]

\smallskip
\noindent\emph{Step 4: replace $\rho_{k,\nu}$ by $k$.}
Since $k\le \rho_{k,\nu}<k+1$, we have
\[
e^{-\frac{\alpha}{4}(r-\rho_{k,\nu})^2}
\le e^{-\frac{\alpha}{8}(r-k)^2 + \frac{\alpha}{4}},
\]
which proves \eqref{eq:ring-sum} for $k\ge1$ after absorbing $e^{\frac{\alpha}{4}}$ into the constant.
\end{proof}

\begin{lemma} \label{lem:single-shell-synthesis}
Let $\al > 0$, $1\le p\le\infty$ and let $\delta\in(0,\frac12]$ and $\mathcal P_\delta$ be the $\delta$-pocket tiling of
$\C$. 
Fix $k\ge0$. For coefficients $c_k=(c_Q)_{Q\in\mathcal P_{\delta,k}}$ set
\[
f_k(z):=\sum_{Q\in\mathcal P_{\delta,k}} c_Q\,\kappa_{\zeta_Q}(z).
\]
Then
\begin{equation}\label{eq:single_block_synthesis}
M_p(f_k,r)\,e^{-\frac{\alpha}{2} r^2}
\ \lesssim\ \delta^{2\bigl(\frac1p-1\bigr)}\,(1+k)^{-\frac1p}\,
e^{-\frac{\alpha}{8}(r-k)^2}\,\|c_k\|_{\ell^p},
\qquad r\ge0,
\end{equation}
where the implicit constant depends only on \(\alpha\) and \(p\), and is independent
of $k,r,\delta$ and
$c_k$.
\end{lemma}
\begin{proof}
Fix $k\ge0$ and $r\ge0$. Consider the linear operator
\[
T_{k,r}:\ell^p(\mathcal P_{\delta,k})\to L^p(\T),\qquad
T_{k,r}(c_k)(\theta):=e^{-\frac{\alpha}{2} r^2}\sum_{Q\in\mathcal P_{\delta,k}} c_Q\,\kappa_{\zeta_Q}(re^{i\theta}).
\]
Then $\|T_{k,r}(c_k)\|_{L^p(\T)}=M_p(f_k,r)\,e^{-\frac{\alpha}{2}r^2}$.

\smallskip
\noindent\emph{Endpoint $p=1$.}
By the triangle inequality,
\[
\|T_{k,r}\|_{\ell^1\to L^1}
\le \sup_{Q\in\mathcal P_{\delta,k}}\int_0^{2\pi}|\kappa_{\zeta_Q}(re^{i\theta})|\,e^{-\frac{\alpha}{2}r^2}\,\frac{d\theta}{2\pi} \, \asymp \, \sup_{Q\in\mathcal P_{\delta,k}} M_1(\kappa_{\zeta_Q},r)\,e^{-\frac{\alpha}{2}r^2}.
\]
Write $\zeta_Q=\rho e^{i\vartheta}$ with $\rho\in[k,k+1)$. Using \eqref{eq:packet-profile} in Lemma \ref{lem:packet-profile}, we obtain
\[
M_1(\kappa_{\zeta_Q},r)\,e^{-\frac{\alpha}{2}r^2} \, \asymp \, (1+\alpha r\rho)^{-\frac12} e^{-\frac{\alpha}{2}(r-\rho)^2}, \quad r \geq 0.
\]
If $r\ge \frac{\rho}{2}$, then the right-hand side is 
$\lesssim (1+\rho)^{-1}e^{-\frac{\alpha}{4}(r-\rho)^2}$. If $0\le r<\frac{\rho}{2}$, then $|r-\rho|\ge\frac{\rho}{2}$, hence 
$$
(1+\alpha r\rho)^{-\frac12} e^{-\frac{\alpha}{2}(r-\rho)^2} \lesssim e^{-\frac{\alpha}{16}\rho^2} e^{-\frac{\alpha}{4}(r-\rho)^2} \lesssim (1+\rho)^{-1}e^{-\frac{\alpha}{4}(r-\rho)^2}.
$$
Consequently,
\[
\|T_{k,r}\|_{\ell^1\to L^1}\ \lesssim \ (1 + \rho)^{-1} e^{-\frac{\alpha}{4}(r - \rho)^2} \ \lesssim\ (1+k)^{-1} e^{-\frac{\alpha}{8}(r-k)^2},
\]
where the last inequality follows from $\rho\in[k,k+1)$.

\smallskip
\noindent\emph{Endpoint $p=\infty$.}
By definition,
\[
\|T_{k,r}\|_{\ell^\infty\to L^\infty}
\le \sup_{\theta\in[0,2\pi]}\sum_{Q\in\mathcal P_{\delta,k}}
|\kappa_{\zeta_Q}(re^{i\theta})|\,e^{-\frac{\alpha}{2}r^2}.
\]
Decompose $\mathcal P_{\delta,k}$ into radial strips indexed by $\nu=0,\dots,N_\delta-1$ and apply
Lemma~\ref{lem:kernel-ring-sum} to each strip:
\[
\sup_{\theta \in [0, 2\pi]}\sum_{Q\in\mathcal P_{\delta,k}}
|\kappa_{\zeta_Q}(re^{i\theta})|\,e^{-\frac{\alpha}{2}r^2}
\le \sum_{\nu=0}^{N_\delta-1}\sup_{\theta \in [0, 2\pi]}\sum_{\mu=0}^{N_{\delta,k,\nu}-1}
|\kappa_{\zeta_{k,\nu,\mu}}(re^{i\theta})|\,e^{-\frac{\alpha}{2}r^2}
\lesssim N_\delta\,\delta^{-1} e^{-\frac{\alpha}{8}(r-k)^2}.
\]
Since $N_\delta\asymp \delta^{-1}$, this gives
\[
\|T_{k,r}\|_{\ell^\infty\to L^\infty}\ \lesssim\ \delta^{-2} e^{-\tfrac{\alpha}{8}(r-k)^2}.
\]

\smallskip
\noindent\emph{Interpolation.}
By the Riesz--Thorin interpolation theorem, for $1<p<\infty$,
\[
\|T_{k,r}\|_{\ell^p\to L^p}
\le \|T_{k,r}\|_{\ell^1\to L^1}^{\frac1p}\,\|T_{k,r}\|_{\ell^\infty\to L^\infty}^{1-\frac1p}
\lesssim \, \delta^{-2\big(1-\frac1p\big)} \, (1+k)^{-\frac1p}\,e^{-\frac{\alpha}{8}(r-k)^2}.
\]
The same bound holds for $p=1$ and $p=\infty$ by the endpoint estimates. Thus,
\[
M_p(f_k,r)e^{-\frac{\alpha}{2}r^2}
=\|T_{k,r}(c_k)\|_{L^p}\le \|T_{k,r}\|_{\ell^p\to L^p}\,\|c_k\|_{\ell^p},
\]
which is \eqref{eq:single_block_synthesis}.
\end{proof}

\begin{lemma}\label{lem:per_shell_synthesis}
Let $\al > 0$, $1\le p\le\infty$ and let $\delta\in(0,\tfrac12]$ and $\mathcal P_\delta$ be the $\delta$-pocket tiling
of $\C$. Then for every sequence $c=(c_Q)_{Q\in\mathcal P_\delta} \in\ell_w^{p,q}$, with blocks
$c_k=(c_Q)_{Q\in\mathcal P_{\delta,k}}$, and for every $\ell\ge0$,
\begin{equation}\label{eq:per_shell_synthesis}
B_p(S^{\mathrm ker}_\delta c;\ell)\ \lesssim\ \delta^{2\big(\frac1p-1\big)}
\sum_{k\ge0} e^{-\frac{\alpha}{16}(\ell-k)^2}\,(1+k)^{-\frac1p}\,\|c_k\|_{\ell^p},
\end{equation}
where the implicit constant depends only on \(\alpha\) and \(p\), and is independent
of $c,\ell$ and $\delta$.
\end{lemma}

\begin{proof}
By Lemma~\ref{lem:kernel-series-entire}, the synthesis series
\(S_\delta^{\ker}c\) defines an entire function for
\(c\in\ell_w^{p,q}\).  Hence the annular quantity
\(B_p(S_\delta^{\ker}c;\ell)\) is well defined.  We estimate it by
decomposing the synthesis series according to annular blocks:
\[
S^{\mathrm ker}_\delta c=\sum_{k\ge0} f_k
\]
with $f_k$ as in Lemma~\ref{lem:single-shell-synthesis}.
Fix $\ell\ge0$ and $r\in[\ell,\ell+1)$. Since $p\ge1$, Minkowski's inequality on $L^p(\T)$ gives
\[
M_p(S^{\mathrm ker}_\delta c,r)\le \sum_{k\ge0} M_p(f_k,r).
\]
Multiplying by $e^{-\frac{\alpha}{2}r^2}$ and applying \eqref{eq:single_block_synthesis} yields
\[
M_p(S^{\mathrm ker}_\delta c,r)e^{-\frac{\alpha}{2}r^2}
\ \lesssim \delta^{2\big(\frac1p-1\big)} \sum_{k\ge0} (1+k)^{-\frac1p}\,
e^{-\frac{\alpha}{8}(r-k)^2}\,\|c_k\|_{\ell^p}.
\]
Taking supremum over $r \in [\ell, \ell + 1)$ and using $|r-k|\ge |\ell-k|-1$, we get
\[
\sup_{r\in[\ell,\ell+1)}e^{-\frac{\alpha}{8}(r-k)^2}
\le e^{\frac{\alpha}{8}} e^{-\frac{\alpha}{16}(\ell-k)^2},
\]
and absorbing $e^{\frac{\alpha}{8}}$ into the implicit constant proves \eqref{eq:per_shell_synthesis}.
\end{proof}

\begin{lemma}[Bounded kernel synthesis]\label{lem:synthesis}
Let $\al > 0$, $1\le p\le\infty$, and $0<q\le\infty$. Let $\mathcal P_\delta$ be a $\delta$-pocket tiling of $\C$ with  $\delta\in(0,\tfrac12]$. Then the synthesis operator
\[
S_\delta^{\mathrm{ker}}:\ell^{p,q}_w\to \Fpq, \qquad c \mapsto \sum_{Q \in \mathcal{P}_{\delta}} c_Q \kappa_{\zeta_Q},
\]
is bounded. More precisely,   
\begin{equation}\label{eq:synthesis}
\|S^{\mathrm ker}_{\delta}c\|_{\Fpq}\ \lesssim \  \delta^{2\big(\frac1p-1\big)}\,\|c\|_{\ell^{p,q}_w},
\qquad c\in \ell^{p,q}_w,
\end{equation}
where the implicit constant depends only on \(\alpha,p\), and \(q\), and is
independent of \(\delta\).
\end{lemma}
\begin{proof}
Let $c\in\ell_w^{p,q}$ and write $c_k=(c_Q)_{Q\in\mathcal P_{\delta,k}}$. By
Lemma~\ref{lem:per_shell_synthesis},
\[
B_p(S^{\mathrm ker}_\delta c;\ell)
\lesssim \delta^{2\big(\frac1p-1\big)}
\sum_{k\ge0} e^{-\frac{\alpha}{16}(\ell-k)^2}\,(1+k)^{-\frac1p}\,\|c_k\|_{\ell^p}.
\]
Set $Y_k:=w_k\|c_k\|_{\ell^p}=(1+k)^{\frac1q-\frac1p}\|c_k\|_{\ell^p}$, so that
$(1+k)^{-\frac1p}\|c_k\|_{\ell^p}=(1+k)^{-\frac1q}Y_k$. Hence
\[
B_p(S^{\mathrm ker}_\delta c;\ell)
\lesssim \delta^{2\big(\frac1p-1\big)} \sum_{k\ge0}K(\ell,k)\,Y_k,
\qquad
K(\ell,k):=e^{-\frac{\alpha}{16}(\ell-k)^2}\,(1+k)^{-\frac1q}.
\]

\smallskip
\noindent\emph{Case $0<q<\infty$.}
Apply the generalized Schur test in Lemma~\ref{lem:gen-schur-test} to obtain
\[
\sum_{\ell\ge0}(1+\ell)\Bigl(\sum_{k\ge0}K(\ell,k)Y_k\Bigr)^q \ \lesssim\ \sum_{k\ge0}Y_k^q.
\]
Combining this with Proposition~\ref{prop:annular} gives
\[
\|S^{\mathrm ker}_{\delta}c\|_{\Fpq}^q
\ \asymp \ \sum_{\ell\ge0}(1+\ell)\,[B_p(S^{\mathrm ker}_{\delta}c;\ell)]^q \ \lesssim \  \delta^{2\big(\frac1p-1\big)q}\sum_{k\ge0}Y_k^q
\ = \ \delta^{2\big(\frac1p-1\big)q}\,\|c\|_{\ell^{p,q}_w}^q,
\]
which is \eqref{eq:synthesis} for $0 < q < \infty$.

\smallskip
\noindent\emph{Case $q=\infty$.}
Using the endpoint form of the generalized Schur test in Lemma~\ref{lem:gen-schur-test},
\[
\sup_{\ell\ge0}\sum_{k\ge0}K(\ell,k)Y_k \ \lesssim\ \sup_{k\ge0}Y_k.
\]
Together with Proposition~\ref{prop:annular} (endpoint form), this yields
\[
\|S^{\mathrm ker}_{\delta}c\|_{\FpqA{p}{\infty}{\alpha}}
=\sup_{\ell\ge0}B_p(S^{\mathrm ker}_{\delta}c;\ell)
\lesssim \delta^{2\big(\frac1p-1\big)}\sup_{k\ge0}Y_k
=\delta^{2\big(\frac1p-1\big)}\|c\|_{\ell_w^{p,\infty}}.
\]
This is \eqref{eq:synthesis} for $q = \infty$.

\smallskip
\noindent
The estimate \eqref{eq:synthesis}  shows that $S_\delta^{\rm ker}$ is bounded with norm
$
\|S_\delta^{\rm ker}\|_{\ell_w^{p,q}\to\Fpq}\lesssim \delta^{2\big(\frac1p-1\big)}$.

\end{proof}

\bigskip
\noindent
The stability analysis starts with a quantitative shellwise estimate for
the remainder operator
\[
R_\delta^{\ker}:=I-S_\delta^{\ker}C_\delta^{\ker},
\]
which will be the key input in proving that $R_\delta^{\ker}$ has small norm for
sufficiently small $\delta$.
\begin{lemma} \label{lem:remainderkernelestimate}
Let $\al>0$, $1\le p\le\infty$, and $0<q\le\infty$. Let $\mathcal P_\delta$ be a
$\delta$-pocket tiling of $\C$ with  $\delta\in(0,\tfrac12]$ and let 
$$
R_\delta^{\ker} = I - S_{\delta}^{\ker} C_{\delta}^{\ker}
$$ 
be the kernel remainder operator from
Definition~\ref{def:kernel-analysis-synthesis}. Then
\[
B_p(R_\delta^{\ker}f;\ell)
\lesssim \,\delta \sum_{k\ge0} e^{-\frac{\al}{32}(\ell-k)^2}\,B_p(f;k)
\]
for every $f\in\mathcal F_\al^{p,q}$ and every $\ell\ge0$, where the implicit
constant depends only on \(\alpha\) and \(p\), and is independent of
\(q,\delta,f\), and \(\ell\).
\end{lemma}

\begin{proof}
By Proposition~\ref{prop:fock-reproducing} and Definition~\ref{def:kernel-analysis-synthesis},
\[
R_\delta^{\ker}f(z)
= f(z)-S_\delta^{\ker}C_\delta^{\ker}f(z) 
= \frac{\alpha}{\pi}\sum_{Q\in\mathcal P_\delta}
\int_Q f(w)e^{-\frac{\alpha}{2}|w|^2}e^{\,i\alpha\, \im(\zeta_Q\overline w)}
\,H(z,\zeta_Q,w)\,dA(w),
\]
where
\[
H(z,\zeta,w)
:=e^{\alpha \overline w z-\frac{\alpha}{2}|w|^2-i\alpha\im(\zeta\overline w)}
-\kappa_\zeta(z).
\]
A direct computation gives
\[
H(z,\zeta,w)=\kappa_\zeta(z)\bigl(e^{U(z,\zeta,w)}-1\bigr) \quad \text{with} \quad 
U(z,\zeta,w)
=\alpha (z-\zeta)\overline{(w-\zeta)}-\frac{\alpha}{2}|w-\zeta|^2.
\]
Hence
\[
|H(z,\zeta,w)|\,e^{-\frac{\alpha}{2}|z|^2}
= e^{-\frac{\alpha}{2}|z-\zeta|^2}\,\bigl|e^{U(z,\zeta,w)}-1\bigr|.
\]
Fix $Q\in\mathcal P_\delta$. For every
 $w\in Q$, Lemma~\ref{lem:geo-pockets}(c) yields
$|w-\zeta_Q|\lesssim \delta$. Therefore
\[
|U(z,\zeta_Q,w)|
\lesssim |w-\zeta_Q|\bigl(|z-\zeta_Q|+|w-\zeta_Q|\bigr)
\leq C\, \delta(1+|z-\zeta_Q|),
\]
where $C$ is a constant depending only on $\alpha$.
Using $|e^u-1|\le |u|e^{|u|}$ and $\delta\le \frac12$, we obtain
\[
|H(z,\zeta_Q,w)|\,e^{-\frac{\alpha}{2}|z|^2}
\lesssim \delta (1+|z-\zeta_Q|)\,
e^{-\frac{\alpha}{2}|z-\zeta_Q|^2} e^{C\delta |z-\zeta_Q|}
\lesssim \delta\, e^{-\frac{\alpha}{4}|z-\zeta_Q|^2}.
\]
Thus, with
\[
a_Q(f):=\int_Q |f(w)|e^{-\frac{\alpha}{2}|w|^2}\,dA(w),
\]
we get
\begin{equation}\label{eq:remainder-pointwise}
|R_\delta^{\ker}f(z)|\,e^{-\frac{\alpha}{2}|z|^2}
\lesssim \delta \sum_{Q\in\mathcal P_\delta} a_Q(f)\,
e^{-\frac{\alpha}{4}|z-\zeta_Q|^2}.
\end{equation}
For each $k\ge0$, let
\[
g_k(z):=\sum_{Q\in\mathcal P_{\delta,k}} a_Q(f)\,
e^{-\frac{\alpha}{4}|z-\zeta_Q|^2}.
\]
Since $p\ge1$, Minkowski's inequality and \eqref{eq:remainder-pointwise} imply
\[
M_p(R_\delta^{\ker}f,r)\,e^{-\frac{\alpha}{2}r^2}
\lesssim \delta \sum_{k\ge0} M_p(g_k,r),
\qquad r\ge0,
\]
where, for measurable functions on the circle, we use $\Mp(h, r)$ to denote the $L^p$-mean of $h(r e^{i\theta})$ in the same notation.

\smallskip
\noindent
Fix $k\ge0$ and $r\ge0$, and define
\[
T_{k,r}:\ell^p(\mathcal P_{\delta,k})\to L^p(\mathbb T),
\qquad
T_{k,r}(b)(\theta)
:=\sum_{Q\in\mathcal P_{\delta,k}} b_Q\,
e^{-\frac{\alpha}{4}|re^{i\theta}-\zeta_Q|^2}.
\]
Then
\[
\|T_{k,r}(a_k(f))\|_{L^p(\mathbb T)}=M_p(g_k,r),
\]
where $a_k(f):=(a_Q(f))_{Q\in\mathcal P_{\delta,k}}$.

\smallskip
\noindent
For the endpoint cases $p=1$ and $p=\infty$, we argue exactly as in the corresponding parts of Lemma~\ref{lem:single-shell-synthesis}, replacing the normalized kernel profile
\[
|\kappa_{\zeta}(re^{i\theta})|\,e^{-\frac{\alpha}{2}r^2}
= e^{-\frac{\alpha}{2}|re^{i\theta}-\zeta|^2}
\]
by the Gaussian profile
\[
e^{-\frac{\alpha}{4}|re^{i\theta}-\zeta|^2}.
\]
Using the same circular Gaussian estimate, we obtain
\[
\|T_{k,r}\|_{\ell^1\to L^1}
\le \sup_{Q\in\mathcal P_{\delta,k}}
\int_0^{2\pi} e^{-\frac{\alpha}{4}|re^{i\theta}-\zeta_Q|^2}\,\frac{d\theta}{2\pi}
\lesssim (1+k)^{-1} e^{-\frac{\alpha}{16}(r-k)^2},
\]
and
\[
\|T_{k,r}\|_{\ell^\infty\to L^\infty}
\le \sup_{\theta\in[0,2\pi]}
\sum_{Q\in\mathcal P_{\delta,k}} e^{-\frac{\alpha}{4}|re^{i\theta}-\zeta_Q|^2}
\lesssim \delta^{-2} e^{-\frac{\alpha}{16}(r-k)^2}.
\]

\smallskip
\noindent
Interpolating by the Riesz--Thorin theorem, we obtain for every
$1\le p\le\infty$,
\begin{equation}\label{eq:Tkrr-bound}
\|T_{k,r}\|_{\ell^p\to L^p}
\lesssim \delta^{-2\left(1-\frac1p\right)}(1+k)^{-\frac1p}
e^{-\frac{\alpha}{16}(r-k)^2}.
\end{equation}
Consequently,
\begin{equation}\label{eq:gk-bound}
M_p(g_k,r)
\lesssim
\delta^{-2\left(1-\frac1p\right)}(1+k)^{-\frac1p}
e^{-\frac{\alpha}{16}(r-k)^2}
\|a_k(f)\|_{\ell^p}.
\end{equation}

\smallskip
\noindent
It remains to estimate $\|a_k(f)\|_{\ell^p}$.
If $1\le p<\infty$, then H\"older's inequality and Lemma~\ref{lem:geo-pockets}(b) give
\[
a_Q(f)^p
\le |Q|^{p-1}\int_Q |f(w)|^p e^{-\frac{\alpha p}{2}|w|^2}\,dA(w)
\lesssim \delta^{2p-2}\int_Q |f(w)|^p e^{-\frac{\alpha p}{2}|w|^2}\,dA(w).
\]
Summing over $Q\in\mathcal P_{\delta,k}$, we obtain
\begin{align*}
\|a_k(f)\|_{\ell^p}^p
\lesssim \
\delta^{2p-2}
\int_{A_k} |f(w)|^p e^{-\frac{\alpha p}{2}|w|^2}\,dA(w) \lesssim \
\delta^{2p-2}(1+k)\, B_p(f;k)^p.
\end{align*}
Therefore,
\[
\|a_k(f)\|_{\ell^p}
\lesssim \delta^{2\big(1-\frac1p\big)}(1+k)^{\frac1p} B_p(f;k),
\qquad 1\le p<\infty.
\]
If $p=\infty$, then for $Q\in\mathcal P_{\delta,k}$,
\[
a_Q(f)\le |Q|\, \sup_{w\in Q}\Bigl(|f(w)|e^{-\frac{\alpha}{2}|w|^2}\Bigr)
\lesssim \delta^2 B_\infty(f;k),
\]
so
\[
\|a_k(f)\|_{\ell^\infty}\lesssim \delta^2 B_\infty(f;k).
\]
In both cases, \eqref{eq:gk-bound} yields
\[
M_p(g_k,r)\lesssim e^{-\frac{\alpha}{16}(r-k)^2} B_p(f;k).
\]

\smallskip
\noindent
Substituting this into the estimate for $R_\delta^{\ker}f$, we get
\[
M_p(R_\delta^{\ker}f,r)\,e^{-\frac{\alpha}{2}r^2}
\lesssim
\delta \sum_{k\ge0} e^{-\frac{\alpha}{16}(r-k)^2} B_p(f;k).
\]
Taking the supremum over $r\in[\ell,\ell+1)$ and using
$|r-k|\ge |\ell-k|-1$, we arrive at
\[
B_p(R_\delta^{\ker}f;\ell)
\lesssim
\delta \sum_{k\ge0} e^{-\frac{\alpha}{32}(\ell-k)^2} B_p(f;k),
\]
which proves the lemma.
\end{proof}

\begin{lemma} \label{lem:stability}
Let $\al>0$, $1\le p\le\infty$, and $0<q\le\infty$. Then there exists
$\delta_0=\delta_0(\al,p,q)\in(0,\tfrac12]$ such that, for  every $\delta$-pocket tiling $\mathcal P_\delta$ of $\C$ with
$0<\delta<\delta_0$,
\[
\|I-S_\delta^{\ker}C_\delta^{\ker}\|_{\mathcal F_\al^{p,q}\to\mathcal F_\al^{p,q}}
\le \frac12.
\]
\end{lemma}

\begin{proof}
By Lemma~\ref{lem:remainderkernelestimate}, there exists a constant
$C_0=C_0(\al,p)>0$ such that
\[
B_p(R_\delta^{\ker}f;\ell)
\le C_0\,\delta \sum_{k\ge0} e^{-\frac{\al}{32}(\ell-k)^2}\,B_p(f;k),
\qquad \ell\ge0,
\]
for every $f\in\mathcal F_\al^{p,q}$.

\smallskip
\noindent
Set
\[
K(\ell,k):=C_0\, e^{-\frac{\al}{32}(\ell-k)^2}(1+k)^{-\frac1q} \quad \text{and} \quad Y_k:=(1+k)^{\frac1q}B_p(f;k), \qquad \ell, k\ge0,
\]
with the standing convention that $\frac1q=0$ when $q=\infty$. 
Then
\[
B_p(R_\delta^{\ker}f;\ell)\le \, \delta \sum_{k\ge0}K(\ell,k)Y_k.
\]
Since $K(\ell,k)$ has exactly the form required in the generalized Schur test (Lemma~\ref{lem:gen-schur-test}) with $\beta=2$,
we obtain
\[
\sum_{\ell\ge0}(1+\ell)\,[B_p(R_\delta^{\ker}f;\ell)]^q
\lesssim
\delta^q\sum_{k\ge0}Y_k^q
=
\delta^q\sum_{k\ge0}(1+k)\,[B_p(f;k)]^q \quad \text{if} \quad 0 < q < \infty
\]
and
\[
\sup_{\ell\ge0} B_p(R_\delta^{\ker}f;\ell)
\lesssim \delta\,\sup_{k\ge0} B_p(f;k) \quad \text{if} \quad  q = \infty.
\]
Applying Proposition~\ref{prop:annular} to both $R_\delta^{\ker}f$ and $f$, we conclude that
\[
\|R_\delta^{\ker}f\|_{\mathcal F_\al^{p,q}}
\le C\,\delta\,\|f\|_{\mathcal F_\al^{p,q}},
\qquad f\in\mathcal F_\al^{p,q},
\]
for some constant $C=C(\al,p,q)>0$. Choosing
\[
\delta_0:=\min\Bigl\{\frac12,\frac{1}{2C}\Bigr\},
\]
we obtain
\[
\|R_\delta^{\ker}\|_{\mathcal F_\al^{p,q}\to\mathcal F_\al^{p,q}}
\le \frac12,
\qquad 0<\delta<\delta_0.
\]
Since $R_\delta^{\ker}=I-S_\delta^{\ker}C_\delta^{\ker}$, the proof is complete.
\end{proof}

\subsubsection{Proof of Theorem~\ref{thm:AD-main} in the angular Banach range}
\label{subsubsec:atomic-banach-exact}

Combining the kernel analysis and synthesis bounds with the stability
estimate from Subsection~\ref{subsubsec:atomic-operators-banach} gives
the exact decomposition and coefficient estimate in
Theorem~\ref{thm:AD-main} for \(1\le p\le\infty\).

\begin{proof}
Let $\delta_0=\delta_0(\alpha,p,q)\in(0,\tfrac12]$ be the constant furnished by
Lemma~\ref{lem:stability}, and fix $0<\delta<\delta_0$.
Let $\mathcal P_\delta$ be a $\delta$-pocket tiling of $\C$ with centers
$\{\zeta_Q\}_{Q\in\mathcal P_\delta}$.

\smallskip
\noindent
\emph{(i) Synthesis.} 
The boundedness of the synthesis operator
\(
S_\delta^{\ker}
\)
from $\ell_w^{p,q}$ to $\mathcal F_\alpha^{p,q}$ was established in
Lemma~\ref{lem:synthesis}. In particular,
\[
\|S_\delta^{\ker}c\|_{\mathcal F_\alpha^{p,q}}
\lesssim \  \delta^{2 \big(\frac{1}{p} -1 \big)}
\,\|c\|_{\ell_w^{p,q}}
\qquad c\in \ell_w^{p,q}.
\]
Since \(\delta>0\) is fixed, this gives
\[
\|S_\delta^{\ker}c\|_{\mathcal F_\alpha^{p,q}}
\lesssim
\|c\|_{\ell_w^{p,q}},
\qquad c\in\ell_w^{p,q},
\]
with an implicit constant depending only on \(\alpha,p,q\), and \(\delta\).
Consequently, for every representation \(f=S_\delta^{\ker}c\),
\[
\|f\|_{\mathcal F_\alpha^{p,q}}
\lesssim
\|c\|_{\ell_w^{p,q}},
\]
and hence
\[
\|f\|_{\mathcal F_\alpha^{p,q}}
\lesssim
\inf\left\{
\|c\|_{\ell_w^{p,q}}:
c\in\ell_w^{p,q},\ f=S_\delta^{\ker}c
\right\}.
\]

\smallskip
\noindent
\emph{(ii) Coefficient recovery and optimality.}
By Lemma~\ref{lem:stability},
\[
\|R_\delta^{\ker}\|_{\mathcal F_\alpha^{p,q}\to \mathcal F_\alpha^{p,q}}
=
\|I-S_\delta^{\ker}C_\delta^{\ker}\|_{\mathcal F_\alpha^{p,q}\to \mathcal F_\alpha^{p,q}}
\le \frac12.
\]
The
Neumann inversion principle in Lemma~\ref{lem:neumann-quasi} applies to the space $\mathcal F_\alpha^{p,q}$ with exponent  $s=\min\{1,q\}$ in the present range $p\ge 1$, hence it yields that
\(
I-R_\delta^{\ker}=S_\delta^{\ker}C_\delta^{\ker}
\)
is invertible on $\mathcal F_\alpha^{p,q}$ and
\[
\|(I-R_\delta^{\ker})^{-1}\|_{\mathcal F_\alpha^{p,q}\to \mathcal F_\alpha^{p,q}}
\lesssim 1.
\]
Let $f\in\mathcal F_\alpha^{p,q}$ and define
\[
g:=(I-R_\delta^{\ker})^{-1}f,
\qquad
c^*:=C_\delta^{\ker}g.
\]
Then, by Lemma~\ref{lem:analysis-bdd}, $c^* \in \ell^{p, q}_w$ and
\[
\|c^*\|_{\ell_w^{p,q}}
=
\|C_\delta^{\ker}g\|_{\ell_w^{p,q}}
\lesssim \delta^{1-\frac1p}\|g\|_{\mathcal F_\alpha^{p,q}}
\lesssim \delta^{1-\frac1p}\|f\|_{\mathcal F_\alpha^{p,q}}.
\]
Since \(\delta>0\) is fixed, this gives
\[
\|c^*\|_{\ell_w^{p,q}}
\lesssim
\|f\|_{\mathcal F_\alpha^{p,q}},
\]
with an implicit constant depending only on \(\alpha,p,q\), and \(\delta\).
Moreover,
\[
S_\delta^{\ker}c^*
=
S_\delta^{\ker}C_\delta^{\ker}g
=
(I-R_\delta^{\ker})g
=
f.
\]
Therefore
\[
\inf\Bigl\{
\|c\|_{\ell_w^{p,q}}:
c\in\ell_w^{p,q},\ f=S_\delta^{\ker}c
\Bigr\}
\le
\|c^*\|_{\ell_w^{p,q}}
\lesssim
\,\|f\|_{\mathcal F_\alpha^{p,q}}.
\]
Combining the preceding two inequalities, we obtain
\[
\|f\|_{\mathcal F_\alpha^{p,q}}
\, \asymp \,
\inf\Bigl\{
\|c\|_{\ell_w^{p,q}}:
c\in\ell_w^{p,q},\ f=S_\delta^{\ker}c
\Bigr\}.
\]
This proves Theorem~\ref{thm:AD-main} in the range $1\le p\le\infty$.
\end{proof}

\subsection{Angular quasi-Banach range $0<p<1$}
\label{subsec:atomic-quasi}

We turn to \(0<p<1\).  The direct kernel analysis--synthesis scheme no longer
gives a small remainder, because the angular estimates are limited to
\(p\)-subadditivity.

\smallskip
\noindent
Finite-order jet atoms replace normalized Fock kernels at the intermediate
stage; they encode local Taylor data and supply the needed cancellation.
The jets are then eliminated locally into floating kernels, re-centered by
Section~\ref{sec:recentering-envelope}, and the remaining small defect is
iterated.

\smallskip
\noindent
Subsections~\ref{subsubsec:atomic-jet-setup}--
\ref{subsubsec:atomic-jet-stability} establish the jet model and exact jet
decomposition.  Subsections~\ref{subsubsec:atomic-jet-elimination} and
\ref{subsubsec:atomic-recentering} eliminate and re-center the jets.  The
\(\varepsilon\)-iteration in Subsection~\ref{subsubsec:atomic-quasi-proof}
gives the normalized-kernel decomposition.

\subsubsection{Jet atoms, jet coefficient spaces, and jet operators}
\label{subsubsec:atomic-jet-setup}

We introduce the finite-order jet atoms and coefficient spaces.  Throughout Subsections~\ref{subsubsec:atomic-jet-setup}--
\ref{subsubsec:atomic-jet-stability}, we fix
\(m\in\mathbb N_0\), \(0<p<1\), \(0<q\le\infty\), and a
\(\delta\)-pocket tiling \(\mathcal P_\delta\) of \(\C\).

\smallskip
\noindent
\begin{definition} \label{def:jet-setup}
For each pocket $Q\in\mathcal P_\delta$ and each jet order $0\le j\le m$, define the jet atom
\[
\psi_{Q,j}(z)
:=\frac{\alpha^j}{j!}(z-\zeta_Q)^j\,\kappa_{\zeta_Q}(z),
\qquad z\in\C.
\]
Thus \(\psi_{Q,0}=\kappa_{\zeta_Q}\); for \(j\ge1\), the factor
\((z-\zeta_Q)^j\) encodes the local jet of order \(j\).

\smallskip
\noindent
Associated with these atoms, we define the jet analysis operator
\[
(C_\delta^m f)_{Q,j}
:=\frac{\alpha}{\pi}\int_Q
f(w)\,e^{-\frac{\alpha}{2}|w|^2}\,
\big(\overline w-\overline{\zeta_Q}\big)^j\,
e^{i\alpha\,\im(\zeta_Q\overline w)}\,
e^{-\frac{\alpha}{2}|w-\zeta_Q|^2}\,dA(w),
\ \ Q\in\mathcal P_\delta,\  0\le j\le m.
\]
For coefficient sequences $c=(c_{Q,j})$, we define the jet synthesis operator by
\[
(S_\delta^m c)(z)
:=\sum_{Q\in\mathcal P_\delta}\sum_{j=0}^m c_{Q,j}\,\psi_{Q,j}(z),
\qquad z\in\C,
\]
whenever the series converges.  The associated remainder operator is
\[
R_\delta^m:=I-S_\delta^m C_\delta^m .
\]
\end{definition}

\smallskip
\noindent
We define the coefficient space for all \(0<p,q\le\infty\), although it will be
used below only for \(0<p<1\).

\begin{definition} 
For each $k\ge0$, write
\[
c_k=(c_{Q,j})_{Q\in\mathcal P_{\delta,k},\,0\le j\le m}.
\]
We equip $c_k$ with the block quasi-norm
\[
\|c_k\|_{\ell^p}
:=
\begin{cases}
\left(\displaystyle\sum_{Q\in\mathcal P_{\delta,k}}\sum_{j=0}^m |c_{Q,j}|^p\right)^{\frac1p},
& 0<p<\infty,\\[3mm]
\displaystyle\sup_{Q\in\mathcal P_{\delta,k},\,0\le j\le m}|c_{Q,j}|,
& p=\infty.
\end{cases}
\]
We then define $\ell_w^{p,q,m}$ to be the space of all sequences
\(
c=(c_{Q,j})_{Q\in\mathcal P_\delta,\,0\le j\le m}
\)
such that
\[
\|c\|_{\ell_w^{p,q,m}}
:=
\begin{cases}
\left(\displaystyle\sum_{k\ge0}\bigl[w_k\|c_k\|_{\ell^p}\bigr]^q\right)^{\frac1q},
& 0<q<\infty,\\[3mm]
\displaystyle\sup_{k\ge0} w_k\|c_k\|_{\ell^p},
& q=\infty,
\end{cases}
\]
is finite, where
\(
w_k:=(1+k)^{\frac1q-\frac1p}\),
with the usual convention that $\frac1\infty=0$.
\end{definition}

\smallskip
\noindent
The next two subsections prove boundedness of \(C_\delta^m\) and
\(S_\delta^m\), and smallness of \(R_\delta^m\) on
\(\mathcal F_\alpha^{p,q}\) for sufficiently small \(\delta\).

\subsubsection{Bounded jet analysis and synthesis}
\label{subsubsec:atomic-jet-boundedness}

In this subsection we establish the boundedness of the jet analysis map
\[
C_\delta^m:\mathcal F_\alpha^{p,q}\to \ell_w^{p,q,m}
\]
and the jet synthesis map
\[
S_\delta^m:\ell_w^{p,q,m}\to \mathcal F_\alpha^{p,q}.
\]
These are the quasi-Banach analogues of the bounded kernel operators from Subsection~\ref{subsec:atomic-banach}.

\begin{lemma} \label{lem:analysisboundedness}
Let $m\in\mathbb N_0$, $\alpha > 0$, $0<p<1$, and $0<q\le\infty$. Let $\mathcal{P}_{\delta}$ be a $\delta$-pocket tiling  of $\C$ with $\delta\in(0,\frac12]$. Then
\begin{equation}\label{eq:analysisboundedness}
\|C_\delta^m f\|_{\ell_w^{p,q,m}}
\lesssim \,\delta^{\,2-\frac{2}{p}}\,
\|f\|_{\Fpq},
\qquad
f\in\mathcal F_\alpha^{p,q},
\end{equation}
where the implicit constant depends only on \(\alpha,p,q, m\), and is
independent of \(\delta\).
In particular, $C_\delta^m:\mathcal F_\alpha^{p,q}\to \ell_w^{p,q,m}$ is a bounded linear operator.
\end{lemma}

\begin{proof}
Let $\mathcal P_\delta^\ast$ be the fattened family from Definition~\ref{def:fattened-pockets}. Fix
$Q\in\mathcal P_\delta$ with center $\zeta_Q$ and $0\le j\le m$. Since
$|e^{i(\cdot)}|=1$, $e^{-\frac{\alpha}{2}|w-\zeta_Q|^2}\le 1$, and
$|w-\zeta_Q|\lesssim \delta$ on $Q$, we obtain
\[
|(C_\delta^m f)_{Q,j}|
\le \frac{\alpha}{\pi}\int_Q |f(w)|e^{-\frac{\alpha}{2}|w|^2}
\, |w-\zeta_Q|^j\,dA(w)
\lesssim \delta^j \int_Q |f(w)|e^{-\frac{\alpha}{2}|w|^2}\,dA(w).
\]
Hence
\begin{equation}\label{eq:CQj-rough-rev}
|(C_\delta^m f)_{Q,j}|^p
\lesssim
\delta^{pj}\,|Q|^p\,
\sup_{w\in Q} |f(w)|^p e^{-\frac{\alpha p}{2}|w|^2}.
\end{equation}

\smallskip
\noindent
To estimate the supremum, define
\[
F_Q(w):=f(w)\,e^{-\alpha w\overline{\zeta_Q}+\frac{\alpha}{2}|\zeta_Q|^2},
\qquad w\in\C.
\]
Then $F_Q$ is entire, so $|F_Q|^p$ is subharmonic. Since each point of $Q$ is contained in a Euclidean ball of radius comparable to $\delta$ inside $Q^\ast$, the sub-mean inequality gives
\[
\sup_{w\in Q}|F_Q(w)|^p
\lesssim \frac{1}{|Q^\ast|}\int_{Q^\ast}|F_Q(u)|^p\,dA(u).
\]
Moreover,
\[
|F_Q(u)|^p
=
|f(u)|^p e^{-\frac{\alpha p}{2}|u|^2} e^{\frac{\alpha p}{2}|u-\zeta_Q|^2},
\]
and since $|u-\zeta_Q|\lesssim\delta\le\frac12$ on $Q^\ast$, we have
\[
\sup_{w\in Q} |f(w)|^p e^{-\frac{\alpha p}{2}|w|^2} \,
 \leq \, \sup_{w \in Q} |F_Q(w)|^p \, \lesssim \,
\frac{1}{|Q^\ast|}\int_{Q^\ast}|f(u)|^p e^{-\frac{\alpha p}{2}|u|^2}\,dA(u).
\]
Combining this with \eqref{eq:CQj-rough-rev} and using $|Q|\asymp |Q^\ast|\asymp \delta^2$, we obtain
\[
|(C_\delta^m f)_{Q,j}|^p
\lesssim
\delta^{pj}\,\delta^{2p-2}
\int_{Q^\ast}|f(u)|^p e^{-\frac{\alpha p}{2}|u|^2}\,dA(u).
\]
Summing over $0\le j\le m$ and using $\sum_{j=0}^m \delta^{pj}\lesssim 1$, we arrive at
\begin{equation}\label{eq:pocketwise-jet-analysis}
\sum_{j=0}^m |(C_\delta^m f)_{Q,j}|^p
\lesssim
\delta^{2p-2}
\int_{Q^\ast}|f(u)|^p e^{-\frac{\alpha p}{2}|u|^2}\,dA(u).
\end{equation}

\smallskip
\noindent
Now fix $k\ge0$. Summing \eqref{eq:pocketwise-jet-analysis} over all
$Q\in\mathcal P_{\delta,k}$ and using the bounded overlap of the family
$\mathcal P_\delta^\ast$, we get
\begin{align*}
\|(C_\delta^m f)_k\|_{\ell^p}^p
=
\sum_{Q\in\mathcal P_{\delta,k}}\sum_{j=0}^m |(C_\delta^m f)_{Q,j}|^p & \lesssim
\delta^{2p-2}
\sum_{Q\in\mathcal P_{\delta,k}}
\int_{Q^\ast}|f(w)|^p e^{-\frac{\alpha p}{2}|w|^2}\,dA(w) \\
&\lesssim
\delta^{2p-2}
\int_{\widetilde{A}_k}
|f(w)|^p e^{-\frac{\alpha p}{2}|w|^2}\,dA(w),
\end{align*}
where $\widetilde{A}_k: = A_{k-1}\cup A_k\cup A_{k+1}$
with the convention $A_{-1}=\emptyset$. Passing to polar coordinates,
\begin{align*}
\int_{\widetilde{A}_k}
|f(w)|^p e^{-\frac{\alpha p}{2}|w|^2}\,dA(w)
 =
\sum_{|n-k|\le1}\int_n^{n+1} M_p^p(f,r)\,e^{-\frac{\alpha p}{2}r^2}\,r\,dr 
 \lesssim
\sum_{|n-k|\le1}(1+n)\,B_p^p(f;n).
\end{align*}
Taking $p$-th roots yields
\begin{equation*}
\|(C_\delta^m f)_k\|_{\ell^p}
\lesssim
\delta^{\,2-\frac{2}{p}}
\Biggl(\sum_{|n-k|\le1}(1+n)\,B_p^p(f;n)\Biggr)^{\frac1p}
\lesssim
\delta^{\,2-\frac{2}{p}}
(1+k)^{\frac1p}\sup_{|n-k|\le1} B_p(f;n).
\end{equation*}

\smallskip
\noindent
Together with Proposition~\ref{prop:annular}, this gives, for $0<q<\infty$,
\begin{align*}
\|C_\delta^m f\|_{\ell_w^{p,q,m}}^q
&=
\sum_{k\ge0}
\Bigl[(1+k)^{\frac1q-\frac1p}\|(C_\delta^m f)_k\|_{\ell^p}\Bigr]^q 
\lesssim
\delta^{\,q\big(2-\frac{2}{p}\big)}
\sum_{k\ge0}
(1+k)\Bigl(\sup_{|n-k|\le1} B_p(f;n)\Bigr)^q \\
& \lesssim \delta^{\,q\big(2-\frac{2}{p}\big)}
\sum_{k\ge0}(1+k)\,B_p^q(f;k) \, \asymp \, \delta^{\,q\big(2-\frac{2}{p}\big)} \|f\|_{\Fpq}^q,
\end{align*}
since each index $n$ contributes to at most three neighboring values of $k$, and $(1+k)\asymp(1+n)$ whenever $|n-k|\le1$;
if $q=\infty$, then
\[
\|C_\delta^m f\|_{\ell_w^{p,\infty,m}}
=
\sup_{k\ge0}(1+k)^{-\frac1p}\|(C_\delta^m f)_k\|_{\ell^p}
\lesssim
\delta^{\,2-\frac{2}{p}}
\sup_{k\ge0}\sup_{|n-k|\le1} B_p(f;n) \, \asymp \, \delta^{\,2-\frac{2}{p}} \|f\|_{\FpqA{p}{\infty}{\alpha}}.
\]
This proves \eqref{eq:analysisboundedness}.
\end{proof}

\begin{lemma} 
\label{lem:jet-series-entire}
Let $m\in\mathbb N_0$, $0<p<1$, and $0<q\le\infty$. Let 
$\mathcal P_\delta$ be a $\delta$-pocket tiling of $\C$ with
$\delta\in(0,\frac12]$. Then, for every coefficient sequence
\(
c=(c_{Q,j})_{Q\in\mathcal P_\delta,\ 0\le j\le m}\in \ell_w^{p,q,m}\),
the series
\[
\sum_{Q\in\mathcal P_\delta}\sum_{j=0}^m c_{Q,j}\,\psi_{Q,j}(z)
\]
converges absolutely and locally uniformly on\/ $\C$. In particular, it defines
an entire function, and hence the jet synthesis operator
\[
S_\delta^m:\ell_w^{p,q,m}\to H(\C),
\qquad
(S_\delta^m c)(z):=\sum_{Q\in\mathcal P_\delta}\sum_{j=0}^m c_{Q,j}\,\psi_{Q,j}(z),
\]
is well-defined.
\end{lemma}

\begin{proof}
Fix $R>0$ and an integer $N\ge R$. Let $Q\in\mathcal P_{\delta,k}$ with $k\ge N$, let
$0\le j\le m$, and let $|z|\le R$. 
Using the inequality
\begin{equation} \label{eq:poly-gauss-jet}
t^j e^{-\frac{\alpha}{2}t^2}\lesssim e^{-\frac{\alpha}{4}t^2},
\qquad t\ge0,\quad 0\le j\le m,
\end{equation}
we get
\[
|\psi_{Q,j}(z)|
=
\frac{\alpha^j}{j!}|z-\zeta_Q|^j
e^{\frac{\alpha}{2}|z|^2}e^{-\frac{\alpha}{2}|z-\zeta_Q|^2} \lesssim e^{-\frac{\alpha}{4} (k - R)^2},
\]
where the implicit constant depends only on $\alpha, m$, and $R$.
Hence
\[
\sum_{Q\in\mathcal P_{\delta,k}}\sum_{j=0}^m
|c_{Q,j}|\,|\psi_{Q,j}(z)|
\lesssim
e^{-\frac{\alpha}{4}(k-R)^2}
\sum_{Q\in\mathcal P_{\delta,k}}\sum_{j=0}^m |c_{Q,j}|.
\]
Since $0<p<1$, we have
\[
\sum_{Q\in\mathcal P_{\delta,k}}\sum_{j=0}^m |c_{Q,j}|
\le
\Biggl(\sum_{Q\in\mathcal P_{\delta,k}}\sum_{j=0}^m |c_{Q,j}|^p\Biggr)^{\frac1p}
=
\|c_k\|_{\ell^p}.
\]
Therefore
\[
\sum_{Q\in\mathcal P_{\delta,k}}\sum_{j=0}^m
|c_{Q,j}|\,|\psi_{Q,j}(z)|
\lesssim
e^{-\frac{\alpha}{4}(k-R)^2}(1+k)^{\frac1p-\frac1q}
\,w_k\|c_k\|_{\ell^p}.
\]

\smallskip
\noindent
The remainder of the proof is the same as in
Lemma~\ref{lem:kernel-series-entire}: after summing over $k\ge N$ and treating
the cases $0<q\le1$, $1<q<\infty$, and $q=\infty$, one finds that the tails
converge to zero uniformly on compact subsets of $\C$. Hence the series
defining $S_\delta^m c$ converges absolutely and locally uniformly on\/ $\C$,
so its sum is entire.
\end{proof}

\begin{lemma} \label{lem:Sr-synth}
Let $m\in\mathbb N_0$, $\alpha > 0$, $0<p<1$, and $0<q\le\infty$. Let $\mathcal P_\delta$ be a $\delta$-pocket tiling of $\C$ with $\delta\in(0,\frac12]$. Then the synthesis operator
\[
S_\delta^{m}:\ell_w^{p,q,m}\to \mathcal F_\alpha^{p,q},
\qquad
c\mapsto \sum_{Q\in\mathcal P_\delta}\sum_{j=0}^m c_{Q,j}\,\psi_{Q,j},
\]
is bounded. More precisely,
\begin{equation}\label{eq:Sru-main}
\|S_\delta^{m}c\|_{\Fpq}
\lesssim \,\|c\|_{\ell_w^{p,q,m}},
\qquad
c\in\ell_w^{p,q,m},
\end{equation}
where the implicit constant depends only on \(\alpha,p,q, m\), and is
independent of \(\delta\).
\end{lemma}

\begin{proof}
Fix $c = (c_{Q, j})_{Q \in \mathcal{P}_{\delta}, 0 \leq j \leq m} \in \ell^{p,q,m}_{w}$.
Let $f=S_\delta^m c$, and, for each $k\ge0$, set
\[
F_k(z):=\sum_{Q\in\mathcal P_{\delta,k}}\sum_{j=0}^m c_{Q,j}\,\psi_{Q,j}(z).
\]

\smallskip
\noindent
Fix $k \geq 0$ and $Q\in\mathcal P_{\delta,k}$, $0\le j\le m$, and $r\ge0$. Using the inequality \eqref{eq:poly-gauss-jet} again, 
we obtain, for $z=re^{i\theta}$,
\[
|\psi_{Q,j}(re^{i\theta})|\,e^{-\frac{\alpha}{2}r^2}
=
\frac{\alpha^j}{j!}\,
|re^{i\theta}-\zeta_Q|^j\,
e^{-\frac{\alpha}{2}|re^{i\theta}-\zeta_Q|^2} \lesssim e^{-\frac{\alpha}{4}|re^{i\theta}-\zeta_Q|^2}
\]
where the implicit constant depends only on $\alpha$ and $m$.
Therefore, by the same circular Gaussian estimate used in the proof of Lemma~\ref{lem:remainderkernelestimate},
\begin{equation}\label{eq:jet-atom-profile}
M_p^p(\psi_{Q,j},r)\,e^{-\frac{\alpha p}{2}r^2} \, \lesssim \,
(1+|\zeta_Q|)^{-1}
e^{-\frac{\alpha p}{8}(r-|\zeta_Q|)^2} 
\, \lesssim \, 
(1+k)^{-1}e^{-\frac{\alpha p}{16}(r-k)^2},
\qquad r\ge0,
\end{equation}
where the implicit constant depends only on $\alpha, m$, and $p$.

\smallskip
\noindent
Fix $\ell \ge 0$.
Since $0<p<1$, 
\[
M_p^p(F_k,r)
\le
\sum_{Q\in\mathcal P_{\delta,k}}\sum_{j=0}^m
|c_{Q,j}|^p\,M_p^p(\psi_{Q,j},r).
\]
Using \eqref{eq:jet-atom-profile} and taking the supremum over $r\in[\ell,\ell+1)$, we get
\begin{equation}\label{eq:block-jet-synthesis}
B_p^p(F_k;\ell)
\lesssim
(1+k)^{-1}
e^{-\frac{\alpha p}{32}(\ell-k)^2}
\|c_k\|_{\ell^p}^p,
\qquad \ell,k\ge0.
\end{equation}

\smallskip
\noindent
Again using $0<p<1$, we obtain for the shellwise sum
\[
M_p^p(f,r)\le \sum_{k\ge0} M_p^p(F_k,r), \quad
\text{hence,} \quad 
B_p^p(f;\ell)\le \sum_{k\ge0} B_p^p(F_k;\ell).
\]
Using this and \eqref{eq:block-jet-synthesis}, we obtain
\begin{equation}\label{eq:jet-shell-conv}
B_p^p(f;\ell)
\lesssim
\sum_{k\ge0}
(1+k)^{-1}e^{-\frac{\alpha p}{32}(\ell-k)^2}\|c_k\|_{\ell^p}^p.
\end{equation}

\smallskip
\noindent
Assume that $0<q<\infty$, and set
\[
q_1:=\frac{q}{p},
\qquad
Y_k:=\bigl(w_k\|c_k\|_{\ell^p}\bigr)^p,
\qquad
K(\ell,k):=e^{-\frac{\alpha p}{32}(\ell-k)^2}(1+k)^{-\frac{1}{q_1}}.
\]
Then \eqref{eq:jet-shell-conv} becomes
\[
B_p^p(f;\ell)\lesssim \sum_{k\ge0} K(\ell,k)\,Y_k.
\]
Using the generalized Schur test
(Lemma~\ref{lem:gen-schur-test}) with exponent $ q_1=\frac{q}{p}$ gives
\[
\sum_{\ell\ge0}(1+\ell)\,B_p^q(f;\ell)
\lesssim
\sum_{\ell\ge0}(1+\ell)\,\Bigl(\sum_{k\ge0} K(\ell,k)\,Y_k\Bigr)^{\frac{q}{p}}
\lesssim
\sum_{k\ge0} Y_k^{\frac{q}{p}}
=
\sum_{k\ge0}\bigl(w_k\|c_k\|_{\ell^p}\bigr)^q.
\]
By the annular discretization principle in Proposition~\ref{prop:annular},
\[
\|f\|_{\Fpq}^q
\, \asymp \,
\sum_{\ell\ge0}(1+\ell)\,B_p^q(f;\ell)
\, \lesssim \,
\sum_{k\ge0}\bigl(w_k\|c_k\|_{\ell^p}\bigr)^q
\, = \,
\|c\|_{\ell_w^{p,q,m}}^q.
\]
If $q=\infty$, then $w_k=(1+k)^{-\frac1p}$ and \eqref{eq:jet-shell-conv} becomes
\[
B_p^p(f;\ell)
\, \lesssim \,
\sum_{k\ge0} e^{-\frac{\alpha p}{32}(\ell-k)^2}\,
\bigl(w_k\|c_k\|_{\ell^p}\bigr)^p.
\]
Applying the endpoint form of Lemma~\ref{lem:gen-schur-test}, we obtain
\[
\sup_{\ell\ge0} B_p^p(f;\ell) \, \lesssim \, \sup_{\ell\ge0} \sum_{k\ge0} e^{-\frac{\alpha p}{32}(\ell-k)^2}\,
\bigl(w_k\|c_k\|_{\ell^p}\bigr)^p \,
\lesssim \,
\sup_{k\ge0}\bigl(w_k\|c_k\|_{\ell^p}\bigr)^p.
\]
Using the endpoint form of Proposition~\ref{prop:annular}, it follows that
\[
\|f\|_{\FpqA{p}{\infty}{\alpha}}
\, \asymp \,
\sup_{\ell\ge0} B_p(f;\ell) \,
\lesssim \,
\sup_{k\ge0} w_k\|c_k\|_{\ell^p}
=
\|c\|_{\ell_w^{p,\infty,m}}.
\]
Thus, in all cases,
\[
\|S_\delta^m c\|_{\Fpq}
=
\|f\|_{\Fpq}
\lesssim
\|c\|_{\ell_w^{p,q,m}},
\]
which proves \eqref{eq:Sru-main}.
\end{proof}

\begin{corollary}[Quasi-Banach kernel synthesis]\label{cor:synthesis<1}
Let $\alpha > 0$, $0<p<1$, and $0<q\le\infty$. Let $\mathcal P_\delta$ be a $\delta$-pocket tiling of $\C$ with $\delta\in(0,\frac12]$. Then the kernel synthesis operator
\[
S_\delta^{\ker}:\ell_w^{p,q}\to \mathcal F_\alpha^{p,q},
\qquad
c\mapsto \sum_{Q\in\mathcal P_\delta} c_Q\,\kappa_{\zeta_Q},
\]
is bounded. More precisely, 
\[
\|S_\delta^{\ker}c\|_{\Fpq}
\lesssim \,\|c\|_{\ell_w^{p,q}},
\qquad
c\in\ell_w^{p,q},
\]
where the implicit constant depends only on \(\alpha, p, q\), and is
independent of \(\delta\).
\end{corollary}

\begin{proof}
This is the special case $m=0$ of Lemma~\ref{lem:Sr-synth}.
\end{proof}

\subsubsection{Jet remainder estimate and exact jet decomposition}
\label{subsubsec:atomic-jet-stability}

We show that the remainder operator
\[
R_\delta^m=I-S_\delta^m C_\delta^m
\]
is small on $\mathcal F_\alpha^{p,q}$ when $\delta>0$ is sufficiently small. This yields an exact decomposition in terms of jet atoms by Neumann inversion.

\begin{lemma} \label{lem:remainderkernelestimate<1}
Let $m\in\mathbb N_0$, $\alpha>0$, $0<p<1$, and $0<q\le\infty$. Let $\mathcal P_\delta$ be a $\delta$-pocket tiling of $\C$ with $\delta\in(0,\frac12]$ and let 
$$
R_\delta^m = I-S_\delta^m C_\delta^m
$$
be the jet remainder operator in Definition \ref{def:jet-setup}.
Then 
\begin{equation}\label{eq:remainderkernelestimate<1}
B_p^p(R_\delta^m f;\ell)
\lesssim
\,\delta^{\,p(m+3)-2}
\sum_{k\ge0} e^{-\frac{\alpha p}{32}(\ell-k)^2}
\max_{|n-k|\le1} B_p^p(f;n)
\end{equation}
for every $f\in\mathcal F_\alpha^{p,q}$ and every $\ell\ge0$, where the implicit
constant depends only on \(\alpha, p, m\), and is independent of
\(q, \delta, f\), and \(\ell\).
\end{lemma}

\begin{proof}
Let $f\in\Fpq$. By the reproducing formula in Proposition \ref{prop:fock-reproducing},
\[
f(z)
=
\frac{\alpha}{\pi}\sum_{Q\in\mathcal P_\delta}
\int_Q f(w)\,e^{-\frac{\alpha}{2}|w|^2+i\alpha\im(\zeta_Q\overline w)}
\,e^{\alpha\overline wz-\frac{\alpha}{2}|w|^2-i\alpha\im(\zeta_Q\overline w)}\,dA(w).
\]
On the other hand, by the definitions of $C_\delta^m$ and $S_\delta^m$,
\begin{align*}
S_\delta^m C_\delta^m f(z)
&=
\frac{\alpha}{\pi}\sum_{Q\in\mathcal P_\delta}
\int_Q
f(w)\,e^{-\frac{\alpha}{2}|w|^2+i\alpha\im(\zeta_Q\overline w)}
e^{-\frac{\alpha}{2}|w-\zeta_Q|^2} \\
&\qquad\qquad\qquad\qquad \times
\sum_{j=0}^m \frac{\big[\alpha(z-\zeta_Q)(\overline w-\overline{\zeta_Q})\big]^j}{j!}\,
\kappa_{\zeta_Q}(z)\,dA(w).
\end{align*}
Hence
\begin{equation}\label{eq:Rudelta<1-rev}
R_\delta^m f(z)
=
\frac{\alpha}{\pi}\sum_{Q\in\mathcal P_\delta}
\int_Q
f(w)\,e^{-\frac{\alpha}{2}|w|^2+i\alpha\im(\zeta_Q\overline w)}
\,H(z,\zeta_Q,w)\,dA(w),
\end{equation}
where
\[
H(z,\zeta_Q,w)
:=
e^{\alpha\overline wz-\frac{\alpha}{2}|w|^2-i\alpha\im(\zeta_Q\overline w)}
-
e^{-\frac{\alpha}{2}|w-\zeta_Q|^2}
\sum_{j=0}^m \frac{[\alpha(z-\zeta_Q)(\overline w-\overline{\zeta_Q})]^j}{j!}
\,\kappa_{\zeta_Q}(z).
\]

\smallskip
\noindent
We estimate the kernel error.  A direct computation gives
\[
e^{\alpha\overline wz-\frac{\alpha}{2}|w|^2-i\alpha\im(\zeta_Q\overline w)}
=
e^{-\frac{\alpha}{2}|w-\zeta_Q|^2}\,
\kappa_{\zeta_Q}(z)\,
e^{\alpha(z-\zeta_Q)(\overline w-\overline{\zeta_Q})}.
\]
Therefore
\[
H(z,\zeta_Q,w)
=
e^{-\frac{\alpha}{2}|w-\zeta_Q|^2}\,
\kappa_{\zeta_Q}(z)\,
\Biggl(
e^{\alpha(z-\zeta_Q)(\overline w-\overline{\zeta_Q})}
-
\sum_{j=0}^m \frac{[\alpha(z-\zeta_Q)(\overline w-\overline{\zeta_Q})]^j}{j!}
\Biggr).
\]
If $w\in Q$, then $|w-\zeta_Q|\lesssim\delta$. Using the Taylor remainder estimate and
\eqref{eq:poly-gauss-jet} for $j = m+1$, we obtain
\begin{align*}
|H(z,\zeta_Q,w)|\,e^{-\frac{\alpha}{2}|z|^2}
&\lesssim
e^{-\frac{\alpha}{2}|z-\zeta_Q|^2}
\,|z-\zeta_Q|^{m+1}|w-\zeta_Q|^{m+1}
e^{\alpha|z-\zeta_Q||w-\zeta_Q|} \\
&\lesssim
|w-\zeta_Q|^{m+1} e^{\alpha|w-\zeta_Q|^2}
\,|z-\zeta_Q|^{m+1}e^{-\frac{\alpha}{4}|z-\zeta_Q|^2} \\
&\lesssim
\delta^{m+1} e^{-\frac{\alpha}{8}|z-\zeta_Q|^2}.
\end{align*}

\smallskip
\noindent
For each $Q\in\mathcal P_\delta$, define
\begin{align*}
f_Q(z)
:= & 
\frac{\alpha}{\pi}\int_Q
f(w)\,e^{-\frac{\alpha}{2}|w|^2+i\alpha\im(\zeta_Q\overline w)}
\,H(z,\zeta_Q,w)\,dA(w) \\
= & 
\frac{\alpha}{\pi}\int_Q
f(w)\,e^{- \alpha \overline{\zeta_Q} w + \frac{\alpha}{2} |\zeta_Q|^2}\, e^{ - \frac{\alpha}{2}|w - \zeta_Q|^2} 
\,H(z,\zeta_Q,w)\,dA(w).
\end{align*}
Then \eqref{eq:Rudelta<1-rev} becomes
\[
R_\delta^m f=\sum_{Q\in\mathcal P_\delta} f_Q .
\]
Since $0<p<1$, we have
\[
M_p^p(R_\delta^m f,r)\le \sum_{Q\in\mathcal P_\delta} M_p^p(f_Q,r).
\]

\smallskip
\noindent
Now fix $Q\in\mathcal P_{\delta,k}$. Using the bound above, we get
\begin{align*}
M_p^p(f_Q,r)e^{-\frac{\alpha p}{2}r^2}
&\lesssim
\delta^{p(m+1)}
\left(\int_0^{2\pi} e^{-\frac{\alpha p}{8}|re^{i\theta}-\zeta_Q|^2}\,\frac{d\theta}{2\pi}\right)
\left(
\int_Q \Bigl|f(w)e^{-\alpha\overline{\zeta_Q}w+\frac{\alpha}{2}|\zeta_Q|^2}\Bigr|\,dA(w)
\right)^p.
\end{align*}
By the same circular Gaussian estimate used earlier in the proof of Lemma \ref{lem:remainderkernelestimate},
\[
\int_0^{2\pi} e^{-\frac{\alpha p}{8}|re^{i\theta}-\zeta_Q|^2}\,\frac{d\theta}{2\pi}
\lesssim
(1+k)^{-1} e^{-\frac{\alpha p}{16}(r-k)^2},
\qquad Q\in\mathcal P_{\delta,k}.
\]
Moreover, exactly as in the proof of Lemma~\ref{lem:analysisboundedness},
\[
\left(
\int_Q \Bigl|f(w)e^{-\alpha\overline{\zeta_Q}w+\frac{\alpha}{2}|\zeta_Q|^2}\Bigr|\,dA(w)
\right)^p
\lesssim
|Q|^p \sup_{w\in Q}\Bigl|f(w)e^{-\alpha\overline{\zeta_Q}w+\frac{\alpha}{2}|\zeta_Q|^2}\Bigr|^p
\]
and
\[
\sup_{w\in Q}\Bigl|f(w)e^{-\alpha\overline{\zeta_Q}w+\frac{\alpha}{2}|\zeta_Q|^2}\Bigr|^p
\lesssim
\frac{1}{|Q^\ast|}
\int_{Q^\ast} |f(u)|^p e^{-\frac{\alpha p}{2}|u|^2}\,dA(u).
\]
Since $|Q|\asymp |Q^\ast|\asymp \delta^2$, it follows that
\[
M_p^p(f_Q,r)e^{-\frac{\alpha p}{2}r^2}
\lesssim
\delta^{p(m+3)-2}
(1+k)^{-1} e^{-\frac{\alpha p}{16}(r-k)^2}
\int_{Q^\ast}|f(u)|^p e^{-\frac{\alpha p}{2}|u|^2}\,dA(u).
\]

\smallskip
\noindent
Now fix $\ell\ge0$ and take the supremum over $r\in[\ell,\ell+1)$. Summing over all
$Q\in\mathcal P_{\delta,k}$ and using the bounded overlap of the fattened pockets, we get
\begin{align*}
B_p^p(R_\delta^m f;\ell)
&\lesssim
\delta^{p(m+3)-2}
\sum_{k\ge0}
(1+k)^{-1} e^{-\frac{\alpha p}{32}(\ell-k)^2}
\int_{A_{k-1}\cup A_k\cup A_{k+1}}
|f(u)|^p e^{-\frac{\alpha p}{2}|u|^2}\,dA(u).
\end{align*}
Passing to polar coordinates and arguing exactly as in the proof of
Lemma~\ref{lem:analysisboundedness}, we obtain
\[
\int_{A_{k-1}\cup A_k\cup A_{k+1}}
|f(u)|^p e^{-\frac{\alpha p}{2}|u|^2}\,dA(u)
\lesssim
(1+k)\max_{|n-k|\le1} B_p^p(f;n).
\]
Substituting this into the previous estimate yields
\[
B_p^p(R_\delta^m f;\ell)
\lesssim
\delta^{p(m+3)-2}
\sum_{k\ge0}
e^{-\frac{\alpha p}{32}(\ell-k)^2}
\max_{|n-k|\le1} B_p^p(f;n),
\]
which is \eqref{eq:remainderkernelestimate<1}.
\end{proof}

\begin{lemma} \label{lem:stability<1}
Let $\alpha>0$, $0<p<1$, $0<q\le\infty$, and
\[
m:=\Bigl\lfloor \frac{2(1-p)}{p}\Bigr\rfloor.
\]
Then there exists $\delta_0=\delta_0(\alpha,p,q)\in(0,\frac12]$ such that for every $\delta$-pocket tiling $\mathcal{P}_{\delta}$ of $\C$ with $0<\delta<\delta_0$,
\[
\|I-S_\delta^m C_\delta^m\|_{\Fpq\to\Fpq}\le \frac12.
\]
\end{lemma}

\begin{proof}
By Lemma~\ref{lem:remainderkernelestimate<1}, there exists
$C_0=C_0(\alpha,p,m)>0$ such that
\begin{equation}\label{eq:bp<1-stability}
B_p^p(R_\delta^m f;\ell)
\le
C_0\,\delta^{p(m+3)-2}
\sum_{k\ge0} e^{-\frac{\alpha p}{32}(\ell-k)^2}
\max_{|n-k|\le1} B_p^p(f;n)
\end{equation}
for every $f\in\Fpq$ and $\ell\ge0$, where $R_\delta^m:=I-S_\delta^m C_\delta^m$.

\smallskip
\noindent
Set
\(
\eta:=m+3-\frac2p
\).
Since
\(
m=\Bigl\lfloor \frac2p-2\Bigr\rfloor\),
we have
\(0<\eta\le 1\).
Assume first that $0<q<\infty$, and put
\[
q_1: = \frac{q}{p}, \qquad Y_k: = \left((1 + k)^{\frac{1}{q}} \max_{|n-k| \leq 1} B_p(f;n)\right)^p, \qquad K(\ell, k): = C_0 e^{-\frac{\alpha p}{32}(\ell-k)^2}\, (1 + k)^{-\frac{1}{q_1}}.
\]
Thus \eqref{eq:bp<1-stability} may be rewritten as
\[
B_p^p(R_\delta^m f;\ell)
\le
\,\delta^{p\eta}
\sum_{k\ge0} K(\ell, k) Y_k.
\]
Applying the generalized Schur test in Lemma~\ref{lem:gen-schur-test} to exponent $q_1 = \frac{q}{p}$,
we obtain
\begin{align*}
\sum_{\ell\ge0}(1+\ell)\,B_p^q(R_\delta^m f;\ell)
 \lesssim  \delta^{q \eta}
\sum_{\ell\ge0}(1+\ell)\,\Bigl(\sum_{k\ge0} K(\ell, k) Y_k\Bigr)^{\frac{q}{p}} 
\lesssim
\delta^{q\eta}\sum_{k\ge0}(1+k)\,\max_{|n-k|\le1} B_p^q(f;n),
\end{align*}
where the implicit constant depends only on $\alpha, p, q$, and $m$.
Since
 each index $n$ contributes to at most three neighboring values of $k$, while
$(1+k)\asymp(1+n)$ whenever $|n-k|\le1$, it follows that
\[
\sum_{\ell\ge0}(1+\ell)\,B_p^q(R_\delta^m f;\ell)
 \lesssim  \delta^{q \eta} 
\sum_{k\ge0}(1+k)\,B_p^q(f;k).
\]
Hence, by the annular discretization principle in Proposition~\ref{prop:annular},
\[
\|R_\delta^m f\|_{\Fpq}^q
\asymp
\sum_{\ell\ge0}(1+\ell)\,B_p^q(R_\delta^m f;\ell)
\lesssim
\delta^{q\eta}\sum_{k\ge0}(1+k)\,B_p^q(f;k)
\asymp
\delta^{q\eta}\|f\|_{\Fpq}^q.
\]
Therefore
\begin{equation}\label{eq:R-stability-qfinite}
\|R_\delta^m f\|_{\Fpq}\lesssim \delta^\eta \|f\|_{\Fpq},
\qquad 0<q<\infty.
\end{equation}
If $q=\infty$, then by \eqref{eq:bp<1-stability},
\[
B_p^p(R_\delta^m f;\ell)
\lesssim
\delta^{p\eta}
\sum_{k\ge0} e^{-\frac{\alpha p}{32}(\ell-k)^2} \max_{|n-k|\le1} B_p^p(f;n).
\]
Applying the endpoint form of Lemma~\ref{lem:gen-schur-test}, we get
\[
\sup_{\ell\ge0} B_p^p(R_\delta^m f;\ell)
\lesssim
\delta^{p\eta}\sup_{k\ge0}\max_{|n-k|\le1} B_p^p(f;n)
\lesssim
\delta^{p\eta}\sup_{n\ge0} B_p^p(f;n).
\]
Using the endpoint form of Proposition~\ref{prop:annular}, it follows that
\begin{equation}\label{eq:R-stability-qinfty}
\|R_\delta^m f\|_{\mathcal F_\alpha^{p,\infty}}
\lesssim
\delta^\eta \|f\|_{\mathcal F_\alpha^{p,\infty}}.
\end{equation}

\smallskip
\noindent
Combining \eqref{eq:R-stability-qfinite} and \eqref{eq:R-stability-qinfty}, there exists a constant
$C=C(\alpha,p,q,m)>0$ such that
\[
\|R_\delta^m f\|_{\Fpq}\le C\,\delta^\eta \|f\|_{\Fpq},
\qquad f\in\Fpq.
\]
Choose $\delta_0=\delta_0(\alpha,p,q,m)\in(0,\frac12]$ so that
\[
C\,\delta_0^\eta\le \frac12.
\]
Then for every $0<\delta<\delta_0$,
\[
\|I-S_\delta^m C_\delta^m\|_{\Fpq\to\Fpq}
=
\|R_\delta^m\|_{\Fpq\to\Fpq}
\le \frac12,
\]
as required.
\end{proof}

\begin{theorem}[Jet atomic decomposition]
\label{thm:atomic-quasi-jet}
Let $\alpha>0$, $0<p<1$, and $0<q\le \infty$, and set
\(
m:=\Bigl\lfloor \frac{2(1-p)}{p}\Bigr\rfloor \).
Then there exists
\(
\delta_0=\delta_0(\alpha,p,q)\in (0,\tfrac12]
\)
such that for every $\delta\in(0,\delta_0)$ and every $\delta$-pocket tiling
$\mathcal P_\delta$ of $\C$ with centers $\{\zeta_Q\}_{Q\in\mathcal P_\delta}$,
the following assertions hold, with
all implicit constants depending only on \(\alpha,p,q\), and \(\delta\).

\smallskip
\noindent
\emph{(i) Jet synthesis.}
The jet synthesis operator
\[
S_\delta^m c(z)
:=
\sum_{Q\in\mathcal P_\delta}\sum_{j=0}^m
c_{Q,j}\,\frac{[\alpha(z-\zeta_Q)]^j}{j!}\,\kappa_{\zeta_Q}(z),
\qquad
c=(c_{Q,j})_{Q\in\mathcal P_\delta,\ 0\le j\le m},
\]
is a bounded linear operator from $\ell_w^{p,q,m}$ to $\mathcal F_\alpha^{p,q}$, and
\[
\|S_\delta^m c\|_{\mathcal F_\alpha^{p,q}}
\lesssim \,\|c\|_{\ell_w^{p,q,m}},
\qquad c\in \ell_w^{p,q,m}.
\]

\smallskip
\noindent
\emph{(ii) Exact jet expansion.}
For every $f\in\mathcal F_\alpha^{p,q}$ there exists
$c=(c_{Q,j})_{Q\in\mathcal P_\delta,\ 0\le j\le m}\in \ell_w^{p,q,m}$ such that
\(
f=S_\delta^m c\),
and
\[
\|c\|_{\ell_w^{p,q,m}}
\lesssim \,\|f\|_{\mathcal F_\alpha^{p,q}}.
\]
\end{theorem}

\begin{proof}
Let $\delta_0=\delta_0(\alpha,p,q)\in(0,\tfrac12]$ be the constant furnished by
Lemma~\ref{lem:stability<1}, and fix $0<\delta<\delta_0$.
Let $\mathcal P_\delta$ be a $\delta$-pocket tiling of $\C$ with centers
$\{\zeta_Q\}_{Q\in\mathcal P_\delta}$.

\smallskip
\noindent
\emph{(i) Jet synthesis.}
The boundedness of the jet synthesis operator
\(
S_\delta^m 
\)
from $\ell_w^{p,q,m}$ to $\mathcal F_\alpha^{p,q}$ was established in
Lemma~\ref{lem:Sr-synth}. In particular,
\[
\|S_\delta^m c\|_{\mathcal F_\alpha^{p,q}}
\lesssim \,\|c\|_{\ell_w^{p,q,m}},
\qquad c\in \ell_w^{p,q,m}.
\]

\smallskip
\noindent
\emph{(ii) Exact jet expansion.}
By Lemma~\ref{lem:stability<1},
\[
\|R_\delta^m\|_{\mathcal F_\alpha^{p,q}\to \mathcal F_\alpha^{p,q}}
=
\|I-S_\delta^m C_\delta^m\|_{\mathcal F_\alpha^{p,q}\to \mathcal F_\alpha^{p,q}}
\le \frac12.
\]
The Neumann inversion principle in Lemma~\ref{lem:neumann-quasi} applies to the space $\mathcal F_\alpha^{p,q}$ with exponent $s=\min\{p,q\}$ in the present range $0<p<1$, hence it
 implies that
\(
I-R_\delta^m=S_\delta^m C_\delta^m
\)
is invertible on $\mathcal F_\alpha^{p,q}$ and
\[
\|(I-R_\delta^m)^{-1}\|_{\mathcal F_\alpha^{p,q}\to\mathcal F_\alpha^{p,q}}
\lesssim 1.
\]
Let $f\in \mathcal F_\alpha^{p,q}$ and define
\[
g:=(I-R_\delta^m)^{-1}f,
\qquad
c:=C_\delta^m g.
\]
Then, by Lemma~\ref{lem:analysisboundedness}, $c\in\ell_w^{p,q,m}$, and
\[
\|c\|_{\ell_w^{p,q,m}}
=
\|C_\delta^m g\|_{\ell_w^{p,q,m}}
\lesssim
\delta^{2-\frac{2}{p}}\|g\|_{\mathcal F_\alpha^{p,q}}
\lesssim
\delta^{2-\frac{2}{p}}\|f\|_{\mathcal F_\alpha^{p,q}}.
\]
Since \(\delta>0\) is fixed, this gives
\[
\|c\|_{\ell_w^{p,q,m}}
\lesssim
\, \|f\|_{\mathcal F_\alpha^{p,q}},
\]
with an implicit constant depending only on \(\alpha,p,q\), and \(\delta\).
Moreover,
\[
S_\delta^m c
=
S_\delta^m C_\delta^m g
=
(I-R_\delta^m)g
=
f.
\]
This proves the theorem.
\end{proof}

\subsubsection{Local jet elimination and floating kernels}
\label{subsubsec:atomic-jet-elimination}

We pass from the exact jet decomposition to normalized Fock kernels through
intermediate floating kernels, centered at finitely many translated points near
each pocket center.  These kernels eliminate the polynomial factors in the jet
atoms and are then re-centered at the distinguished pocket centers by
Section~\ref{sec:recentering-envelope}.

\smallskip
\noindent
Let \(\Lambda\subset\C\) be finite.  For each pocket \(Q\) with center
\(\zeta_Q\) and each \(u\in\Lambda\), define
\[
\kappa_{Q,u}(z)
:= e^{i\alpha \im\big({\overline \zeta_Q} u\big) }\,\kappa_{\zeta_Q+u}(z),
\qquad z\in\C.
\]
The phase factor is included for algebraic convenience; in particular,
$|\kappa_{Q,u}(z)|=|\kappa_{\zeta_Q+u}(z)|$ for all $z\in\C$.

\smallskip
\noindent
We also use the associated floating coefficient spaces.

\begin{definition}[Floating coefficient spaces and synthesis]
\label{def:floating-setup}
Let $\Lambda\subset\C$ be a finite set and let $\mathcal{P}_{\delta}$ be a $\delta$-pocket tiling of $\C$ with $\delta \in (0, \frac{1}{2}]$. For each $k\ge0$,  we write
\(
d_k=(d_{Q,u})_{Q\in\mathcal P_{\delta,k},\,u\in\Lambda}\) and equip $d_k$ with the block quasi-norm
\[
\|d_k\|_{\ell^p}
:=
\begin{cases}
\biggl(\displaystyle\sum_{Q\in\mathcal P_{\delta,k}}\sum_{u\in\Lambda}|d_{Q,u}|^p\biggr)^{\frac1p},
& 0<p<\infty,\\[2ex]
\displaystyle\sup_{Q\in\mathcal P_{\delta,k},\,u\in\Lambda}|d_{Q,u}|,
& p=\infty.
\end{cases}
\]
We then define $\ell_w^{p,q,\Lambda}$ to be the space of all sequences
\(
d=(d_{Q,u})_{Q\in\mathcal P_\delta,\,u\in\Lambda}
\)
such that
\[
\|d\|_{\ell_w^{p,q,\Lambda}}
:=
\begin{cases}
\biggl(\displaystyle\sum_{k\ge0}\bigl(w_k\|d_k\|_{\ell^p}\bigr)^q\biggr)^{\frac1q},
& 0<q<\infty,\\[2ex]
\displaystyle\sup_{k\ge0} w_k\|d_k\|_{\ell^p},
& q=\infty,
\end{cases}
\]
is finite, where
\(
w_k:=(1+k)^{\frac1q-\frac1p}\),
with the usual convention that $\frac{1}{\infty}=0$.

\smallskip
\noindent
For $d\in \ell_w^{p,q,\Lambda}$, we define the floating synthesis operator by
\[
(S_{\Lambda}^{\mathrm{flt}} d)(z)
:=\sum_{Q\in\mathcal P_\delta}\sum_{u\in\Lambda} d_{Q,u}\,\kappa_{Q,u}(z),
\qquad z\in\C,
\]
whenever the series converges.
\end{definition}

\smallskip
\noindent
\begin{remark} 
For \(0<p,q\le \infty\), the sequence spaces
\(\ell^{p,q}, \ell_w^{p,q}, \ell_w^{p,q,m}\), and \( \ell_w^{p,q,\Lambda}\)
are complete quasi-Banach spaces. Their quasi-norms satisfy the usual power-triangle inequality with exponent $\min\{1,p,q\}$.
\end{remark}

\medskip
\noindent
The next lemma gives the synthesis bound for floating kernels.  We then choose
a finite patch \(\Lambda\) to replace each jet atom by a finite combination of
floating kernels, up to a small error in \(\mathcal F_\alpha^{p,q}\).

\begin{lemma}
\label{lem:floating-synthesis}
Let $\alpha>0$, $0<p < 1$, $0<q\le\infty$, and let $\Lambda\subset\C$ be finite and let $\mathcal{P}_{\delta}$ be a $\delta$-pocket tiling of $\C$ with $\delta \in (0, \frac{1}{2}]$.
Then the floating synthesis operator
\(
S_{\Lambda}^{\mathrm{flt}}\) 
defined in Definition~\ref{def:floating-setup} is a bounded linear operator from \(\ell_w^{p,q,\Lambda}\) to \(\mathcal F_{\alpha}^{p,q}
\). More precisely,
\[
\|S_{\Lambda}^{\mathrm{flt}} d\|_{\mathcal F_{\alpha}^{p,q}}
\lesssim \,\|d\|_{\ell_w^{p,q,\Lambda}},
\qquad d\in \ell_w^{p,q,\Lambda},
\]
where the implicit constant depends only on \(\alpha,p,q, \Lambda\), and is
independent of \(\delta\).
\end{lemma}
\begin{proof}
Let
\(
R_\Lambda:=\max_{u\in\Lambda}|u|\).
We show that the series defining $S_{\Lambda}^{\mathrm{flt}}d$ converges absolutely and
locally uniformly on $\C$. Fix $R>0$ and choose an integer
\(
N\ge R+R_\Lambda+1\).

\smallskip
\noindent
If $Q\in\mathcal P_{\delta,k}$ with $k\ge N$, $u\in\Lambda$, and $|z|\le R$, then
$|\zeta_Q|\in[k,k+1)$, and hence
\[
|\kappa_{Q,u}(z)|
=
|\kappa_{\zeta_Q+u}(z)|
=
e^{\frac{\alpha}{2}|z|^2}e^{-\frac{\alpha}{2}|z-(\zeta_Q+u)|^2}
\lesssim
e^{-\frac{\alpha}{2}(k-R_\Lambda-R)^2}
\lesssim
e^{-\frac{\alpha}{4}(k-R)^2},
\]
where the implicit constant depends only on $\alpha$, $\Lambda$, and $R$. Therefore,
\[
\sum_{Q\in\mathcal P_{\delta,k}}\sum_{u\in\Lambda}
|d_{Q,u}|\,|\kappa_{Q,u}(z)|
\lesssim
e^{-\frac{\alpha}{4}(k-R)^2}
\sum_{Q\in\mathcal P_{\delta,k}}\sum_{u\in\Lambda}|d_{Q,u}|.
\]
Arguing exactly as in Lemma~\ref{lem:jet-series-entire}, with the finite index set
$\Lambda$ in place of $\{0,\dots,m\}$, it follows that the series defining
$S_{\Lambda}^{\mathrm{flt}}d$ converges absolutely and locally uniformly on $\C$, and hence
defines an entire function.

\smallskip
\noindent
Now fix $Q\in\mathcal P_{\delta,k}$ and $u\in\Lambda$. Since $|\zeta_Q|\in[k,k+1)$, we have
\begin{equation}\label{eq:floating-shift-radius-rev}
\bigl||\zeta_Q+u|-k\bigr|\le 1+R_\Lambda .
\end{equation}
By the single-kernel estimate used in the proof of Lemma~\ref{lem:Sr-synth},
\[
B_p^p(\kappa_{\zeta_Q+u};\ell)
\lesssim
(1+|\zeta_Q+u|)^{-1}e^{-\frac{\alpha p}{8}(\ell-|\zeta_Q+u|)^2},
\qquad \ell\ge0,
\]
with implicit constant depending only on $\alpha$ and $p$. Since $\Lambda$ is finite,
\eqref{eq:floating-shift-radius-rev} implies
\(
1+|\zeta_Q+u|\asymp 1+k
\)
with constants depending only on $\Lambda$. Using also
\(
|\kappa_{Q,u}|=|\kappa_{\zeta_Q+u}|,
\)
we obtain
\begin{equation}\label{eq:floating-kernel-profile-rev}
B_p^p(\kappa_{Q,u};\ell)
\le
C_0(1+k)^{-1}e^{-\frac{\alpha p}{16}(\ell-k)^2},
\qquad \ell\ge0,
\end{equation}
where $C_0$ depends only on $\alpha$, $p$, and $\Lambda$.

\smallskip
\noindent
For each $k\ge0$, define
\[
F_k(z):=\sum_{Q\in\mathcal P_{\delta,k}}\sum_{u\in\Lambda} d_{Q,u}\,\kappa_{Q,u}(z).
\]
Since $0<p<1$, the $p$-subadditivity of $x\mapsto x^p$ yields
\[
M_p^p(F_k,r)
\le
\sum_{Q\in\mathcal P_{\delta,k}}\sum_{u\in\Lambda}
|d_{Q,u}|^p\,M_p^p(\kappa_{Q,u},r).
\]
Taking the supremum over $r\in[\ell,\ell+1)$ and applying
\eqref{eq:floating-kernel-profile-rev}, we obtain
\begin{equation}\label{eq:floating-block-estimate-rev}
B_p^p(F_k;\ell)
\le
C_0(1+k)^{-1}e^{-\frac{\alpha p}{16}(\ell-k)^2}\,\|d_k\|_{\ell^p}^p.
\end{equation}
Let
\(
f:=S_{\Lambda}^{\mathrm{flt}}d.
\)
Since \(f=\sum_{k\ge0}F_k\), another application of \(p\)-subadditivity gives
\[
M_p^p(f,r)\le \sum_{k\ge0}M_p^p(F_k,r),
\qquad\text{hence}\qquad
B_p^p(f;\ell)\le \sum_{k\ge0}B_p^p(F_k;\ell).
\]
Combining this with \eqref{eq:floating-block-estimate-rev}, we get
\begin{equation}\label{eq:floating-convolution-estimate-rev}
B_p^p(f;\ell)
\le
C_0\sum_{k\ge0}(1+k)^{-1}e^{-\frac{\alpha p}{16}(\ell-k)^2}\,\|d_k\|_{\ell^p}^p.
\end{equation}
Applying the same shellwise-to-mixed-norm argument as in Lemma~\ref{lem:Sr-synth}, the
discrete Gaussian convolution in \eqref{eq:floating-convolution-estimate-rev} yields
\[
\|f\|_{\mathcal F_\alpha^{p,q}}
\le
C\,\|d\|_{\ell_w^{p,q,\Lambda}},
\]
where \(C\) depends only on \(\alpha\), \(p\), \(q\), and \(\Lambda\), and is independent of
\(\delta\in(0,\tfrac12]\). This proves the lemma.
\end{proof}

\begin{lemma}[Local jet elimination]
\label{lem:local-jet-elimination}
Let $\alpha>0$, $0<p<1$, and $0<q\le\infty$, and set
\(
m:=\Bigl\lfloor \frac{2(1-p)}{p}\Bigr\rfloor\).
Let $\mathcal P_\delta$ be a $\delta$-pocket tiling of $\C$ with
$\delta\in(0,\tfrac12]$. Given $\varepsilon\in(0,1)$, there exist a number
$\tau \in(0,1)$, a finite set
\(
\Lambda\subset \{u\in\C:\ |u|< \tau\}\),
and coefficients
\(
\{a_{j,u}\}_{0\le j\le m,\ u\in\Lambda}\),
independent of $Q$, such that the following hold.

\smallskip
\noindent
\emph{(i) Local elimination profile.}
For every $k\ge0$, every $Q\in\mathcal P_{\delta,k}$, and every $0\le j\le m$, define
\[
\Delta_{Q,j}
:=
\psi_{Q,j}-\sum_{u\in\Lambda} a_{j,u}\,\kappa_{Q,u}.
\]
Then
\begin{equation}\label{eq:local-jet-elimination-profile}
B_p(\Delta_{Q,j};\ell)
\le
\varepsilon\,(1+k)^{-\frac{1}{p}}e^{-\frac{\alpha}{32}(\ell-k)^2},
\qquad \ell\ge0.
\end{equation}

\smallskip
\noindent
\emph{(ii) Bounded local elimination map.}
The operator $E_{\mathrm{loc}}$ defined by
\[
(E_{\mathrm{loc}}c)_{Q,u}:=\sum_{j=0}^m a_{j,u}\,c_{Q,j},
\qquad Q\in\mathcal P_\delta,\ u\in\Lambda,
\]
is a bounded linear operator from $\ell_w^{p,q,m}$ to $\ell_w^{p,q,\Lambda}$, with
operator norm independent of $\delta$.

\smallskip
\noindent
\emph{(iii) Small defect estimate.}
Moreover, there exists a constant
\(C_E=C_E(\alpha,p,q)>0\), independent of \(\varepsilon\), \(\delta\), and
\(c\), such that
\begin{equation}\label{eq:local-jet-elimination-defect}
\|S_{\Lambda}^{\mathrm{flt}}E_{\mathrm{loc}}c-S_\delta^m c\|_{\mathcal F_\alpha^{p,q}}
\le
C_E \,\varepsilon\,\|c\|_{\ell_w^{p,q,m}},
\qquad c\in \ell_w^{p,q,m}.
\end{equation}
\end{lemma}

\begin{proof}
Set $N:=m+1$. Fix distinct points $v_1,\dots,v_N\in\C$ with $|v_h|<1$, and let
\(
u_h:=\tau  v_h\) for \( 1\le h \le N\),
where $\tau \in(0,1)$ will be chosen later. Define
\[
\Lambda:=\{u_1,\dots,u_N\}.
\]
Consider the matrix
\[
\mathsf M_{s,h}
:=
e^{-\frac{\alpha}{2}|u_h|^2}\,\overline{u_h}^{\,s},
\qquad 0\le s\le m,\ 1\le h \le N.
\]
Then
\[
\mathsf M = D(\tau)\,V\,W(\tau),
\]
where
\[
D(\tau)=\mathrm{diag}(\tau^s)_{s=0}^m,\qquad
V_{s,h}=\overline{v_h}^{\,s},\qquad
W(\tau)=\mathrm{diag}\!\Bigl(e^{-\frac{\alpha}{2}\tau^2|v_h|^2}\Bigr)_{h=1}^N.
\]
Here $V$ is a Vandermonde matrix in the distinct nodes
$\overline{v_1},\dots,\overline{v_N}$, and is therefore invertible. Since
$D(\tau)$ and $W(\tau)$ are also invertible, it follows that $\mathsf M$ is invertible, with
\begin{equation}\label{eq:local-M-inverse}
\mathsf M^{-1}=W(\tau)^{-1}V^{-1}D(\tau)^{-1}.
\end{equation}
For each $0\le j\le m$, let
\[
\mathsf e^{(j)}:=(\delta_{s,j})_{s=0}^m,
\qquad
\mathsf a^{(j)}:=(a_{j,u_1},\dots,a_{j,u_N})^T,
\]
and define $\mathsf a^{(j)}$ as the unique solution of the equation 
\(
\mathsf M\,\mathsf a^{(j)}=\mathsf e^{(j)}\).
Equivalently,
\begin{equation}\label{eq:local-moment-identity}
\sum_{u\in\Lambda} a_{j,u}\,e^{-\frac{\alpha}{2}|u|^2}\,\overline{u}^{\,s}
=
\delta_{s,j},
\qquad 0\le s\le m.
\end{equation}

\smallskip
\noindent
For (i), use the identity
\[
e^{i\alpha\im(\overline{\zeta_Q}u)}\,\kappa_{\zeta_Q+u}(z)
=
\kappa_{\zeta_Q}(z)\,e^{\alpha(z-\zeta_Q)\overline u-\frac{\alpha}{2}|u|^2},
\]
we obtain
\begin{align*}
\sum_{u\in\Lambda} a_{j,u}\,\kappa_{Q,u}(z)
&=
\kappa_{\zeta_Q}(z)\sum_{u\in\Lambda}
a_{j,u}\,e^{-\frac{\alpha}{2}|u|^2}e^{\alpha(z-\zeta_Q)\overline u} \\
&=
\kappa_{\zeta_Q}(z)\sum_{s\ge0}\frac{\alpha^s}{s!}(z-\zeta_Q)^s
\Bigl(\sum_{u\in\Lambda} a_{j,u}\,e^{-\frac{\alpha}{2}|u|^2}\,\overline u^{\,s}\Bigr).
\end{align*}
By \eqref{eq:local-moment-identity}, all terms with $0\le s\le m$ vanish except $s=j$, and therefore
\[
\sum_{u\in\Lambda} a_{j,u}\,\kappa_{Q,u}
=
\psi_{Q,j}+\kappa_{\zeta_Q}R_{Q,j},
\]
where
\[
R_{Q,j}(z)
=
\sum_{s\ge m+1}\frac{\alpha^s}{s!}(z-\zeta_Q)^s\,H_{j,s},
\qquad
H_{j,s}:=
\sum_{u\in\Lambda} a_{j,u}\,e^{-\frac{\alpha}{2}|u|^2}\,\overline u^{\,s}.
\]
Hence
\[
\Delta_{Q,j}=-\kappa_{\zeta_Q}R_{Q,j}.
\]
To bound the coefficients $H_{j,s}$, use the factorization
\eqref{eq:local-M-inverse}, we have
\[
\mathsf a^{(j)}
=
W(\tau)^{-1}V^{-1}D(\tau)^{-1}\mathsf e^{(j)}
=
\tau^{-j}W(\tau)^{-1}V^{-1}\mathsf e^{(j)},
\]
since
\(
D(\tau)^{-1}\mathsf e^{(j)}=\tau^{-j}\mathsf e^{(j)}\).
Now $V^{-1}$ is fixed once the nodes $v_1,\dots,v_N$ are chosen, and
\[
W(\tau)^{-1}
=
\mathrm{diag} \!\Bigl(e^{\frac{\alpha}{2}\tau^2|v_h|^2}\Bigr)_{h=1}^N
\]
is uniformly bounded for $0< \tau <1$, since $|v_h|<1$ for all $h$. Hence
\[
\|W(\tau)^{-1}V^{-1}\mathsf e^{(j)}\|_{\infty}\le C,
\]
where $C$ depends only on $\alpha$ and the fixed nodes $v_1,\dots,v_N$. Therefore each component of
$\mathsf a^{(j)}$ satisfies
\[
|a_{j,u}|\le C\, \tau^{-j},
\qquad 0\le j\le m,\ u\in\Lambda.
\]
Using this, for $s\ge m+1$, we obtain
\[
|H_{j,s}|
\le
C\sum_{u\in\Lambda} \tau^{-j}|u|^s
\le
C\, \tau^{\,s-j},
\]
after enlarging $C$. Since $s-j\ge1$, it follows that
\[
|R_{Q,j}(z)|
\le
C\sum_{s\ge m+1}\frac{\alpha^s|z-\zeta_Q|^s}{s!}\, \tau^{s-j}
\le
C\, \tau \,|z-\zeta_Q|^{m+1}e^{\alpha \tau |z-\zeta_Q|}.
\]
Consequently,
\[
|\Delta_{Q,j}(z)|e^{-\frac{\alpha}{2}|z|^2}
\le
C\, \tau \,|z-\zeta_Q|^{m+1}e^{\alpha \tau |z-\zeta_Q|}
e^{-\frac{\alpha}{2}|z-\zeta_Q|^2}.
\]
Using the bounds
\[
t^{m+1}e^{-\frac{\alpha}{2}t^2}\lesssim e^{-\frac{\alpha}{4}t^2},
\qquad
e^{\alpha \tau t}e^{-\frac{\alpha}{4}t^2}\lesssim e^{-\frac{\alpha}{8}t^2}
\qquad (0<\tau<1),
\]
we get
\begin{equation}\label{eq:local-difference-envelope}
|\Delta_{Q,j}(z)|e^{-\frac{\alpha}{2}|z|^2}
\lesssim
 \tau \,e^{-\frac{\alpha}{8}|z-\zeta_Q|^2},
\end{equation}
where the implicit constant depends only on $\alpha$, $m$, and the fixed nodes
$v_1,\dots,v_N$.

\smallskip
\noindent
Now let $k\ge0$, $Q\in\mathcal P_{\delta,k}$, and $0\le j\le m$. From
\eqref{eq:local-difference-envelope},
\[
M_p^p(\Delta_{Q,j},r)e^{-\frac{\alpha p}{2}r^2}
\lesssim
\tau^p\int_0^{2\pi}
e^{-\frac{\alpha p}{8}|re^{i\theta}-\zeta_Q|^2}\,\frac{d\theta}{2\pi}.
\]
Using the same Gaussian circle-mean estimate as in the proof of
Lemma~\ref{lem:Sr-synth}, and the fact that $|\zeta_Q|\in[k,k+1)$, we obtain
\[
B_p^p(\Delta_{Q,j};\ell)
\lesssim
\tau^p(1+k)^{-1}e^{-\frac{\alpha p}{32}(\ell-k)^2},
\qquad \ell\ge0.
\]
Taking the $p$th root, we get
\[
B_p(\Delta_{Q,j};\ell)
\le
C \tau \,(1+k)^{-\frac1p}e^{-\frac{\alpha}{32}(\ell-k)^2},
\]
where $C$ depends only on $\alpha$, $p$, and the fixed nodes
$v_1,\dots,v_N$. Choosing $\tau =\tau(\varepsilon)$ so that $C \tau \le\varepsilon$, we obtain
\eqref{eq:local-jet-elimination-profile}.

\smallskip
\noindent
For (ii), since $0<p<1$, for each $Q\in\mathcal P_\delta$,
\[
\sum_{u\in\Lambda}\Bigl|\sum_{j=0}^m a_{j,u}c_{Q,j}\Bigr|^p
\le
\sum_{u\in\Lambda}\sum_{j=0}^m |a_{j,u}|^p|c_{Q,j}|^p
\le
A^p\sum_{j=0}^m |c_{Q,j}|^p,
\]
where
\[
A:=\max_{0\le j\le m}
\Bigl(\sum_{u\in\Lambda}|a_{j,u}|^p\Bigr)^{\frac1p}<\infty.
\]
Summing over $Q\in\mathcal P_{\delta,k}$ yields
\[
\|(E_{\rm loc}c)_k\|_{\ell^p}\le A\,\|c_k\|_{\ell^p},
\]
and therefore
\[
\|E_{\rm loc}c\|_{\ell_w^{p,q,\Lambda}}
\le
A\,\|c\|_{\ell_w^{p,q,m}}.
\]
Thus $E_{\rm loc}$ is bounded, with operator norm independent of $\delta$.

\smallskip
\noindent
For (iii), let $c\in\ell_w^{p,q,m}$. Then
\[
S_\delta^m c-S_{\Lambda}^{\mathrm{flt}}E_{\rm loc}c
=
\sum_{Q\in\mathcal P_\delta}\sum_{j=0}^m c_{Q,j}\,\Delta_{Q,j}.
\]
For each $k\ge0$, set
\[
F_k:=\sum_{Q\in\mathcal P_{\delta,k}}\sum_{j=0}^m c_{Q,j}\,\Delta_{Q,j}.
\]
Since $0<p<1$, $p$-subadditivity and \eqref{eq:local-jet-elimination-profile} give
\[
B_p^p(F_k;\ell)
\le
\varepsilon^p(1+k)^{-1}e^{-\frac{\alpha p}{32}(\ell-k)^2}\,\|c_k\|_{\ell^p}^p.
\]
Summing in $k$ and using $p$-subadditivity once more, we obtain
\[
B_p^p\Bigl(\sum_{k\ge0}F_k;\ell\Bigr)
\le
\varepsilon^p
\sum_{k\ge0}(1+k)^{-1}e^{-\frac{\alpha p}{32}(\ell-k)^2}\,\|c_k\|_{\ell^p}^p.
\]
Now applying the same shellwise-to-mixed-norm argument as in the proof of
Lemma~\ref{lem:Sr-synth}, we obtain
\[
\Bigl\|\sum_{k\ge0}F_k\Bigr\|_{\mathcal F_\alpha^{p,q}}
\le
C_E\,\varepsilon\,\|c\|_{\ell_w^{p,q,m}},
\]
where $C_E$ depends only on $\alpha$, $p$, and $q$. Since
\[
\sum_{k\ge0}F_k=S_\delta^m c-S_{\Lambda}^{\mathrm{flt}}E_{\rm loc}c,
\]
this proves \eqref{eq:local-jet-elimination-defect}.
\end{proof}

\subsubsection{Re-centering floating kernels}
\label{subsubsec:atomic-recentering}

We convert the floating kernels from Definition~\ref{def:floating-setup} into
kernels centered at the distinguished pocket centers, using the envelope
theorem from Section~\ref{sec:recentering-envelope}.

\smallskip
\noindent
By the uniform ball control in Lemma~\ref{lem:geo-pockets} and the
bi-Lipschitz reindexing in Lemma~\ref{lem:biLip}, the \(\delta\)-pocket
tiling satisfies the hypotheses of Theorem~\ref{thm:envelope-recentering}.

\begin{proposition} 
\label{prop:recentering-floating}
Let $\alpha,\gamma>0$, $0<p,\beta<1$, and $0<q\le\infty$.
Let $\Lambda\subset\C$ be finite, and let $\delta_0$ be the constant from
Theorem~\ref{thm:envelope-recentering}. Let $\mathcal P_\delta$ be a
$\delta$-pocket tiling of $\C$ with $\delta\in(0,\delta_0)$.
For each $\eta\in\C$, choose a coefficient sequence
\[
b(\eta)=\bigl(b_{Q'}(\eta)\bigr)_{Q'\in\mathcal P_\delta}
\]
as in Theorem~\ref{thm:envelope-recentering}, so that
\[
\kappa_\eta=\sum_{Q'\in\mathcal P_\delta} b_{Q'}(\eta)\,\kappa_{\zeta_{Q'}}
\quad\text{in }\mathcal F_\alpha^2,
\qquad\text{and}\qquad
|b_{Q'}(\eta)|\le C_{\mathrm{env}} \,e^{-\gamma|\zeta_{Q'}-\eta|^\beta},
\quad Q'\in\mathcal P_\delta,\ \eta\in\C,
\]
where $C_{\mathrm{env}}$ depends only on $\alpha,\beta,\gamma,\delta$ and the geometric constants
of $\mathcal P_\delta$.

\smallskip
\noindent
Define the linear map
\begin{equation}\label{eq:recentering-map}
E_{\rm rec}:\ell_w^{p,q,\Lambda}\to \ell_w^{p,q},
\qquad
(E_{\rm rec}d)_{Q'}
:=
\sum_{Q\in\mathcal P_\delta}\sum_{u\in\Lambda}
d_{Q,u}\,e^{i\alpha\im(\overline{\zeta_Q}u)}\,b_{Q'}(\zeta_Q+u),
\quad Q'\in\mathcal P_\delta.
\end{equation}
Then the following hold:
\begin{enumerate}
\item[\textup{(i)}]
The map $E_{\rm rec}$ is bounded from $\ell_w^{p,q,\Lambda}$ to $\ell_w^{p,q}$. More precisely, there exists
a constant
\[
C_{\mathrm{rec}}
=
C_{\mathrm{rec}}(\alpha,p,q,\beta,\gamma,\Lambda,\delta)>0
\]
such that
\[
\|E_{\mathrm{rec}}d\|_{\ell_w^{p,q}}
\le
C_{\mathrm{rec}}\|d\|_{\ell_w^{p,q,\Lambda}},
\qquad d\in\ell_w^{p,q,\Lambda}.
\]

\item[\textup{(ii)}]
For every $d\in\ell_w^{p,q,\Lambda}$,
\begin{equation}\label{eq:recentering-identity}
S_\Lambda^{\mathrm{flt}}d
=
S_\delta^{\ker}(E_{\rm rec}d)
\end{equation}
as an identity of entire functions.
\end{enumerate}
\end{proposition}

\begin{proof}
Throughout the proof, \(C\) denotes a positive constant which may change from
line to line and depends only on
\(\alpha,p,q,\beta,\gamma,\Lambda,\delta\), and the geometric constants of
\(\mathcal P_\delta\).

\medskip
\noindent
\emph{Step 1: A shellwise $\ell^p$ envelope.}
For $\eta\in\C$ and $\ell\ge0$, Theorem~\ref{thm:envelope-recentering} gives
\[
\mathscr{B}_\ell(\eta)^p
:=
\sum_{Q'\in\mathcal P_{\delta,\ell}} |b_{Q'}(\eta)|^p
\le
\, C_{\mathrm{env}}^p\sum_{Q'\in\mathcal P_{\delta,\ell}} e^{-\gamma p|\zeta_{Q'}-\eta|^\beta}.
\]
Let
\(
\rho_\ell(\eta):=\inf\{|z-\eta|:\ z\in A_\ell\}\).
Since $A_\ell=\{z\in\C:\ell\le |z|<\ell+1\}$, we have
\[
\rho_\ell(\eta)\ge \bigl(|\ell-|\eta||-1\bigr)_+,
\qquad (t)_+:=\max\{t,0\}.
\]
For each integer $n\ge0$, define
\[
E_n(\eta,\ell)
:=
\{Q'\in\mathcal P_{\delta,\ell}:\ n\le |\zeta_{Q'}-\eta|<n+1\}.
\]
Then $E_n(\eta,\ell)=\emptyset$ whenever $n+1<\rho_\ell(\eta)$, and therefore
\[
\mathscr{B}_\ell(\eta)^p
\le \, C_{\mathrm{env}}^p\sum_{n\ge (\rho_\ell(\eta)-1)_+}
\#E_n(\eta,\ell)\,e^{-\gamma p n^\beta}.
\]
By Lemma~\ref{lem:geo-pockets}(d),
\(
\#E_n(\eta,\ell)\lesssim \delta^{-2}(n+1)^2\),
hence
\[
\mathscr{B}_\ell(\eta)^p
\le
C\sum_{n\ge (\rho_\ell(\eta)-1)_+}(n+1)^2 e^{-\gamma p n^\beta}.
\]
Since polynomial factors are dominated by stretched exponentials,
\[
\sum_{n\ge N}(n+1)^2 e^{-\gamma p n^\beta}
\lesssim
e^{-\frac{\gamma p}{2}N^\beta},
\qquad N\ge0.
\]
It follows that
\[
\mathscr{B}_\ell(\eta)^p
\le
C\,e^{-\frac{\gamma p}{2}(\rho_\ell(\eta)-1)_+^\beta}
\le
C\,e^{-\frac{\gamma p}{4}|\ell-|\eta||^\beta}.
\]
Thus
\begin{equation}\label{eq:recentering-shell-envelope}
\sum_{Q'\in\mathcal P_{\delta,\ell}} |b_{Q'}(\eta)|^p
\le
C\,e^{-\frac{\gamma p}{4}|\ell-|\eta||^\beta},
\qquad \eta\in\C,\ \ell\ge0.
\end{equation}

\medskip
\noindent
\emph{Step 2: A blockwise $p$-estimate.}
Fix $\ell\ge0$.
Since $0<p<1$, the map $x\mapsto |x|^p$ is subadditive, so
\begin{align*}
\|(E_{\rm rec}d)_\ell\|_{\ell^p}^p
&=
\sum_{Q'\in\mathcal P_{\delta,\ell}}
\Bigl|
\sum_{Q\in\mathcal P_\delta}\sum_{u\in\Lambda}
d_{Q,u}\,e^{i\alpha\im(\overline{\zeta_Q}u)}\,b_{Q'}(\zeta_Q+u)
\Bigr|^p \\
&\le
\sum_{Q'\in\mathcal P_{\delta,\ell}}
\sum_{Q\in\mathcal P_\delta}\sum_{u\in\Lambda}
|d_{Q,u}|^p\,|b_{Q'}(\zeta_Q+u)|^p \\
&=
\sum_{Q\in\mathcal P_\delta}\sum_{u\in\Lambda}
|d_{Q,u}|^p
\sum_{Q'\in\mathcal P_{\delta,\ell}} |b_{Q'}(\zeta_Q+u)|^p.
\end{align*}
If $Q\in\mathcal P_{\delta,k}$, then $|\zeta_Q|\in[k,k+1)$, and hence
\[
\bigl||\zeta_Q+u|-k\bigr|\le 1+R_\Lambda,
\qquad
R_\Lambda:=\max_{u\in\Lambda}|u|.
\]
Applying \eqref{eq:recentering-shell-envelope} with $\eta=\zeta_Q+u$, we obtain
\[
\sum_{Q'\in\mathcal P_{\delta,\ell}} |b_{Q'}(\zeta_Q+u)|^p
\le
C\,e^{-\frac{\gamma p}{4}|\ell-|\zeta_Q+u||^\beta}
\le
C\,e^{-\frac{\gamma p}{8}|\ell-k|^\beta}.
\]
Grouping the sum over $Q$ by annular blocks gives
\begin{equation}\label{eq:recentering-block-estimate}
\|(E_{\rm rec}d)_\ell\|_{\ell^p}^p
\le
C\sum_{k\ge0} e^{-\frac{\gamma p}{8}|\ell-k|^\beta}\,\|d_k\|_{\ell^p}^p,
\qquad \ell\ge0.
\end{equation}

\medskip
\noindent
\emph{Step 3: Outer $\ell^q$ boundedness.}
We pass from \eqref{eq:recentering-block-estimate} to the mixed sequence norm.

\smallskip
\noindent
\emph{Case $0<q\le p$.}
Let $s:=\frac{q}{p}\le1$. Raising \eqref{eq:recentering-block-estimate} to the power $s$
and using subadditivity gives
\[
\|(E_{\rm rec}d)_\ell\|_{\ell^p}^q
\le
C\sum_{k\ge0} e^{-\frac{\gamma q}{8}|\ell-k|^\beta}\,\|d_k\|_{\ell^p}^q.
\]
Multiplying by $w_\ell^q=(1+\ell)^{1-\frac{q}{p}}$ and summing in $\ell$, we obtain
\[
\sum_{\ell\ge0}\Bigl(w_\ell\|(E_{\rm rec}d)_\ell\|_{\ell^p}\Bigr)^q
\le
C\sum_{\ell\ge0}\sum_{k\ge0}
(1+\ell)^{1-\frac{q}{p}} e^{-\frac{\gamma q}{8}|\ell-k|^\beta}\,\|d_k\|_{\ell^p}^q.
\]
By the weighted Peetre inequality in Lemma~\ref{lem:peetre-weighted},
\[
(1+\ell)^{1-\frac{q}{p}}
\lesssim
(1+k)^{1-\frac{q}{p}}(1+|\ell-k|)^{1-\frac{q}{p}},
\]
and since, by the weighted exponential sums in Lemma~\ref{lem:exp-rowsum},
\[
\sum_{j\in\mathbb Z}(1+|j|)^M e^{-\frac{\gamma q}{8}|j|^\beta}<\infty
\qquad (M\ge0),
\]
it follows that
\[
\sum_{\ell\ge0}\Bigl(w_\ell\|(E_{\rm rec}d)_\ell\|_{\ell^p}\Bigr)^q
\le
C\sum_{k\ge0}(1+k)^{1-\frac{q}{p}}\|d_k\|_{\ell^p}^q
=
C\sum_{k\ge0}\bigl(w_k\|d_k\|_{\ell^p}\bigr)^q.
\]
That is,
\[
\|E_{\rm rec}d\|_{\ell_w^{p,q}}
\le
C\|d\|_{\ell_w^{p,q,\Lambda}}.
\]

\smallskip
\noindent
\emph{Case $p<q<\infty$.}
Let $s:=\frac{q}{p}>1$.
Since
\(
\sum_{k\ge0} e^{-\frac{\gamma p}{8}|\ell-k|^\beta}\lesssim 1
\)
uniformly in $\ell$, Jensen's inequality yields
\[
\Bigl(\sum_{k\ge0} e^{-\frac{\gamma p}{8}|\ell-k|^\beta}\|d_k\|_{\ell^p}^p\Bigr)^s
\lesssim
\sum_{k\ge0} e^{-\frac{\gamma p}{8}|\ell-k|^\beta}\|d_k\|_{\ell^p}^{ps}
=
\sum_{k\ge0} e^{-\frac{\gamma p}{8}|\ell-k|^\beta}\|d_k\|_{\ell^p}^{q}.
\]
Using this in \eqref{eq:recentering-block-estimate}, multiplying by
$w_\ell^q=(1+\ell)^{1-\frac{q}{p}}$, and arguing exactly as above with the same weighted
summation argument in Lemma~\ref{lem:peetre-weighted} and Lemma~\ref{lem:exp-rowsum},
we obtain
\[
\sum_{\ell\ge0}\Bigl(w_\ell\|(E_{\rm rec}d)_\ell\|_{\ell^p}\Bigr)^q
\le
C\sum_{k\ge0}\bigl(w_k\|d_k\|_{\ell^p}\bigr)^q,
\qquad\text{i.e.,}\qquad
\|E_{\rm rec}d\|_{\ell_w^{p,q}}
\le
C\|d\|_{\ell_w^{p,q,\Lambda}}.
\]

\smallskip
\noindent
\emph{Case $q=\infty$.}
Now $w_\ell^p=(1+\ell)^{-1}$, so \eqref{eq:recentering-block-estimate} gives
\[
\Bigl(w_\ell\|(E_{\rm rec}d)_\ell\|_{\ell^p}\Bigr)^p
\le
C\sum_{k\ge0}\frac{1+k}{1+\ell}e^{-\frac{\gamma p}{8}|\ell-k|^\beta}
\bigl(w_k\|d_k\|_{\ell^p}\bigr)^p.
\]
Applying the same weighted summation argument as above in Lemma~\ref{lem:peetre-weighted} and Lemma~\ref{lem:exp-rowsum}, we get
\[
w_\ell\|(E_{\rm rec}d)_\ell\|_{\ell^p}
\le
C\sup_{k\ge0} w_k\|d_k\|_{\ell^p}.
\]
Taking the supremum over $\ell$ yields
\[
\|E_{\rm rec}d\|_{\ell_w^{p,\infty}}
\le
C\|d\|_{\ell_w^{p,\infty,\Lambda}}.
\]
Combining the three cases, we have proved that
\[
\|E_{\mathrm{rec}}d\|_{\ell_w^{p,q}}
\le
C_{\mathrm{rec}}\|d\|_{\ell_w^{p,q,\Lambda}},
\qquad d\in\ell_w^{p,q,\Lambda},
\]
where
\[
C_{\mathrm{rec}}
=
C_{\mathrm{rec}}(\alpha,p,q,\beta,\gamma,\Lambda,\delta)>0.
\]
This proves the boundedness of \(E_{\mathrm{rec}}\).

\medskip
\noindent
\emph{Step 4: Exact synthesis identity.}
By Lemma~\ref{lem:floating-synthesis}, the series defining
$S_\Lambda^{\mathrm{flt}}d$ converges absolutely and locally uniformly on~$\C$.
On the other hand, Step~3 shows that $E_{\rm rec}d\in \ell_w^{p,q}$, so by
Corollary~\ref{cor:synthesis<1}, the series defining
$S_\delta^{\ker}(E_{\rm rec}d)$ also converges absolutely and locally uniformly on~$\C$.
It therefore suffices to justify the rearrangement leading to
\eqref{eq:recentering-identity} on compact subsets of~$\C$.

\smallskip
\noindent
Fix $R>0$ and set
\[
K_R:=\{z\in\C:\ |z|\le R\}.
\]
In this step, the constant \(C\) may also depend on \(R\).
For each $(Q,u)$, Theorem~\ref{thm:envelope-recentering} gives
\[
\kappa_{\zeta_Q+u}
=
\sum_{Q'\in\mathcal P_\delta} b_{Q'}(\zeta_Q+u)\,\kappa_{\zeta_{Q'}}
\quad\text{in }\mathcal F_\alpha^2,
\]
hence also locally uniformly on~$\C$.
Multiplying by
\(
d_{Q,u}e^{i\alpha\im(\overline{\zeta_Q}u)}
\)
and summing formally in $(Q,u)$ gives
\begin{equation}\label{eq:recentering-formal}
S_\Lambda^{\mathrm{flt}}d(z)
=
\sum_{Q\in\mathcal P_\delta}\sum_{u\in\Lambda}
d_{Q,u}e^{i\alpha\im(\overline{\zeta_Q}u)}
\sum_{Q'\in\mathcal P_\delta}
b_{Q'}(\zeta_Q+u)\,\kappa_{\zeta_{Q'}}(z).
\end{equation}
We show that the triple series in \eqref{eq:recentering-formal}
converges absolutely and uniformly on $K_R$.

\smallskip
\noindent
If $Q'\in\mathcal P_{\delta, \ell}$ and $z\in K_R$, then exactly as in
Lemma~\ref{lem:kernel-series-entire} and Step~1 of Lemma~\ref{lem:floating-synthesis},
\[
|\kappa_{\zeta_{Q'}}(z)|
=
e^{\frac{\alpha}{2}|z|^2}e^{-\frac{\alpha}{2}|z-\zeta_{Q'}|^2}
\le
C e^{-\frac{\alpha}{4}(\ell-R)^2}
\]
for a constant $C$ depending only on $\alpha$ and $R$.
Also, if  \(Q\in\mathcal P_{\delta,k}\),
and \(u\in\Lambda\), then
\[
|\zeta_{Q'}-(\zeta_Q+u)|
\ge
\bigl||\zeta_{Q'}|-|\zeta_Q+u|\bigr|
\ge
|\ell-k|-(2+R_\Lambda).
\]
Hence, by the envelope estimate in Theorem~\ref{thm:envelope-recentering},
\[
|b_{Q'}(\zeta_Q+u)|
\le
C_{\mathrm{env}}
e^{-\gamma|\zeta_{Q'}-(\zeta_Q+u)|^\beta}
\le
C e^{-\frac{\gamma}{2}|\ell-k|^\beta},
\]
where \(C\) may depend on \(\gamma,\beta,\Lambda,\delta\), and the geometric
constants of \(\mathcal P_\delta\).
Moreover,
\[
\#\mathcal P_{\delta,k}\lesssim \delta^{-2}(1+k),
\qquad
\#\mathcal P_{\delta, \ell}\lesssim \delta^{-2}(1+\ell),
\]
and
\[
|d_{Q,u}|
\le
\|d_k\|_{\ell^p}
\le
w_k^{-1}\|d\|_{\ell_w^{p,q,\Lambda}}
\le
(1+k)^{\frac{1}{p}}\|d\|_{\ell_w^{p,q,\Lambda}}.
\]
Hence
\begin{align*}
&\sup_{z\in K_R}
\sum_{Q\in\mathcal P_\delta}\sum_{u\in\Lambda}\sum_{Q'\in\mathcal P_\delta}
|d_{Q,u}|\,|b_{Q'}(\zeta_Q+u)|\,|\kappa_{\zeta_{Q'}}(z)| \\
\lesssim{}&
\delta^{-4}
\|d\|_{\ell_w^{p,q,\Lambda}}
\sum_{\ell\ge0}(1+\ell)e^{-\frac{\alpha}{4}(\ell-R)^2}
\sum_{k\ge0}(1+k)^{1+\frac1p}e^{-\frac{\gamma}{2}|\ell-k|^\beta}.
\end{align*}
By Lemma~\ref{lem:peetre-weighted} and Lemma~\ref{lem:exp-rowsum}, the inner sum is
\(
\lesssim (1+\ell)^{1+\frac{1}{p}}\).
Hence the whole expression is bounded by
\[
C\,\|d\|_{\ell_w^{p,q,\Lambda}}
\sum_{\ell\ge0}(1+\ell)^{2+\frac{1}{p}}e^{-\frac{\alpha}{4}(\ell-R)^2},
\]
which is finite.
Thus the triple series in \eqref{eq:recentering-formal} is absolutely and uniformly
convergent on $K_R$, so we may interchange the sums and obtain, for every $z\in K_R$,
\begin{align*}
S_\Lambda^{\mathrm{flt}}d(z)
&=
\sum_{Q'\in\mathcal P_\delta}
\Bigl(
\sum_{Q\in\mathcal P_\delta}\sum_{u\in\Lambda}
d_{Q,u}e^{i\alpha\im(\overline{\zeta_Q}u)}\,b_{Q'}(\zeta_Q+u)
\Bigr)\kappa_{\zeta_{Q'}}(z) \\
&=
\sum_{Q'\in\mathcal P_\delta}(E_{\rm rec}d)_{Q'}\,\kappa_{\zeta_{Q'}}(z)
=
S_\delta^{\ker}(E_{\rm rec}d)(z).
\end{align*}
Since $R>0$ was arbitrary, \eqref{eq:recentering-identity} holds on all of~$\C$.
\end{proof}

\subsubsection{\(\varepsilon\)-iteration and normalized Fock kernel coefficients}
\label{subsubsec:atomic-quasi-proof}
We combine the exact jet decomposition, local jet elimination, and re-centering
operator.  After choosing \(\varepsilon>0\) sufficiently small, the resulting
error operator is a strict contraction on \(\mathcal F_\alpha^{p,q}\).
Fix once and for all auxiliary parameters $0<\beta<1$ and $\gamma>0$, which enter only through Proposition~\ref{prop:recentering-floating}.

\begin{proposition}[\(\varepsilon\)-iteration for kernel coefficient recovery]
\label{prop:epsilon-iteration-kernel}
Let $\alpha>0$, $0<p<1$, $0<q\le \infty$, and set
\(
m:=\Bigl\lfloor \frac{2(1-p)}{p}\Bigr\rfloor\) , \(
s:=\min\{1,p,q\}\).

\smallskip
\noindent
Assume that $\delta_0>0$ is chosen so that the conclusions of
Theorem~\ref{thm:atomic-quasi-jet} and
Proposition~\ref{prop:recentering-floating} hold.
Fix a $\delta$-pocket tiling $\mathcal P_\delta$ of $\C$ with
$\delta\in(0,\delta_0)$.

\smallskip
\noindent
By Theorem~\ref{thm:atomic-quasi-jet}, the operator
\(
J:=C_\delta^m(I-R_\delta^m)^{-1}:
\mathcal F_\alpha^{p,q}\to \ell_w^{p,q,m}
\)
is bounded, with $\|J\|\le C_J$, and satisfies
\(
S_\delta^mJ=I\) on \( \mathcal F_\alpha^{p,q}\), where $C_J$ depends only on $\alpha, p, q$ and $\delta$.

\smallskip
\noindent
Let $C_{\rm E}=C_{\rm E}(\alpha,p,q)>0$ denote the constant in
\eqref{eq:local-jet-elimination-defect} from
Lemma~\ref{lem:local-jet-elimination}. Choose $\varepsilon\in(0,1)$ so that
\[
\varepsilon_0:=C_{\rm E}C_J\varepsilon<1.
\]
Fix such an $\varepsilon$ throughout.
Let $\Lambda_\varepsilon\subset\C$ and
\(
E_{{\rm loc},\varepsilon}:\ell_w^{p,q,m}\to \ell_w^{p,q,\Lambda_\varepsilon}
\)
be the operator furnished by Lemma~\ref{lem:local-jet-elimination}, with
\(
\|E_{{\rm loc},\varepsilon}\|\le C_{{\rm loc},\varepsilon}\).
Then \eqref{eq:local-jet-elimination-defect} becomes
\[
\|S^{\mathrm{flt}}_{\Lambda_\varepsilon}E_{{\rm loc},\varepsilon}c-S_\delta^m c\|_{\mathcal F_\alpha^{p,q}}
\le
C_{\rm E}\varepsilon\,\|c\|_{\ell_w^{p,q,m}},
\qquad c\in \ell_w^{p,q,m}.
\]
For this set $\Lambda_\varepsilon$, let
\(
E_{{\rm rec},\varepsilon}:\ell_w^{p,q,\Lambda_\varepsilon}\to \ell_w^{p,q}
\)
be the re-centering operator furnished by
Proposition~\ref{prop:recentering-floating}, with
\(
\|E_{{\rm rec},\varepsilon}\|\le C_{{\rm rec},\varepsilon}\),
and
\(
S^{\mathrm{flt}}_{\Lambda_\varepsilon}d
=
S_\delta^{\ker}(E_{{\rm rec},\varepsilon}d)\) for every \( d\in \ell_w^{p,q,\Lambda_\varepsilon}\).

\smallskip
\noindent
Define
\[
A_\varepsilon:=E_{{\rm rec},\varepsilon}E_{{\rm loc},\varepsilon}J\qquad \text{and}
\qquad
T_\varepsilon:=S_\delta^{\ker}A_\varepsilon.
\]
Then the series
\[
A_\delta^{\ker}f
:=
\sum_{n=0}^\infty A_\varepsilon (I-T_\varepsilon)^n f
\]
converges in $\ell_w^{p,q}$ for every $f\in\mathcal F_\alpha^{p,q}$, defines a bounded
linear operator
\(
A_\delta^{\ker}:\mathcal F_\alpha^{p,q}\to \ell_w^{p,q}\),
and satisfies
\(
S_\delta^{\ker}A_\delta^{\ker}=I\) on \( \mathcal F_\alpha^{p,q}\).
Moreover,
\begin{equation}\label{eq:epsilon-iteration-bound}
\|A_\delta^{\ker}f\|_{\ell_w^{p,q}}
\le
\frac{C_{{\rm rec},\varepsilon}C_{{\rm loc},\varepsilon}C_J}{(1-\varepsilon_0^s)^{\frac{1}{s}}}
\,\|f\|_{\mathcal F_\alpha^{p,q}},
\qquad
f\in\mathcal F_\alpha^{p,q}.
\end{equation}
\end{proposition}

\begin{proof}
By the defining property of $E_{\rm rec, \varepsilon}$ and
\(T_\varepsilon\), we have $S_\delta^{\ker}E_{\rm rec, \varepsilon}=S^{\mathrm{flt}}_{\Lambda_{\varepsilon}}$, and hence
\[
T_\varepsilon
=
S_\delta^{\ker}E_{\rm rec, \varepsilon}E_{{\rm loc},\varepsilon}J
=
S^{\mathrm{flt}}_{\Lambda_{\varepsilon}}E_{{\rm loc},\varepsilon}J.
\]
For $f\in\mathcal F_\alpha^{p,q}$, using \(
S_\delta^mJ=I \)  on \(\mathcal F_\alpha^{p,q}\) and then \eqref{eq:local-jet-elimination-defect} with $c=Jf$, we obtain
\[
f-T_\varepsilon f
=
S_\delta^m(Jf)-S^{\mathrm{flt}}_{\Lambda_{\varepsilon}}E_{{\rm loc},\varepsilon}(Jf),
\]
hence
\[
\|f-T_\varepsilon f\|_{\mathcal F_\alpha^{p,q}}
\le
C_{\rm E}\varepsilon\,\|Jf\|_{\ell_w^{p,q,m}}
\le
C_{\rm E}C_J\varepsilon\,\|f\|_{\mathcal F_\alpha^{p,q}}
=
\varepsilon_0 \,\|f\|_{\mathcal F_\alpha^{p,q}}.
\]
Thus
\[
\|I-T_\varepsilon\|_{\mathcal F_\alpha^{p,q}\to \mathcal F_\alpha^{p,q}}
\le \varepsilon_0 <1.
\]
Now define
\(
e^{(n)}:=(I-T_\varepsilon)^n f \) and \(
c^{(n)}:=A_\varepsilon e^{(n)}\).
Then
\(
\|e^{(n)}\|_{\mathcal F_\alpha^{p,q}}
\le
\varepsilon_0 ^n\|f\|_{\mathcal F_\alpha^{p,q}}
\),
and therefore
\[
\|c^{(n)}\|_{\ell_w^{p,q}}
\le
\|A_\varepsilon\|\,\|e^{(n)}\|_{\Fpq}
\le
C_{\rm rec, \varepsilon} C_{\rm loc, \varepsilon} C_J\, \varepsilon_0 ^n\,\|f\|_{\mathcal F_\alpha^{p,q}}.
\]
Since $\ell_w^{p,q}$ is a complete quasi-Banach space, the series
\(
\sum_{n=0}^\infty c^{(n)}
\)
converges in $\ell_w^{p,q}$ to some $c=A_\delta^{\ker}f$, and
\[
\|c\|_{\ell_w^{p,q}}^s
\le
\sum_{n=0}^\infty \|c^{(n)}\|_{\ell_w^{p,q}}^s
\le
\big(C_{\rm rec, \varepsilon} C_{\rm loc, \varepsilon} C_J\big)^s
\sum_{n=0}^\infty \varepsilon_0 ^{ns}
\,\|f\|_{\mathcal F_\alpha^{p,q}}^s.
\]
This yields \eqref{eq:epsilon-iteration-bound}.

\smallskip
\noindent
Since $S_\delta^{\ker}$ is bounded on $\ell_w^{p,q}$ (Corollary \ref{cor:synthesis<1}),
\[
S_\delta^{\ker}c
=
\sum_{n=0}^\infty S_\delta^{\ker}c^{(n)}
=
\sum_{n=0}^\infty T_\varepsilon(I-T_\varepsilon)^n f.
\]
The partial sums telescope:
\[
\sum_{n=0}^N T_\varepsilon(I-T_\varepsilon)^n f
=
f-(I-T_\varepsilon)^{N+1}f,
\]
and the partial sums converge to $f$ in $\mathcal F_\alpha^{p,q}$ because
\[
\|(I-T_\varepsilon)^{N+1}f\|_{\Fpq}
\le
\varepsilon_0 ^{N+1}\|f\|_{\Fpq}\to0.
\]
Hence
\(
S_\delta^{\ker}A_\delta^{\ker}f=f\). Moreover, linearity is immediate from the definition of $A_\delta^{\ker}$.
\end{proof}

\subsubsection{Proof of Theorem~\ref{thm:AD-main} in the angular quasi-Banach range}
\label{subsubsec:atomic-quasi-banach-exact}

The proof in the range $0<p<1$ is obtained by combining the exact jet decomposition,
the local elimination of jet atoms to floating kernels, the re-centering of floating kernels,
and the $\varepsilon$-iteration argument from the preceding subsections.

\begin{proof}[Proof of Theorem~\ref{thm:AD-main} in the range $0<p<1$]

Fix once and for all auxiliary parameters \(0<\beta<1\) and
\(\gamma>0\), for instance \(\beta= \frac{1}{2}\) and \(\gamma=1\), for the
re-centering step in Proposition~\ref{prop:recentering-floating}.  Then
choose \(\varepsilon>0\) as in
Proposition~\ref{prop:epsilon-iteration-kernel}.  After these auxiliary
choices are fixed, their contribution is absorbed into the implicit
constants, which then depend only on \(\alpha,p,q\), and \(\delta\).

\smallskip
\noindent
Set
\(
m:=\Bigl\lfloor \frac{2(1-p)}{p}\Bigr\rfloor\).
Let \(\delta_0=\delta_0(\alpha,p,q)>0\) be small enough so that the conclusions of
Theorem~\ref{thm:atomic-quasi-jet} and
Proposition~\ref{prop:recentering-floating} hold.
Fix $0<\delta<\delta_0$, and let $\mathcal P_\delta$ be a $\delta$-pocket tiling of $\C$
with centers $\{\zeta_Q\}_{Q\in\mathcal P_\delta}$.

\smallskip
\noindent
\emph{(i) Synthesis.}
By Corollary~\ref{cor:synthesis<1}, the kernel synthesis operator
\(
S_\delta^{\ker}\)
is bounded from $\ell_w^{p,q}$ to $\mathcal F_\alpha^{p,q}$. In particular,
for every $c=(c_Q)_{Q\in\mathcal P_\delta}\in\ell_w^{p,q}$,
\[
\|S_\delta^{\ker}c\|_{\mathcal F_\alpha^{p,q}}
\lesssim
\|c\|_{\ell_w^{p,q}}.
\]
Consequently, for every representation \(f=S_\delta^{\ker}c\),
\(
\|f\|_{\mathcal F_\alpha^{p,q}}
\lesssim
\|c\|_{\ell_w^{p,q}}\),
and therefore
\[
\|f\|_{\mathcal F_\alpha^{p,q}}
\lesssim
\inf\left\{
\|c\|_{\ell_w^{p,q}}:
c\in\ell_w^{p,q},\ f=S_\delta^{\ker}c
\right\}.
\]

\smallskip
\noindent
\emph{(ii) Coefficient recovery and optimality.}
By Proposition~\ref{prop:epsilon-iteration-kernel}, there exists a bounded linear operator
\(
A_\delta^{\ker}:\mathcal F_\alpha^{p,q}\to \ell_w^{p,q}
\)
such that
\[
S_\delta^{\ker}A_\delta^{\ker}=I
\quad\text{on }\mathcal F_\alpha^{p,q} \quad
\text{ and } \quad 
\|A_\delta^{\ker}f\|_{\ell_w^{p,q}}
\lesssim
\|f\|_{\mathcal F_\alpha^{p,q}},
\qquad
f\in\mathcal F_\alpha^{p,q}.
\]
Thus, for every $f\in\mathcal F_\alpha^{p,q}$, the coefficient sequence
\(
c^\ast:=A_\delta^{\ker}f
\)
belongs to \(\ell_w^{p,q}\),
satisfies \(f=S_\delta^{\ker}c^\ast\), and  obeys
\[
\|c^\ast\|_{\ell_w^{p,q}}
\lesssim \,  
\|f\|_{\mathcal F_\alpha^{p,q}}.
\]
Therefore
\[
\inf\Bigl\{
\|c\|_{\ell_w^{p,q}}:\ c\in\ell_w^{p,q},\ f=S_\delta^{\ker}c
\Bigr\}
\lesssim \, 
\|f\|_{\mathcal F_\alpha^{p,q}}.
\]
Combining the two inequalities, we obtain
\[
\|f\|_{\mathcal F_\alpha^{p,q}}
\, \asymp \,
\inf\Bigl\{
\|c\|_{\ell_w^{p,q}}:\ c\in\ell_w^{p,q},\ f=S_\delta^{\ker}c
\Bigr\}.
\]
This proves
Theorem~\ref{thm:AD-main} in the range $0<p<1$.
\end{proof}

\subsection{The mixed-norm-normalized formulation}
\label{subsec:atomic-consequences}
Using \(\ell^{p,q}\), the mixed-norm-normalized kernels
\(\widetilde\kappa_{\zeta_Q}\), and Proposition~\ref{prop:T1}, we reformulate
Theorem~\ref{thm:AD-main}.

\begin{corollary}[Mixed-norm-normalized kernel form of the atomic decomposition]
\label{cor:normalized-atomic}
Let $\alpha>0$ and $0<p,q\le\infty$. Let $\delta_0=\delta_0(\alpha,p,q)$ be the constant from
Theorem~\ref{thm:AD-main}. For every $\delta\in(0,\delta_0)$ and every $\delta$-pocket tiling
$\mathcal P_\delta$ of $\C$ with centers $\{\zeta_Q\}_{Q\in\mathcal P_\delta}$, the following assertions hold, with
all implicit constants depending only on \(\alpha,p,q\), and \(\delta\).

\smallskip
\noindent
\emph{(i) Synthesis.}
The mixed-norm-normalized synthesis operator
\[
S_\delta d
:=
\sum_{Q\in\mathcal P_\delta} d_Q\,\widetilde{\kappa}_{\zeta_Q},
\qquad d=(d_Q)_{Q\in\mathcal P_\delta},
\]
is a bounded linear operator from $\ell^{p,q}$ to $\mathcal F_\alpha^{p,q}$, and
\[
\|S_\delta d\|_{\mathcal F_\alpha^{p,q}}
\lesssim
\,\|d\|_{\ell^{p,q}},
\qquad d\in\ell^{p,q}.
\]

\smallskip
\noindent
\emph{(ii) Coefficient recovery and optimality.}
Conversely, for every $f\in\mathcal F_\alpha^{p,q}$ there exists
$d=(d_Q)_{Q\in\mathcal P_\delta}\in\ell^{p,q}$ such that
\(
f= S_\delta d\),
and moreover
\[
\|f\|_{\mathcal F_\alpha^{p,q}}
\, \asymp \,
\inf\Bigl\{
\|d\|_{\ell^{p,q}}:
d\in\ell^{p,q},\ f = S_\delta d
\Bigr\}.
\]
\end{corollary}

\begin{proof}
Compare \(S_\delta\) with \(S_\delta^{\ker}\).
By Proposition \ref{prop:T1}, for every $k\ge0$ and every
$Q\in\mathcal P_{\delta,k}$,
\[
\|\kappa_{\zeta_Q}\|_{\mathcal F_\alpha^{p,q}} \asymp (1 + |\zeta_Q|)^{\frac{1}{q} - \frac{1}{p}} \asymp (1 + k)^{\frac{1}{q} - \frac{1}{p}} = w_k.
\]
Hence, if we define
\[
d_Q:=c_Q\,\|\kappa_{\zeta_Q}\|_{\mathcal F_\alpha^{p,q}},
\qquad Q\in\mathcal P_\delta,
\]
then
\[
\sum_{Q\in\mathcal P_\delta} c_Q\,\kappa_{\zeta_Q}
=
\sum_{Q\in\mathcal P_\delta} d_Q\,\widetilde{\kappa}_{\zeta_Q}.
\]
Moreover, for the block vectors
\(
c_k:=(c_Q)_{Q\in\mathcal P_{\delta,k}}\) and \(
d_k:=(d_Q)_{Q\in\mathcal P_{\delta,k}}\), with \(k\ge0\),
we have
\(
\|d_k\|_{\ell^p}\asymp w_k\|c_k\|_{\ell^p}\).
Therefore
\(
\|d\|_{\ell^{p,q}}\asymp \|c\|_{\ell_w^{p,q}}\).

\smallskip
\noindent
The corollary now follows immediately from Theorem~\ref{thm:AD-main}: part~(i) transfers the
bounded synthesis estimate from $\ell_w^{p,q}$ to $\ell^{p,q}$, while part~(ii) transfers the
existence of coefficients and the optimality estimate in exactly the same way.
\end{proof}

\begin{remark}
The preceding mixed-norm-normalized formulation makes explicit the two-level nature of the
coefficient space. In the classical Fock scale, atomic decompositions may be stated
using normalized kernels indexed by a uniform lattice and coefficients in a single
\(\ell^p\)-space. In the mixed-norm setting, the \(\delta\)-pocket tiling
\(\mathcal P_\delta\) is better adapted to the norm: the pockets separate the angular
or local behavior inside each annulus from the radial summation over annuli. The
weight
\[
w_k=(1+k)^{\frac1q-\frac1p}
\]
accounts for the change between local \(\ell^p\)-summability and the outer
\(\ell^q\)-summation. Thus the atomic decomposition is a shellwise discretization,
rather than a lattice discretization of the classical Fock theory, adapted to
the mixed norm. This is analogous in spirit to the annular--angular discretizations
used in mixed-norm Bergman spaces, although here the localization is governed by
Gaussian kernels on the Euclidean plane.
\end{remark}


\section{Carleson measures and vanishing Carleson measures}\label{sec:carleson}

This section proves the Carleson and vanishing Carleson theorems stated in
Subsection~\ref{subsec:carmeas}.  We fix the standing notation in
Subsection~\ref{subsec:carleson-notation}, isolate the required discrete
shell principles in Subsection~\ref{subsec:carleson-discrete}, and then
prove the bounded and compact characterizations in
Subsections~\ref{subsec:carleson-bounded} and
\ref{subsec:carleson-vanishing}.

\subsection{Standing notation}\label{subsec:carleson-notation}

Throughout this section, we use the notation of Subsection~\ref{subsec:carmeas}.  Fix
\[
\alpha>0,\qquad 0<p,q\le \infty,\qquad 0<s<\infty,
\]
and let \(\delta_0=\delta_0(\alpha,p,q)\) be as in Theorem~\ref{thm:AD-main}.  Fix
\(0<\delta<\delta_0\) and work with the corresponding \(\delta\)-pocket tiling
\(
\mathcal P_\delta=\bigsqcup_{k\ge0}\mathcal P_{\delta,k}
\)
of \(\mathbb C\), with distinguished centers \(\zeta_Q\in Q\) and the uniform ball control
\[
B(\zeta_Q,c_{\mathrm{in}}\delta) \subset Q\subset B(\zeta_Q,C_{\mathrm{out}}\delta), \qquad Q\in\mathcal P_\delta,
\]
where \(c_{\mathrm{in}}>0\) and \(C_{\mathrm{out}}>0\) are independent of \(\delta\).  As before, write
\[
A_k=\{z\in\mathbb C:\ k\le |z|<k+1\},\qquad k\in\mathbb N_0.
\]
Let \(\mu\) be a positive Borel measure on \(\mathbb C\). We consider the
Gaussian \(L^s\)-space
\[
L_\alpha^s(\mu)
=
\left\{
f \text{ measurable on }\mathbb C:
\|f\|_{L_\alpha^s(\mu)}^s
:=
\int_{\mathbb C}|f(z)|^s e^{-\frac{\alpha s}{2}|z|^2}\,d\mu(z)
<\infty
\right\},
\]
and denote by
\[
i_\mu:\mathcal F_\alpha^{p,q}\to L_\alpha^s(\mu),\qquad i_\mu f=f,
\]
the natural embedding map.

\smallskip
\noindent
Fix a radius \(R > C_{\mathrm{out}}\delta\).  Define the continuous and discrete ball-mass quantities by
\[
U_{\mu,R}(\zeta):=\|\kappa_\zeta\|_{\mathcal F_\alpha^{p,q}}^{-s}\, \mu(B(\zeta,R)),
\quad \zeta\in\mathbb C, \qquad \text{and} \qquad 
u_{\mu,R}(Q):=U_{\mu,R}(\zeta_Q),
\quad Q\in\mathcal P_\delta.
\]
We also retain the generalized conjugate exponents \(p_s^\ast\) and \(q_s^\ast\), as defined in
Subsection~\ref{subsec:carmeas}, together with the corresponding shell space
\(L_{\mathrm{sh}}^{p_s^\ast,q_s^\ast}\). Note that
\[
1\le p_s^\ast,\ q_s^\ast\le\infty.
\]

\smallskip
\noindent
We use throughout this section the mixed-norm-normalized
formulation of the atomic decomposition from
Subsection~\ref{subsec:atomic-consequences}. Namely, we write
\[
\widetilde\kappa_{\zeta_Q}
:=
\frac{\kappa_{\zeta_Q}}
{\|\kappa_{\zeta_Q}\|_{\mathcal F_\alpha^{p,q}}},
\qquad Q\in\mathcal P_\delta,
\]
and denote by
\[
S_\delta d
:=
\sum_{Q\in\mathcal P_\delta}
d_Q\widetilde\kappa_{\zeta_Q},
\qquad d=(d_Q)_{Q\in\mathcal P_\delta} \in\ell^{p,q},
\]
the mixed-norm-normalized synthesis operator.

\smallskip
\noindent
By Corollary~\ref{cor:normalized-atomic} and the coefficient-recovery
operators constructed in the proof of Theorem~\ref{thm:AD-main}, after the
normalization in Subsection~\ref{subsec:atomic-consequences}, 
we fix a bounded linear operator
\begin{equation}\label{eq:inverse_S}
A_\delta:\mathcal F_\alpha^{p,q}\to \ell^{p,q} \qquad \text{ such that } 
\qquad
S_\delta A_\delta = I \quad
 \text{ on } \ \mathcal F_\alpha^{p,q}.
\end{equation}
More precisely, we choose \(A_\delta\) from the mixed-norm-normalized
formulation of the atomic decomposition as follows.  In the angular
Banach range \(1\le p\le\infty\), it is induced by the corrected kernel
analysis map
\(
C_\delta^{\ker}(I-R_\delta^{\ker})^{-1}\), which appears in the proof of Theorem~\ref{thm:AD-main} in
Subsection~\ref{subsubsec:atomic-banach-exact}. In the angular quasi-Banach range
\(0<p<1\), it is induced by \(A_\delta^{\ker}\) from
Proposition~\ref{prop:epsilon-iteration-kernel}.

\smallskip
\noindent
Unless otherwise stated, all implicit constants in this section may depend on
\(\alpha,p,q,s,\delta\), and on the fixed radius \(R\), but are independent of \(\mu\).

\subsection{Auxiliary discrete shell principles}\label{subsec:carleson-discrete}
We record the discrete weighted embedding principles used in both the
bounded and vanishing Carleson arguments. Let
\(\omega=(\omega_Q)_{Q\in\mathcal P_\delta}\) be a nonnegative weight on
the pocket index set, and define the weighted sequence space
\[
\ell^s(\omega)
:=
\left\{
c=(c_Q)_{Q\in\mathcal P_\delta}:
\|c\|_{\ell^s(\omega)}^s
:=
\sum_{Q\in\mathcal P_\delta}\omega_Q |c_Q|^s
<\infty
\right\}, 
\qquad 0<s<\infty.
\]
We denote by
\[
i_\omega:\ell^{p,q}\longrightarrow \ell^s(\omega),
\qquad i_\omega c=c,
\]
the corresponding weighted identity map whenever it is well defined.

\smallskip
\noindent
The first result treats a single finite block.  The second applies this
one-level criterion on each shell and gives the boundedness of
\(i_\omega\) in terms of the mixed shell norm of \(\omega\).  We then
record a finite-shell truncation criterion and use it to characterize
compactness of \(i_\omega\).

\smallskip
\noindent
For each \(k\ge0\), let
\[
i_\omega^{(k)}:\ell^p(\mathcal P_{\delta,k})\to \ell^s(\omega|_{\mathcal P_{\delta,k}}),
\qquad i_\omega^{(k)}(a)=a,
\]
denote the restriction of \(i_\omega\) to the \(k\)-th block. Since \(\mathcal P_{\delta,k}\) is finite, \(i_\omega^{(k)}\) is automatically bounded. 

\smallskip
\noindent
We compute the block norms \(\|i^{(k)}_{\omega}\|\), and then use them to
characterize the boundedness of the identity map
\(
i_\omega\).

\begin{lemma} \label{lem:block-norm}
Let  \(\omega=(\omega_Q)_{Q\in\mathcal P_\delta}\)
be a nonnegative weight on \(\mathcal P_\delta\). For each \(k\ge0\), write
\(
\omega_k:=(\omega_Q)_{Q\in\mathcal P_{\delta,k}}\). Then
\[
\|i_\omega^{(k)}\|^s
= \|\omega_k\|_{\ell^{p_s^{\ast}}}.
\]
\end{lemma}

\begin{proof}
By definition,
\[
\|i_\omega^{(k)}\|^s
=
\sup\left\{
\sum_{Q\in\mathcal P_{\delta,k}}\omega_Q |a_Q|^s:
a=(a_Q)_{Q\in\mathcal P_{\delta,k}},\ 
\|a\|_{\ell^p(\mathcal P_{\delta,k})}=1
\right\}.
\]
If \(0<p\le s\), then $p_s^{\ast} = \infty$ and \(\|a\|_{\ell^s(\mathcal P_{\delta,k})}\le \|a\|_{\ell^p(\mathcal P_{\delta,k})}=1\), hence
\[
 \sum_{Q\in\mathcal P_{\delta,k}}\omega_Q |a_Q|^s
\le
\Bigl(\sup_{Q\in\mathcal P_{\delta,k}}\omega_Q\Bigr)
\sum_{Q\in\mathcal P_{\delta,k}}|a_Q|^s
\leq  \sup_{Q\in\mathcal P_{\delta,k}}\omega_Q = 
\|\omega_k\|_{\ell^{p_s^{\ast}}}.
\]
Hence $\|i_\omega^{(k)}\|^s
\leq \|\omega_k\|_{\ell^{p_s^{\ast}}}$.  
The reverse inequality is obtained by choosing \(a\) supported at a point
where \(\omega_Q\) attains its maximum on $\mathcal{P}_{\delta, k}$.

\smallskip
\noindent
If \(s<p<\infty\), then H\"older's inequality with conjugate exponents
\(\frac{p}{s}\) and \(p_s^\ast=\frac{p}{p-s}\) gives
\[
\sum_{Q\in\mathcal P_{\delta,k}}\omega_Q |a_Q|^s
\le
\Bigl(\sum_{Q\in\mathcal P_{\delta,k}}\omega_Q^{\,p_s^\ast}\Bigr)^{\frac{1}{p_s^\ast}}
\Bigl(\sum_{Q\in\mathcal P_{\delta,k}}|a_Q|^p\Bigr)^{\frac{s}{p}}.
\]
Taking the supremum over \(\|a\|_{\ell^p(\mathcal P_{\delta,k})}=1\) yields
\(
\|i_\omega^{(k)}\|^s\le
 \|\omega_k\|_{\ell^{p_s^{\ast}}}\).
Since \(\mathcal P_{\delta,k}\) is finite, equality is attained by the usual H\"older extremizer.

\smallskip
\noindent
If \(p=\infty\), then $p_s^{\ast} = 1$ and \(|a_Q|\le1\) for all \(Q\), so
\[
\sum_{Q\in\mathcal P_{\delta,k}}\omega_Q |a_Q|^s
\le
\sum_{Q\in\mathcal P_{\delta,k}}\omega_Q = \|\omega_k\|_{\ell^{p_s^{\ast}}}.
\]
Hence
$\|i_\omega^{(k)}\|^s
\leq \|\omega_k\|_{\ell^{p_s^{\ast}}}$. 
The reverse inequality is obtained by taking \(a_Q=1\) for all
\(Q\in\mathcal P_{\delta,k}\).
\end{proof}

\begin{proposition}\label{prop:boundedness-iomega}
Let \(\omega=(\omega_Q)_{Q\in\mathcal P_\delta}\) be a nonnegative weight on \(\mathcal P_\delta\).
Then the identity map
\(
i_\omega:\ell^{p,q}\to \ell^s(\omega)
\)
is bounded if and only if the sequence
\(
\omega\) belongs to $\ell^{p_s^\ast, q_s^\ast}$.
Moreover,
\[
\|i_\omega\|^s
 = \|\omega\|_{\ell^{p_s^{\ast}, q_s^{\ast}}}.
\]
\end{proposition}

\begin{proof}
Set $W_k(\omega):=\|i_\omega^{(k)}\|^s$, $W:=\bigl(W_k(\omega)\bigr)_{k\ge0}$, and 
\[
\ell^s\bigl(W\bigr)
:=
\left\{
x=(x_k)_{k\ge0}:
\|x\|_{\ell^s(W)}^s
:=
\sum_{k\ge0}W_k(\omega)|x_k|^s
<\infty
\right\}.
\]
Let
\(
i_{W}:\ell^q\to \ell^s\bigl(W\bigr)\) with \( i_{W}(x)=x\),
denote the corresponding identity map.

\smallskip
\noindent
For \(c=(c_Q)_{Q\in\mathcal P_\delta}\in \ell^{p,q}\), write
\(
x_k:=\|c_k\|_{\ell^p}\) for each \( k\ge0\).
Then \(\|x\|_{\ell^q}=\|c\|_{\ell^{p,q}}\), and by the definition of \(W_k(\omega)\),
\[
\sum_{Q\in\mathcal P_{\delta,k}}\omega_Q|c_Q|^s
\le
W_k(\omega)\,\|c_k\|_{\ell^p}^s
=
W_k(\omega)\,x_k^s .
\]
Summing over \(k\ge0\), we obtain
\[
\|c\|_{\ell^s(\omega)}^s
=
\sum_{k\ge0}\sum_{Q\in\mathcal P_{\delta,k}}\omega_Q|c_Q|^s
\le
\sum_{k\ge0}W_k(\omega)x_k^s
=
\|x\|_{\ell^s(W)}^s .
\]
Hence,
\(
\|i_\omega\|\le \|i_{W}\|\).

\smallskip
\noindent
Conversely, since each block \(\mathcal P_{\delta,k}\) is finite, the supremum defining
\( \|i_\omega^{(k)}\|^s\) is attained. Thus, for each \(k\ge0\), there exists
\(a^{(k)}=\bigl(a_Q^{(k)}\bigr)_{Q\in\mathcal P_{\delta,k}}\) such that
\[
\|a^{(k)}\|_{\ell^p(\mathcal P_{\delta,k})}=1
\qquad\text{and}\qquad
\sum_{Q\in\mathcal P_{\delta,k}}\omega_Q|a_Q^{(k)}|^s=W_k(\omega).
\]
Let \(x=(x_k)_{k\ge0}\in \ell^q\) be nonnegative, and define \(c=(c_Q)_{Q\in\mathcal P_\delta}\) by
\(
c_Q:=x_k\,a_Q^{(k)}\) for \( Q\in\mathcal P_{\delta,k}\).
Then
\(
\|c_k\|_{\ell^p}=x_k\) for each \( k\ge0\),
and therefore
\(
\|c\|_{\ell^{p,q}}=\|x\|_{\ell^q}\).
Moreover,
\[
\|c\|_{\ell^s(\omega)}^s
=
\sum_{k\ge0}\sum_{Q\in\mathcal P_{\delta,k}}\omega_Q|x_k a_Q^{(k)}|^s
=
\sum_{k\ge0}W_k(\omega)x_k^s
=
\|x\|_{\ell^s(W)}^s.
\]
Taking the supremum over all nonnegative sequences \(x\) with \(\|x\|_{\ell^q}\le1\), we get
\(
\|i_W\|\le \|i_\omega\|\).
Thus
\[
\|i_\omega\|=\|i_W\|.
\]
It remains to compute \(\|i_W\|\). By the same one-level argument as in
Lemma~\ref{lem:block-norm}, applied to the outer index set \(\mathbb N_0\), with \(p\) replaced by \(q\) and the block decomposition given by singletons \(\{k\}\), we obtain
\(
\|i_W\|^s=\|W\|_{\ell^{q_s^\ast}}\).
Therefore
\[
\|i_\omega\|^s
=
\|i_W\|^s
=
\|W\|_{\ell^{q_s^\ast}}
=
\Bigl\|\bigl(\|\omega_k\|_{\ell^{p_s^\ast}}\bigr)_{k\ge0}\Bigr\|_{\ell^{q_s^\ast}}
=
\|\omega\|_{\ell^{p_s^\ast,q_s^\ast}},
\]
where the penultimate identity follows from Lemma~\ref{lem:block-norm}.
\end{proof}

\medskip
\noindent
We characterize the compactness of the identity map
\(
i_\omega:\ell^{p,q}\to \ell^s(\omega)\)
in terms of block norms \(\|\omega_k\|_{\ell^{p_s^{\ast}}}\).

\smallskip
\noindent
For \(N\in\mathbb N_0\), set
\[
\mathcal P_{\delta,N}:=\bigcup_{k=0}^N \mathcal P_{\delta,k},
\qquad
\mathcal P_\delta^{\,N}:=\bigcup_{k>N} \mathcal P_{\delta,k}.
\]
For \(c=(c_Q)_{Q\in\mathcal P_\delta}\), define the truncations
\[
c_N:=c\,\mathbf 1_{\mathcal P_{\delta,N}},
\qquad
c^N:=c\,\mathbf 1_{\mathcal P_\delta^{\,N}}.
\]
We write
\(
P_{\omega,N}(c):=c_N\) and \(
R_{\omega,N}(c):=c^N\),
viewed as maps from \(\ell^{p,q}\) into \(\ell^s(\omega)\).

\begin{lemma}\label{lem:finite-shell-tails}
For each \(N\in\mathbb N_0\), the operator
\[
P_{\omega,N}:\ell^{p,q}\to \ell^s(\omega)
\]
is finite rank. Moreover, the operator $i_{\omega}: \ell^{p, q} \to \ell^s(\omega)$ is compact if and only if 
\[
\|R_{\omega,N}\|\to0
\qquad\text{as }\quad N\to\infty.
\]
\end{lemma}

\begin{proof}
Since \(\mathcal P_{\delta,N}\) is finite, every sequence in the range of \(P_{\omega,N}\)
is supported on a finite set. Hence \(P_{\omega,N}\) has finite-dimensional range, and
therefore is finite rank.

\smallskip
\noindent
Also, for every \(c\in\ell^{p,q}\),
\[
c=c_N+c^N,
\qquad \text{so} \qquad
i_\omega=P_{\omega,N}+R_{\omega,N}.
\]
Assume first that
\(
\|R_{\omega,N}\|\to0\)  as \(N\to\infty\).
Then \(i_\omega\) is the operator-norm limit of the finite-rank operators \(P_{\omega,N}\),
hence \(i_\omega\) is compact.

\smallskip
\noindent
Conversely, assume that the operator $i_{\omega}: \ell^{p, q} \to \ell^s(\omega)$ is compact.
Then the image of the unit ball of \(\ell^{p,q}\) is relatively compact in $\ell^s(\omega)$. By the standard compactness criterion for weighted $\ell^s$-spaces,  equivalently
by uniform approximation of compact sets by finite-coordinate truncations, the tails vanish uniformly on this image; hence
\[
\|R_{\omega,N}\|
=
\sup_{\|c\|_{\ell^{p,q}}\le1}\|R_{\omega,N}(c)\|_{\ell^s(\omega)}
\to0.
\]
This proves the lemma.
\end{proof}
\begin{proposition}\label{prop:compactness-iomega}
Let \(\omega=(\omega_Q)_{Q\in\mathcal P_\delta}\) be a nonnegative weight on \(\mathcal P_\delta\).
Then the following assertions are equivalent:
\begin{enumerate}
\item[\textup{(i)}] The operator \(i_\omega:\ell^{p,q}\to \ell^s(\omega)\) is compact.

\item[\textup{(ii)}]
\[
\bigl\|\bigl(\|\omega_k\|_{\ell^{p_s^\ast}}\bigr)_{k>N}\bigr\|_{\ell^{q_s^\ast}}\to0
\qquad\text{as }N\to\infty.
\]

\item[\textup{(iii)}] One has:
\begin{itemize}
    \item[\textup{(a)}] if $0<q\le s$, then \(
\|\omega_k\|_{\ell^{p_s^\ast}}\to0\)  as \(k\to\infty\);
\item[\textup{(b)}] if $s<q\le\infty$, then $\omega\in \ell^{p_s^\ast,q_s^\ast}$.
\end{itemize}
\end{enumerate}
\end{proposition}

\begin{proof}
For \(N\in\mathbb N_0\), the operator \(R_{\omega,N}\) coincides with the identity map
associated with the tail weight
\[
\omega^N:=\omega\,\mathbf 1_{\mathcal P_\delta^{\,N}},
\]
viewed as a map from $\ell^{p,q}$ into \(\ell^s(\omega^N)\). 
Therefore, by Lemma \ref{lem:block-norm},
\[
\|i^{(k)}_{\omega^N}\|^s= \|(\omega^N)_k\|_{\ell^{p_s^{\ast}}} =
\begin{cases}
0, & 0\le k\le N,\\[1mm]
\|\omega_k\|_{\ell^{p_s^\ast}}, & k>N.
\end{cases}
\]
Applying Proposition~\ref{prop:boundedness-iomega} to \(\omega^N\), we obtain
\[
\|R_{\omega,N}\|^s
=
\|i_{\omega^N}\|^s
=
\bigl\|\bigl(\|\omega_k\|_{\ell^{p_s^\ast}}\bigr)_{k>N}\bigr\|_{\ell^{q_s^\ast}}.
\]
Hence \((i)\Longleftrightarrow(ii)\) follows from Lemma~\ref{lem:finite-shell-tails}.

\smallskip
\noindent
If \(0<q\le s\), then \(q_s^\ast=\infty\), so
\[
\bigl\|\bigl(\|\omega_k\|_{\ell^{p_s^\ast}}\bigr)_{k>N}\bigr\|_{\ell^{q_s^\ast}}
=
\sup_{k>N}\|\omega_k\|_{\ell^{p_s^\ast}},
\]
and thus \((ii)\) is equivalent to
\(
\|\omega_k\|_{\ell^{p_s^\ast}}\to0 \) as $k \to \infty$. 

\smallskip
\noindent
If \(s<q \leq \infty\), then \(q_s^\ast<\infty\), and hence, \((ii)\) is equivalent to
\(
\bigl(\|\omega_k\|_{\ell^{p_s^\ast}}\bigr)_{k\ge0}\in \ell^{q_s^\ast}\), that is, 
\(
\omega\in \ell^{p_s^\ast,q_s^\ast}\).
Therefore
\((ii)\Longleftrightarrow(iii)\).
\end{proof}

\begin{remark}
The preceding discrete criteria isolate the point at which the mixed-norm structure
enters the Carleson measure problem. The analytic part of the proof below reduces the
Carleson embedding to weighted sequence embeddings of the form
\[
i_\omega:\ell^{p,q}\longrightarrow \ell^s(\omega),
\]
where the weight \(\omega\) is later chosen to be either the auxiliary weight
\(u^\sharp\) or the ball-mass weight \(u_{\mu,R}\). In contrast with the classical
one-parameter Fock scale, where localized masses are measured by a one-level
condition, the mixed-norm setting requires a shellwise organization: the inner
exponent \(p_s^*\) measures the distribution within each shell, while the outer
exponent \(q_s^*\) measures the radial distribution across shells. The compact
criterion also explains the dichotomy in Theorem~\ref{thm:vanishing-carleson-intro}:
when \(s<q\le\infty\), boundedness and compactness of \(i_\omega\) are described by
the same mixed summability condition, whereas when \(0<q\le s\), compactness requires
shellwise decay.
\end{remark}

\subsection{Carleson measures}\label{subsec:carleson-bounded}

We prove Theorem~\ref{thm:carleson-intro}.  The argument reduces the
embedding to weighted sequence inequalities, compares the auxiliary weight
\(u^\sharp\) with the localized ball-mass weight \(u_{\mu,R}\), and then
passes between the discrete and continuous shell conditions.  The last two
steps are contained in Propositions~\ref{prop:carleson-discrete-characterization}
and \ref{prop:carleson-discrete-continuous}.

\subsubsection{Atomic reduction and an auxiliary weighted inequality}
\label{subsubsec:carleson-atomic-reduction}

We begin with two complementary reductions for the bounded Carleson problem.
The necessity part is most naturally formulated in terms of the localized ball-mass weight
\(u_{\mu,R}\). For the sufficiency part, we introduce an auxiliary discrete weight
\(u^\sharp\) on the coefficient space and reduce the argument to the boundedness of the
corresponding weighted identity map.

\begin{proposition}
\label{prop:carleson-necessity-reduction}
If \(\mu\) is an \(s\)-Carleson measure for \(\mathcal F_\alpha^{p,q}\), then the identity map
\(
i_{u_{\mu,R}}:\ell^{p,q}\to \ell^s(u_{\mu,R})
\)
is bounded. Moreover,
\[
\|i_{u_{\mu,R}}\| \lesssim \|i_\mu\|,
\]
where the implicit constant depends only on \(\alpha,p,q,s,\delta\), and \(R\).
\end{proposition}

\begin{proof}
Let \(d=(d_Q)_{Q\in\mathcal P_\delta}\in \ell^{p,q}\) and \((\varepsilon_Q(t))_{Q\in\mathcal P_\delta}\) be a Rademacher family on \((0,1)\). For each \(t\in(0,1)\), we define
\[
d^{(t)}:=(\varepsilon_Q(t)d_Q)_{Q\in\mathcal P_\delta} \qquad \text{and} \qquad f_t:=S_\delta d^{(t)}=
\sum_{Q\in\mathcal P_\delta}\varepsilon_Q(t)\,d_Q\,\widetilde\kappa_{\zeta_Q}.
\]
Then $d^{(t)}$
belongs to \(\ell^{p,q}\) and satisfies
\(
\|d^{(t)}\|_{\ell^{p,q}}=\|d\|_{\ell^{p,q}}\).
By Corollary~\ref{cor:normalized-atomic}, the  mixed-norm-normalized synthesis operator
\(S_\delta:\ell^{p,q}\to \mathcal F_\alpha^{p,q}\) is bounded, and therefore 
\[
\|f_t\|_{\mathcal F_\alpha^{p,q}}
\le
\|S_\delta\|\,\|d\|_{\ell^{p,q}}
\]
uniformly in \(t\in(0,1)\). As \(
i_\mu:\mathcal F_\alpha^{p,q}\to L_\alpha^s(\mu)
\) is bounded, it follows that
\[
\int_{\C}|f_t(z)|^s e^{-\frac{\alpha s}{2}|z|^2}\,d\mu(z)
\le
\|i_\mu\|^s\,\|S_\delta\|^s\,\|d\|_{\ell^{p,q}}^s.
\]
Integrating in \(t\in(0,1)\) and applying Fubini's theorem, we obtain
\[
\int_{\C}
\left(
\int_0^1
\left|\sum_{Q\in\mathcal P_\delta}\varepsilon_Q(t)\,d_Q\,\widetilde\kappa_{\zeta_Q}(z)\right|^s
dt
\right)
e^{-\frac{\alpha s}{2}|z|^2}\,d\mu(z)
\le
\|i_\mu\|^s\,\|S_\delta\|^s\,\|d\|_{\ell^{p,q}}^s.
\]
By Khinchine's inequality for Rademacher series, this yields
\[
\int_{\C}
\left(
\sum_{Q\in\mathcal P_\delta}
|d_Q|^2\,|\widetilde\kappa_{\zeta_Q}(z)|^2
\right)^{\frac{s}{2}}
e^{-\frac{\alpha s}{2}|z|^2}\,d\mu(z)
\lesssim
\|i_\mu\|^s\,\|S_\delta\|^s\,\|d\|_{\ell^{p,q}}^s,
\]
where the implicit constant depends only on $s$.
Set
\[
G_d(z):=
\left(
\sum_{Q\in\mathcal P_\delta}
|d_Q|^2\,|\widetilde\kappa_{\zeta_Q}(z)|^2
\right)^{\frac{s}{2}}
e^{-\frac{\alpha s}{2}|z|^2}.
\]
We derive a lower bound for \(G_d\) in terms of \(u_{\mu,R}\).
If \(z\in B(\zeta_Q,R)\), then
\[
|\widetilde\kappa_{\zeta_Q}(z)|\,e^{-\frac{\alpha}{2}|z|^2}
=
\|\kappa_{\zeta_Q}\|_{\mathcal F_\alpha^{p,q}}^{-1}
e^{-\frac{\alpha}{2}|z-\zeta_Q|^2}
\ge
e^{-\frac{\alpha}{2}R^2}\,\|\kappa_{\zeta_Q}\|_{\mathcal F_\alpha^{p,q}}^{-1}.
\]
Therefore, whenever \(z\in B(\zeta_Q,R)\),
\[
G_d(z)\ge
e^{-\frac{\alpha s}{2}R^2}\,
|d_Q|^s\,
\|\kappa_{\zeta_Q}\|_{\mathcal F_\alpha^{p,q}}^{-s}.
\]
Hence, for every \(z\in\C\),
\[
G_d(z)\ge
e^{-\frac{\alpha s}{2}R^2}
\max_{Q:\, z\in B(\zeta_Q,R)}
\Bigl(
|d_Q|^s\,
\|\kappa_{\zeta_Q}\|_{\mathcal F_\alpha^{p,q}}^{-s}
\Bigr).
\]

\smallskip
\noindent
By the counting estimate in Lemma~\ref{lem:geo-pockets}(d), there exists
\(N_R\ge 1\), depending only on \(R\) and \(\delta\), such that
\[
\sum_{Q\in\mathcal P_\delta}\textbf{1}_{B(\zeta_Q,R)}(z)\le N_R,
\qquad z\in\C.
\]
Therefore
\[
\max_{Q:\, z\in B(\zeta_Q,R)}
\Bigl(
|d_Q|^s\,
\|\kappa_{\zeta_Q}\|_{\mathcal F_\alpha^{p,q}}^{-s}
\Bigr)
\ge
\frac1{N_R}
\sum_{Q\in\mathcal P_\delta}
|d_Q|^s\,
\|\kappa_{\zeta_Q}\|_{\mathcal F_\alpha^{p,q}}^{-s}
\,\textbf{1}_{B(\zeta_Q,R)}(z).
\]
Combining the last two estimates and integrating against \(d\mu(z)\), we get
\[
\int_{\C} G_d(z)\,d\mu(z)
\ge
\frac{e^{-\frac{\alpha s}{2}R^2}}{N_R}
\sum_{Q\in\mathcal P_\delta}
|d_Q|^s\,
\|\kappa_{\zeta_Q}\|_{\mathcal F_\alpha^{p,q}}^{-s}\,
\mu(B(\zeta_Q, R)).
\]
By the definition of \(u_{\mu,R}\), this becomes
\[
\int_{\C} G_d(z)\,d\mu(z)
\ge
\frac{e^{-\frac{\alpha s}{2}R^2}}{N_R}
\sum_{Q\in\mathcal P_\delta}u_{\mu,R}(Q)\,|d_Q|^s.
\]
Together with the upper bound obtained earlier, we conclude that
\[
\sum_{Q\in\mathcal P_\delta}u_{\mu,R}(Q)\,|d_Q|^s
\lesssim
\|i_\mu\|^s\,\|S_\delta\|^s\,\|d\|_{\ell^{p,q}}^s.
\]
Thus \(i_{u_{\mu,R}}:\ell^{p,q}\to \ell^s(u_{\mu,R})\) is bounded, and
\[
\|i_{u_{\mu,R}}\| \lesssim \|i_\mu\|,
\]
where the implicit constant depends only on \(\alpha,p,q,s,\delta\), and \(R\).
\end{proof}

\smallskip
\noindent
For the converse direction, the sufficiency argument is first reduced to an
auxiliary weighted inequality on the coefficient space.

\smallskip
\noindent
For this purpose, we introduce the auxiliary weight
\(u^\sharp=(u^\sharp(Q))_{Q\in\mathcal P_\delta}\) by
\[
u^\sharp(Q)
:=
\|\kappa_{\zeta_Q}\|_{\mathcal F_\alpha^{p,q}}^{-s}
\int_{\C} e^{-\frac{\alpha s}{4}|z-\zeta_Q|^2}\,d\mu(z),
\qquad Q\in\mathcal P_\delta.
\]

\begin{lemma}
\label{lem:carleson-diagonalization}
For every \(d\in\ell^{p,q}\), one has
\[
\|S_\delta d\|_{L_\alpha^s(\mu)}
\lesssim \|d\|_{\ell^s(u^{\sharp})},
\]
where the implicit constant depends only on \(\alpha,s\), and \(\delta\).
\end{lemma}

\begin{proof}
Since
\[
\bigl|\widetilde\kappa_{\zeta_Q}(z)\bigr|
e^{-\frac{\alpha}{2}|z|^2}
=
\|\kappa_{\zeta_Q}\|_{\mathcal F_\alpha^{p,q}}^{-1} \,
e^{-\frac{\alpha}{2}|z-\zeta_Q|^2},
\]
we obtain
\[
|S_\delta d(z)|e^{-\frac{\alpha}{2}|z|^2}
\le
\sum_{Q\in\mathcal P_\delta}
|d_Q|\,
\|\kappa_{\zeta_Q}\|_{\mathcal F_\alpha^{p,q}}^{-1} \,
e^{-\frac{\alpha}{2}|z-\zeta_Q|^2}.
\]
Assume first that \(0<s\le 1\). Since \(x\mapsto x^s\) is subadditive on \([0,\infty)\),
 integrating against \(d\mu(z)\), we get
\begin{align*}
\|S_\delta d\|_{L_\alpha^s(\mu)}^s
\le
\sum_{Q\in\mathcal P_\delta}|d_Q|^s \|\kappa_{\zeta_Q}\|_{\mathcal F_\alpha^{p,q}}^{-s} \,
\int_{\C}
e^{-\frac{\alpha s}{2}|z-\zeta_Q|^2}\,d\mu(z)
\leq
\sum_{Q\in\mathcal P_\delta}u^\sharp(Q)\,|d_Q|^s = \|d\|^s_{\ell^s(u^{\sharp})}.
\end{align*}
This proves the lemma when \(0<s\le 1\).

\smallskip
\noindent
Now assume that \(1<s<\infty\).
We split the Gaussian as
\(
e^{-\frac{\alpha}{2}|z-\zeta_Q|^2}
=
e^{-\frac{\alpha}{4}|z-\zeta_Q|^2}
e^{-\frac{\alpha}{4}|z-\zeta_Q|^2}
\)
and apply H\"older's inequality with exponents \(s\) and \(s'=\frac{s}{s-1}\) to get
\[
|S_\delta d(z)|^s e^{-\frac{\alpha s}{2}|z|^2}
\le
\Bigg(
\sum_{Q\in\mathcal P_\delta}
|d_Q|^s
\|\kappa_{\zeta_Q}\|_{\mathcal F_\alpha^{p,q}}^{-s}
e^{-\frac{\alpha s}{4}|z-\zeta_Q|^2}
\Bigg)
\Bigg(
\sum_{Q\in\mathcal P_\delta}
e^{-\frac{\alpha s'}{4}|z-\zeta_Q|^2}
\Bigg)^{\frac{s}{s'}}.
\]
By the counting estimate in Lemma~\ref{lem:geo-pockets}(d) and the same annular
decomposition used in the proof of Lemma~\ref{lem:gramian-localized}, one has
\begin{equation}\label{eq:uniform-gaussian-sum-pockets}
\sup_{z\in\C}\sum_{Q\in\mathcal P_\delta}
e^{-\frac{\alpha s'}{4}|z-\zeta_Q|^2}<\infty,
\end{equation}
with a bound depending only on \(\alpha\), \(s\), and \(\delta\).
Hence
\[
|S_\delta d(z)|^s e^{-\frac{\alpha s}{2}|z|^2}
\lesssim
\sum_{Q\in\mathcal P_\delta}
|d_Q|^s
\|\kappa_{\zeta_Q}\|_{\mathcal F_\alpha^{p,q}}^{-s}
e^{-\frac{\alpha s}{4}|z-\zeta_Q|^2}.
\]
Integrating against \(d\mu(z)\) gives
\[
\|S_\delta d\|_{L_\alpha^s(\mu)}^s
\lesssim
\sum_{Q\in\mathcal P_\delta}
|d_Q|^s
\|\kappa_{\zeta_Q}\|_{\mathcal F_\alpha^{p,q}}^{-s}
\int_{\C}
e^{-\frac{\alpha s}{4}|z-\zeta_Q|^2}\,d\mu(z) =  \sum_{Q\in\mathcal P_\delta}u^\sharp(Q)\,|d_Q|^s = \|d\|^s_{\ell^s(u^{\sharp})}.
\]
Taking $s$-th roots yields the desired estimate.
\end{proof}

\begin{proposition}
\label{prop:carleson-sufficiency-reduction}
If the identity map
\(
i_{u^\sharp}:\ell^{p,q}\to \ell^s(u^\sharp)
\)
is bounded, then \(\mu\) is an \(s\)-Carleson measure for \(\mathcal F_\alpha^{p,q}\). Moreover,
\[
\|i_\mu\| \lesssim \|i_{u^\sharp}\|,
\]
where the implicit constant depends only on \(\alpha, p, q, s\), and \(\delta\).
\end{proposition}

\begin{proof}
Let \(f\in \mathcal F_\alpha^{p,q}\), and  set
\(
d:=A_\delta f\in \ell^{p,q}\). 
Since \(S_\delta A_\delta=I\) on \(\mathcal F_\alpha^{p,q}\) by \eqref{eq:inverse_S}, we have
\(
f=S_\delta d\).
Hence, by Lemma~\ref{lem:carleson-diagonalization},
\[
\|f\|_{L_\alpha^s(\mu)}
=
\|S_\delta d\|_{L_\alpha^s(\mu)}
\lesssim
\|d\|_{\ell^s(u^\sharp)}.
\]
Since \(i_{u^\sharp}:\ell^{p,q}\to \ell^s(u^\sharp)\) is bounded, it follows that
\[
\|d\|_{\ell^s(u^\sharp)}
\le
\|i_{u^\sharp}\|\,\|d\|_{\ell^{p,q}}.
\]
Using the boundedness of \(A_\delta:\mathcal F_\alpha^{p,q}\to \ell^{p,q}\), we obtain
\[
\|f\|_{L_\alpha^s(\mu)}
\lesssim
\|i_{u^\sharp}\|\,\|f\|_{\mathcal F_\alpha^{p,q}}.
\]
Therefore \(i_\mu:\mathcal F_\alpha^{p,q}\to L_\alpha^s(\mu)\) is bounded, and
\(
\|i_\mu\| \lesssim \|i_{u^\sharp}\|\).
\end{proof}

\subsubsection{Comparison of the auxiliary and ball weights}
\label{subsubsec:carleson-weight-comparison}

We compare the auxiliary weight \(u^\sharp\) with the localized ball-mass
weight \(u_{\mu,R}\).  A pointwise lower estimate and a Gaussian smoothing
upper estimate yield equivalence of the corresponding mixed sequence norms.

\begin{lemma}
\label{lem:pointwise-comparison-u-sharp-ball}
Let
\(
\sigma:=s\left(\frac1p-\frac1q\right)\). For every \(Q\in\mathcal P_\delta\), one has
\[
u_{\mu,R}(Q) \; \lesssim \; u^\sharp(Q) \;
\lesssim   \; \sum_{Q'\in\mathcal P_\delta}
\left(\frac{1+|\zeta_Q|}{1+|\zeta_{Q'}|}\right)^\sigma
e^{-\frac{\alpha s}{8}|\zeta_Q-\zeta_{Q'}|^2}
u_{\mu,R}(Q'),
\]
where the implicit constants depend only on \(\alpha,p,q,s,\delta\), and \(R\).
\end{lemma}

\begin{proof}
For the lower estimate, by definition,
\[
u^\sharp(Q)
=
\|\kappa_{\zeta_Q}\|_{\mathcal F_\alpha^{p,q}}^{-s}
\int_{\C} e^{-\frac{\alpha s}{4}|z-\zeta_Q|^2}\,d\mu(z).
\]
Restricting the integral to \(B(\zeta_Q,R)\), we obtain
\[
u^\sharp(Q)
\ge
e^{-\frac{\alpha s}{4}R^2}
\|\kappa_{\zeta_Q}\|_{\mathcal F_\alpha^{p,q}}^{-s} \,
\mu(B(\zeta_Q,R))
=
e^{-\frac{\alpha s}{4}R^2}u_{\mu,R}(Q).
\]

\smallskip
\noindent
For the upper estimate, since \(R>C_{\rm out}\delta\) is part of the standing notation,  Lemma~\ref{lem:geo-pockets}(a),(c)
implies that
\[
\C=\bigcup_{Q'\in\mathcal P_\delta}B(\zeta_{Q'},R).
\]
Hence
\[
u^\sharp(Q)
\le
\|\kappa_{\zeta_Q}\|_{\mathcal F_\alpha^{p,q}}^{-s}
\sum_{Q'\in\mathcal P_\delta}
\int_{B(\zeta_{Q'},R)}
e^{-\frac{\alpha s}{4}|z-\zeta_Q|^2}\,d\mu(z).
\]
If \(z\in B(\zeta_{Q'},R)\), then
\[
|z-\zeta_Q|^2
\ge
\frac12|\zeta_{Q'}-\zeta_Q|^2-R^2.
\]
Therefore
\[
e^{-\frac{\alpha s}{4}|z-\zeta_Q|^2}
\le
e^{\frac{\alpha s}{4}R^2}
e^{-\frac{\alpha s}{8}|\zeta_{Q'}-\zeta_Q|^2}.
\]
It follows that
\[
u^\sharp(Q)
\le e^{\frac{\alpha s}{4}R^2} 
\sum_{Q'\in\mathcal P_\delta}
e^{-\frac{\alpha s}{8}|\zeta_Q-\zeta_{Q'}|^2}
\|\kappa_{\zeta_Q}\|_{\mathcal F_\alpha^{p,q}}^{-s} \,
\mu(B(\zeta_{Q'},R)).
\]
By Proposition~\ref{prop:T1},
\[
\frac{\|\kappa_{\zeta_Q}\|_{\mathcal F_\alpha^{p,q}}^{-s}}
     {\|\kappa_{\zeta_{Q'}}\|_{\mathcal F_\alpha^{p,q}}^{-s}}
\asymp
\left(\frac{1+|\zeta_Q|}{1+|\zeta_{Q'}|}\right)^\sigma .
\]
Using the definition of \(u_{\mu,R}(Q')\), we obtain
\[
u^\sharp(Q)
\lesssim
\sum_{Q'\in\mathcal P_\delta}
\left(\frac{1+|\zeta_Q|}{1+|\zeta_{Q'}|}\right)^\sigma
e^{-\frac{\alpha s}{8}|\zeta_Q-\zeta_{Q'}|^2}
u_{\mu,R}(Q'),
\]
as claimed.
\end{proof}

\begin{proposition}
\label{prop:equivalence-u-sharp-ball}
One has
\[
\|u^\sharp\|_{\ell^{p_s^*,q_s^*}}
\asymp
\|u_{\mu,R}\|_{\ell^{p_s^*,q_s^*}},
\]
where the implicit constants depend only on \(\alpha,p,q,s,\delta\), and \(R\).
\end{proposition}

\begin{proof}
By the lower bound in Lemma~\ref{lem:pointwise-comparison-u-sharp-ball},
\[
u_{\mu,R}(Q)\lesssim u^\sharp(Q),
\qquad Q\in\mathcal P_\delta .
\]
Therefore, by monotonicity of the \(\ell^{p_s^*}\)-norm on each block,
\[
\|(u_{\mu,R})_k\|_{\ell^{p_s^*}(\mathcal P_{\delta,k})}
\lesssim
\|(u^\sharp)_k\|_{\ell^{p_s^*}(\mathcal P_{\delta,k})},
\qquad k\ge0.
\]
Taking the outer \(\ell^{q_s^*}\)-norm gives
\[ \|u_{\mu,R}\|_{\ell^{p_s^*,q_s^*}}
\lesssim \|u^\sharp\|_{\ell^{p_s^*,q_s^*}}.
\]

\smallskip
\noindent
For the converse estimate, the upper bound in Lemma~\ref{lem:pointwise-comparison-u-sharp-ball} gives
\[
u^\sharp(Q)
\lesssim
\sum_{Q'\in\mathcal P_\delta}
K(Q,Q')u_{\mu,R}(Q'),
\qquad Q\in\mathcal P_\delta,
\]
where
\[
K(Q,Q')
:=
\left(\frac{1+|\zeta_Q|}{1+|\zeta_{Q'}|}\right)^\sigma
e^{-\frac{\alpha s}{8}|\zeta_Q-\zeta_{Q'}|^2},
\qquad
\sigma=s\left(\frac1p-\frac1q\right).
\]
Since
\[
\left(\frac{1+|\zeta_Q|}{1+|\zeta_{Q'}|}\right)^\sigma
\le
(1+|\zeta_Q-\zeta_{Q'}|)^{|\sigma|},
\]
the polynomial factor is absorbed by the Gaussian. Hence, after changing the implicit
constant,
\[
K(Q,Q')\lesssim
e^{-\frac{\alpha s}{16}|\zeta_Q-\zeta_{Q'}|^2}.
\]
Moreover, if \(Q\in\mathcal P_{\delta,k}\) and \(Q'\in\mathcal P_{\delta,m}\), then
\(
|k-m|
\le
|\zeta_Q-\zeta_{Q'}|+1\).
Thus
\[
K(Q,Q')
\lesssim
e^{-\frac{\alpha s}{64}|k-m|^2}
e^{-\frac{\alpha s}{32}|\zeta_Q-\zeta_{Q'}|^2}.
\]
For \(k,m\ge0\), define
\[
(T_{k,m}b)_Q
:=
\sum_{Q'\in\mathcal P_{\delta,m}}
K(Q,Q')b_{Q'},
\qquad Q\in\mathcal P_{\delta,k}.
\]
By the same Gaussian summability estimate over pocket centers used in the
proof of
\eqref{eq:uniform-gaussian-sum-pockets}, with
\(\frac{\alpha s}{32}\) in place of \(\frac{\alpha s'}{4}\), we have
\[
\sup_{Q\in\mathcal P_{\delta,k}}
\sum_{Q'\in\mathcal P_{\delta,m}}
e^{-\frac{\alpha s}{32}|\zeta_Q-\zeta_{Q'}|^2}
<\infty \qquad \text{and} \qquad
\sup_{Q'\in\mathcal P_{\delta,m}}
\sum_{Q\in\mathcal P_{\delta,k}}
e^{-\frac{\alpha s}{32}|\zeta_Q-\zeta_{Q'}|^2}
<\infty .
\]
Consequently,
\[
\sup_{Q\in\mathcal P_{\delta,k}}
\sum_{Q'\in\mathcal P_{\delta,m}}K(Q,Q')
+
\sup_{Q'\in\mathcal P_{\delta,m}}
\sum_{Q\in\mathcal P_{\delta,k}}K(Q,Q')
\lesssim
e^{-\frac{\alpha s}{64}|k-m|^2}.
\]
Since \(p_s^*\in[1,\infty]\), the standard \(\ell^1\)- and \(\ell^\infty\)-bounds for
nonnegative matrices, together with interpolation, give
\[
\|T_{k,m}b\|_{\ell^{p_s^*}(\mathcal P_{\delta,k})}
\lesssim
e^{-\frac{\alpha s}{64}|k-m|^2}
\|b\|_{\ell^{p_s^*}(\mathcal P_{\delta,m})}.
\]
Therefore, for every \(k\ge0\),
\begin{equation}\label{eq:u-sharp-ball}
\|(u^\sharp)_k\|_{\ell^{p_s^*}(\mathcal P_{\delta,k})}
\lesssim
\sum_{m\ge0}
\|T_{k,m}(u_{\mu,R})_m\|_{\ell^{p_s^*}(\mathcal P_{\delta,k})} \lesssim
\sum_{m\ge0}
e^{-\frac{\alpha s}{64}|k-m|^2}
\|(u_{\mu,R})_m\|_{\ell^{p_s^*}(\mathcal P_{\delta,m})}.
\end{equation}
Since the sequence \((e^{- \frac{\alpha s}{64} n^2})_{n\ge0}\) belongs to \(\ell^1(\mathbb N_0)\), and since
\(q_s^*\in[1,\infty]\), Young's inequality for discrete convolution gives
\[
\left\|
\left(
\|(u^\sharp)_k\|_{\ell^{p_s^*}(\mathcal P_{\delta,k})}
\right)_{k\ge0}
\right\|_{\ell^{q_s^*}}
\lesssim
\left\|
\left(
\|(u_{\mu,R})_k\|_{\ell^{p_s^*}(\mathcal P_{\delta,k})}
\right)_{k\ge0}
\right\|_{\ell^{q_s^*}},  \text{i.e., } \  
\|u^\sharp\|_{\ell^{p_s^*,q_s^*}} \lesssim \|u_{\mu,R}\|_{\ell^{p_s^*,q_s^*}}.
\]
Combining this with the first estimate proves the desired equivalence.
\end{proof}

\subsubsection{Discrete characterization of Carleson measures}
\label{subsubsec:carleson-discrete-characterization}

Combining the reductions with the comparison between \(u^\sharp\) and
\(u_{\mu,R}\) gives the discrete characterization and the equivalence of
assertions \((i)\) and \((ii)\) in Theorem~\ref{thm:carleson-intro}.
\begin{proposition}
\label{prop:carleson-discrete-characterization}
Assertions \((i)\) and \((ii)\) in Theorem~\ref{thm:carleson-intro} are equivalent.
More precisely, \(\mu\) is an \(s\)-Carleson measure for
\(\mathcal F_\alpha^{p,q}\) if and only if the discrete ball-mass sequence
\(u_{\mu,R}\) belongs to \(\ell^{p_s^*,q_s^*}\). Moreover,
\[
\|i_\mu\|^s
\asymp
\|u_{\mu,R}\|_{\ell^{p_s^*,q_s^*}},
\]
where the implicit constants depend only on \(\alpha,p,q,s,\delta\), and \(R\).
\end{proposition}

\begin{proof}
Assume first that \(\mu\) is an \(s\)-Carleson measure for \(\mathcal F_\alpha^{p,q}\).
By Proposition~\ref{prop:carleson-necessity-reduction}, the discrete embedding
\(
i_{u_{\mu,R}}:\ell^{p,q}\longrightarrow \ell^s(u_{\mu,R})
\)
is bounded and
\(
\|i_{u_{\mu,R}}\|\lesssim \|i_\mu\|\).
Therefore, by Proposition~\ref{prop:boundedness-iomega},
\[
u_{\mu,R}\in \ell^{p_s^*,q_s^*}
\qquad \text{and} \qquad
\|u_{\mu,R}\|_{\ell^{p_s^*,q_s^*}}
=
\|i_{u_{\mu,R}}\|^s
\lesssim
\|i_\mu\|^s.
\]

\smallskip
\noindent
Conversely, assume that \(u_{\mu,R}\in \ell^{p_s^*,q_s^*}\). By
Proposition~\ref{prop:equivalence-u-sharp-ball},
\[
u^\sharp\in \ell^{p_s^*,q_s^*}
\qquad\text{and}\qquad
\|u^\sharp\|_{\ell^{p_s^*,q_s^*}}
\asymp
\|u_{\mu,R}\|_{\ell^{p_s^*,q_s^*}}.
\]
Applying Proposition~\ref{prop:boundedness-iomega} to \(u^\sharp\), we obtain the
boundedness of
\(
i_{u^\sharp}:\ell^{p,q}\longrightarrow \ell^s(u^\sharp)
\)
and
\[
\|i_{u^\sharp}\|^s
=
\|u^\sharp\|_{\ell^{p_s^*,q_s^*}}
\asymp
\|u_{\mu,R}\|_{\ell^{p_s^*,q_s^*}}.
\]
Hence Proposition~\ref{prop:carleson-sufficiency-reduction} implies that \(\mu\) is an
\(s\)-Carleson measure for \(\mathcal F_\alpha^{p,q}\), and
\(
\|i_\mu\|
\lesssim
\|i_{u^\sharp}\|\).
Consequently,
\[
\|i_\mu\|^s
\lesssim
\|i_{u^\sharp}\|^s
\asymp
\|u_{\mu,R}\|_{\ell^{p_s^*,q_s^*}}.
\]
Combining the two estimates gives both the equivalence of \((i)\) and \((ii)\), and the
norm comparison
\(
\|i_\mu\|^s
\asymp
\|u_{\mu,R}\|_{\ell^{p_s^*,q_s^*}}\).
\end{proof}

\subsubsection{Discrete--continuous transfer for Carleson measures}
\label{subsubsec:carleson-discrete-to-continuous}

It remains to pass from the discrete ball-mass sequence \(u_{\mu,R}\) to the continuous
ball-mass function \(U_{\mu,R}\). This is a purely geometric transfer: the bounded
Carleson embedding has already been characterized discretely in
Proposition~\ref{prop:carleson-discrete-characterization}. The next proposition proves
the equivalence between assertions \((ii)\) and \((iii)\) in
Theorem~\ref{thm:carleson-intro}.

\begin{proposition}[Discrete--continuous equivalence]
\label{prop:carleson-discrete-continuous}
Assertions \((ii)\) and \((iii)\) in Theorem~\ref{thm:carleson-intro} are equivalent.
More precisely, the discrete ball-mass sequence \(u_{\mu,R}\) belongs to
\(\ell^{p_s^*,q_s^*}\) if and only if the continuous ball-mass function
\(U_{\mu,R}\) belongs to \(L_{\rm sh}^{p_s^*,q_s^*}(\C)\). Moreover,
\[
\|u_{\mu,R}\|_{\ell^{p_s^*,q_s^*}}
\asymp
\|U_{\mu,R}\|_{L_{\rm sh}^{p_s^*,q_s^*}},
\]
where the implicit constants depend only on \(\alpha,p,q,s,\delta\), and \(R\).
\end{proposition}

\begin{proof}
For simplicity, write
\[
u:=u_{\mu,R},\qquad U:=U_{\mu,R}.
\]
We shall use repeatedly the following local consequence of Proposition~\ref{prop:T1}: for
each fixed \(\rho>0\),
\begin{equation}
\label{eq:local-kappa-comparison}
\|\kappa_z\|_{\mathcal F_\alpha^{p,q}}^{-s}
\asymp
\|\kappa_w\|_{\mathcal F_\alpha^{p,q}}^{-s},
\qquad |z-w|\le \rho,
\end{equation}
where the implicit constants depend only on \(\alpha,p,q,s\), and \(\rho\).

\smallskip
\noindent
We prove
\[
\|U\|_{L_{\rm sh}^{p_s^*,q_s^*}}
\lesssim
\|u\|_{\ell^{p_s^*,q_s^*}}.
\]
Fix \(Q\in\mathcal P_\delta\) and \(z\in Q\). If \(w\in B(z,R)\) and \(Q'\in\mathcal P_\delta\)
is such that \(w\in Q'\), then Lemma~\ref{lem:geo-pockets}(c) gives
\[
|\zeta_{Q'}-\zeta_Q|
\le
|\zeta_{Q'}-w|+|w-z|+|z-\zeta_Q|
\le
R+2C_{\rm out}\delta .
\]
Set
\(
R_1:=R+2C_{\rm out}\delta\).
Then
\[
B(z,R)
\subset
\bigcup_{\zeta_{Q'}\in B(\zeta_Q,R_1)} B(\zeta_{Q'},R).
\]
Moreover, if \(z\in Q\) and \(\zeta_{Q'}\in B(\zeta_Q,R_1)\), then
\(
|z-\zeta_{Q'}|\le C_{\rm out}\delta+R_1\).
Hence, by \eqref{eq:local-kappa-comparison} with $\rho = C_{\rm out}\delta+R_1$,
\[
\|\kappa_z\|_{\mathcal F_\alpha^{p,q}}^{-s}
\lesssim
\|\kappa_{\zeta_{Q'}}\|_{\mathcal F_\alpha^{p,q}}^{-s}.
\]
Consequently, for every $z\in Q$,
\[
U(z)
=
\|\kappa_z\|_{\mathcal F_\alpha^{p,q}}^{-s} \, \mu(B(z,R))
\lesssim
\sum_{\zeta_{Q'}\in B(\zeta_Q,R_1)}
\|\kappa_{\zeta_{Q'}}\|_{\mathcal F_\alpha^{p,q}}^{-s}\, \mu(B(\zeta_{Q'},R)) = 
\sum_{\zeta_{Q'}\in B(\zeta_Q,R_1)} u(Q').
\]

\smallskip
\noindent
Now fix \(k\ge0\). Since all pockets have the same area by Lemma~\ref{lem:geo-pockets}(b),
and since Lemma~\ref{lem:geo-pockets}(d) implies that for each \(Q\) the set
\(
\{Q'\in\mathcal P_\delta: \zeta_{Q'} \in B(\zeta_Q, R_1)\}
\)
has uniformly bounded cardinality, we get for \(1 \leq p_s^\ast<\infty\),
\begin{align*}
\|U\|_{L^{p_s^\ast}(A_k)}^{p_s^\ast}
&=
\sum_{Q\in\mathcal P_{\delta,k}} \int_Q |U(z)|^{p_s^\ast}\,dA(z) \\
&\lesssim
\sum_{Q\in\mathcal P_{\delta,k}}
\left(
\sum_{\zeta_{Q'} \in B(\zeta_Q, R_1)} u(Q')
\right)^{p_s^\ast} \quad \lesssim
\sum_{Q\in\mathcal P_{\delta,k}}
\sum_{\zeta_{Q'} \in B(\zeta_Q, R_1)} u(Q')^{p_s^\ast}.
\end{align*}
On the other hand, for $\zeta_{Q'} \in B(\zeta_Q, R_1)$ with \(Q\in\mathcal P_{\delta,k}\) and \(Q'\in\mathcal P_{\delta,m}\), we have
\[
|m-k| \le |\zeta_Q - \zeta_{Q'}| + 1 \leq \lceil R_1\rceil  + 1 =: R_2.
\]
Using again Lemma~\ref{lem:geo-pockets}(d), now in the reverse direction, we find that for
each fixed \(Q'\) the number of pockets \(Q\in\mathcal P_{\delta,k}\) with
\(\zeta_Q \in B(\zeta_{Q'}, R_1)\) is uniformly bounded. Hence
\[
\|U\|_{L^{p_s^\ast}(A_k)}^{p_s^\ast}
\lesssim
\sum_{|m-k|\le R_2}
\sum_{Q'\in\mathcal P_{\delta,m}} u(Q')^{p_s^\ast}
=
\sum_{|m-k|\le R_2}
\|u_m\|_{\ell^{p_s^\ast}(\mathcal P_{\delta,m})}^{p_s^\ast}.
\]
Therefore, since \(p_s^{\ast} \geq 1\),
\[
\|U\|_{L^{p_s^\ast}(A_k)}
\lesssim
\sum_{|m-k|\le R_2}
\|u_m\|_{\ell^{p_s^\ast}(\mathcal P_{\delta,m})}.
\]
If \(p_s^\ast=\infty\), the same argument gives directly
\[
\|U\|_{L^\infty(A_k)}
\lesssim
\sum_{|m-k|\le R_2}
\|u_m\|_{\ell^\infty(\mathcal P_{\delta,m})}.
\]
Taking the outer \(\ell^{q_s^\ast}\)-norm in \(k\), and using Young's inequality for discrete
convolution with the finitely supported kernel
\(\mathbf 1_{\{|n|\le R_2\}}\), we conclude that
\[
\|U\|_{L_{\mathrm{sh}}^{p_s^\ast,q_s^\ast}}
\lesssim
\|u\|_{\ell^{p_s^\ast,q_s^\ast}}.
\]

\smallskip
\noindent
We prove the reverse estimate
\[
\|u\|_{\ell^{p_s^\ast,q_s^\ast}}
\lesssim
\|U\|_{L_{\mathrm{sh}}^{p_s^\ast,q_s^\ast}}.
\]
Fix \(Q\in\mathcal P_\delta\). By Fubini's theorem,
\[
\int_{B(\zeta_Q,R)} \mu(B(z,R))\,dA(z)
=
\int_{\C} |B(\zeta_Q,R)\cap B(w,R)|\,d\mu(w).
\]
If \(w\in B(\zeta_Q,R)\), then
\[
B\!\left(\frac{\zeta_Q+w}{2},\frac{R}{2}\right)
\subset
B(\zeta_Q,R)\cap B(w,R), \quad \text{ hence } \quad
|B(\zeta_Q,R)\cap B(w,R)|\ge \frac{\pi R^2}{4}.
\]
Thus
\[
\mu(B(\zeta_Q,R))
\le
\frac{4}{\pi R^2}
\int_{B(\zeta_Q,R)} \mu(B(z,R))\,dA(z).
\]
Moreover, for \(z\in B(\zeta_Q,R)\), \eqref{eq:local-kappa-comparison} with $\rho = R$ yields
\[
\|\kappa_{\zeta_Q}\|_{\mathcal F_\alpha^{p,q}}^{-s}
\lesssim
\|\kappa_z\|_{\mathcal F_\alpha^{p,q}}^{-s}.
\]
Therefore
\[
u(Q)
=
\|\kappa_{\zeta_Q}\|_{\mathcal F_\alpha^{p,q}}^{-s}\mu(B(\zeta_Q,R))
\lesssim
\int_{B(\zeta_Q,R)} U(z)\,dA(z).
\]

\smallskip
\noindent
Fix \(k\ge0\). If \(1 \leq p_s^\ast<\infty\), H\"older's inequality gives
\[
u(Q)^{p_s^\ast}
\lesssim
\int_{B(\zeta_Q,R)} U(z)^{p_s^\ast}\,dA(z),
\qquad Q\in\mathcal P_{\delta,k}.
\]
Summing over \(Q\in\mathcal P_{\delta,k}\), using Lemma~\ref{lem:geo-pockets}(d) to obtain
uniformly bounded overlap of the family
\(
\{B(\zeta_Q,R):\ Q\in\mathcal P_{\delta,k}\}\),
and noting that \(B(\zeta_Q,R)\) can meet only shells \(A_m\) with
\(
|m-k|\le \lceil R\rceil+1 =: R_3\),
we get
\[
\|u_k\|_{\ell^{p_s^\ast}(\mathcal P_{\delta,k})}^{p_s^\ast}
\lesssim
\sum_{|m-k|\le R_3}
\int_{A_m} U(z)^{p_s^\ast}\,dA(z)
=
\sum_{|m-k|\le R_3}
\|U\|_{L^{p_s^\ast}(A_m)}^{p_s^\ast}.
\]
Hence, since $p_s^{\ast} \geq 1$,
\[
\|u_k\|_{\ell^{p_s^\ast}(\mathcal P_{\delta,k})}
\lesssim
\sum_{|m-k|\le R_3}
\|U\|_{L^{p_s^\ast}(A_m)}.
\]
If \(p_s^\ast=\infty\), then
\[
u(Q)
\lesssim
\sup_{z\in B(\zeta_Q,R)} U(z)
\lesssim
\sum_{|m-k|\le R_3}
\|U\|_{L^\infty(A_m)},
\qquad Q\in\mathcal P_{\delta,k},
\]
and therefore
\[
\|u_k\|_{\ell^\infty(\mathcal P_{\delta,k})}
\lesssim
\sum_{|m-k|\le R_3}
\|U\|_{L^\infty(A_m)}.
\]
Taking the outer \(\ell^{q_s^\ast}\)-norm in \(k\), and using Young's inequality for discrete
convolution once more, we conclude that
\[
\|u\|_{\ell^{p_s^\ast,q_s^\ast}}
\lesssim
\|U\|_{L_{\mathrm{sh}}^{p_s^\ast,q_s^\ast}}.
\]
Combining the two estimates proves the desired equivalence and the norm comparability.
\end{proof}

\subsection{Vanishing Carleson measures}
\label{subsec:carleson-vanishing}

We prove Theorem~\ref{thm:vanishing-carleson-intro}.  The proof follows
the bounded case, with compactness and
Proposition~\ref{prop:compactness-iomega} replacing boundedness and
Proposition~\ref{prop:boundedness-iomega}.  The compact reductions,
weight comparison, and discrete--continuous transfer are given in
Propositions~\ref{prop:vanishing-discrete-characterization} and
\ref{prop:vanishing-discrete-continuous}.

\subsubsection{Atomic reductions for compactness}
\label{subsubsec:vanishing-atomic-reduction}

We begin with the compact analogues of the two reduction steps in
Subsection~\ref{subsubsec:carleson-atomic-reduction}.  The necessity
argument gives compactness of the discrete embedding associated with the
ball-mass weight \(u_{\mu,R}\), while the sufficiency argument reduces
compactness of the Carleson embedding to compactness of the auxiliary
discrete embedding associated with \(u^\sharp\).

\begin{proposition}
\label{prop:compact-necessity-ball}
If \(\mu\) is a vanishing \(s\)-Carleson measure for
\(\mathcal F_\alpha^{p,q}\), then the identity map
\(
i_{u_{\mu,R}}:\ell^{p,q}\longrightarrow \ell^s(u_{\mu,R})
\)
is compact.
\end{proposition}

\begin{proof}
Assume that \(\mu\) is a vanishing \(s\)-Carleson measure. If \(s<q\le\infty\), then
\(i_\mu\) is bounded, so Theorem~\ref{thm:carleson-intro} gives
\[
u_{\mu,R}\in\ell^{p_s^*,q_s^*}.
\]
By Proposition~\ref{prop:compactness-iomega}, this is equivalent to compactness of
\(i_{u_{\mu,R}}\).

\smallskip
\noindent
It remains to treat \(0<q\le s\). By Proposition~\ref{prop:compactness-iomega}, it is
enough to prove
\[
\|(u_{\mu,R})_k\|_{\ell^{p_s^*}(\mathcal P_{\delta,k})}\to0.
\]
\smallskip
\noindent
\emph{Step 1. } We record a consequence of compactness. If \((\mathcal G_k)_{k\ge0}\) is a family of subsets of
\(\mathcal F_\alpha^{p,q}\) such that
\[
\sup_{k\ge0}\sup_{g\in\mathcal G_k}\|g\|_{\mathcal F_\alpha^{p,q}}<\infty
\]
and, for every compact set \(K\Subset\C\),
\[
\sup_{g\in\mathcal G_k}\sup_{z\in K}|g(z)|\to0,
\qquad k\to\infty,
\]
then compactness of \(i_\mu\) implies
\[
\sup_{g\in\mathcal G_k}\|g\|_{L_\alpha^s(\mu)}\to0.
\]
Suppose this conclusion fails. Then there exist \(\varepsilon>0\), integers
\(k_j\to\infty\), and functions \(g_j\in\mathcal G_{k_j}\) such that
\[
\|g_j\|_{L_\alpha^s(\mu)}\ge \varepsilon,
\qquad j\ge1.
\]
By the uniform boundedness assumption, \((g_j)_j\) is bounded in
\(\mathcal F_\alpha^{p,q}\). Since \(i_\mu\) is compact, after passing to a subsequence we
may assume that \((g_j)_j\) converges in \(L_\alpha^s(\mu)\) to some function \(g\). 
On the other hand, the local uniform vanishing assumption implies that
\[
g_j(z)\to0,\qquad z\in\C.
\]
Set
\(
d\nu(z):=e^{-\frac{\alpha s}{2}|z|^2}\,d\mu(z)\),
so that \(L_\alpha^s(\mu)=L^s(\nu)\).
Since convergence in \(L^s(\nu)\) implies, after passing to a further
subsequence, almost everywhere convergence to the \(L^s(\nu)\)-limit, we must have
\(g=0\) \(\nu\)-almost everywhere, equivalently \(\mu\)-almost everywhere.
Hence \(g_j\to0\) in \(L_\alpha^s(\mu)\), contradicting
\(\|g_j\|_{L_\alpha^s(\mu)}\ge\varepsilon\). Therefore
\[
\sup_{g\in\mathcal G_k}\|g\|_{L_\alpha^s(\mu)}\to0.
\]

\smallskip
\noindent
\emph{Step 2.} 
For each \(k\ge0\), by the block norm estimate in Lemma \ref{lem:block-norm}, choose
\(
c^{(k)}=(c_Q^{(k)})_{Q\in\mathcal P_{\delta,k}}
\)
with
\[
\|c^{(k)}\|_{\ell^p(\mathcal P_{\delta,k})}=1
\qquad \text{ and } \qquad
\sum_{Q\in\mathcal P_{\delta,k}}
u_{\mu,R}(Q)|c_Q^{(k)}|^s
\ge
\frac12
\|(u_{\mu,R})_k\|_{\ell^{p_s^*}(\mathcal P_{\delta,k})}.
\]
Let \((\varepsilon_Q(t))_{Q\in\mathcal P_\delta}\) be a Rademacher family on \((0,1)\). For
\(k\ge0\) and \(0<t<1\), define
\[
g_{k,t}(z)
:=
\sum_{Q\in\mathcal P_{\delta,k}}
\varepsilon_Q(t)c_Q^{(k)}\widetilde\kappa_{\zeta_Q}(z).
\]
We verify three properties of this family.

\noindent
First, \((g_{k,t})_{k, t}\) is uniformly bounded in \(\mathcal F_\alpha^{p,q}\). If
\(d^{(k,t)}\) denotes the sequence supported on \(\mathcal P_{\delta,k}\) with
\[
d_Q^{(k,t)}=\varepsilon_Q(t)c_Q^{(k)},
\qquad Q\in\mathcal P_{\delta,k},
\]
then
\[
\|d^{(k,t)}\|_{\ell^{p,q}}=\|c^{(k)}\|_{\ell^p(\mathcal P_{\delta,k})}=1
\qquad \text{and} \qquad g_{k,t}=S_\delta d^{(k,t)}.
\]
Hence the boundedness of
\(S_\delta:\ell^{p,q}\to\mathcal F_\alpha^{p,q}\) gives
\[
\sup_{k,t}\|g_{k,t}\|_{\mathcal F_\alpha^{p,q}}<\infty.
\]
Second, the family \((g_{k,t})_{k, t}\) tends to zero locally uniformly, uniformly in \(t\) as \(k \to \infty\).
Let \(K\Subset\C\), and choose \(R_K>0\) such that \(K\subset B(0,R_K)\). As in the proof
of Lemma~\ref{lem:kernel-series-entire}, for all sufficiently large \(k\), for all \(Q\in\mathcal P_{\delta,k}\) and
all \(z\in K\), one has
\[
|\kappa_{\zeta_Q}(z)|
=
e^{\frac{\alpha}{2}|z|^2}
e^{-\frac{\alpha}{2}|z-\zeta_Q|^2}
\lesssim
e^{-\frac{\alpha}{4}(k-R_K)^2},
\]
where the implicit constant depends only on \(\alpha\) and \(K\). By
Proposition~\ref{prop:T1},
\[
\|\kappa_{\zeta_Q}\|_{\mathcal F_\alpha^{p,q}}^{-1}
\lesssim
(1+k)^{\left|\frac1p-\frac1q\right|},
\qquad Q\in\mathcal P_{\delta,k}.
\]
Hence
\[
\sup_{z\in K}|\widetilde\kappa_{\zeta_Q}(z)|
\lesssim
(1+k)^{\left|\frac1p-\frac1q\right|}
e^{-\frac{\alpha}{4}(k-R_K)^2},
\qquad Q\in\mathcal P_{\delta,k}.
\]
It remains only to estimate the \(\ell^1\)-size of \(c^{(k)}\). Since \(\delta\) is fixed, the explicit construction of the pocket tiling gives
\[
\#\mathcal P_{\delta,k}\lesssim 1+k,
\qquad k\ge0.
\]
Using this and
\(\|c^{(k)}\|_{\ell^p(\mathcal P_{\delta,k})}=1\),  we have
\[
\sum_{Q\in\mathcal P_{\delta,k}} |c_Q^{(k)}|
\lesssim 
(1+k)^{\max\{1-\frac1p,0\}},
\]
where, as usual, \(\frac{1}{\infty}=0\).  This follows from H\"older's
inequality when \(1\le p\le\infty\), and from
\(\sum_Q |c_Q^{(k)}|\le \sum_Q |c_Q^{(k)}|^p=1\) when \(0<p<1\).
Therefore
\[
\sup_{0<t<1}\sup_{z\in K}|g_{k,t}(z)|
\le
\sum_{Q\in\mathcal P_{\delta,k}}
|c_Q^{(k)}|
\sup_{z\in K}|\widetilde\kappa_{\zeta_Q}(z)| \lesssim
(1+k)^{\left|\frac1p-\frac1q\right| + \max\{1-\frac1p,0\}} e^{-\frac{\alpha}{4}(k-R_K)^2}
\to0.
\]

\smallskip
\noindent
Third, we claim that the randomized packets satisfy the lower estimate
\[
\int_0^1
\|g_{k,t}\|_{L_\alpha^s(\mu)}^s\,dt
\gtrsim
\sum_{Q\in\mathcal P_{\delta,k}}
u_{\mu,R}(Q)|c_Q^{(k)}|^s,
\]
where the implicit constant depends only on \(\alpha,s, \delta\), and \(R\), and is independent
of \(k\), \(\mu\), and \(c^{(k)}\).
This is the same randomization and localization argument used in the proof of
Proposition~\ref{prop:carleson-necessity-reduction}, restricted to the single shell
\(\mathcal P_{\delta,k}\). By Fubini's theorem and Khintchine's inequality,
\[
\int_0^1
\|g_{k,t}\|_{L_\alpha^s(\mu)}^s\,dt
\gtrsim
\int_{\C} G_k(z)\,d\mu(z),
\]
where
\[
G_k(z)
:=
\left(
\sum_{Q\in\mathcal P_{\delta,k}}
|c_Q^{(k)}|^2
|\widetilde\kappa_{\zeta_Q}(z)|^2
\right)^{\frac{s}{2}}
e^{-\frac{\alpha s}{2}|z|^2}.
\]
For \(z\in B(\zeta_Q,R)\), 
\[
G_k(z) \geq |c_Q^{(k)}|^s  |\widetilde\kappa_{\zeta_Q}(z)|^s e^{-\frac{\alpha s}{2}|z|^2}
\ge
e^{-\frac{\alpha s}{2}R^2}
|c_Q^{(k)}|^s
\|\kappa_{\zeta_Q}\|_{\mathcal F_\alpha^{p,q}}^{-s}.
\]
Hence, with the convention that the maximum over an empty set is zero,
\[
G_k(z)
\ge
e^{-\frac{\alpha s}{2}R^2}
\max_{Q\in\mathcal P_{\delta,k}:z\in B(\zeta_Q,R)}
\left(
|c_Q^{(k)}|^s
\|\kappa_{\zeta_Q}\|_{\mathcal F_\alpha^{p,q}}^{-s}
\right).
\]
Moreover, by the upper counting estimate in
Lemma~\ref{lem:geo-pockets}(d), there exists \(N_R\ge1\), depending only on \(R\) and
\(\delta\), such that
\[
\sum_{Q\in\mathcal P_{\delta,k}}
\mathbf 1_{B(\zeta_Q,R)}(z)
\le N_R,
\qquad z\in\C,\ k\ge0,
\]
which implies that
\[
\max_{Q\in\mathcal P_{\delta,k}:z\in B(\zeta_Q,R)}
\left(
|c_Q^{(k)}|^s
\|\kappa_{\zeta_Q}\|_{\mathcal F_\alpha^{p,q}}^{-s}
\right)
\ge
\frac1{N_R}
\sum_{Q\in\mathcal P_{\delta,k}}
|c_Q^{(k)}|^s
\|\kappa_{\zeta_Q}\|_{\mathcal F_\alpha^{p,q}}^{-s}
\mathbf 1_{B(\zeta_Q,R)}(z).
\]
Integrating this estimate with respect to \(d\mu(z)\), we obtain
\[
\int_{\C}G_k(z)\,d\mu(z)
\gtrsim
\sum_{Q\in\mathcal P_{\delta,k}}
|c_Q^{(k)}|^s
\|\kappa_{\zeta_Q}\|_{\mathcal F_\alpha^{p,q}}^{-s} \,
\mu(B(\zeta_Q,R))  =
\sum_{Q\in\mathcal P_{\delta,k}}
u_{\mu,R}(Q)|c_Q^{(k)}|^s.
\]
Together with the Khintchine lower bound above, this proves the claimed estimate.

\smallskip
\noindent
\emph{Step 3.}
Now set
\[
\mathcal G_k:=\{g_{k,t}:0<t<1\}.
\]
The first two properties in Step 2 and the compactness observation in Step 1 imply
\[
\sup_{0<t<1}\|g_{k,t}\|_{L_\alpha^s(\mu)}\to0.
\]
Therefore,
\[
\|(u_{\mu,R})_k\|_{\ell^{p_s^*}(\mathcal P_{\delta,k})}
\lesssim
\sum_{Q\in\mathcal P_{\delta,k}}
u_{\mu,R}(Q)|c_Q^{(k)}|^s \lesssim
\int_0^1\|g_{k,t}\|_{L_\alpha^s(\mu)}^s\,dt \le
\sup_{0<t<1}\|g_{k,t}\|_{L_\alpha^s(\mu)}^s
\to0.
\]
Thus the shellwise decay condition holds. By
Proposition~\ref{prop:compactness-iomega}, the embedding
\(i_{u_{\mu,R}}\) is compact.
\end{proof}

\begin{proposition}
\label{prop:compact-sufficiency-u-sharp}
If the identity map
\(
i_{u^\sharp}:\ell^{p,q}\longrightarrow \ell^s(u^\sharp)
\)
is compact, then \(\mu\) is a vanishing \(s\)-Carleson measure for
\(\mathcal F_\alpha^{p,q}\).
\end{proposition}

\begin{proof}
Let \(B\) denote the unit ball of \(\mathcal F_\alpha^{p,q}\). Since
\(A_\delta:\mathcal F_\alpha^{p,q}\to \ell^{p,q}\) is bounded, the set
\(
\mathcal B:=A_\delta(B)
\)
is bounded in \(\ell^{p,q}\). By compactness of
\(i_{u^\sharp}:\ell^{p,q}\to\ell^s(u^\sharp)\), the set \(\mathcal B\), viewed in
\(\ell^s(u^\sharp)\), is relatively compact.

\smallskip
\noindent
Fix \(\varepsilon>0\). After replacing \(\varepsilon\) by a smaller number if necessary, there exist finitely many sequences
\(d^{(1)},\ldots,d^{(N)}\in\mathcal B\) such that, for every \(f\in B\), one can choose
\(1\le j\le N\) with
\[
\|A_\delta f-d^{(j)}\|_{\ell^s(u^\sharp)}<\varepsilon.
\]
Set
\[
f_j:=S_\delta d^{(j)},\qquad 1\le j\le N.
\]
Since \(S_\delta A_\delta=I\) on \(\mathcal F_\alpha^{p,q}\) by \eqref{eq:inverse_S}, Lemma~\ref{lem:carleson-diagonalization}
gives
\[
\|f-f_j\|_{L_\alpha^s(\mu)}
=
\|S_\delta(A_\delta f-d^{(j)})\|_{L_\alpha^s(\mu)}
\le C
\|A_\delta f-d^{(j)}\|_{\ell^s(u^\sharp)}
< C\varepsilon,
\]
where $C$ is the implicit constant in Lemma \ref{lem:carleson-diagonalization}, independent of \(f\), \(j\), and \(\varepsilon\).

\smallskip
\noindent
Thus \(i_\mu(B)\) is totally bounded in \(L_\alpha^s(\mu)\). Since
\(L_\alpha^s(\mu)\) is complete, total boundedness implies relative compactness. Hence the embedding
\(
i_\mu:\mathcal F_\alpha^{p,q}\longrightarrow L_\alpha^s(\mu)\)
is compact. 
\end{proof}

\subsubsection{Compact comparison of the auxiliary and ball weights}
\label{subsubsec:vanishing-weight-comparison}

We transfer compactness between the auxiliary weight \(u^\sharp\) and the ball-mass
weight \(u_{\mu,R}\), the compact analogue of
Subsection~\ref{subsubsec:carleson-weight-comparison}.

\begin{proposition}
\label{prop:compact-equivalence-u-sharp-ball}
The identity map
\(
i_{u^\sharp}:\ell^{p,q}\longrightarrow \ell^s(u^\sharp)
\)
is compact if and only if the identity map
\(
i_{u_{\mu,R}}:\ell^{p,q}\longrightarrow \ell^s(u_{\mu,R})
\)
is compact.
\end{proposition}

\begin{proof}
We consider the two regimes separately.

\smallskip
\noindent
First assume that \(s<q\le\infty\). By Proposition~\ref{prop:compactness-iomega}, compactness
of \(i_\omega:\ell^{p,q}\to\ell^s(\omega)\) is equivalent to
\[
\omega\in\ell^{p_s^*,q_s^*}.
\]
Therefore the desired equivalence follows immediately from
Proposition~\ref{prop:equivalence-u-sharp-ball}.

\smallskip
\noindent
It remains to treat the regime \(0<q\le s\). By Proposition~\ref{prop:compactness-iomega},
the compactness of \(i_\omega\) is equivalent to the shellwise decay condition
\[
\|\omega_k\|_{\ell^{p_s^*}(\mathcal P_{\delta,k})}\to0, \qquad k \to \infty.
\]
The lower estimate in Lemma~\ref{lem:pointwise-comparison-u-sharp-ball} gives
\[
u_{\mu,R}(Q)\lesssim u^\sharp(Q),
\qquad Q\in\mathcal P_\delta.
\]
Hence
\[
\|(u_{\mu,R})_k\|_{\ell^{p_s^*}(\mathcal P_{\delta,k})}
\lesssim
\|(u^\sharp)_k\|_{\ell^{p_s^*}(\mathcal P_{\delta,k})},
\qquad k\ge0.
\]
Thus shellwise decay of \(u^\sharp\) implies shellwise decay of \(u_{\mu,R}\).

\smallskip
\noindent
Conversely, assume that
\[
\|(u_{\mu,R})_k\|_{\ell^{p_s^*}(\mathcal P_{\delta,k})}\to0.
\]
By the shellwise estimate \eqref{eq:u-sharp-ball} in the proof of
Proposition~\ref{prop:equivalence-u-sharp-ball}, which follows from the
upper estimate in Lemma~\ref{lem:pointwise-comparison-u-sharp-ball}, we have 
\[
\|(u^\sharp)_k\|_{\ell^{p_s^*}(\mathcal P_{\delta,k})}
\lesssim
\sum_{m\ge0}e^{-\frac{\alpha s}{64}|k-m|^2}
\|(u_{\mu,R})_m\|_{\ell^{p_s^*}(\mathcal P_{\delta,m})}.
\]
The right-hand side tends to zero: extend the sequence
\[
b_m:=\|(u_{\mu,R})_m\|_{\ell^{p_s^*}(\mathcal P_{\delta,m})},
\qquad m\ge0,
\]
by zero to \(\mathbb Z\), and set
\[
H_n:=e^{-\frac{\alpha s}{64}|n|^2},
\qquad n\in\mathbb Z.
\]
Then \(b\in c_0(\mathbb Z)\) and \(H\in\ell^1(\mathbb Z)\). Since convolution by an \(\ell^1(\mathbb Z)\)-sequence maps \(c_0(\mathbb Z)\) into itself,
the convolution term on the right-hand side tends to zero as \(k\to\infty\). Hence
\[
\|(u^\sharp)_k\|_{\ell^{p_s^*}(\mathcal P_{\delta,k})}
\to 0.
\]
Therefore \(i_{u^\sharp}\) is compact if and only if \(i_{u_{\mu,R}}\) is compact.
\end{proof}

\subsubsection{Discrete characterization of vanishing Carleson measures}
\label{subsubsec:vanishing-discrete-characterization}

Combining the compact reductions, the compact comparison, and
Proposition~\ref{prop:compactness-iomega} gives the discrete
characterization in terms of the ball-mass sequence \(u_{\mu,R}\).

\begin{proposition}
\label{prop:vanishing-discrete-characterization}
Assertions \((i)\) and \((ii)\) in Theorem~\ref{thm:vanishing-carleson-intro} are
equivalent. More precisely, \(\mu\) is a vanishing \(s\)-Carleson measure for
\(\mathcal F_\alpha^{p,q}\) if and only if the discrete ball-mass sequence \(u_{\mu,R}\)
satisfies:
\begin{itemize}
    \item[\textup{(a)}] \(
\|(u_{\mu,R})_k\|_{\ell^{p_s^*}(\mathcal P_{\delta,k})}\to0\) as \( k\to\infty\)
  when \(0<q\le s\);
  \item[\textup{(b)}] \(
u_{\mu,R}\in \ell^{p_s^*,q_s^*}
\) when \(s<q\le\infty\).
\end{itemize}
\end{proposition}

\begin{proof}
Assume first that \(\mu\) is a vanishing \(s\)-Carleson measure. By
Proposition~\ref{prop:compact-necessity-ball}, the identity map
\(
i_{u_{\mu,R}}:\ell^{p,q}\longrightarrow \ell^s(u_{\mu,R})
\)
is compact. Hence Proposition~\ref{prop:compactness-iomega} gives condition \({\rm(a)}\)
when \(0<q\le s\), and condition \({\rm(b)}\) when \(s<q\le\infty\).

\smallskip
\noindent
Conversely, assume that \(u_{\mu,R}\) satisfies the corresponding condition in
(a) or (b), according as \(0<q\le s\) or \(s<q\le\infty\).
 By
Proposition~\ref{prop:compactness-iomega}, the identity map
\(
i_{u_{\mu,R}}:\ell^{p,q}\longrightarrow \ell^s(u_{\mu,R})
\)
is compact. Proposition~\ref{prop:compact-equivalence-u-sharp-ball} then implies that
\(
i_{u^\sharp}:\ell^{p,q}\longrightarrow \ell^s(u^\sharp)
\)
is compact. Therefore Proposition~\ref{prop:compact-sufficiency-u-sharp} implies that
\(\mu\) is a vanishing \(s\)-Carleson measure.
\end{proof}

\subsubsection{Discrete--continuous transfer for vanishing conditions}
\label{subsubsec:vanishing-discrete-to-continuous}
It remains to pass from the discrete vanishing condition for \(u_{\mu,R}\) to the
continuous shell condition for \(U_{\mu,R}\). This is the compact analogue of
Subsection~\ref{subsubsec:carleson-discrete-to-continuous}. The case
\(s<q\le\infty\) follows directly from Proposition~\ref{prop:carleson-discrete-continuous};
the case \(0<q\le s\) follows from the shellwise estimates proved in the proof of that
proposition.

\begin{proposition}
\label{prop:vanishing-discrete-continuous}
Assertions \((ii)\) and \((iii)\) in
Theorem~\ref{thm:vanishing-carleson-intro} are equivalent. More precisely:
\begin{enumerate}
\item[\textup{(a)}] If \(0<q\le s\), then
\(
\|(u_{\mu,R})_k\|_{\ell^{p_s^*}(\mathcal P_{\delta,k})}\to0
\)
if and only if
\(
\|U_{\mu,R}\|_{L^{p_s^*}(A_k)}\to0
\) as \(k \to \infty\).
\item[\textup{(b)}] If \(s<q\le\infty\), then
\(
u_{\mu,R}\in\ell^{p_s^*,q_s^*}
\)
if and only if
\(
U_{\mu,R}\in L_{\rm sh}^{p_s^*,q_s^*}(\C)\).
\end{enumerate}
\end{proposition}

\begin{proof}
The equivalence in part \({\rm(b)}\) follows directly from
Proposition~\ref{prop:carleson-discrete-continuous}, restricted to the regime
\(s<q\le\infty\). It remains to prove the shellwise vanishing equivalence in
part \({\rm(a)}\).

\smallskip
\noindent
For simplicity, write
\[
u:=u_{\mu,R},\qquad U:=U_{\mu,R}.
\]
The proof of
Proposition~\ref{prop:carleson-discrete-continuous} gives an integer \(N_R\ge0\),
depending only on \(R\) and \(\delta\), such that for every \(k\ge0\),
\[
\|U\|_{L^{p_s^*}(A_k)}
\lesssim
\sum_{|m-k|\le N_R}
\|u_m\|_{\ell^{p_s^*}(\mathcal P_{\delta,m})}
\qquad \text{and} \qquad
\|u_k\|_{\ell^{p_s^*}(\mathcal P_{\delta,k})}
\lesssim
\sum_{|m-k|\le N_R}
\|U\|_{L^{p_s^*}(A_m)}.
\]
Since these are finite shell convolutions, they imply
\[
\|u_k\|_{\ell^{p_s^*}(\mathcal P_{\delta,k})}\to0
\quad\Longleftrightarrow\quad
\|U\|_{L^{p_s^*}(A_k)}\to0.
\]
This completes the proof.
\end{proof}

\section{Dual spaces}
\label{sec:dual-spaces}

This section proves the duality theorems stated in
Subsection~\ref{subsec:dual-spaces-intro}.  The main task is to identify
each continuous linear functional with a unique entire representative and
to determine the precise weighted mixed-norm space to which this
representative belongs.  The proof is intrinsic to the analytic Fock
spaces: representatives are constructed from reproducing kernels, and
their mixed-norm estimates are obtained by testing functionals on
families adapted to circles and annular shells.

\smallskip
\noindent
We organize the section as follows.  In
Subsection~\ref{subsec:dual-prelim}, we fix the representing spaces, Gaussian
and transported pairings, dilation notation, and canonical representing functions associated with functionals.  In Subsection~\ref{subsec:dual-equal}, we prove the
equal-parameter duality theorem under the Gaussian pairing
\(\langle\cdot,\cdot\rangle_\alpha\); this is the core of the section.  The
proof has two directions: functions in the proposed representing space
define bounded functionals, and every bounded functional gives a
canonical representative with the required weighted mixed-norm estimates.  The
polynomial correction \(\sigma(p,q)\) arises in this step from the interaction
between the angular and radial exponents.
In Subsection~\ref{subsec:dual-transport}, we transport the equal-parameter theorem
to arbitrary Gaussian pairing parameters by means of the dilation operator.
Subsection~\ref{subsec:dual-little} proves the corresponding
duality theorem for the little endpoint space \(f_\alpha^{p,\infty}\), where the endpoint \(q=\infty\) is reflected in the
outer \(\ell^1\)-type condition for the representative.

\subsection{Representing spaces, pairings, and preliminary facts}
\label{subsec:dual-prelim}
We begin by recalling the weighted mixed-norm scale introduced in
Subsection~\ref{subsec:dual-spaces-intro} and by fixing the notation
used throughout the duality arguments.  We record the annular
discretization of the weighted norm, the dilation lemma for changing
Gaussian parameters, the Gaussian and transported pairings, and the
canonical representing function associated with a continuous linear
functional.

\smallskip
\noindent
For \(\alpha>0\), \(\sigma\in\mathbb R\), and \(0<p,q\le\infty\), the
weighted mixed-norm Fock space \(\mathcal F_{\alpha;\sigma}^{p,q}\) consists of
all entire functions \(f\) such that
\[
\|f\|_{\mathcal F_{\alpha;\sigma}^{p,q}}
:=
\begin{cases}
\displaystyle
\left(
\int_0^\infty
\left[M_p(f,r)(1+r)^\sigma\right]^q
\,d\lambda_{\alpha q}(r)
\right)^{\frac{1}{q}}, & 0<q<\infty,\\[3mm]
\displaystyle
\sup_{r\ge0} M_p(f,r)e^{-\frac{\alpha}{2}r^2}(1+r)^\sigma,
& q=\infty,
\end{cases}
\]
is finite, with the usual interpretation when \(p=\infty\).  In particular,
\(\mathcal F_{\alpha;0}^{p,q}=\mathcal F_\alpha^{p,q}\).

\smallskip
\noindent
We also use the generalized conjugates \(p^*\) and \(q^*\) in the sense fixed in
the notation section, that is,
\[
p^*=
\begin{cases}
\infty, & 0<p\le1,\\
\dfrac{p}{p-1}, & 1<p<\infty,\\
1, & p=\infty,
\end{cases}
\qquad
q^*=
\begin{cases}
\infty, & 0<q\le1,\\
\dfrac{q}{q-1}, & 1<q<\infty,\\
1, & q=\infty.
\end{cases}
\]
We set the polynomial correction
\[
\sigma(p,q)
:=
\left(\frac1p-1\right)_+
-
\left(\frac1q-1\right)_+,
\qquad
(x)_+:=\max\{x,0\},
\]
with the convention \(\frac{1}{\infty}=0\).

\smallskip
\noindent
We record the annular form of the weighted norm, used below to pass from
radial estimates to shell estimates.
\begin{lemma}[Weighted annular discretization]
\label{lem:dual-weighted-annular}
Let \(\alpha>0\), \(\sigma\in\mathbb R\), and
\(0<p, q\le\infty\). For \(f\in H(\mathbb C)\) and \(k\ge0\), set
\[
 B_{p,\sigma}(f;k)
:=
\sup_{r\in[k,k+1)}
M_p(f,r)e^{-\frac{\alpha}{2}r^2}(1+r)^\sigma .
\]
Then
\[
\|f\|_{\mathcal F_{\alpha;\sigma}^{p,\infty}}
=
\sup_{k\ge0} B_{p,\sigma}(f;k) \qquad \text{ and }
\qquad
\|f\|_{\mathcal F_{\alpha;\sigma}^{p,q}}^q
\asymp
\sum_{k=0}^\infty
(1+k) B^q_{p,\sigma}(f;k) \quad (0<q<\infty),
\]
where the implicit constants depend only on \(\alpha,\sigma,p\), and \(q\).
\end{lemma}

\begin{proof}
The case \(q=\infty\) follows immediately from the definition.
Assume \(0<q<\infty\). Decomposing the defining integral over the
intervals \([k,k+1)\),
we obtain
\[
\begin{aligned}
\|f\|_{\mathcal F_{\alpha;\sigma}^{p,q}}^q
 & \asymp \, \sum_{k=0}^\infty
\int_k^{k+1}
\left[
M_p(f,r)e^{-\frac{\alpha}{2}r^2}(1+r)^\sigma
\right]^q r\,dr  \\
&\le \,
\sum_{k=0}^\infty
 B^q_{p,\sigma}(f;k)
\int_k^{k+1} r\,dr   \, \asymp \,
\sum_{k=0}^\infty
(1+k) B^q_{p,\sigma}(f;k).
\end{aligned}
\]
For the reverse inequality, we argue as in Proposition~\ref{prop:annular}.
By Lemma~\ref{lem:window},
\[
\sup_{r\in[k,k+1)}M_p^q(f,r)e^{-\frac{\alpha q}{2}r^2}
\lesssim
\frac1{1+k}
\int_{[k-2,k+2]\cap[0,\infty)}
M_p^q(f,s)e^{-\frac{\alpha q}{2}s^2}\,s\,ds .
\]
Multiplying by \((1+k)^{\sigma q}\), and using
\(
(1+k)^\sigma \asymp (1+s)^\sigma\) for \(s\in [k-2,k+2]\cap[0,\infty)\),
we obtain
\[
(1+k) B^q_{p,\sigma}(f;k)
\lesssim
\int_{[k-2,k+2]\cap[0,\infty)}
M_p^q(f,s)e^{-\frac{\alpha q}{2}s^2}(1+s)^{\sigma q}s\,ds .
\]
Summing over \(k\ge0\) and using the bounded overlap of the intervals
\([k-2,k+2]\cap[0,\infty)\), we get
\[
\sum_{k=0}^\infty (1+k) B^q_{p,\sigma}(f;k)
\lesssim
\int_0^\infty
M_p^q(f,s)e^{-\frac{\alpha q}{2}s^2}(1+s)^{\sigma q}s\,ds
\asymp
\|f\|_{\mathcal F_{\alpha;\sigma}^{p,q}}^q .
\]
This proves the desired equivalence.
\end{proof}

\smallskip
\noindent
We record the behavior of the weighted scale under dilations. This mechanism
passes from the equal-parameter pairing to
arbitrary Gaussian pairing parameters.
For \(\tau >0\), let \(D_{\tau}\) denote the dilation operator
\[
(D_{\tau}g)(z):=g({\tau} z), \qquad z\in\mathbb C.
\]
\begin{lemma}
\label{lem:dual-dilation-weighted}
Let \(\alpha>0\), \(\tau >0\), \(\sigma\in\mathbb R\), and \(0<p,q\le\infty\).
Then \(D_{\tau}\) is an isomorphism from
\(
\mathcal F_{\alpha/\tau^2;\sigma}^{p,q}\) onto \(
\mathcal F_{\alpha;\sigma}^{p,q}\),
and
\[
\|D_\tau f\|_{\mathcal F_{\alpha;\sigma}^{p,q}}
\asymp
\|f\|_{\mathcal F_{\alpha/\tau^2;\sigma}^{p,q}},
\]
where the implicit constants depend only on \(\tau,\sigma,p\), and \(q\).  The inverse of \(D_{\tau}\) is
\(D_{1/\tau}\).
\end{lemma}

\begin{proof}
Since \(M_p(D_\tau f,r)=M_p(f,\tau r)\), the assertion follows from the change of
variables \(s=\tau r\).  For \(0<q<\infty\),
\[
d\lambda_{\alpha q}(r)
=
\alpha q e^{-\frac{\alpha q}{2}r^2}r\,dr
=
\frac{\alpha q}{\tau^2}
e^{-\frac{\alpha q}{2\tau^2}s^2}s\,ds
=
d\lambda_{(\alpha/\tau^2)q}(s),
\]
and \((1+\frac{s}{\tau})^\sigma\asymp(1+s)^\sigma\).  The case \(q=\infty\) is
identical, using the same comparison of polynomial weights.
\end{proof}

\smallskip
\noindent
The dilation lemma transports equal-parameter duality to arbitrary Gaussian
pairing parameters.  We use the sesquilinear Gaussian pairing
\[
\langle f,g\rangle_\alpha
:=
\frac{\alpha}{\pi}
\int_{\mathbb C}
f(z)\overline{g(z)}e^{-\alpha|z|^2}\,dA(z),
\]
whenever the integral is absolutely convergent. For \(\alpha,\gamma>0\), we define the transported \(\gamma\)-pairing with
base parameter \(\alpha\) by
\[
\langle f,g\rangle_{\alpha\to\gamma}^{\rm tr}
:=
\big\langle f,D_{\alpha/\gamma}g\big\rangle_\alpha,
\]
whenever the right-hand side is well defined.

\smallskip
\noindent
We fix the kernel notation and define the canonical representing function.
Since several Gaussian parameters occur in the duality argument, we keep the
parameter in the notation for kernels:
\[
K_{\alpha,w}(z):=e^{\alpha\overline w z},
\qquad
\kappa_{\alpha,w}(z):=
e^{\alpha\overline w z-\frac{\alpha}{2}|w|^2}.
\]
By Proposition~\ref{prop:T1}, \(K_{\alpha,w}\in\mathcal F_\alpha^{p,q}\)
for all \(0<p,q\le\infty\). Moreover, the estimate
\[
M_p(K_{\alpha,w},r)e^{-\frac{\alpha}{2}r^2}
\le
e^{\alpha r|w|-\frac{\alpha}{2}r^2}
\to0,
\qquad r\to\infty,
\]
shows that \(K_{\alpha,w}\in f_\alpha^{p,\infty}\) for every
\(0<p\le\infty\).

\smallskip
\noindent
For a continuous linear functional  \(F\)  on
\(\mathcal F_\alpha^{p,q}\), where \(0<p\le\infty\) and \(0<q<\infty\), or on
the little space \(f_\alpha^{p,\infty}\), where \(0<p\le\infty\), we define
its \emph{canonical representing function} by
\[
g_F(w):=\overline{F(K_{\alpha,w})},
\qquad w\in\mathbb C.
\]
The next lemma shows that the canonical representing function is entire and
identifies its Taylor coefficients.
\begin{lemma}
\label{lem:dual-gF-expansion}
Let \(\alpha>0\), \(0<p\le\infty\), and \(0<q<\infty\). For each 
\(
F\in(\mathcal F_\alpha^{p,q})^*\)
or
\(
F\in(f_\alpha^{p,\infty})^*\), its canonical representing function \(g_F\) is entire and admits the
locally uniformly convergent expansion
\begin{equation}
\label{eq:dual-gF-expansion}
g_F(w)
=
\sum_{n=0}^{\infty}
\frac{\alpha^n}{n!}\overline{F(z^n)}\,w^n,
\qquad w\in\mathbb C.
\end{equation}
\end{lemma}

\begin{proof}
For \(w\in\mathbb C\), we use the formal expansion
\[
K_{\alpha,w}(z)
=
\sum_{n=0}^{\infty}
\frac{\alpha^n}{n!}\overline w^{\,n}z^n.
\]
We justify convergence in the underlying space, locally uniformly with
respect to \(w\).  The monomial norm computation from Lemma~\ref{lem-sh}
gives, for
\(n\ge1\),
\[
\|z^n\|_{\Fpq} \asymp \Big(\frac{n}{e\alpha}\Big)^{\frac{n}{2}}\; n^{\frac{1}{2q}}, \quad n \geq 1.
\]
Hence, by Stirling's formula, for every \(R>0\),
\[
\sum_{n=0}^{\infty}
\left(
\frac{\alpha^nR^n}{n!}
\|z^n\|_{\mathcal F_\alpha^{p,q}}
\right)^s
<\infty,
\qquad
s:=\min\{1,p,q\}.
\]
The same argument gives the corresponding summability in
\(f_\alpha^{p,\infty}\), with \(s=\min\{1,p\}\).
Since \(\mathcal F_\alpha^{p,q}\) and \(f_\alpha^{p,\infty}\) are complete, the kernel expansion converges
in these spaces, uniformly for \(|w|\le R\).

\smallskip
\noindent
Therefore, by the continuity of \(F\), we may apply \(F\) term by term and
obtain
\[
g_F(w)
=
\overline{F(K_{\alpha,w})}
=
\sum_{n=0}^{\infty}
\frac{\alpha^n}{n!}\overline{F(z^n)}\,w^n,
\]
locally uniformly in \(w\). Hence \(g_F\) is entire.
\end{proof}

\noindent
We shall also use the polynomial core in the completion and uniqueness
arguments. By
Corollary~\ref{cor:density-poly}, polynomials are dense in
\(\mathcal F_\alpha^{p,q}\) whenever \(0<p\le\infty\) and \(0<q<\infty\),
and they are dense in \(f_\alpha^{p,\infty}\).

\subsection{Equal-parameter duality}
\label{subsec:dual-equal}

We prove the duality theorem in the equal-parameter case
\(\alpha=\beta=\gamma\).  The proof establishes boundedness of the
Gaussian pairing, estimates the canonical representative in the proposed
space, and then uses polynomial density and Gaussian orthogonality for
representation and uniqueness.

\subsubsection{Boundedness of the pairing}
\label{subsubsec:dual-forward}
We start with the forward direction, namely the boundedness of the Gaussian
pairing on the proposed pair of spaces.

\begin{proposition}
\label{prop:dual-forward}
Let \(\alpha>0\), \(0<p\le\infty\), and \(0<q<\infty\). Then every 
\(g\in  \mathcal F_{\alpha;\sigma(p,q)}^{p^*,q^*}\) defines a continuous linear functional \(F_g\) on \(\mathcal F_\alpha^{p,q}\) by 
\[
F_g(f):=\langle f,g\rangle_\alpha,\qquad f\in\mathcal F_\alpha^{p,q},
\quad\text{and}\quad
\|F_g\|_{(\mathcal F_\alpha^{p,q})^*}
\lesssim
\|g\|_{\mathcal F_{\alpha;\sigma(p,q)}^{p^*,q^*}},
\]
where the implicit constant depends only on \(\alpha,p\), and \(q\).
\end{proposition}

\begin{proof}
Let \(f\in\mathcal F_\alpha^{p,q}\) and
\(g\in\mathcal F_{\alpha;\sigma(p,q)}^{p^*,q^*}\).  Put
\[
A_f(r):=M_p(f,r)e^{-\frac{\alpha}{2}r^2},
\qquad
B_g(r):=M_{p^*}(g,r)e^{-\frac{\alpha}{2}r^2}(1+r)^{\sigma(p,q)}.
\]
We use the following circle estimate.  If \(1\le p\le\infty\), then
H\"older's inequality on the circle gives
\[
\frac1{2\pi}\int_0^{2\pi}
|f(re^{i\theta})g(re^{i\theta})|\,d\theta
\le
M_p(f,r)M_{p^*}(g,r).
\]
If \(0<p<1\), then, using \eqref{eq:angular-comparison} with \(p_1=p\) and
\(p_2=1\),
\[
\frac1{2\pi}\int_0^{2\pi}
|f(re^{i\theta})g(re^{i\theta})|\,d\theta
\le
M_1(f,r)M_\infty(g,r)
\lesssim
(1+r)^{\frac1p-1}M_p(f,r)M_\infty(g,r).
\]
Hence, for all \(0<p\le\infty\),
\[
\frac1{2\pi}\int_0^{2\pi}
|f(re^{i\theta})g(re^{i\theta})|\,d\theta
\lesssim
(1+r)^{(\frac1p-1)_+}M_p(f,r)M_{p^*}(g,r).
\]
Therefore,
\[
\begin{aligned}
|\langle f,g\rangle_\alpha|
&\lesssim
\int_0^\infty
A_f(r)B_g(r)
(1+r)^{(\frac1p-1)_+-\sigma(p,q)}\,r\,dr  \\
&=
\int_0^\infty
A_f(r)B_g(r)
(1+r)^{(\frac1q-1)_+}\,r\,dr,
\end{aligned}
\]
because
\[
\sigma(p,q)
=
\left(\frac1p-1\right)_+
-
\left(\frac1q-1\right)_+ .
\]
If \(1<q<\infty\), then \(\left(\frac{1}{q}-1\right)_+=0\) and \(q^*=q'\).  H\"older's
inequality in the radial variable gives
\[
|\langle f,g\rangle_\alpha|
\lesssim
\left(\int_0^\infty A_f^q(r) r\,dr\right)^{\frac{1}{q}}
\left(\int_0^\infty B_g^{q'}(r) r\,dr\right)^{\frac{1}{q'}}
\lesssim
\|f\|_{\mathcal F_\alpha^{p,q}}
\|g\|_{\mathcal F_{\alpha;\sigma(p,q)}^{p^*,q^*}}.
\]
It remains to consider \(0<q\le1\).  In this case \(q^*=\infty\), and so
\[
B_g(r)\le
\|g\|_{\mathcal F_{\alpha;\sigma(p,q)}^{p^*,\infty}}.
\]
Thus
\[
|\langle f,g\rangle_\alpha|
\lesssim
\|g\|_{\mathcal F_{\alpha;\sigma(p,q)}^{p^*,\infty}}
\int_0^\infty A_f(r)(1+r)^{\frac1q-1}r\,dr.
\]
If \(q=1\), the last integral is directly comparable to
\(\|f\|_{\mathcal F_\alpha^{p,1}}\).  If \(0<q<1\), then by the sharp
\(M_p\)-estimate \eqref{eq-newest} in Theorem~\ref{thm:estimate},
\[
A_f(r)=M_p(f,r)e^{-\frac{\alpha}{2}r^2}
\lesssim
\|f\|_{\mathcal F_\alpha^{p,q}}(1+r)^{-\frac1q}.
\]
Thus
\[
A_f^{1-q}(r)
\lesssim
\|f\|_{\mathcal F_\alpha^{p,q}}^{1-q}
(1+r)^{1-\frac1q}.
\]
Consequently,
\[
\begin{aligned}
\int_0^\infty A_f(r)(1+r)^{\frac1q-1}r\,dr
&=
\int_0^\infty A_f^q(r) A_f^{1-q}(r)
(1+r)^{\frac1q-1}r\,dr \\
&\lesssim
\|f\|_{\mathcal F_\alpha^{p,q}}^{1-q}
\int_0^\infty A_f(r)^q r\,dr  
\quad \asymp \quad
\|f\|_{\mathcal F_\alpha^{p,q}} .
\end{aligned}
\]
Hence, for all \(0<q\le1\),
\[
|\langle f,g\rangle_\alpha|
\lesssim
\|f\|_{\mathcal F_\alpha^{p,q}}
\|g\|_{\mathcal F_{\alpha;\sigma(p,q)}^{p^*,\infty}}.
\]
The preceding estimates show that the Gaussian integral is absolutely
convergent and that \(F_g(f):=\langle f,g\rangle_\alpha\) defines a continuous
linear functional on \(\mathcal F_\alpha^{p,q}\), with
\[
\|F_g\|_{(\mathcal F_\alpha^{p,q})^*}
\lesssim
\|g\|_{\mathcal F_{\alpha;\sigma(p,q)}^{p^*,q^*}}.
\]
This completes the proof.
\end{proof}

\subsubsection{Representer estimates}
\label{subsubsec:dual-representer-estimates}

We prove the converse estimates.  Given a continuous linear functional
\(F\), we estimate the canonical representing function $g_F$ of $F$
in the proposed representing space. The proof is divided according to the
inner exponent \(p\).  When \(0<p\le1\), the representing space has inner
exponent \(p^*=\infty\), and the argument relies on pointwise kernel testing
together with one-point-per-shell synthesis.  When \(1<p\le\infty\), one uses
dual testing on the angular variable, encoded through the circle-test operators \(T_r\).

\smallskip
\noindent
We begin with the quasi-Banach inner range \(0<p\le1\).  The following
one-point-per-shell synthesis estimate is the basic testing device needed to
recover the outer \(q^*\)-summability of \(g_F\).

\begin{lemma} 
\label{lem:dual-one-point-shell}
Let \(\alpha>0\), \(0<p\le1<q<\infty\), and let
\(\{w_k\}_{k\ge0}\subset\mathbb C\) satisfy \(k\le |w_k|<k+1\).  If
\(b=(b_k)_{k\ge0}\) satisfies
\[
\sum_{k=0}^\infty |b_k|^q(1+k)^{1-\frac qp}<\infty,
\]
then the series
\(
\sum_{k=0}^\infty b_k\kappa_{\alpha,w_k}
\)
converges  in
\(\mathcal F_\alpha^{p,q}\). Moreover, its sum satisfies
\[
\left\|
\sum_{k=0}^\infty b_k\kappa_{\alpha,w_k}
\right\|_{\mathcal F_\alpha^{p,q}}
\lesssim
\left(
\sum_{k=0}^\infty |b_k|^q(1+k)^{1-\frac qp}
\right)^{\frac1q},
\]
where the implicit constant depends only on \(\alpha,p\), and \(q\).
\end{lemma}

\begin{proof}
Let
\[
S_b:=
\left(
\sum_{k=0}^{\infty}|b_k|^q(1+k)^{1-\frac{q}{p}}
\right)^{1/q}<\infty.
\]
For a finite set \(E\subset\mathbb N_0\), put
\[
f_E(z):=\sum_{k\in E}b_k\kappa_{\alpha,w_k}(z).
\]
We prove the uniform estimate
\[
\|f_E\|_{\mathcal F_\alpha^{p,q}}
\lesssim
\left(
\sum_{k\in E}|b_k|^q(1+k)^{1-\frac{q}{p}}
\right)^{1/q}.
\]
By the single-kernel shell estimate obtained from the packet profile
estimate \eqref{eq:packet-profile} and used in \eqref{eq:jet-atom-profile},
\[
\sup_{r\in[\ell,\ell+1)}
M_p^p(\kappa_{\alpha,w_k},r)e^{-\frac{\alpha p}{2}r^2}
\lesssim
(1+k)^{-1}e^{-\frac{\alpha p}{8}|\ell-k|^2},
\qquad \ell,k\ge0.
\]
where the implicit constant depends only on $\alpha$ and $p$. Since \(0<p\le1\), \(p\)-subadditivity gives
\[
\begin{aligned}
B_p^p(f_E;\ell)
&=
\sup_{r\in[\ell,\ell+1)}
M_p^p(f_E,r)e^{-\frac{\alpha p}{2}r^2}  \\
&\le
\sum_{k\in E}|b_k|^p
\sup_{r\in[\ell,\ell+1)}
M_p^p(\kappa_{\alpha,w_k},r)e^{-\frac{\alpha p}{2}r^2}  \lesssim
\sum_{k\in E}|b_k|^p(1+k)^{-1}
e^{-\frac{\alpha p}{8}|\ell-k|^2}.
\end{aligned}
\]
Set
\[
Y_k:=|b_k|^p(1+k)^{\frac{p}{q}-1}\mathbf 1_E(k).
\]
Then
\(
|b_k|^p(1+k)^{-1}
=
(1+k)^{-\frac{p}{q}}Y_k\).
Hence
\[
B_p^p(f_E;\ell)
\lesssim
\sum_{k=0}^{\infty}
e^{-\frac{\alpha p}{8}|\ell-k|^2}(1+k)^{-\frac{p}{q}}Y_k.
\]
By the generalized Schur test in Lemma~\ref{lem:gen-schur-test} applied
with exponent \(\frac{q}{p}\), we get
\[
\sum_{\ell=0}^{\infty}(1+\ell)B_p^q(f_E;\ell)
=
\sum_{\ell=0}^{\infty}(1+\ell)
\bigl(B_p^p(f_E;\ell)\bigr)^{\frac{q}{p}}
\lesssim
\sum_{k=0}^{\infty}Y_k^{\frac{q}{p}} = \sum_{k\in E}|b_k|^q(1+k)^{1-\frac{q}{p}}.
\]
By the annular discretization in Proposition~\ref{prop:annular},
\[
\|f_E\|_{\mathcal F_\alpha^{p,q}}
\lesssim
\left(
\sum_{k\in E}|b_k|^q(1+k)^{1-\frac{q}{p}}
\right)^{\frac{1}{q}}.
\]
Now let
\[
f_n:=\sum_{k=0}^{n}b_k\kappa_{\alpha,w_k},\qquad n\in\mathbb N.
\]
Applying the preceding estimate to the differences \(f_m-f_n\), we see
that \((f_n)_{n \geq 0}\) is a Cauchy sequence in \(\mathcal F_\alpha^{p,q}\), because
\[
\sum_{k=0}^{\infty}|b_k|^q(1+k)^{1-\frac{q}{p}}<\infty.
\]
Since the space
\(\mathcal F_\alpha^{p,q}\) is complete,  \(f_n\) converges in
\(\mathcal F_\alpha^{p,q}\) to some \(f\). Passing to the limit in the
finite-set estimate gives
\[
\|f\|_{\mathcal F_\alpha^{p,q}}
\lesssim
\left(
\sum_{k=0}^{\infty}|b_k|^q(1+k)^{1-\frac{q}{p}}
\right)^{\frac1q}.
\]
This proves the lemma.
\end{proof}

\begin{proposition} 
\label{prop:dual-rep-inner-quasi}
Let \(\alpha>0\), \(0<p\le1\), and \(0<q<\infty\).  For each
\(F\in(\mathcal F_\alpha^{p,q})^*\), the canonical representing function  $g_F$ of $F$
belongs to 
\(
\mathcal F_{\alpha;\sigma(p,q)}^{\infty,q^*}
\)
and satisfies
\[
\|g_F\|_{\mathcal F_{\alpha;\sigma(p,q)}^{\infty,q^*}}
\lesssim
\|F\|_{(\mathcal F_\alpha^{p,q})^*},
\]
where the implicit constant depends only on \(\alpha,p\), and \(q\).
\end{proposition}

\begin{proof}
By the preliminary discussion in Subsection~\ref{subsec:dual-prelim}, the function
\(g_F\) is entire.  We split the proof according to the outer exponent \(q\).

\smallskip
\noindent
First suppose that \(0<q\le1\).  Then \(q^*=\infty\) and
\(
\sigma(p,q)=\frac1p-\frac1q\).
Using the kernel-size estimate in Proposition \ref{prop:T1}, we get
\[
|g_F(w)|
=
|F(K_{\alpha,w})|  
\le
\|F\|_{(\mathcal F_\alpha^{p,q})^*}
\|K_{\alpha,w}\|_{\mathcal F_\alpha^{p,q}}  \lesssim
\|F\|_{(\mathcal F_\alpha^{p,q})^*}
(1+|w|)^{\frac1q-\frac1p}e^{\frac{\alpha}{2}|w|^2}.
\]
Therefore
\[
\sup_{w\in\mathbb C}
|g_F(w)|e^{-\frac{\alpha}{2}|w|^2}
(1+|w|)^{\frac1p-\frac1q}
\lesssim
\|F\|_{(\mathcal F_\alpha^{p,q})^*}.
\]
This is exactly
\[
\|g_F\|_{\mathcal F_{\alpha;\frac1p-\frac1q}^{\infty,\infty}}
\lesssim
\|F\|_{(\mathcal F_\alpha^{p,q})^*},
\]
which proves the assertion in the range \(0<q\le1\).

\smallskip
\noindent
It remains to consider \(1<q<\infty\).  In this case \(q^*=q'\) and
\(
\sigma(p,q)=\frac1p-1\).
For \(k\ge0\), set
\[
H_k
:= B_{\infty, \frac{1}{p} - 1}(g_F; k) = 
\sup_{w\in A_k}
|g_F(w)|e^{-\frac{\alpha}{2}|w|^2}
(1+|w|)^{\frac1p-1}.
\]
We prove
\[
\left(\sum_{k=0}^\infty (1+k)H_k^{q'}\right)^{\frac{1}{q'}}
\lesssim
\|F\|_{(\mathcal F_\alpha^{p,q})^*},
\]
where the implicit constant depends only on $\alpha, p$, and $q$. 
Fix \(\varepsilon\in(0,1)\). For each \(k\ge0\), choose
\(w_k\in A_k\) such that
\[
|g_F(w_k)|e^{-\frac{\alpha}{2}|w_k|^2}
(1+|w_k|)^{\frac1p-1}
\ge
(1-\varepsilon)H_k .
\]
We put
\[
\xi_k :=
\begin{cases}
\dfrac{g_F(w_k)}{|g_F(w_k)|}, & g_F(w_k)\ne0,\\[2mm]
1, & g_F(w_k)=0.
\end{cases}
\]
Let \(a=(a_k)_{k\ge0}\) be a nonnegative sequence satisfying
\[
\sum_{k=0}^\infty a_k^q(1+k)^{1-q}<\infty .
\]
Define
\[
b_k:=a_k\xi_k(1+|w_k|)^{\frac1p-1},
\qquad k\ge0.
\]
Since \(k\le |w_k|<k+1\), we have
\[
\sum_{k=0}^\infty |b_k|^q(1+k)^{1-\frac qp}
\asymp
\sum_{k=0}^\infty a_k^q(1+k)^{1-q}<\infty .
\]
By Lemma~\ref{lem:dual-one-point-shell}, the series
\[
\sum_{k=0}^\infty
a_k\xi_k(1+|w_k|)^{\frac1p-1}\kappa_{\alpha,w_k}(z)
\]
defines a function $f_a$  in \(\mathcal F_\alpha^{p,q}\), and
\[
\|f_a\|_{\mathcal F_\alpha^{p,q}}
\lesssim
\left(
\sum_{k=0}^\infty a_k^q(1+k)^{1-q}
\right)^{\frac{1}{q}}.
\]
Using the continuity of \(F\) and the choice of \(w_k\), we get
\[
\begin{aligned}
F(f_a)
&=
\sum_{k=0}^\infty
a_k\xi_k(1+|w_k|)^{\frac1p-1}
e^{-\frac{\alpha}{2}|w_k|^2}\overline{g_F(w_k)} \\
&=
\sum_{k=0}^\infty
a_k(1+|w_k|)^{\frac1p-1}
e^{-\frac{\alpha}{2}|w_k|^2}|g_F(w_k)| \ge
(1-\varepsilon)\sum_{k=0}^\infty a_kH_k.
\end{aligned}
\]
On the other hand,
\[
|F(f_a)|
\le
\|F\|_{(\mathcal F_\alpha^{p,q})^*}
\|f_a\|_{\mathcal F_\alpha^{p,q}}
\lesssim
\|F\|_{(\mathcal F_\alpha^{p,q})^*}
\left(
\sum_{k=0}^\infty a_k^q(1+k)^{1-q}
\right)^{\frac{1}{q}}.
\]
Therefore
\[
(1-\varepsilon)\sum_{k=0}^\infty a_kH_k
\lesssim
\|F\|_{(\mathcal F_\alpha^{p,q})^*}
\left(
\sum_{k=0}^\infty a_k^q(1+k)^{1-q}
\right)^{\frac{1}{q}}.
\]
Letting \(\varepsilon\downarrow0\), we obtain
\[
\sum_{k=0}^\infty a_kH_k
\lesssim
\|F\|_{(\mathcal F_\alpha^{p,q})^*}
\left(
\sum_{k=0}^\infty a_k^q(1+k)^{1-q}
\right)^{\frac{1}{q}}
\]
for every nonnegative sequence \(a=(a_k)_{k\ge0}\) satisfying
\(
\sum_{k=0}^\infty a_k^q(1+k)^{1-q}<\infty\).
Since \(H_k\ge0\), applying this estimate to \(|a|\) gives
\[
\left|\sum_{k=0}^\infty a_kH_k\right|
\le
\sum_{k=0}^\infty |a_k|H_k
\lesssim
\|F\|_{(\mathcal F_\alpha^{p,q})^*}
\left(
\sum_{k=0}^\infty |a_k|^q(1+k)^{1-q}
\right)^{\frac{1}{q}}
\]
for every \(a\in L^q(\mathbb N_0,d\mu)\), where
\(
\mu(\{k\})=(1+k)^{1-q}\), \( k\in\mathbb N_0\).
Thus the functional
\[
a\longmapsto \sum_{k=0}^\infty a_kH_k
\]
is bounded on \(L^q(\mathbb N_0,d\mu)\).
By the \(L^q\)-duality theorem \cite[Theorem~6.6]{Rud87}, it follows that
\[
\left(
\sum_{k=0}^\infty (1+k)H_k^{q'}
\right)^{\frac{1}{q'}} = \left(
\sum_{k=0}^\infty
\left[
\frac{H_k}{(1+k)^{1-q}}
\right]^{q'}
(1+k)^{1-q}
\right)^{\frac{1}{q'}}
\lesssim
\|F\|_{(\mathcal F_\alpha^{p,q})^*}.
\]
By the weighted annular discretization Lemma~\ref{lem:dual-weighted-annular},
this is precisely
\[
\|g_F\|_{\mathcal F_{\alpha;\frac1p-1}^{\infty,q'}}
\lesssim
\|F\|_{(\mathcal F_\alpha^{p,q})^*}.
\]
The proof is complete.
\end{proof}

\smallskip
\noindent
We turn to the range \(1<p\le\infty\).  In this regime the inner dual
exponent \(p^*\) is finite when \(1<p<\infty\) and equals \(1\) when
\(p=\infty\).  Thus pointwise kernel testing alone is no longer sufficient to
recover the full inner \(L^{p^*}\)-norm of \(g_F\) on each circle.  We therefore
use a circlewise testing argument, based on the functions \(T_r(a)\), to capture
the angular dual norm before summing in the radial variable.

\begin{lemma} 
\label{lem:dual-circle-tests}
Let \(\alpha>0\), \(1<p\le\infty\), and \(0<q<\infty\).  For each fixed \(r \geq 0\), the map $T_r$ defined as 
\[
(T_r a)(z)
:=
\frac{1}{2\pi}\int_0^{2\pi}
a(\theta)\kappa_{\alpha,re^{i\theta}}(z)\,d\theta, \qquad a\in L^p(\T),
\]
is a bounded linear operator from \(L^p(\T)\) into \(\mathcal F_{\alpha}^{p, q}\), and
\[
\|T_r a\|_{\mathcal F_\alpha^{p,q}}
\lesssim
(1+r)^{\frac1q-1}\|a\|_{L^p(\T)},
\]
where the implicit constant depends only on \(\alpha,p\), and \(q\).
Moreover, for each \(F\in(\mathcal F_\alpha^{p,q})^*\), one has
\[
F(T_r a)
=
e^{-\frac{\alpha}{2}r^2}
\frac1{2\pi}\int_0^{2\pi}
a(\theta)\overline{g_F(re^{i\theta})}\,d\theta ,
\]
where  \(g_F\) is the canonical representing function of $F$.
\end{lemma}

\begin{proof}
The case \(r=0\) is immediate: since \(\kappa_{\alpha,0}\equiv1\), we have
\[
(T_0a)(z)=\frac1{2\pi}\int_0^{2\pi}a(\theta)\,d\theta .
\]
Thus \(T_0a\) is constant,
\[
\|T_0a\|_{\mathcal F_\alpha^{p,q}}
=
\left|
\frac1{2\pi}\int_0^{2\pi}a(\theta)\,d\theta
\right|
\le
\|a\|_{L^p(\T)},
\]
and the asserted norm estimate follows because
\((1+0)^{\frac{1}{q}-1}=1\). Moreover,
\[
F(T_0a)
=
\left(\frac1{2\pi}\int_0^{2\pi}a(\theta)\,d\theta\right)\,F(K_{\alpha,0})
=
\frac1{2\pi}\int_0^{2\pi}
a(\theta)\overline{g_F(0)}\,d\theta,
\]
which is the desired identity at \(r=0\).
Hence, in the rest of the proof, we may assume \(r>0\).

\smallskip
\noindent
The defining integral for \(T_ra\) converges locally uniformly in
\(z\), since \(a\in L^1(\T)\) and the kernel is uniformly bounded on
compact subsets of \(\mathbb C\). Hence \(T_ra\) is entire.

\smallskip
\noindent
For the norm estimate, put
\[
G_r(z):=\exp\left(\alpha rz-\frac{\alpha}{2}r^2\right).
\]
Fix \(\rho>0\) and write \(z=\rho e^{i\varphi}\). Set
\[
k_{\rho,r}(\psi):=G_r(\rho e^{i\psi}),\qquad \psi\in\mathbb R.
\]
Then
\[
(T_ra)(\rho e^{i\varphi})
=
\frac1{2\pi}\int_0^{2\pi}
a(\theta)k_{\rho,r}(\varphi-\theta)\,d\theta.
\]
Thus \(\varphi\mapsto (T_ra)(\rho e^{i\varphi})\) is the normalized circular
convolution \(a*k_{\rho,r}\). Young's inequality on the circle gives, for \(1<p\le\infty\),
\[
M_p(T_ra,\rho)
=
\|a*k_{\rho,r}\|_{L^p(\T)}
\le
\|a\|_{L^p(\T)}\|k_{\rho,r}\|_{L^1(\T)}.
\]
Since
\(
\|k_{\rho,r}\|_{L^1(\T)}=M_1(G_r,\rho)\),
we obtain
\[
M_p(T_ra,\rho)
\le
\|a\|_{L^p(\T)}M_1(G_r,\rho).
\]
Therefore, 
\[
\|T_ra\|_{\mathcal F_\alpha^{p,q}}
\le
\|a\|_{L^p(\T)}\,
\|G_r\|_{\mathcal F_\alpha^{1,q}}.
\]
By Proposition~\ref{prop:T1}, applied with inner exponent \(1\), we have
\[
\|G_r\|_{\mathcal F_\alpha^{1,q}}
\asymp
(1+r)^{\frac1q-1}.
\]
Thus
\[
\|T_ra\|_{\mathcal F_\alpha^{p,q}}
\lesssim
(1+r)^{\frac1q-1}\|a\|_{L^p(\T)},
\]
which also shows that \(T_ra\in\mathcal F_\alpha^{p,q}\).

\smallskip
\noindent
It remains to prove the identity involving \(F\):
\[
F(T_ra)
=
e^{-\frac{\alpha}{2}r^2}
\frac1{2\pi}\int_0^{2\pi}
a(\theta)\overline{g_F(re^{i\theta})}\,d\theta, \qquad a\in L^p(\T).
\]

\smallskip
\noindent
First suppose that \(a\) is a trigonometric polynomial. Since
\[
\kappa_{\alpha,re^{i\theta}}(z)
=
e^{-\frac{\alpha}{2}r^2}
K_{\alpha,re^{i\theta}}(z) = e^{-\frac{\alpha}{2}r^2}\sum_{n=0}^{\infty}
\frac{\alpha^n r^n}{n!}e^{-in\theta}z^n,
\]
\(T_ra\) is then a finite linear combination of monomials. Thus we may
apply \(F\) term by term to obtain
\[
F(T_ra)
=
e^{-\frac{\alpha}{2}r^2}
\frac1{2\pi}\int_0^{2\pi}
a(\theta)F(K_{\alpha,re^{i\theta}})\,d\theta.
\]
Since \(F(K_{\alpha,re^{i\theta}})=\overline{g_F(re^{i\theta})}\), this gives
\[
F(T_ra)
=
e^{-\frac{\alpha}{2}r^2}
\frac1{2\pi}\int_0^{2\pi}
a(\theta)\overline{g_F(re^{i\theta})}\,d\theta
\]
for trigonometric polynomials \(a\).

\smallskip
\noindent
Now let \(a\in L^p(\T)\) and 
\(\sigma_n a\) denote the \(n\)-th Fej\'er mean of \(a\). Since
\(a\in L^p(\T)\subset L^1(\T)\) for \(1<p\le\infty\), and since
the Fej\'er kernels form a summability kernel, Katznelson's theorem for
homogeneous Banach spaces gives
\[
\sigma_n a\to a
\quad\text{in }L^1(\T);
\]
see \cite[Chapter~I, Theorem~2.11]{Kat04}.
Moreover, Young's inequality on the circle also gives
\[
M_p(T_rb,s)
\le
\|b\|_{L^1(\T)}M_p(G_r,s),
\qquad b\in L^1(\T).
\]
Hence,
\[
\|T_rb\|_{\mathcal F_\alpha^{p,q}}
\le
\|G_r\|_{\mathcal F_\alpha^{p,q}}\|b\|_{L^1([0,2\pi])}.
\]
Applying this to \(b=\sigma_n a-a\), we get
\[
T_r(\sigma_n a)\to T_ra
\quad\text{in }\mathcal F_\alpha^{p,q}.
\]
Therefore, by the continuity of \(F\),
\[
F(T_r\sigma_n a)\to F(T_ra).
\]
On the other hand, since \(g_F\) is entire, the function
\(\theta\mapsto \overline{g_F(re^{i\theta})}\) is bounded on
\([0,2\pi]\). Thus the \(L^1\)-convergence of \(\sigma_n a\) also implies
\[
\int_0^{2\pi}
\sigma_n a(\theta)\overline{g_F(re^{i\theta})}\,d\theta
\to
\int_0^{2\pi}
a(\theta)\overline{g_F(re^{i\theta})}\,d\theta.
\]
Passing to the limit in the identity for the trigonometric polynomials
\(\sigma_n a\), we obtain
\[
F(T_ra)
=
e^{-\frac{\alpha}{2}r^2}
\frac1{2\pi}\int_0^{2\pi}
a(\theta)\overline{g_F(re^{i\theta})}\,d\theta.
\]
The proof is complete.
\end{proof}

\begin{lemma} 
\label{lem:dual-circle-extraction}
Let \(\alpha>0\), \(1<p\le\infty\), and \(0<q<\infty\). Let
\(F\in(\mathcal F_\alpha^{p,q})^*\), and let \(g_F\) be the canonical
representing function of \(F\). Then, for every \(r>0\),
\[
M_{p^*}(g_F,r)e^{-\frac{\alpha}{2}r^2}
\lesssim
(1+r)^{\frac1q-1}\|F\|_{(\mathcal F_\alpha^{p,q})^*},
\]
where the implicit constant depends only on $\alpha, p$, and $q$.
\end{lemma}

\begin{proof}
Fix \(r>0\). By Lemma~\ref{lem:dual-circle-tests} and the boundedness of \(F\), for every
\(a\in L^p(\T)\), we get
\[
\begin{aligned}
e^{-\frac{\alpha}{2}r^2} \left|
\frac1{2\pi}
\int_0^{2\pi}
a(\theta)\overline{g_F(re^{i\theta})}\,d\theta
\right| & = |F(T_ra)| \leq 
\|F\|_{(\mathcal F_\alpha^{p,q})^*} \|T_ra\|_{\mathcal F_\alpha^{p,q}} \\
& 
\lesssim
(1+r)^{\frac1q-1}
\|F\|_{(\mathcal F_\alpha^{p,q})^*}
\|a\|_{L^p(\T)}.
\end{aligned}
\]
If \(M_{p^*}(g_F,r)=0\), there is nothing to prove. Otherwise,  choose
\[
a(\theta):=
\begin{cases}
g_F(re^{i\theta})\bigl|g_F(re^{i\theta})\bigr|^{p^*-2}, & g_F(re^{i\theta})\ne0,\\[3pt]
0, & g_F(re^{i\theta})=0.
\end{cases}
\]
Then \(a\in L^p(\T)\), moreover, 
\[
\|a\|_{L^p(\T)}
=
M_{p^*}^{p^*-1}(g_F,r) \qquad \text{ and } \qquad \frac1{2\pi}\int_0^{2\pi}
a(\theta)\overline{g_F(re^{i\theta})}\,d\theta
=
M_{p^*}^{p^*}(g_F,r).
\]
From this and the estimate above, it follows that
\[
e^{-\frac{\alpha}{2}r^2}M_{p^*}^{p^*}(g_F,r)
\lesssim
(1+r)^{\frac1q-1} \|F\|_{(\mathcal F_\alpha^{p,q})^*}
M_{p^*}^{p^*-1}(g_F,r).
\]
Equivalently,
\[
M_{p^*}(g_F,r)e^{-\frac{\alpha}{2}r^2}
\lesssim
(1+r)^{\frac1q-1}\|F\|_{(\mathcal F_\alpha^{p,q})^*}.
\]
The proof is complete.
\end{proof}

\begin{lemma} 
\label{lem:dual-circle-shell-synthesis}
Let \(\alpha>0\), \(1<p\le\infty\), and \(1<q<\infty\). For each
\(k\ge0\), let \(r_k\in[k,k+1)\), and let
\(a_k\in L^p(\T)\) satisfy
\(
\|a_k\|_{L^p(\T)}\le1\).
If \(c=(c_k)_{k=0}^{\infty}\) is a scalar sequence such that
\[
\sum_{k=0}^{\infty}|c_k|^q(1+k)^{1-q}<\infty,
\]
then the series
\(
\sum_{k=0}^{\infty}c_kT_{r_k}a_k(z)
\)
converges  in \(\mathcal F_\alpha^{p,q}\), and its sum satisfies
\[
\left\|\sum_{k=0}^{\infty} c_kT_{r_k}a_k\right\|_{\mathcal F_\alpha^{p,q}}
\lesssim
\left(
\sum_{k=0}^{\infty}|c_k|^q(1+k)^{1-q}
\right)^{\frac1q},
\]
where the implicit constant depends only on \(\alpha,p\), and \(q\).
\end{lemma}

\begin{proof}
The proof follows the same finite-set approximation argument as in
Lemma~\ref{lem:dual-one-point-shell}. Let \(E\subset\mathbb N_0\) be finite and put
\[
h_E(z):=\sum_{k\in E}c_kT_{r_k}a_k(z).
\]
We prove the uniform estimate
\[
\|h_E\|_{\mathcal F_\alpha^{p,q}}
\lesssim
\left(
\sum_{k\in E}|c_k|^q(1+k)^{1-q}
\right)^{\frac1q}.
\]
By the proof of Lemma~\ref{lem:dual-circle-tests}, since
\(\|a_k\|_{L^p(\T)}\le1\), we have
\[
M_p(T_{r_k}a_k,\rho)
\le
M_1(G_{r_k},\rho),
\qquad \rho>0, \qquad \text{ where } \quad
G_r(z):=\exp\left(\alpha rz-\frac{\alpha}{2}r^2\right).
\]
By the single-kernel shell estimate obtained from the packet profile estimate
\eqref{eq:packet-profile} and used in \eqref{eq:jet-atom-profile}, 
\[
\sup_{\rho\in[\ell,\ell+1)}
M_1(G_{r_k},\rho)e^{-\frac{\alpha}{2}\rho^2}
\lesssim
(1+k)^{-1}e^{-\frac{\alpha}{8}|\ell-k|^2}, \qquad \ell, k \geq 0.
\]
Consequently, for each \(\ell\ge0\),  since $1 < p \leq \infty$, Minkowski's inequality on the circle gives
\[
\begin{aligned}
B_p(h_E;\ell)
&=
\sup_{\rho\in[\ell,\ell+1)}
M_p(h_E,\rho)e^{-\frac{\alpha}{2}\rho^2} \\
&\le
\sum_{k\in E}|c_k|
\sup_{\rho\in[\ell,\ell+1)}
M_p(T_{r_k}a_k,\rho)e^{-\frac{\alpha}{2}\rho^2}  \lesssim
\sum_{k\in E}|c_k|(1+k)^{-1}
e^{-\frac{\alpha}{8}|\ell-k|^2}.
\end{aligned}
\]
For \(k\ge0\), set
\(
Y_k:=|c_k|(1+k)^{\frac1q-1}\).
Then
\(
|c_k|(1+k)^{-1}
=
(1+k)^{-\frac1q}Y_k\).
Hence
\[
B_p(h_E;\ell)
\lesssim
\sum_{k=0}^{\infty}
e^{-\frac{\alpha}{8}|\ell-k|^2}(1+k)^{-\frac1q}
Y_k\mathbf 1_E(k).
\]
By the generalized Schur test, Lemma~\ref{lem:gen-schur-test}, applied with
exponent \(q\), we obtain
\[
\sum_{\ell=0}^{\infty}(1+\ell)B_p^q(h_E;\ell)
\lesssim
\sum_{k\in E}Y_k^q
=
\sum_{k\in E}|c_k|^q(1+k)^{1-q}.
\]
Using the annular discretization in Proposition~\ref{prop:annular}, we get
\[
\|h_E\|_{\mathcal F_\alpha^{p,q}}
\lesssim
\left(
\sum_{k\in E}|c_k|^q(1+k)^{1-q}
\right)^{\frac1q}.
\]
Let
\[
h_n:=\sum_{k=0}^{n}c_kT_{r_k}a_k,\qquad n\in\mathbb N.
\]
Applying the preceding estimate to  \(h_m-h_n\), we see
that \((h_n)_{n \geq 0}\) is a Cauchy sequence in \(\mathcal F_\alpha^{p,q}\), because
\[
\sum_{k=0}^{\infty}|c_k|^q(1+k)^{1-q}<\infty.
\]
Since the space
\(\mathcal F_\alpha^{p,q}\) is complete,  \(h_n\) converges in
\(\mathcal F_\alpha^{p,q}\) to some \(h\). Passing to the limit in the
finite-set estimate gives
\[
\|h\|_{\mathcal F_\alpha^{p,q}}
\lesssim
\left(
\sum_{k=0}^{\infty}|c_k|^q(1+k)^{1-q}
\right)^{\frac1q}.
\]
This proves the lemma.
\end{proof}

\begin{proposition} 
\label{prop:dual-rep-inner-banach}
Let \(\alpha>0\), \(1<p\le\infty\), and \(0<q<\infty\). For each
\(F\in(\mathcal F_\alpha^{p,q})^*\), the canonical
representing function \(g_F\) of \(F\) belongs to \(
\mathcal F_{\alpha;\sigma(p,q)}^{p^*,q^*}\)
and
\[
\|g_F\|_{\mathcal F_{\alpha;\sigma(p,q)}^{p^*,q^*}}
\lesssim
\|F\|_{(\mathcal F_\alpha^{p,q})^*},
\]
where the implicit constant depends only on \(\alpha,p\), and \(q\).
\end{proposition}

\begin{proof}
We split the proof according to the outer exponent \(q\).

\smallskip
\noindent
\emph{Case 1: \(0<q\le1\).}
Since \(1<p\le\infty\), we have
\[
\sigma(p,q)
=
-\left(\frac1q-1\right)
=
1-\frac1q,
\qquad q^*=\infty.
\]
Then, by Lemma~\ref{lem:dual-circle-extraction}, for every \(r>0\),
\[
M_{p^*}(g_F,r)e^{-\frac{\alpha}{2}r^2}(1+r)^{\sigma(p,q)} =
M_{p^*}(g_F,r)e^{-\frac{\alpha}{2}r^2}(1+r)^{1-\frac1q}
\lesssim
\|F\|_{(\mathcal F_\alpha^{p,q})^*}.
\]
Taking the supremum over \(r>0\), we get
\[
\|g_F\|_{\mathcal F_{\alpha;\sigma(p,q)}^{p^*,\infty}}
\lesssim
\|F\|_{(\mathcal F_\alpha^{p,q})^*}.
\]
This proves the assertion in the case \(0<q\le1\).

\smallskip
\noindent
\emph{Case 2: \(1<q<\infty\).}
In this case \(q^*=q'=\frac{q}{q-1}\) and
\(
\sigma(p,q)=0\).
It remains to prove that \(g_F\in\mathcal F_\alpha^{p^*,q'}\).
For \(k\ge0\), set
\[
H_k
:=
B_{p^*}(g_F;k)
=
\sup_{r\in[k,k+1)}
M_{p^*}(g_F,r)e^{-\frac{\alpha}{2}r^2}.
\]
We shall prove
\[
\left(
\sum_{k=0}^{\infty}(1+k)H_k^{q'}
\right)^{\frac1{q'}}
\lesssim
\|F\|_{(\mathcal F_\alpha^{p,q})^*}.
\]
Fix  \(\varepsilon\in(0,1)\). For each
\(k\geq 0\), choose \(r_k\in[k,k+1)\), with \(r_k>0\), such that
\[
M_{p^*}(g_F,r_k)e^{-\frac{\alpha}{2}r_k^2}
\ge
(1-\varepsilon)H_k.
\]
If \(M_{p^*}(g_F,r_k)=0\), put \(a_k=0\). Otherwise, define
\[
a_k(\theta)
:=
\frac{
g_F(r_ke^{i\theta})
\bigl|g_F(r_ke^{i\theta})\bigr|^{p^*-2}
}{
M_{p^*}(g_F,r_k)^{p^*-1}
},
\qquad 1<p^*<\infty,
\]
and, in the endpoint case \(p^*=1\),
\[
a_k(\theta)
:=
\begin{cases}
\dfrac{g_F(r_ke^{i\theta})}
{|g_F(r_ke^{i\theta})|},
& g_F(r_ke^{i\theta})\ne0,\\[2mm]
0,
& g_F(r_ke^{i\theta})=0.
\end{cases}
\]
Then \(a_k\in L^p(\T)\), \(\|a_k\|_{L^p(\T)}\le1\), and
\[
\frac1{2\pi}\int_0^{2\pi}
a_k(\theta)\overline{g_F(r_ke^{i\theta})}\,d\theta
=
M_{p^*}(g_F,r_k).
\]
Therefore, by Lemma~\ref{lem:dual-circle-tests},
\begin{equation}
\label{eq:circle-test-extract-Hk}
F(T_{r_k}a_k)
=
e^{-\frac{\alpha}{2}r_k^2}
M_{p^*}(g_F,r_k)
\ge
(1-\varepsilon)H_k.
\end{equation}
Let \(c=(c_k)_{k\ge0}\) be a nonnegative sequence such that
\[
\sum_{k=0}^{\infty}c_k^q(1+k)^{1-q}<\infty.
\]
By Lemma~\ref{lem:dual-circle-shell-synthesis}, the series
\(
h:=\sum_{k=0}^{\infty}c_kT_{r_k}a_k
\)
converges in \(\mathcal F_\alpha^{p,q}\), and
\[
\|h\|_{\mathcal F_\alpha^{p,q}}
\lesssim
\left(
\sum_{k=0}^{\infty}c_k^q(1+k)^{1-q}
\right)^{\frac1q}.
\]
Since \(F\) is continuous on \(\mathcal F_\alpha^{p,q}\), using \eqref{eq:circle-test-extract-Hk} and the nonnegativity of \(c_k\), we have
\[
(1-\varepsilon)\sum_{k=0}^{\infty}c_kH_k \leq \sum_{k=0}^{\infty}c_kF(T_{r_k}a_k) = F(h) \leq |F(h)| 
\leq \|F\|_{(\mathcal F_\alpha^{p,q})^*}\,
\|h\|_{\mathcal F_\alpha^{p,q}}. 
\]
Hence
\[
(1-\varepsilon)\sum_{k=0}^{\infty}c_kH_k
\lesssim
\|F\|_{(\mathcal F_\alpha^{p,q})^*}
\left(
\sum_{k=0}^{\infty}c_k^q(1+k)^{1-q}
\right)^{\frac1q}.
\]
Letting \(\varepsilon\downarrow0\), we conclude that
\[
\sum_{k=0}^{\infty}c_kH_k
\lesssim
\|F\|_{(\mathcal F_\alpha^{p,q})^*}
\left(
\sum_{k=0}^{\infty}c_k^q(1+k)^{1-q}
\right)^{\frac1q}
\]
for every nonnegative sequence \(c=(c_k)_{k\ge0}\) satisfying
\[
\sum_{k=0}^{\infty}c_k^q(1+k)^{1-q}<\infty.
\]
Since \(H_k\ge0\), the preceding estimate applied to \(|c|\) gives
\[
\left|
\sum_{k=0}^{\infty}c_kH_k
\right|
\lesssim
\|F\|_{(\mathcal F_\alpha^{p,q})^*}
\left(
\sum_{k=0}^{\infty}|c_k|^q(1+k)^{1-q}
\right)^{\frac1q},
\]
for every \(c \in L^q(\mathbb N_0,d\mu)\), where
\(
\mu(\{k\})=(1+k)^{1-q}\).
Thus the functional
\[
c\longmapsto \sum_{k=0}^{\infty}c_kH_k
\]
is bounded on the weighted space \(L^q(\mathbb N_0,d\mu)\).
By the \(L^q\)-duality theorem \cite[Theorem~6.6]{Rud87}, it follows that
\[
\left(\sum_{k=0}^{\infty}(1+k)H_k^{q'}
\right)^{\frac1{q'}} = \left(
\sum_{k=0}^{\infty}
\left[\frac{H_k}{(1+k)^{1-q}}\right]^{q'}
(1+k)^{1-q}
\right)^{\frac1{q'}}
\lesssim
\|F\|_{(\mathcal F_\alpha^{p,q})^*}.
\]
By Proposition~\ref{prop:annular} applied to \(g_F\) with exponents
\((p^*,q')\),
\[
\|g_F\|_{\mathcal F_\alpha^{p^*,q'}}^{q'}
\asymp
\sum_{k=0}^{\infty}(1+k)H_k^{q'}.
\]
Thus
\[
\|g_F\|_{\mathcal F_\alpha^{p^*,q'}}
\lesssim
\|F\|_{(\mathcal F_\alpha^{p,q})^*}.
\]
Since \(\sigma(p,q)=0\) and \(q^*=q'\) in the present case, this is exactly
\[
\|g_F\|_{\mathcal F_{\alpha;\sigma(p,q)}^{p^*,q^*}}
\lesssim
\|F\|_{(\mathcal F_\alpha^{p,q})^*}.
\]
Combining the two cases completes the proof.
\end{proof}

\subsubsection{Completion of the equal-parameter duality theorem}
\label{subsubsec:dual-equal-completion}

Combining boundedness of the pairing with the representative estimates,
it remains to identify a functional with its canonical representing
function on the whole space.  This follows from polynomial density and
Gaussian orthogonality.

\begin{lemma} 
\label{lem:dual-polynomial-completion}
Let \(\alpha>0\), \(0<p\le\infty\), and \(0<q<\infty\). For each
\(F\in(\mathcal F_\alpha^{p,q})^*\), let \(g_F\) be the canonical
representing function of \(F\). 
Then
\[
F(f)=\langle f,g_F\rangle_\alpha,
\qquad f\in\mathcal F_\alpha^{p,q}.
\]
Moreover, \(g_F\) is the unique function in
\(\mathcal F_{\alpha;\sigma(p,q)}^{p^*,q^*}\) representing \(F\).
\end{lemma}

\begin{proof}
By Propositions~\ref{prop:dual-rep-inner-quasi} and
\ref{prop:dual-rep-inner-banach}, we have
\(
g_F\in \mathcal F_{\alpha;\sigma(p,q)}^{p^*,q^*}\).
Hence, by Proposition~\ref{prop:dual-forward}, the functional
\[
F_{g_F}(f):=\langle f,g_F\rangle_\alpha,
\qquad f\in\mathcal F_\alpha^{p,q},
\]
is continuous on \(\mathcal F_\alpha^{p,q}\).

\smallskip
\noindent
We show that \(F\) and \(F_{g_F}\) agree on monomials. By
\eqref{eq:dual-gF-expansion}, we have
\[
\overline{g_F(w)}
=
\sum_{m=0}^{\infty}
\frac{\alpha^m}{m!}F(z^m)\,\overline{w}^{\,m},
\qquad w\in\mathbb C,
\]
with locally uniform convergence. Fix \(n\in\mathbb N_0\). For \(R>0\),
termwise integration over \(\{|w|\le R\}\) is justified by locally uniform
convergence. Therefore, by angular orthogonality,
\[
\frac{\alpha}{\pi}
\int_{|w|\le R}
w^n\overline{g_F(w)}e^{-\alpha|w|^2}\,dA(w)
=
\left(\frac{\alpha^n}{n!}F(z^n) \right) \,
\frac{\alpha}{\pi}
\int_{|w|\le R}
|w|^{2n}e^{-\alpha|w|^2}\,dA(w).
\]
Letting \(R\to\infty\), and using
\[
\frac{\alpha}{\pi}
\int_{\mathbb C}
|w|^{2n}e^{-\alpha|w|^2}\,dA(w)
=
\frac{n!}{\alpha^n},
\]
we obtain
\[
F_{g_F}(z^n)
=
\langle z^n,g_F\rangle_\alpha
=
F(z^n).
\]
Thus \(F(P)=F_{g_F}(P)\) for every polynomial \(P\).
Since polynomials are dense in \(\mathcal F_\alpha^{p,q}\), and both
\(F\) and \(F_{g_F}\) are continuous on \(\mathcal F_\alpha^{p,q}\), it follows
that
\[
F(f)=F_{g_F}(f)=\langle f,g_F\rangle_\alpha,
\qquad f\in\mathcal F_\alpha^{p,q}.
\]
It remains to prove uniqueness. Suppose that
\(
g_1,g_2\in
\mathcal F_{\alpha;\sigma(p,q)}^{p^*,q^*}
\)
represent the same functional on \(\mathcal F_\alpha^{p,q}\). Then
\(
\langle P,g_1-g_2\rangle_\alpha=0
\)
for every polynomial \(P\). Write
\[
g_1(z)-g_2(z)=\sum_{m=0}^{\infty}c_mz^m.
\]
Testing against \(P(z)=z^n\) and using Gaussian orthogonality gives
\[
0
=
\langle z^n,g_1 -g_2\rangle_\alpha =\overline{c_n}\frac{n!}{\alpha^n},
\qquad n\ge0.
\]
Hence \(c_n=0\) for every \(n\ge0\), and therefore \(g_1=g_2\).
\end{proof}

\begin{theorem}[Equal-parameter duality]
\label{thm:dual-equal}
Let \(\alpha>0\), \(0<p\le\infty\), and \(0<q<\infty\). Then the
continuous dual of \(\mathcal F_\alpha^{p,q}\) can be identified with
\(
\mathcal F_{\alpha;\sigma(p,q)}^{p^*,q^*}
\)
under the Gaussian pairing \(\langle\cdot,\cdot\rangle_\alpha\). More precisely,
every \(g\in \mathcal F_{\alpha;\sigma(p,q)}^{p^*,q^*}\) defines a continuous
linear functional on \(\mathcal F_\alpha^{p,q}\) by
\[
F_g(f)=\langle f,g\rangle_\alpha,
\qquad f\in\mathcal F_\alpha^{p,q},
\]
and every \(F\in(\mathcal F_\alpha^{p,q})^*\) is represented uniquely in this
way by its canonical representing function
\(
g_F\in \mathcal F_{\alpha;\sigma(p,q)}^{p^*,q^*}\).
Moreover,
\[
\|F\|_{(\mathcal F_\alpha^{p,q})^*}
\asymp
\|g_F\|_{\mathcal F_{\alpha;\sigma(p,q)}^{p^*,q^*}},
\]
where the implicit constants depend only on \(\alpha,p\), and \(q\).
\end{theorem}

\begin{proof}
By Proposition~\ref{prop:dual-forward}, every
\(g\in \mathcal F_{\alpha;\sigma(p,q)}^{p^*,q^*}\) defines a continuous
linear functional \(F_g\) on \(\mathcal F_\alpha^{p,q}\), and
\[
\|F_g\|_{(\mathcal F_\alpha^{p,q})^*}
\lesssim
\|g\|_{\mathcal F_{\alpha;\sigma(p,q)}^{p^*,q^*}}.
\]
Conversely, let \(F\in(\mathcal F_\alpha^{p,q})^*\), and let \(g_F\) be its
canonical representing function. Then
Proposition~\ref{prop:dual-rep-inner-quasi} for \(0<p\le1\) and Proposition~\ref{prop:dual-rep-inner-banach} for \(1<p\le\infty\) give
\[
g_F\in \mathcal F_{\alpha;\sigma(p,q)}^{p^*,q^*} \qquad \text{and} \qquad \|g_F\|_{\mathcal F_{\alpha;\sigma(p,q)}^{p^*,q^*}}
\lesssim
\|F\|_{(\mathcal F_\alpha^{p,q})^*}.
\]
By Lemma~\ref{lem:dual-polynomial-completion},
\(
F(f)=\langle f,g_F\rangle_\alpha\) for all \(f\in\mathcal F_\alpha^{p,q}\), 
and \(g_F\) is the unique representing function in
\(\mathcal F_{\alpha;\sigma(p,q)}^{p^*,q^*}\).

\smallskip
\noindent
It remains only to obtain the opposite norm estimate. Since the representation
identity just proved gives \(F=F_{g_F}\), Proposition~\ref{prop:dual-forward}
implies
\[
\|F\|_{(\mathcal F_\alpha^{p,q})^*}
=
\|F_{g_F}\|_{(\mathcal F_\alpha^{p,q})^*}
\lesssim
\|g_F\|_{\mathcal F_{\alpha;\sigma(p,q)}^{p^*,q^*}}.
\]
Together with the estimate obtained above, this gives the asserted norm
equivalence and completes the proof. In particular, this proves
Theorem~\ref{thm:duality-intro} in the case \(\gamma=\alpha\).
\end{proof}

\begin{remark}[Comparison with the projection method]
In the Banach mixed range \(1\leq p<\infty\) and \(1<q<\infty\), the
equal-parameter duality
\[
(\mathcal F_\alpha^{p,q})^*
\simeq
\mathcal F_\alpha^{p^*,q^*}
\]
can also be obtained by a projection argument; see Liu~\cite[Theorem~7]{Liu24}.
In that approach one proves boundedness of the Fock projection \(P_\alpha\) on
the ambient mixed-norm space \(L_\alpha^{p,q}\), extends functionals from
\(\mathcal F_\alpha^{p,q}\) to \(L_\alpha^{p,q}\) by the Hahn--Banach theorem,
uses the Benedek--Panzone duality theorem for mixed-norm Lebesgue
spaces~\cite[Theorem~1, p.~304]{BP61}, and then projects the resulting
representative back to the analytic subspace.

\smallskip
\noindent
This method is effective in the Banach projection range, but it depends on two
inputs: boundedness of the ambient Fock projection and Banach-space duality.
Thus it does not give a uniform proof of the quasi-Banach cases, the endpoint
\(q=\infty\), or the weighted representing spaces appearing in
Theorems~\ref{thm:duality-intro} and \ref{thm:little-duality-intro}.  More
general projection problems, allowing different Gaussian parameters and mixed
exponents in the source and target ambient spaces, are natural in this context,
but they are separate from the intrinsic duality argument used here.

\smallskip
\noindent
The proof in this paper does not use boundedness of the Fock projection.
Instead, for a functional \(F\), it constructs the canonical representing
function
\[
g_F(w)=F(K_{\alpha,w})
\]
and estimates it directly by kernel tests, circle tests, and shellwise sequence
estimates.  This gives a single argument covering the Banach, quasi-Banach, and
endpoint regimes treated in the duality theorem.
\end{remark}

\subsection{Transport to arbitrary Gaussian pairing parameters}
\label{subsec:dual-transport}

We prove Theorem~\ref{thm:duality-intro}.  By the equal-parameter theorem,
it remains to transport the representing space from \(\alpha\) to
\(\beta\) and identify the resulting pairing.  This is done by
\(D_{\alpha/\gamma}\), where \(\gamma = \sqrt{\alpha\beta}\).

\begin{proof}[Proof of Theorem~\ref{thm:duality-intro}]
Throughout the proof, the pairing between
\(\mathcal F_\alpha^{p,q}\) and
\(\mathcal F_{\beta;\sigma(p,q)}^{p^*,q^*}\) is the transported pairing
\begin{equation}
\label{eq:transported-pairing-identity}
\langle f,g\rangle_{\alpha\to\gamma}^{\rm tr}
:=
\big\langle f,D_{\alpha/\gamma}g\big\rangle_\alpha .
\end{equation}

\smallskip
\noindent
Since \(\gamma^2=\alpha\beta\), Lemma~\ref{lem:dual-dilation-weighted}, applied
with \(\tau=\frac{\alpha}{\gamma}\) and with exponents \(p^*,q^*\), shows that the dilation
\(D_{\alpha/\gamma}\) is an isomorphism from
\(
\mathcal F_{\beta;\sigma(p,q)}^{p^*,q^*}
\)
onto
\(
\mathcal F_{\alpha;\sigma(p,q)}^{p^*,q^*}\).
Moreover,
\begin{equation}
\label{eq:transport-dilation-norm}
\|D_{\alpha/\gamma}g\|_{\mathcal F_{\alpha;\sigma(p,q)}^{p^*,q^*}}
\asymp
\|g\|_{\mathcal F_{\beta;\sigma(p,q)}^{p^*,q^*}}, \qquad g \in \mathcal F_{\beta;\sigma(p,q)}^{p^*,q^*}.
\end{equation}

\smallskip
\noindent
Every
\(
g\in \mathcal F_{\beta;\sigma(p,q)}^{p^*,q^*}
\)
defines a continuous linear functional on \(\mathcal F_\alpha^{p,q}\):
set
\(
h:=D_{\alpha/\gamma}g\).
Then
\(
h\in \mathcal F_{\alpha;\sigma(p,q)}^{p^*,q^*}\).
By the equal-parameter duality theorem, Theorem~\ref{thm:dual-equal}, the map
\[
F_g(f)
:=
\langle f,g\rangle_{\alpha\to\gamma}^{\rm tr}
=
\langle f,h\rangle_\alpha,
\qquad f\in\mathcal F_\alpha^{p,q},
\]
belongs to \((\mathcal F_\alpha^{p,q})^*\), and
\(
\|F_g\|_{(\mathcal F_\alpha^{p,q})^*}
\asymp
\|h\|_{\mathcal F_{\alpha;\sigma(p,q)}^{p^*,q^*}}\). Together with \eqref{eq:transport-dilation-norm}, this gives
\[
\|F_g\|_{(\mathcal F_\alpha^{p,q})^*}
\asymp
\|g\|_{\mathcal F_{\beta;\sigma(p,q)}^{p^*,q^*}}.
\]

\smallskip
\noindent
Conversely, let
\(
F\in(\mathcal F_\alpha^{p,q})^*\).
Again by Theorem~\ref{thm:dual-equal}, there exists a unique
\(
h\in\mathcal F_{\alpha;\sigma(p,q)}^{p^*,q^*}
\)
such that
\[
F(f)=\langle f,h\rangle_\alpha,
\qquad f\in\mathcal F_\alpha^{p,q} \qquad \text{ and } \qquad
\|h\|_{\mathcal F_{\alpha;\sigma(p,q)}^{p^*,q^*}}
\asymp
\|F\|_{(\mathcal F_\alpha^{p,q})^*}.
\]
Define
\(
g:=D_{\gamma/\alpha}h\).
Since \(D_{\gamma/\alpha}\) is the inverse of \(D_{\alpha/\gamma}\), Lemma~\ref{lem:dual-dilation-weighted} gives
\[
g\in\mathcal F_{\beta;\sigma(p,q)}^{p^*,q^*}
\qquad \text{ and } \qquad
\|g\|_{\mathcal F_{\beta;\sigma(p,q)}^{p^*,q^*}}
\asymp
\|h\|_{\mathcal F_{\alpha;\sigma(p,q)}^{p^*,q^*}}
\asymp
\|F\|_{(\mathcal F_\alpha^{p,q})^*}.
\]
Moreover,
\(
D_{\alpha/\gamma}g=h\).
Hence, for every \(f\in\mathcal F_\alpha^{p,q}\),
\[
F(f)
=
\langle f,h\rangle_\alpha
=
\big\langle f,D_{\alpha/\gamma}g\big\rangle_\alpha
=
\langle f,g\rangle_{\alpha\to\gamma}^{\rm tr}.
\]
Thus every continuous linear functional on \(\mathcal F_\alpha^{p,q}\) is
represented by an element of
\(\mathcal F_{\beta;\sigma(p,q)}^{p^*,q^*}\).

\smallskip
\noindent
It remains only to prove uniqueness. Suppose that
\(
g_1,g_2\in\mathcal F_{\beta;\sigma(p,q)}^{p^*,q^*}
\)
represent the same functional under the transported pairing. Then
\[
\big\langle f,D_{\alpha/\gamma}(g_1-g_2)\big\rangle_\alpha=0,
\qquad f\in\mathcal F_\alpha^{p,q}.
\]
By the uniqueness part of Theorem~\ref{thm:dual-equal},
\(
D_{\alpha/\gamma}(g_1-g_2)=0\).
Since \(D_{\alpha/\gamma}\) is injective on entire functions, we obtain
\(
g_1=g_2\).

\smallskip
\noindent
Therefore the conjugate-linear map
\[
g\longmapsto F_g,\qquad
F_g(f)=\langle f,g\rangle_{\alpha\to\gamma}^{\rm tr},
\]
identifies
\(
\mathcal F_{\beta;\sigma(p,q)}^{p^*,q^*}
\)
with
\(
(\mathcal F_\alpha^{p,q})^*
\)
with equivalent norms. This proves Theorem~\ref{thm:duality-intro}.
\end{proof}

\begin{remark}
The pairing used in Theorem~\ref{thm:duality-intro} is the transported pairing
\eqref{eq:transported-pairing-identity}. On polynomial pairs it agrees with the Gaussian pairing at parameter \(\gamma\): for \(m,n\in\mathbb N_0\),
\[
\langle z^n, z^m\rangle_{\alpha\to\gamma}^{\rm tr}  = \big\langle z^n,D_{\alpha/\gamma}z^m\big\rangle_\alpha
=
\left(\frac{\alpha}{\gamma}\right)^n
\frac{n!}{\alpha^n}\delta_{n,m}
=
\frac{n!}{\gamma^n}\delta_{n,m}
=
\langle z^n,z^m\rangle_\gamma.
\]
By linearity in the first variable and conjugate-linearity in the second
variable, the same identity holds for all polynomial pairs.
\end{remark}

\subsection{The little mixed-norm Fock space}
\label{subsec:dual-little}

We finish with the duality theorem for the little endpoint space
\(f_\alpha^{p,\infty}\).  The argument parallels the full-space proof,
with \(q=\infty\) replacing the finite outer exponent, followed by the
same dilation transport as in Subsection~\ref{subsec:dual-transport}.

\smallskip
\noindent
We prove the equal-parameter little-space duality.

\begin{proposition}[Little endpoint duality]
\label{prop:dual-little-equal}
Let \(\alpha>0\) and \(0<p\le\infty\). Then the continuous dual of
\(f_\alpha^{p,\infty}\) can be identified with
\(
\mathcal F_{\alpha;\sigma(p,\infty)}^{p^*,1}
\)
under the Gaussian pairing \(\langle\cdot,\cdot\rangle_\alpha\).
More precisely, every
\(
g\in\mathcal F_{\alpha;\sigma(p,\infty)}^{p^*,1}
\)
defines a continuous linear functional on \(f_\alpha^{p,\infty}\) by
\[
F_g(f):=\langle f,g\rangle_\alpha,
\qquad f\in f_\alpha^{p,\infty},
\]
and every
\(
F\in(f_\alpha^{p,\infty})^*
\)
is represented uniquely in this way by its canonical representing function
\(
g_F\in\mathcal F_{\alpha;\sigma(p,\infty)}^{p^*,1}\).
Moreover,
\[
\|F\|_{(f_\alpha^{p,\infty})^*}
\asymp
\|g_F\|_{\mathcal F_{\alpha;\sigma(p,\infty)}^{p^*,1}},
\]
where the implicit constants depend only on \(\alpha\) and \(p\).
\end{proposition}

\begin{proof}
The proof follows the equal-parameter argument for
\(\mathcal F_\alpha^{p,q}\), with \(q=\infty\).

\smallskip
\noindent
\emph{Step 1: Boundedness of the pairing.}
Every
\(
g\in\mathcal F_{\alpha;\sigma(p,\infty)}^{p^*,1}
\)
induces a bounded linear functional $F_g$ on \(f_\alpha^{p,\infty}\). Let
\(f\in f_\alpha^{p,\infty}\).

\smallskip
\noindent
If \(0<p\le1\), then \(p^*=\infty\) and
\(
\sigma(p,\infty)=\frac1p-1\).
Using polar coordinates and \eqref{eq:angular-comparison} with
\((p_1,p_2)=(p,1)\), we obtain
\[
\begin{aligned}
|\langle f,g\rangle_\alpha|
&\le
2\alpha
\int_0^\infty
M_1(f,r)M_\infty(g,r)e^{-\alpha r^2}r\,dr  \\
&\lesssim
\|f\|_{\mathcal F_\alpha^{p,\infty}}
\int_0^\infty
M_\infty(g,r)e^{-\frac{\alpha}{2}r^2}
(1+r)^{\frac1p-1}r\,dr  \ \asymp \
\|f\|_{\mathcal F_\alpha^{p,\infty}}
\|g\|_{\mathcal F_{\alpha;\sigma(p,\infty)}^{\infty,1}} .
\end{aligned}
\]
If \(1<p\le\infty\), then \(\sigma(p,\infty)=0\). By H\"older's inequality
on circles, with the usual interpretation when \(p=\infty\),
\[
\begin{aligned}
|\langle f,g\rangle_\alpha|
&\le
2\alpha
\int_0^\infty
M_p(f,r)M_{p^*}(g,r)e^{-\alpha r^2}r\,dr  \\
&\le
2\alpha
\|f\|_{\mathcal F_\alpha^{p,\infty}}
\int_0^\infty
M_{p^*}(g,r)e^{-\frac{\alpha}{2}r^2}r\,dr  \ \asymp \ 
\|f\|_{\mathcal F_\alpha^{p,\infty}}
\|g\|_{\mathcal F_\alpha^{p^*,1}} .
\end{aligned}
\]
Thus the map
\(
F_g(f):=\langle f,g\rangle_\alpha, f\in f_\alpha^{p,\infty}\),
defines an element of \((f_\alpha^{p,\infty})^*\), and
\begin{equation}
\label{eq:little-forward-bound}
\|F_g\|_{(f_\alpha^{p,\infty})^*}
\lesssim
\|g\|_{\mathcal F_{\alpha;\sigma(p,\infty)}^{p^*,1}}.
\end{equation}

\smallskip
\noindent
\emph{Step 2: The canonical representing function belongs to the proposed
representing space.}
We estimate the canonical representing function. The endpoint
\(q=\infty\) leads to an outer \(\ell^1\)-condition in the representing space,
which is obtained by finite shell testing.

\smallskip
\noindent
Let
\(
F\in(f_\alpha^{p,\infty})^*\).
Then its canonical representing function is defined by
\[
g_F(w):=\overline{F(K_{\alpha,w})},
\qquad w\in\mathbb C.
\]
By Lemma~\ref{lem:dual-gF-expansion}, \(g_F\) is entire and admits the
locally uniformly convergent expansion
\begin{equation}
\label{eq:little-gF-expansion}
g_F(w)
=
\sum_{n=0}^{\infty}
\frac{\alpha^n}{n!}\overline{F(z^n)}\,w^n .
\end{equation}
We prove that
\(
g_F\in\mathcal F_{\alpha;\sigma(p,\infty)}^{p^*,1}
\)
and
\begin{equation}
\label{eq:little-rep-estimate}
\|g_F\|_{\mathcal F_{\alpha;\sigma(p,\infty)}^{p^*,1}}
\lesssim
\|F\|_{(f_\alpha^{p,\infty})^*}.
\end{equation}

\smallskip
\noindent
\emph{Case 1: \(0<p\le1\).}
In this case
\(
p^*=\infty\) and \(
\sigma(p,\infty)=\frac1p-1\).
As in the proof of Proposition \ref{prop:dual-rep-inner-quasi}, for \(k\ge0\), set
\[
H_k
:= B_{\infty, \frac{1}{p} - 1}(g_F; k) = 
\sup_{w\in A_k}
|g_F(w)|e^{-\frac{\alpha}{2}|w|^2}
(1+|w|)^{\frac1p-1}.
\]
We shall prove
\begin{equation}
\label{eq:little-quasi-shell-l1}
\sum_{k=0}^{\infty}(1+k)H_k
\lesssim
\|F\|_{(f_\alpha^{p,\infty})^*}.
\end{equation}
Fix a finite set \(E\subset\mathbb N_0\) and an arbitrary number
\(\varepsilon\in(0,1)\). For each \(k\in E\), choose \(w_k\in A_k\) such
that
\[
|g_F(w_k)|e^{-\frac{\alpha}{2}|w_k|^2}
(1+|w_k|)^{\frac1p-1}
\ge
(1-\varepsilon)H_k.
\]
We choose
\[
\xi_k :=
\begin{cases}
\dfrac{g_F(w_k)}{|g_F(w_k)|}, & g_F(w_k)\ne0,\\[2mm]
1, & g_F(w_k)=0.
\end{cases}
\]
Let \(c=(c_k)_{k\in E}\) be a finite nonnegative sequence and define
\[
h_E(z)
:=
\sum_{k\in E}
c_k \xi_k (1+|w_k|)^{\frac1p-1}
\kappa_{\alpha,w_k}(z).
\]
Since \(h_E\) is a finite linear combination of normalized kernels,
\(h_E\in f_\alpha^{p,\infty}\). We claim that
\begin{equation}
\label{eq:little-quasi-finite-synthesis}
\|h_E\|_{\mathcal F_\alpha^{p,\infty}}
\lesssim
\sup_{k\in E}\frac{c_k}{1+k}.
\end{equation}
By the same single-kernel shell estimate used in
Lemma~\ref{lem:dual-one-point-shell} and  \(p\)-subadditivity, we get
\[
\begin{aligned}
B_p^p(h_E;\ell) & = \sup_{r\in[\ell,\ell+1)}
M_p^p(h_E, r) e^{-\frac{\alpha p}{2}r^2}
\lesssim
\sum_{k\in E}
c_k^p(1+|w_k|)^{1-p}
(1+k)^{-1}e^{-\frac{\alpha p}{8}|\ell-k|^2} \\
&\asymp
\sum_{k\in E}
\left(\frac{c_k}{1+k}\right)^p
e^{-\frac{\alpha p}{8}|\ell-k|^2} \lesssim
\left(\sup_{k\in E}\frac{c_k}{1+k}\right)^p .
\end{aligned}
\]
Taking the supremum over \(\ell\) and using the annular discretization in Proposition \ref{prop:annular} proves
\eqref{eq:little-quasi-finite-synthesis}.

\smallskip
\noindent
On the other hand, by the definition of \(g_F\),
\[
F(\kappa_{\alpha,w_k})
=
e^{-\frac{\alpha}{2}|w_k|^2}F(K_{\alpha,w_k})
=
e^{-\frac{\alpha}{2}|w_k|^2}\overline{g_F(w_k)}.
\]
Hence, by the choice of $w_k$ and $\xi_k$,
\[
F(h_E)
=
\sum_{k\in E}
c_k  (1+|w_k|)^{\frac1p-1}
e^{-\frac{\alpha}{2}|w_k|^2}|g_F(w_k)| \ge
(1-\varepsilon)\sum_{k\in E}c_kH_k.
\]
Therefore, by \eqref{eq:little-quasi-finite-synthesis},
\[
(1-\varepsilon)\sum_{k\in E}c_kH_k
\le
|F(h_E)|
\lesssim
\|F\|_{(f_\alpha^{p,\infty})^*}
\sup_{k\in E}\frac{c_k}{1+k}.
\]
Letting \(\varepsilon\downarrow0\), we obtain
\begin{equation}
\label{eq:little-quasi-finite-dual}
\sum_{k\in E}c_kH_k
\lesssim
\|F\|_{(f_\alpha^{p,\infty})^*}
\left(\sup_{k\in E}\frac{c_k}{1+k}\right).
\end{equation}
Now choose \(c_k=1+k\) for \(k\in E\). Then
\[
\sum_{k\in E}(1+k)H_k
\lesssim
\|F\|_{(f_\alpha^{p,\infty})^*}.
\]
Taking the supremum over all finite \(E\subset\mathbb N_0\) proves
\eqref{eq:little-quasi-shell-l1}. By the annular discretization in Lemma \ref{lem:dual-weighted-annular} applied to
the weighted space \(\mathcal F_{\alpha;\sigma(p,\infty)}^{\infty,1}\),
we get
\[
\|g_F\|_{\mathcal F_{\alpha;\sigma(p,\infty)}^{\infty,1}}
\lesssim
\|F\|_{(f_\alpha^{p,\infty})^*}.
\]

\smallskip
\noindent
\emph{Case 2: \(1<p\le\infty\).}
In this case
\(
\sigma(p,\infty)=0\). As in the proof of Proposition \ref{prop:dual-rep-inner-banach},
for \(k\ge0\), set
\[
H_k
:=
B_{p^*}(g_F;k)
=
\sup_{r\in[k,k+1)}
M_{p^*}(g_F,r)e^{-\frac{\alpha}{2}r^2}.
\]
We shall prove
\begin{equation}
\label{eq:little-banach-shell-l1}
\sum_{k=0}^{\infty}(1+k)H_k
\lesssim
\|F\|_{(f_\alpha^{p,\infty})^*}.
\end{equation}
Fix a finite set \(E\subset\mathbb N_0\) and an arbitrary number
\(\varepsilon\in(0,1)\). For each \(k\in E\), choose \(r_k\in[k,k+1)\) such
that
\[
M_{p^*}(g_F,r_k)e^{-\frac{\alpha}{2}r_k^2}
\ge
(1-\varepsilon)H_k.
\]
If \(M_{p^*}(g_F,r_k)=0\), put \(a_k=0\). Otherwise, choose
\(a_k\in L^p(\T)\), as in the proof of
Proposition~\ref{prop:dual-rep-inner-banach}, such that
\[
\|a_k\|_{L^p(\T)}\le1
\qquad \text{ and } \qquad
\frac1{2\pi}\int_0^{2\pi}
a_k(\theta)\overline{g_F(r_ke^{i\theta})}\,d\theta
=
M_{p^*}(g_F,r_k).
\]
We note that \(T_r a\in f_\alpha^{p,\infty}\) for every fixed
\(r\ge0\) and every \(a\in L^p(\T)\). The proof of
Lemma~\ref{lem:dual-circle-tests} gives
\[
M_p(T_r a,\rho)
\le
\|a\|_{L^p(\T)}M_1(G_r,\rho),
\quad \rho>0, \quad \text{ where }
 \quad 
G_r(z):=\exp\left(\alpha rz-\frac{\alpha}{2}r^2\right).
\]
By the packet profile estimate \eqref{eq:packet-profile}, with \(p=1\),
\[
M_1(G_r,\rho)e^{-\frac{\alpha}{2}\rho^2}
\lesssim
(1+\alpha r\rho)^{-\frac{1}{2}}
e^{-\frac{\alpha}{2}(\rho-r)^2}.
\]
For fixed \(r\), the right-hand side tends to \(0\) as \(\rho\to\infty\).
Hence
\(
M_p(T_r a,\rho)e^{-\frac{\alpha}{2}\rho^2}\to0\) as \(
\rho\to\infty\),
and therefore \(T_r a\in f_\alpha^{p,\infty}\).

\smallskip
\noindent
Consequently, the proof of Lemma~\ref{lem:dual-circle-tests} applies
verbatim to functionals on \(f_\alpha^{p,\infty}\), and gives
\[
F(T_r a)
=
e^{-\frac{\alpha}{2}r^2}
\frac1{2\pi}\int_0^{2\pi}
a(\theta)\overline{g_F(re^{i\theta})}\,d\theta .
\]
Thus
\begin{equation}
\label{eq:little-circle-extract-Hk}
F(T_{r_k}a_k) = M_{p^*}(g_F,r_k) e^{-\frac{\alpha}{2}r_k^2} 
\ge
(1-\varepsilon)H_k.
\end{equation}
Let \(c=(c_k)_{k\in E}\) be a finite nonnegative sequence and define
\[
h_E(z):=\sum_{k\in E}c_kT_{r_k}a_k(z).
\]
Then \(h_E\in f_\alpha^{p,\infty}\). We claim that
\begin{equation}
\label{eq:little-circle-finite-synthesis}
\|h_E\|_{\mathcal F_\alpha^{p,\infty}}
\lesssim
\sup_{k\in E}\frac{c_k}{1+k}.
\end{equation}
By the estimate used in Lemma~\ref{lem:dual-circle-shell-synthesis}, 
\[
M_p(T_{r_k}a_k,\rho)
\le
M_1(G_{r_k},\rho),
\qquad \rho>0,
\]
and hence
\[
\sup_{\rho\in[\ell,\ell+1)}
M_p(T_{r_k}a_k,\rho)e^{-\frac{\alpha}{2}\rho^2} \leq \sup_{\rho\in[\ell,\ell+1)}
M_1(G_{r_k},\rho)e^{-\frac{\alpha}{2}\rho^2} 
\lesssim
(1+k)^{-1}e^{-\frac{\alpha}{8}|\ell-k|^2}.
\]
Minkowski's inequality on the circle gives
\[
B_p(h_E;\ell) = \sup_{\rho\in[\ell,\ell+1)}
M_p(h_E,\rho)e^{-\frac{\alpha}{2}\rho^2}
\lesssim
\sum_{k\in E}
c_k(1+k)^{-1}e^{-\frac{\alpha}{8}|\ell-k|^2}  \lesssim
\sup_{k\in E}\frac{c_k}{1+k}.
\]
Taking the supremum over \(\ell\) and using the annular discretization in Proposition  \ref{prop:annular} proves
\eqref{eq:little-circle-finite-synthesis}.

\smallskip
\noindent
On the other hand, using \eqref{eq:little-circle-extract-Hk}, we obtain
\[
(1-\varepsilon)\sum_{k\in E}c_kH_k
\le
F(h_E).
\]
Hence, by \eqref{eq:little-circle-finite-synthesis},
\[
(1-\varepsilon)\sum_{k\in E}c_kH_k
\le
|F(h_E)|
\lesssim
\|F\|_{(f_\alpha^{p,\infty})^*}
\left(\sup_{k\in E}\frac{c_k}{1+k}\right).
\]
Letting \(\varepsilon\downarrow0\), we get
\begin{equation}
\label{eq:little-circle-finite-dual}
\sum_{k\in E}c_kH_k
\lesssim
\|F\|_{(f_\alpha^{p,\infty})^*}
\left(\sup_{k\in E}\frac{c_k}{1+k}\right).
\end{equation}
Taking \(c_k=1+k\) for \(k\in E\) and then taking the supremum over all
finite \(E\subset\mathbb N_0\), we obtain \eqref{eq:little-banach-shell-l1}.
By Proposition~\ref{prop:annular}, applied with exponents \((p^*,1)\), this
gives
\[
\|g_F\|_{\mathcal F_\alpha^{p^*,1}}
\lesssim
\|F\|_{(f_\alpha^{p,\infty})^*}.
\]
Combining the two cases proves \eqref{eq:little-rep-estimate}.

\smallskip
\noindent
\emph{Step 3: Polynomial completion and uniqueness.}
The polynomial-completion argument from
Lemma~\ref{lem:dual-polynomial-completion} applies verbatim. By Step~1 the functional
\(F_{g_F}:f\mapsto \langle f,g_F\rangle_\alpha\) is continuous on
\(f_\alpha^{p,\infty}\), and by \eqref{eq:little-gF-expansion} and Gaussian
orthogonality we have
\[
F_{g_F}(z^n)=F(z^n),\qquad n\ge0.
\]
Hence \(F_{g_F}=F\) on polynomials. Since polynomials are dense in
\(f_\alpha^{p,\infty}\), the identity extends to all
\(f\in f_\alpha^{p,\infty}\).

\smallskip
\noindent
For uniqueness, suppose that
\(
g_1,g_2\in
\mathcal F_{\alpha;\sigma(p,\infty)}^{p^*,1}
\)
represent the same functional on \(f_\alpha^{p,\infty}\). Then
\(
\langle P,g_1-g_2\rangle_\alpha=0
\)
for every polynomial \(P\). Write
\[
g_1(z)-g_2(z)=\sum_{m=0}^{\infty}c_m z^m.
\]
Testing against \(P(z)=z^n\) and using Gaussian orthogonality gives
\(
\overline{c_n}\frac{n!}{\alpha^n}=0\) for all \(n\ge0\).
Thus all Taylor coefficients vanish, and \(g_1=g_2\).

\smallskip
\noindent
\emph{Step 4: Norm equivalence.}
Estimate \eqref{eq:little-rep-estimate} gives one side of the norm
equivalence. The opposite estimate follows from the boundedness of the
pairing proved in Step~1:
\[
\|F\|_{(f_\alpha^{p,\infty})^*}
=
\|F_{g_F}\|_{(f_\alpha^{p,\infty})^*}
\lesssim
\|g_F\|_{\mathcal F_{\alpha;\sigma(p,\infty)}^{p^*,1}}.
\]
Hence
\[
\|F\|_{(f_\alpha^{p,\infty})^*}
\asymp
\|g_F\|_{\mathcal F_{\alpha;\sigma(p,\infty)}^{p^*,1}}.
\]
This completes the proof.
\end{proof}

\smallskip
\noindent
We transport the equal-parameter little-space duality to arbitrary
Gaussian pairing parameters by the same dilation argument as in
Subsection~\ref{subsec:dual-transport}, with
Proposition~\ref{prop:dual-little-equal} replacing
Theorem~\ref{thm:dual-equal}.

\begin{proof}[Proof of Theorem~\ref{thm:little-duality-intro}]
Throughout the proof, the pairing between \(f_\alpha^{p,\infty}\) and
\(\mathcal F_{\beta;\sigma(p,\infty)}^{p^*,1}\) is the transported pairing
\[
\langle f,g\rangle_{\alpha\to\gamma}^{\rm tr}
:=
\big\langle f,D_{\alpha/\gamma}g\big\rangle_\alpha .
\]
Since \(\gamma^2=\alpha\beta\), Lemma~\ref{lem:dual-dilation-weighted},
applied with \(\tau= \frac{\alpha}{\gamma}\) and with exponents \(p^*,1\), shows that
\(D_{\alpha/\gamma}\) is an isomorphism from
\(
\mathcal F_{\beta;\sigma(p,\infty)}^{p^*,1}
\)
onto
\(
\mathcal F_{\alpha;\sigma(p,\infty)}^{p^*,1}\).
Moreover,
\begin{equation}
\label{eq:little-transport-dilation-norm}
\|D_{\alpha/\gamma}g\|_{\mathcal F_{\alpha;\sigma(p,\infty)}^{p^*,1}}
\asymp
\|g\|_{\mathcal F_{\beta;\sigma(p,\infty)}^{p^*,1}}, \qquad g\in\mathcal F_{\beta;\sigma(p,\infty)}^{p^*,1}.
\end{equation}
Let
\(
g\in\mathcal F_{\beta;\sigma(p,\infty)}^{p^*,1}\)
and set
\(
h:=D_{\alpha/\gamma}g\).
Then
\(
h\in\mathcal F_{\alpha;\sigma(p,\infty)}^{p^*,1}\).
By Proposition~\ref{prop:dual-little-equal}, the map
\[
F_g(f)
:=
\langle f,g\rangle_{\alpha\to\gamma}^{\rm tr}
=
\langle f,h\rangle_\alpha,
\qquad f\in f_\alpha^{p,\infty},
\]
belongs to \((f_\alpha^{p,\infty})^*\), and
\(
\|F_g\|_{(f_\alpha^{p,\infty})^*}
\asymp
\|h\|_{\mathcal F_{\alpha;\sigma(p,\infty)}^{p^*,1}}\).
Together with \eqref{eq:little-transport-dilation-norm}, this gives
\[
\|F_g\|_{(f_\alpha^{p,\infty})^*}
\asymp
\|g\|_{\mathcal F_{\beta;\sigma(p,\infty)}^{p^*,1}}.
\]

\smallskip
\noindent
Conversely, let
\(
F\in(f_\alpha^{p,\infty})^*\).
By Proposition~\ref{prop:dual-little-equal}, there exists a unique
\(
h\in\mathcal F_{\alpha;\sigma(p,\infty)}^{p^*,1}
\)
such that
\[
F(f)=\langle f,h\rangle_\alpha,
\qquad f\in f_\alpha^{p,\infty},
\qquad \text{ and } \qquad
\|h\|_{\mathcal F_{\alpha;\sigma(p,\infty)}^{p^*,1}}
\asymp
\|F\|_{(f_\alpha^{p,\infty})^*}.
\]
Define
\(
g:=D_{\gamma/\alpha}h\).
Since \(D_{\gamma/\alpha}\) is the inverse of \(D_{\alpha/\gamma}\),
Lemma~\ref{lem:dual-dilation-weighted} gives
\[
g\in
\mathcal F_{\beta;\sigma(p,\infty)}^{p^*,1}
\qquad \text{ and } \qquad
\|g\|_{\mathcal F_{\beta;\sigma(p,\infty)}^{p^*,1}}
\asymp
\|h\|_{\mathcal F_{\alpha;\sigma(p,\infty)}^{p^*,1}}
\asymp
\|F\|_{(f_\alpha^{p,\infty})^*}.
\]
Moreover,
\(
D_{\alpha/\gamma}g=h\).
Therefore, for every \(f\in f_\alpha^{p,\infty}\),
\[
F(f)
=
\langle f,h\rangle_\alpha
=
\big\langle f,D_{\alpha/\gamma}g\big\rangle_\alpha
=
\langle f,g\rangle_{\alpha\to\gamma}^{\rm tr}.
\]
The uniqueness follows in the same way. If
\(
g_1,g_2\in
\mathcal F_{\beta;\sigma(p,\infty)}^{p^*,1}
\)
represent the same functional under the transported pairing, then
\[
\big\langle f,D_{\alpha/\gamma}(g_1-g_2)\big\rangle_\alpha=0,
\qquad f\in f_\alpha^{p,\infty}.
\]
By the uniqueness part of Proposition~\ref{prop:dual-little-equal},
\(
D_{\alpha/\gamma}(g_1-g_2)=0\).
Since \(D_{\alpha/\gamma}\) is injective on entire functions, \(g_1=g_2\).

\smallskip
\noindent
Thus the conjugate-linear map
\[
g\longmapsto F_g,
\qquad
F_g(f)=\langle f,g\rangle_{\alpha\to\gamma}^{\rm tr},
\]
identifies
\(
\mathcal F_{\beta;\sigma(p,\infty)}^{p^*,1}
\)
with
\(
(f_\alpha^{p,\infty})^*
\)
with equivalent norms. This proves Theorem~\ref{thm:little-duality-intro}.
\end{proof}


\section*{Acknowledgements}

We thank Shenzhao Hou for informing us of related unpublished work on embeddings between
mixed-norm Fock spaces.  X. F. is
supported by a grant from the National Science and Technology Council, Taiwan
(NSTC 114-2115-M-A49-003-MY3).

\bigskip
\textit{AI Statement.}
The authors used artificial-intelligence tools for language editing,
\LaTeX\ formatting, and limited assistance with local mathematical
reasoning.  The proof strategy and mathematical development are the
authors' own; all mathematical content was independently written and checked by the authors, who take full responsibility for it.


\bibliographystyle{amsplain}

\begin{thebibliography}{99}

\bibitem{AD16}
E. Abakumov and E. Doubtsov,
\emph{Volterra type operators on growth Fock spaces},
Arch. Math. (Basel) \textbf{107} (2016), no.~6, 543--553.

\bibitem{AS}
M. Abramowitz and I. A. Stegun (eds.),
\emph{Handbook of Mathematical Functions with Formulas, Graphs, and Mathematical Tables},
National Bureau of Standards Applied Mathematics Series, vol.~55,
U.S. Government Printing Office, Washington, DC, 1964.

\bibitem{A16}
I. Ar\'evalo,
\emph{Corrigendum to ``A characterization of the inclusions between mixed norm spaces''
[J. Math. Anal. Appl. 429 (2015), no.~2, 942--955]},
J. Math. Anal. Appl. \textbf{433} (2016), 1904--1905.

\bibitem{BP61}
A. Benedek and R. Panzone,
\emph{The spaces \(L^p\), with mixed norm},
Duke Math. J. \textbf{28} (1961), 301--324.

\bibitem{B95}
O. Blasco,
\emph{Multipliers on spaces of analytic functions},
Canad. J. Math. \textbf{47} (1995), 44--64.

\bibitem{BG24}
O. Blasco and A. Galbis,
\emph{Boundedness and compactness of Hausdorff operators on Fock spaces},
Trans. Amer. Math. Soc. \textbf{377} (2024), 5165--5196.

\bibitem{Buckley00}
S. M. Buckley,
\emph{Mixed norms and analytic function spaces},
Math. Proc. R. Ir. Acad. \textbf{100A} (2000), no.~1, 1--9.

\bibitem{Christensen}
O. Christensen,
\emph{An Introduction to Frames and Riesz Bases},
2nd ed., Birkh\"auser, Cham, 2016.

\bibitem{ConwayFA}
J. B. Conway,
\emph{A Course in Functional Analysis},
2nd ed., Graduate Texts in Mathematics, vol.~96,
Springer, New York, 1990.

\bibitem{CP16}
O. Constantin and J. A. Pel\'aez,
\emph{Integral operators, embedding theorems and a Littlewood--Paley formula on weighted Fock spaces},
J. Geom. Anal. \textbf{26} (2016), no.~2, 1109--1154.

\bibitem{DurenHp}
P. L. Duren,
\emph{Theory of $H^p$ Spaces},
Pure and Applied Mathematics, Vol. 38,
Academic Press, New York--London, 1970.

\bibitem{DS04}
P. Duren and A. Schuster,
\emph{Bergman Spaces},
Mathematical Surveys and Monographs, vol.~100,
American Mathematical Society, Providence, RI, 2004.

\bibitem{D19}
R. Durrett,
\emph{Probability: Theory and Examples},
Cambridge Series in Statistical and Probabilistic Mathematics, vol.~49,
Cambridge University Press, Cambridge, 2019.

\bibitem{FT23}
X. Fang and P. T. Tien,
\emph{Two problems on random analytic functions in Fock spaces},
Canad. J. Math. \textbf{75} (2023), no.~4, 1176--1198.


\bibitem{Flett72}
T. M. Flett,
\emph{Lipschitz spaces of functions on the circle and the disc},
J. Math. Anal. Appl. \textbf{39} (1972), 125--158.

\bibitem{Folland89}
G. B. Folland,
\emph{Harmonic Analysis in Phase Space},
Annals of Mathematics Studies, vol.~122,
Princeton University Press, Princeton, NJ, 1989.

\bibitem{FG05}
M.~Fornasier and K.~Gr\"ochenig,
\emph{Intrinsic localization of frames},
Constr. Approx. \textbf{22} (2005), no.~3, 395--415.

\bibitem{Gad88}
S. Gadbois,
\emph{Mixed-norm generalizations of Bergman spaces and duality},
Proc. Amer. Math. Soc. \textbf{104} (1988), no.~4, 1171--1180.

\bibitem{Gro04}
K.~Gr\"ochenig,
\emph{Localization of frames, Banach frames, and the invertibility of the frame
operator},
J. Fourier Anal. Appl. \textbf{10} (2004), no.~2, 105--132.

\bibitem{GL06}
K. Gr\"ochenig and M. Leinert,
\emph{Symmetry and inverse-closedness of matrix algebras and functional calculus for infinite matrices},
Trans. Amer. Math. Soc. \textbf{358} (2006), no.~6, 2695--2711.

\bibitem{GW92}
K.~Gr\"ochenig and D.~Walnut,
\emph{A Riesz basis for Bargmann--Fock space related to sampling and
interpolation},
Ark. Mat. \textbf{30} (1992), no.~2, 283--295.

\bibitem{Gue92}
D. Gu,
\emph{Bergman projections and duality in weighted mixed-norm spaces of analytic functions},
Michigan Math. J. \textbf{39} (1992), no.~1, 71--84.

\bibitem{HKZ00}
H. Hedenmalm, B. Korenblum, and K. Zhu,
\emph{Theory of Bergman Spaces},
Graduate Texts in Mathematics, vol.~199,
Springer, New York, 2000.

\bibitem{HL11}
Z. Hu and X. Lv,
\emph{Toeplitz operators from one Fock space to another},
Integral Equations Operator Theory \textbf{70} (2011), 541--559.

\bibitem{KoosisHp}
P. Koosis,
\emph{Introduction to $H_p$ Spaces},
2nd ed., Cambridge Tracts in Mathematics, Vol. 115,
Cambridge University Press, Cambridge, 1999.

\bibitem{IZ10}
J. Isralowitz and K. Zhu,
\emph{Toeplitz operators on the Fock space},
Integral Equations Operator Theory \textbf{66} (2010), 593--611.

\bibitem{Jev87}
M. Jevti\'c,
\emph{Bounded projections and duality in mixed-norm spaces of analytic functions},
Complex Variables Theory Appl. \textbf{8} (1987), no.~3--4, 293--301.

\bibitem{K22}
B. Karapetrovi\'c,
\emph{Randomization in generalized mixed norm spaces},
Complex Anal. Oper. Theory \textbf{16} (2022), Art.~25.

\bibitem{Kat04}
Y. Katznelson,
\emph{An Introduction to Harmonic Analysis},
3rd ed., Cambridge University Press, Cambridge, 2004.

\bibitem{L92}
M. Laczkovich,
\emph{Uniformly spread discrete sets in \(\mathbb R^d\)},
J. London Math. Soc. (2) \textbf{46} (1992), no.~1, 39--57.

\bibitem{Liu24}
Y. Liu,
\emph{Fock projections on mixed norm spaces},
Mediterr. J. Math. \textbf{21} (2024), Art.~171.

\bibitem{Lyub92}
Y.~Lyubarskii,
\emph{Frames in the Bargmann space of entire functions},
in \emph{Entire and Subharmonic Functions},
Adv. Soviet Math., vol.~11, Amer. Math. Soc., Providence, RI, 1992,
167--180.

\bibitem{Meng13}
T. Mengestie,
\emph{Volterra type and weighted composition operators on weighted Fock spaces},
Integral Equations Operator Theory \textbf{76} (2013), no.~1, 81--94.

\bibitem{MP26}
\'A. M. Moreno and J. \'A. Pel\'aez,
\emph{Duality of mixed norm spaces induced by radial one-sided doubling weights},
J. Geom. Anal. \textbf{36} (2026), no.~1, Art.~31.


\bibitem{Pav86}
M. Pavlovi\'c,
\emph{Mixed norm spaces of analytic and harmonic functions. I},
Publ. Inst. Math. (Beograd) (N.S.) \textbf{40}(54) (1986), 117--141.

\bibitem{Pav87}
M. Pavlovi\'c,
\emph{Mixed norm spaces of analytic and harmonic functions. II},
Publ. Inst. Math. (Beograd) (N.S.) \textbf{41}(55) (1987), 97--110.

\bibitem{PP08}
M. Pavlovi\'c and J. A. Pel\'aez,
\emph{An equivalence for weighted integrals of an analytic function and its derivative},
Math. Nachr. \textbf{281} (2008), no.~11, 1612--1623.

\bibitem{PRS2019}
J. A. Pel\'aez, J. R\"atty\"a, and K. Sierra,
\emph{Atomic decomposition and Carleson measures for weighted mixed norm spaces},
J. Geom. Anal. \textbf{29} (2019), 2055--2084.

\bibitem{Seip92}
K.~Seip,
\emph{Density theorems for sampling and interpolation in the Bargmann--Fock
space. I},
J. Reine Angew. Math. \textbf{429} (1992), 91--106.

\bibitem{SW92}
K.~Seip and R.~Wallst\'en,
\emph{Density theorems for sampling and interpolation in the Bargmann--Fock
space. II},
J. Reine Angew. Math. \textbf{429} (1992), 107--114.



\bibitem{Rud87}
W. Rudin,
\emph{Real and Complex Analysis},
3rd ed., McGraw--Hill, New York, 1987.

\bibitem{Zhu}
K. Zhu,
\emph{Analysis on Fock Spaces},
Graduate Texts in Mathematics, vol.~263,
Springer, New York, 2012.

\end{thebibliography}

\end{document}